\documentclass[leqno,12pt]{book}
\usepackage{amsmath,amsthm,amsfonts,amssymb,mathtools}
\usepackage{mathrsfs}
\allowdisplaybreaks

\makeatletter
\def\@seccntformat#1{\csname the#1\endcsname.~}
\def\section{\@startsection{section}{1}{\z@}%
              {-3.5ex \@plus -1ex \@minus -.2ex}%
              {2.3ex \@plus.2ex}%
              {\large\bf}}
\def\subsection{\@startsection{subsection}{1}{\z@}%
              {-3.5ex \@plus -1ex \@minus -.2ex}%
              {2.3ex \@plus.2ex}%
              {\large\bf}}

\newcommand\chappreface{
    \chapter*{Preface
        \@mkboth{%
           \MakeUppercase{Preface}}{\MakeUppercase{Preface}}}%
    }

\newcommand\chapintro{
    \chapter*{Introduction
        \@mkboth{%
           \MakeUppercase{Introduction}}{\MakeUppercase{Introduction}}}%
    }

\makeatother

\numberwithin{equation}{section}

\theoremstyle{plain}
\swapnumbers
\newtheorem{theorem}{Theorem}[section]
\newtheorem{corollary}[theorem]{Corollary}
\newtheorem{proposition}[theorem]{Proposition}
\newtheorem{lemma}[theorem]{Lemma}
\theoremstyle{definition}

\newtheorem{example}[theorem]{Example}
\newtheorem{remark}[theorem]{Remark}

 \makeatletter
 \let\c@equation\c@theorem
 \let\cl@equation\cl@theorem
 
 \makeatother

\newcommand{\ltwo}{\mathscr{L}_2}
\newcommand{\stwo}{\mathscr{S}_2}
\newcommand{\htwo}{{\mathcal{H}^{\otimes 2}}}
\newcommand{\D}{{\mathrm{D}}}
\newcommand{\q}{\mathfrak{q}}
\newcommand{\one}{\boldsymbol{1}}
\newcommand{\dettwo}{{{\det}_2}}
\newcommand{\op}{{\text{\rm op}}}
\newcommand{\tr}{{\text{\rm tr}}}
\newcommand{\f}{{\boldsymbol{\mathfrak{f}}}}
\newcommand{\kyosu}{{\mathrm{i}}}
\newcommand{\p}{\mathfrak{p}}

\newcommand{\bR}{{\boldsymbol{R}}}
\newcommand{\bS}{{\boldsymbol{S}}}
\newcommand{\bU}{{\boldsymbol{U}}}
\newcommand{\bV}{{\boldsymbol{V}}}

\newcommand{\e}{{\mathfrak{e}}}
\newcommand{\sh}{{\widehat{\mathfrak{sh}}}}
\newcommand{\ch}{{\mathfrak{ch}}}
\newcommand{\tnh}{{\widehat{\mathfrak{th}}}}
\newcommand{\snh}{{\mathfrak{sh}}}

\newcommand{\eqltwo}{\overset{(\ltwo)}{=}}
\newcommand{\as}{{\mathrm{AS}}}
\newcommand{\A}{{\boldsymbol{A}}}
\newcommand{\sa}{{\mathfrak{S}}}
\newcommand{\ba}{{\boldsymbol{a}}}
\newcommand{\I}{{\mathscr{I}}}
\newcommand{\J}{{\mathbf{J}}}

\begin{document}

\title{\bf
\huge
Quadratic Wiener functionals
\\[10pt]
\Large
---transformations and quadratic forms---
}
\author{\large
Setsuo TANIGUCHI
\thanks{email: {\tt se2otngc@kyudai.jp}}}
\date{{\tt February 20, 2026}}

\maketitle

\pagenumbering{roman}

\tableofcontents

\chappreface
The integral $\int_{\mathcal{W}} f e^{\q} d\mu$ on the Wiener space
$(\mathcal{W},\mu)$ (or the Laplace transformation  
$\mathbb{R}\ni\lambda\mapsto 
  \int_{\mathcal{W}} f e^{\lambda \q} d\mu$)
appears in many situations.
One may remember the Feynman-Kac formula for Schr\"odinger operator,
the Girsanov-Maruyama transformation, the pricing formula with
risk-neutral probability measure (Mathematical finance) and so on.
Evaluating the integral is one of fundamental problems in stochastic
analysis. 
Furthermore, continuing the evaluation holomorphically, we obtain 
the Fourier transformation  
$\mathbb{R}\ni\lambda\mapsto
 \int_{\mathcal{W}} f e^{\kyosu\lambda \q} d\mu$, 
where $\kyosu^2=-1$.
When $f=1$, it is exactly the characteristic function of
$\mathfrak{q}$ (a basic object in Probability theory).

Historically speaking, the study of evaluation of the
integral with quadratic $\q$ (an element of the direct sum
of Wiener chaos of order less than or equal to two) goes
back to the pioneering works by R.H.~Cameron and
W.T.~Martin \cite{CM0,CM}, M.~Kac \cite{kac}, and 
P.~L\'evy \cite{levy} in the mid-1900s.  
Cameron and Martin dealt with weighted square norms of
sample path of the one-dimensional Wiener process.
Kac did the square norm of sample path of the
one-dimensional Wiener process, which corresponds to the
Schr\"odinger operator for harmonic oscillator.
L\'evy did the stochastic area surrounded by the
two-dimensional Wiener process, which relates to the
Schr\"odinger operator for constant magnetic field.
Their approaches are different from each other, and each one
corresponds to the famous achievements named after them.
Cameron and Martin applied linear transformations on Wiener
space; it is a generalization of the Cameron-Martin formula.
Kac used the Fourier series expansion in the $L^2$-space
over the time interval; the method connects to the Feynman-Kac
formula. 
L\'evy took advantage of the development of Wiener process
in $\mathcal{W}$, which is originally used by N.~Wiener to
construct the Wiener measure $\mu$ and generalized by
K.~It\^o and M.~Nisio \cite{ito-nisio}.
The evaluation of Laplace transformation by L\'evy is called 
L\'evy's stochastic area formula.
After them, there are a lot of works concerning quadratic
$\q$s.
The author also made several contributions
[13--15, 17, 23, 24, 26, 27, 34--47].

In this monograph, we take the advantage of transformations
on the Wiener space to make the systematic investigation on
the integral $\int_{\mathcal{W}} fe^{\q} d\mu$ with
quadratic $\q$ so to establish a uniform framework to handle
such an integral, and revisit the author's contributions by
applying the uniform framework.

\chapintro
Let $d\in \mathbb{N}$ and $T>0$.
Denote by $\mathcal{W}$ the space of $\mathbb{R}^d$-valued continuous
functions on $[0,T]$ vanishing at $0$, and by $\mu$ the Wiener 
measure on $\mathcal{W}$.
Let $\mathcal{C}_n$ for $n\in \mathbb{N}\cup\{0\}$ be the Wiener 
chaos of order $n$.
We call an element of $\mathcal{C}_2$ {\it a quadratic form}, and that
of the direct sum 
$\mathcal{C}_0\oplus\mathcal{C}_1\oplus\mathcal{C}_2$ 
{\it a quadratic Wiener functional}.
Recall that each quadratic form is represented as
\[
    \q_\eta
    \equiv\int_0^T\biggl\langle \int_0^t
      \eta(t,s)d\theta(s),d\theta(t)\biggr\rangle
    =\sum_{i,j=1}^d \int_0^T \biggl(
       \int_0^t \eta_j^i(t,s) d\theta^j(s)\biggr)
       d\theta^i(t),
\]
where 
$\eta=\bigl(\eta_j^i\bigr)_{1\le i,j\le d}$
is a square integrable function on $[0,T]^2$ with values in 
$\mathbb{R}^{d\times d}$($\equiv$the space of $d\times d$ real
matrices) satisfying that
$\eta_i^j(t,s)=\eta_j^i(s,t)$ for $1\le i,j\le d$ and 
$(t,s)\in[0,T]^2$
(see \cite{nualart} or Theorem~\ref{t.w.chaos} below).
Here and in what follows, 
\begin{enumerate}
\item
$\langle\cdot,\cdot\rangle$ is the Euclidean inner product,
\item
every element of $\mathbb{R}^d$ is thought of as a column
vector, and $\mathbb{R}^{d\times d}$ acts on $\mathbb{R}^d$
from left as
$Mx=\Bigl(\sum\limits_{j=1}^d M_j^ix^j\Bigr)_{1\le i\le d}$ for 
$M=(M_j^i)_{1\le i,j\le d}\in \mathbb{R}^{d\times d}$ and
$x=(x^i)_{1\le i\le d}\in \mathbb{R}^d$,
\item
$\{\theta(t)=(\theta^1(t),\dots,\theta^d(t))\}_{t\in[0,T]}$
is the coordinate process on $\mathcal{W}$, i.e.,
$\theta(t)(w)=w(t)$ for $t\in[0,T]$ and $w\in\mathcal{W}$,
\item
each $d\theta^i(t)$ stands for the It\^o integral with 
respect to $\{\theta^i(t)\}_{t\in[0,T]}$; and
\item
$d\theta(t)=(d\theta^1(t),\dots,d\theta^d(t))$.
\end{enumerate}
Similarly, every quadratic Wiener functional is represented
as 
$\q_\eta+\int_0^T \langle f(t),d\theta(t)\rangle+c$
with square integrable $f:[0,T]\to \mathbb{R}^d$ and
$c\in \mathbb{R}$ (see \cite{nualart}).

Let $\mathcal{H}$ be the Cameron-Martin subspace of $\mathcal{W}$:
the space of absolutely continuous $h\in \mathcal{W}$ possessing
square integrable derivative $h^\prime$ on $[0,T]$. 
The inner product 
$\langle\cdot,\cdot\rangle_{\mathcal{H}}$ in $\mathcal{H}$
is given by 
\[
    \langle h,g\rangle_{\mathcal{H}}
    =\int_0^T \langle h^\prime(t),g^\prime(t)
           \rangle dt
    \quad\text{for } h,g\in \mathcal{H}.
\]
Its norm is given as
$\|\cdot\|_{\mathcal{H}}
 =\sqrt{\langle\cdot,\cdot\rangle_{\mathcal{H}}}$.
Denote by $\D$ the $\mathcal{H}$-derivative in the sense of
Malliavin calculus and by $\D^*$ its adjoint operator
(see \cite[Chapter~5]{mt-cambridge}). 

In Chapter~\ref{chap.qwf}, we establish a series expansion
of quadratic forms as  
\[
   \q_\eta=\frac12\sum_{n,m=1}^\infty 
     \langle B_\eta h_n,h_m\rangle_{\mathcal{H}}
     \{(\D^*h_n)(\D^*h_m)-\delta_{nm}\}
\]
for any orthonormal basis $\{h_n\}_{n=1}^\infty$ of 
$\mathcal{H}$, where $B_\eta:\mathcal{H}\to \mathcal{H}$
is the Hilbert-Schmidt operator with kernel $\eta$, i.e.,
it is defined as
\[
    \langle B_\eta h,g\rangle_{\mathcal{H}}
    =\int_0^T \biggl\langle \int_0^T \eta(t,s)
     h^\prime(s)ds, g^\prime(t)\biggr\rangle dt
    \quad\text{for }h,g\in \mathcal{H}.
\]
In the series expansion, 
$\D^*h_n$ is defined by thinking of $h_n$ as a constant
$\mathcal{H}$-valued Wiener functional and applying $\D^*$,
and the convergence takes place in every $L^p$-spaces with respect to 
$\mu$ for $p\in(1,\infty)$.
This series expansion will be seen by applying Malliavin
calculus, especially, representing $\q_\eta$ as
$\frac12(\D^*)^2B_\eta$ and using the continuity of the
adjoint operator $\D^*$. 
As an application of the series expansion, we revisit
L\'evy's (\cite{levy}) and Kac's (\cite{kac}) developments
of quadratic Wiener functionals.  

Chapter~\ref{chap.transf} is devoted to the investigation on 
{\it the transformation of order one} 
$\iota+F_\kappa$:
the sum of the identity mapping 
$\iota:\mathcal{W} \to \mathcal{W}$
and the Wiener functional $F_\kappa:\mathcal{W}\to \mathcal{H}$ 
determined with square integrable 
$\kappa:[0,T]^2\to \mathbb{R}^{d\times d}$ as
\[
    \langle F_\kappa,h\rangle_{\mathcal{H}}
    =\int_0^T \biggl\langle 
         \int_0^T \kappa(t,s)d\theta(s),h^\prime(t)
         \biggr\rangle dt
       \quad\text{for }h\in \mathcal{H}.
\]
Define $\eta(\kappa):[0,T]^2\to \mathbb{R}^{d\times d}$ by
\[
     \eta(\kappa)(t,s)
     =-\biggl\{\kappa(t,s)+\kappa(s,t)^\dagger
       +\int_0^T \kappa(u,t)^\dagger\kappa(u,s)du
       \biggr\}
       \quad\text{for }(t,s)\in[0,T]^2,
\]
where $M^\dagger$ is the transpose of
$M\in \mathbb{R}^{d\times d}$.
By using the development of quadratic forms achieved in
Chapter~\ref{chap.qwf}, we shall see that the transformation 
$\iota+F_\kappa$ of order one gives rise to a quadratic form
$\q_{\eta(\kappa)}$ as an exponent of Radon-Nikodym density
of change of variables formulas; if 
$\sup\limits_{\|h\|_{\mathcal{H}}=1}
 \langle B_{\eta(\kappa)}h,h\rangle_{\mathcal{H}}<1$, 
then it holds that
\[
    |\dettwo(I+B_\kappa)| 
    \int_{\mathcal{W}} f(\iota+F_\kappa)
       e^{\q_{\eta(\kappa)}} d\mu
    =e^{\frac12 \|\kappa\|_2^2} \int_{\mathcal{W}} f d\mu
    \quad\text{for every }f\in C_b(\mathcal{W}),
\]
where $\dettwo$ stands for the regularized determinant and
$C_b(\mathcal{W})$ is the space of bounded and continuous
functions on $\mathcal{W}$ with values in $\mathbb{R}$.
See Theorem~\ref{t.transf}.
On the way, it will be also shown that, in order for $\q_\eta$ to be  
exponentially integrable, it is necessary and sufficient that
$\sup\limits_{\|h\|_{\mathcal{H}}=1}
 \langle B_\eta h,h\rangle_{\mathcal{H}}<1$.

If  $\kappa$ satisfies that 
$\sup\limits_{\|h\|_{\mathcal{H}}=1}
 \langle B_{\eta(\kappa)}h,h\rangle_{\mathcal{H}}<1$, 
then $I+B_\kappa$ has a continuous inverse
(see Remark~\ref{r.transf}).
In the same chapter, letting $\widehat{\kappa}$ be 
the kernel of the  Hilbert-Schmidt operator
$(I+B_\kappa)^{-1}-I$, we shall show that
$\iota+F_{\widehat{\kappa}}$ is the inverse transformation
of $\iota+F_\kappa$ and the change of variables formula
turns into  
\[
    |\dettwo(I+B_\kappa)| 
    \int_{\mathcal{W}} f e^{\q_{\eta(\kappa)}} d\mu
    =e^{\frac12 \|\kappa\|_2^2} \int_{\mathcal{W}} 
     f(\iota+F_{\widehat{\kappa}}) d\mu.
\]
See Theorem~\ref{t.inv.transf}.
At the end of the chapter, the above result is applied to the 
linear transformations studied by Cameron and Martin \cite{CM}. 
A comparison between our results and theirs is also discussed. 

In Chapter~\ref{chap.twoways}, we shall show the following
four conditions are equivalent.
\\
(i)~$e^{\q_\eta}\in L^1(\mu)$,
\\
(ii)~$\Lambda(B_\eta)<1$,
\\
(iii)~there is a $\tau\in\stwo$ satisfying that
$\dettwo(I+B_\tau)\ne0$ and 
\[
    \eta(t,s)=-2\tau(t,s)-\int_0^T \tau(t,u)\tau(u,s)du
    \quad\text{for a.e.}~(t,s)\in[0,T]^2,
\]
and
\\
(iv)~
there is a $\rho\in\stwo$ satisfying that
\[
    \eta(t,s)=\rho(t,s)
      -\int_{t\vee s}^T \rho(t,u)\rho(u,s)du
    \quad\text{for a.e.}~(t,s)\in[0,T]^2,
\]
where we have used the abbreviation ``a.e.'' for 
``almost every'' with respect to the two-dimensional Lebesgue 
measure, and $t\vee s$ for $\max\{t,s\}$.
With $\kappa_A(\rho)\in\ltwo$ defined as
\[
    \kappa_A(\rho)(t,s)=-\one_{[0,t)}(s)\rho(t,s)
    \quad\text{for }(t,s)\in[0,T]^2, 
\]
it will be proved  that if these conditions hold, then  
\[
    \eta(\tau)=\eta(\kappa_A(\rho))=\eta\]
and the following two types of change of variables formulas hold:
\begin{align*}
    \int_{\mathcal{W}} f e^{\q_\eta} d\mu
    & = \{\dettwo(I-B_\eta)\}^{-\frac12}
        \int_{\mathcal{W}} f(\iota+F_{\widehat{\tau}}) d\mu
    \\
    & =e^{\frac14\|\rho\|_2^2} \int_{\mathcal{W}}
         f(\iota+F_{\widehat{\kappa_A(\rho)}}) d\mu
      \quad\text{for every }f\in C_b(\mathcal{W}).
\end{align*}
See Theorem~\ref{t.tways}.
Furthermore, restricting the domain, we shall see two types of 
bijectivity of the mapping 
$\ltwo\ni\kappa\mapsto\eta(\kappa)\in\stwo$.
In the same chapter, it will be also shown that the Girsanov
theorem is derived from the transformation
$\iota+F_{\widehat{\kappa_A(\rho)}}$ of oder one.

In Chapter~\ref{chap.lin.adapted}, we investigate transformations
$\iota+F_{\kappa_A(\rho)}$ which are continuous linear operators of
$\mathcal{W}$ to itself.
To be precise, define $G_\phi:\mathcal{W}\to \mathcal{H}$ and 
$\p_\sigma:\mathcal{W}\to \mathbb{R}$ for continuous 
$\phi,\sigma:[0,T]\to \mathbb{R}^{d\times d}$ as
\[
    \langle G_\phi,h\rangle_{\mathcal{H}}
      =-\int_0^T\langle \phi(t)\theta(t),h^\prime(t)
           \rangle dt
      \quad\text{for }h\in \mathcal{H}
    \quad\text{and}\quad
    \p_\sigma=\int_0^T\langle\sigma(t)\theta(t),
        d\theta(t)\rangle.
\]
$G_\phi$ is equal to $F_{\kappa_A(\rho)}$ with
$\rho\in\stwo$ given as 
\[
    \rho(t,s)=\one_{[0,t)}(s)\phi(t)
    +\one_{(t,T]}(s)\phi(s)^\dagger
    \quad\text{for }(t,s)\in[0,T]^2.
\]
Applying the results in the preceding chapter, we shall show
the equivalence between the integrability of $e^{\p_\sigma}$
and the solvability in the classical sense of 
the Riccati ordinary differential equation (ODE in short)
\[
    \bR^\prime=-\bR^2-\sigma^\dagger \bR -\bR\sigma
        -\sigma^\dagger\sigma,
    \quad \bR(T)=0.
\]
When $\sigma$ is of $C^1$-class, the solvability is
furthermore equivalent to the non-singularity of the solution $\bS$ 
to the $\mathbb{R}^{d\times d}$-valued second order ODE
\[
    \bS^{\prime\prime}-2\sigma_A \bS^\prime 
    -\sigma^\prime\bS=0,\quad 
    \bS(T)=I_d,~\bS^\prime(T)=\sigma(T),
\]
where we have said that $\bS$ is non-singular if $\det\bS(t)\ne0$ for
any $t\in[0,T]$. 
Using these $\bR$ and $\bS$, we will achieve another explicit
expression of Laplace transformation 
$\int_{\mathcal{W}} fe^{\p_\sigma} d\mu$.
The expression will be applied to evaluate the Feynman-Kac
density function and the kernel on a two-step nilpotent Lie
group.
Furthermore, it will be taken advantage of to achieve 
stochastic representations of algebraic or analytic objects:  
(a)~Euler, Bernoulli, and Eulerian polynomials
and 
(b)~reflectionless potentials and soliton solutions to the
KdV equation.

In Chapter~\ref{chap.sqrt}, we consider applications of
$\iota+F_\tau$ for $\tau\in\stwo$ appearing in the 
condition~(ii) above.
We first investigate the Wiener functional
\[
    \int_0^T \biggl\langle x,\int_0^T \kappa(t,s)d\theta(s)
     \biggr\rangle^2 dt
    \quad\text{with }x\in \mathbb{R}^d
    \text{ and }\kappa\in\ltwo.
\]
When $d=1$, $T=1$, $x=1$, and 
$\kappa(t,s)=\one_{[0,t)}(s)$ for $(t,s)\in[0,1]^2$, this is equal to
the Wiener functional studied by Kac \cite{kac}.
Next we compute the logarithm of the characteristic function
of $\q_\eta$ and show that the distribution of $\q_\eta$ is
self-decomposable. 
Thirdly, the evaluation of Laplace transformations will
be extended to pinned measures obtained as pull-backs of
Dirac measures by non-degenerate Wiener functionals. 
Fourthly, we will also show how the evaluation varies
accordingly to the change of pinning.
The variation is described in terms of Pl\"ucker coordinates.  
Finally, several examples on explicit expressions of
Pl\"ucker coordinates will be given in the last section of
the chapter. 

Chapters~\ref{chap.anal.alg}, \ref{chap.without.MC},
and \ref{chap.lap.another} are appendices.
In Chapter~\ref{chap.anal.alg}, analytic or algebraic
assertions used in Chapters from \ref{chap.qwf} to \ref{chap.sqrt} 
are seen. 
Chapter \ref{chap.without.MC} is devoted to proofs of the
results in Chapters~\ref{chap.qwf} and \ref{chap.transf}
without using Malliavin calculus.
In Chapter~\ref{chap.lap.another}, another proof of the
change of variables formula \eqref{t.tways.3} with $h=0$
is given by taking advantage of the eigenvalue expansion of
the Hilbert-Schmidt operator associated with the quadratic
form without using transformations of order one.

A list of notation appearing frequently in this monograph is 
presented before the bibliography.

\chapter{Quadratic forms}
\label{chap.qwf}

\pagenumbering{arabic}
\setcounter{page}{1}

\begin{quote}{\small
A bijective correspondence between quadratic forms and
symmetric Hilbert-Schmidt operators on $\mathcal{H}$ is
shown with the help of Malliavin calculus.
It is used to develop quadratic forms in $L^p$-spaces.
Taking advantage of the developments,  L\'evy's and Kac's
developments in the mid-1900s are revisited.
}
\end{quote}

\section{Series expansion of quadratic forms}
\label{sec.development}
We continue to use the symbols $\mathcal{W}$, $\mu$, and 
$\mathcal{H}$ presented in Introduction:
$\mathcal{W}$ denotes the $d$-dimensional Wiener space over $[0,T]$,
$\mu$ does the Wiener measure on $\mathcal{W}$, 
and $\mathcal{H}$ does the Cameron-Martin subspace.  
For a real separable Hilbert space $E$, let $L^p(\mu;E)$ be 
the space of $p$th integrable $E$-valued functions on 
$\mathcal{W}$ with respect to $\mu$. 
We write simply as $L^p(\mu)$ for $L^p(\mu;\mathbb{R})$.

Denote by $\ltwo$ the space of 
$\mathbb{R}^{d\times d}$-valued square integrable functions 
on $[0,T]^2$ with respect to the two-dimensional Lebesgue
measure.
Put
\[
    \stwo=
    \Bigl\{\eta\in\ltwo \,\Big|\,
      \eta(t,s)^\dagger=\eta(s,t)
      \text{ for }(t,s)\in[0,T]^2\Bigr\},
\]
where $M^\dagger$ is the transpose of 
$M\in \mathbb{R}^{d\times d}$. 
For $\eta=\bigl(\eta_j^i\bigr)_{1\le i,j\le d}\in\stwo$,
$\q_\eta:\mathcal{W}\to \mathbb{R}$ was defined in
Introduction as 
\[
    \q_\eta=\int_0^T\biggl\langle \int_0^t
      \eta(t,s)d\theta(s),d\theta(t)\biggr\rangle
    =\sum_{i,j=1}^d \int_0^T \biggl(
       \int_0^t \eta_j^i(t,s) d\theta^j(s)\biggr)
       d\theta^i(t).
\]
If we set
\[
    \|\kappa\|_2
    =\biggl(\int_0^T \int_0^T |\kappa(t,s)|^2 ds dt
     \biggr)^{\frac12}
    \quad\text{for }\kappa\in\ltwo,
\]
then, by the isometry for It\^o integral, we have that
\[
    \int_{\mathcal{W}}\q_\eta^2 d\mu
    =\frac12 \|\eta\|_2^2.
\]

For $k\in \mathbb{N}\cup\{0\}$, $p\in(1,\infty)$, 
and a real separable Hilbert space $E$, denote by 
$\mathbb{D}^{k,p}(E)$ the space of $k$-times 
$\mathcal{H}$-differentiable $E$-valued Wiener functionals
in the sense of Malliavin calculus, whose
$\mathcal{H}$-derivatives of  all orders up to $k$ are $p$th
integrable with respect to $\mu$. 
Set 
\[
    \mathbb{D}^\infty(E)
    =\bigcap_{k\in \mathbb{N}\cup\{0\}}
      \bigcap_{p\in(1,\infty)} \mathbb{D}^{k,p}(E).
\]
We write simply $\mathbb{D}^{k,p}$ and 
$\mathbb{D}^\infty$ for $\mathbb{D}^{k,p}(\mathbb{R})$ 
and $\mathbb{D}^\infty(\mathbb{R})$, respectively.
The symbols $\D$ and $\D^*$ stands for the
$\mathcal{H}$-derivative and its adjoint operator,
respectively.
As for Malliavin calculus, see
\cite[Chapter\,5]{mt-cambridge}.
Let $\htwo$ be the space of Hilbert-Schmidt operators from
$\mathcal{H}$ to itself, and set
\[
    \mathcal{S}(\htwo)=\{B\in\htwo \mid B^*=B\}, 
\]
where $B^*$ is the adjoint operator of $B$.
For $\kappa\in\ltwo$, define $B_\kappa\in\htwo$ by 
\[
    \langle B_\kappa h,g\rangle_{\mathcal{H}}
    =\int_0^T \biggl\langle \int_0^T 
       \kappa(t,s)h^\prime(s)ds,g^\prime(t)\biggr\rangle
       dt
    \quad\text{for } h,g\in \mathcal{H}.
\]
Our first aim of this section is to characterize $\q_\eta$
by using $B_\eta$ and give its series expansion in
$L^p(\mu)$ for every $p\in(1,\infty)$.

\begin{theorem}\label{t.q.eta}
Let $\eta\in\stwo$.
\\
{\rm(i)}
It holds that
\[
    \q_\eta=\frac12(\D^*)^2B_\eta,
\]
where the right term is defined by regarding 
$B_\eta\in \mathcal{S}(\htwo)$ as a constant Wiener
functional in $\mathbb{D}^\infty(\htwo)$ and applying $\D^*$
twice. 
\\
{\rm(ii)}
For any orthonormal basis (ONB in short) 
$\{h_n\}_{n=1}^\infty$ of $\mathcal{H}$, it holds that
\begin{equation}\label{t.q.eta.1}
   \q_\eta=\frac12\sum_{n,m=1}^\infty 
     \langle B_\eta h_n,h_m\rangle_{\mathcal{H}}
     \{(\D^*h_n)(\D^*h_m)-\delta_{nm}\},
\end{equation}
where the series converges in $L^p(\mu)$ for every
$p\in(1,\infty)$ and $\D^*h_n$ is defined 
by thinking of $h_n$ as a constant Wiener functional in
$\mathbb{D}^\infty(\mathcal{H})$ and applying $\D^*$.
\end{theorem}

For the proof, we prepare a lemma, which will be used
frequently in this monograph.
Define $h\otimes g\in\htwo$ for $h,g\in \mathcal{H}$ by 
\[
    (h\otimes g)k=\langle h,k\rangle_{\mathcal{H}}\, g
    \quad\text{for }k\in \mathcal{H}.
\]
Recalling that the inner product 
$\langle\cdot,\cdot\rangle_{\htwo}$ in $\htwo$ is given by 
\[
    \langle B,B^\prime\rangle_{\htwo}
    =\sum_{n=1}^\infty \langle Bh_n,B^\prime h_n
       \rangle_{\mathcal{H}}
    \quad\text{for }B,B^\prime\in\htwo,
\]
where $\{h_n\}_{n=1}^\infty$ is an (and any) ONB of
$\mathcal{H}$, we see that
\[
    \langle B,h\otimes g\rangle_{\htwo}
    =\langle Bh,g\rangle_{\mathcal{H}}
    \quad\text{for }B\in\htwo
    \text{ and }h,g\in \mathcal{H}
\]
and 
\begin{equation}\label{eq.hs}
    \|B_\kappa\|_{\htwo}=\|\kappa\|_2
    \quad\text{for }\kappa\in\ltwo,
\end{equation}
where 
$\|\cdot\|_{\htwo}
 =\langle \cdot,\cdot\rangle_{\htwo}^{\frac12}$.

\begin{lemma}\label{l.D^3=0}
If $\D^3\Phi=0$ for $\Phi\in \mathbb{D}^\infty$, then
$\D^2\Phi$ is a constant, say $=B$, and it holds that
\[
   \Phi=\frac12(\D^*)^2B+\D^*h+c,
   \quad\text{where }
   h=\int_{\mathcal{W}} \D\Phi d\mu
   \text{ and }
   c=\int_{\mathcal{W}} \Phi d\mu.
\]
\end{lemma}

\begin{proof}
See \cite[Proposition\,5.7.4]{mt-cambridge}.
\end{proof}

\begin{proof}[Proof of Theorem~\ref{t.q.eta}]
(i)
Since $\eta\in\stwo$, as an easy exercise of Malliavin
calculus (\cite{mt-cambridge}), we see that  
\[
    \langle \D\q_\eta,g\rangle_{\mathcal{H}}
     = \int_0^T \biggl\langle \int_0^t 
          \eta(t,s)g^\prime(s)ds, d\theta(t)
          \biggr\rangle
     +\int_0^T \biggl\langle \int_0^t 
        \eta(t,s)d\theta(s), g^\prime(t)
        \biggr\rangle dt
\]
and
\begin{align*}
   \langle \D^2\q_\eta,h\otimes g\rangle_{\htwo}
   &  = \int_0^T \biggl\langle \int_0^t 
        \eta(t,s)g^\prime(s) ds, h^\prime(t)
        \biggr\rangle dt
      +\int_0^T \biggl\langle \int_0^t 
        \eta(t,s)h^\prime(s) ds, g^\prime(t)
        \biggr\rangle dt
    \\
    & =\int_0^T \biggl\langle \int_0^T
        \eta(t,s)h^\prime(s) ds, g^\prime(t)
        \biggr\rangle dt
      =\langle B_\eta,h\otimes g\rangle_{\htwo}
\end{align*}
for $h,g\in \mathcal{H}$.
Hence it holds that
\[
    \int_{\mathcal{W}}\D\q_\eta d\mu=0
    \quad\text{and}\quad
    \D^2\q_\eta=B_\eta.
\]
Noting in addition that 
\[
    \int_{\mathcal{W}}\q_\eta d\mu=0,
\]
by Lemma~\ref{l.D^3=0}, we have the desired identity.  

\noindent
(ii)
Let $\{h_n\}_{n=1}^\infty$ be an ONB of $\mathcal{H}$.
Develop $B_\eta$ in $\htwo$ as 
\[
    B_\eta=\sum_{n,m=1}^\infty 
       \langle B_\eta h_n,h_m\rangle_{\mathcal{H}} 
       h_n\otimes h_m.
\]
For $N\in \mathbb{N}$, set 
\[
    B_\eta^{(N)}=\sum_{n,m=1}^N
      \langle B_\eta h_n,h_m\rangle_{\mathcal{H}} 
       h_n\otimes h_m.
\]
Then $B_\eta^{(N)}$ converges to $B_\eta$ in $\htwo$.
Due to the continuity of $\D^*$, 
$(\D^*)^2B_\eta^{(N)}$ converges to $(\D^*)^2B_\eta$ in
$L^p(\mu)$ for any $p\in(1,\infty)$.
Remembering that 
\[
    (\D^*)^2 (h_n\otimes h_m)
     =(\D^*h_n)(\D^*h_m)-\delta_{nm}
    \quad\text{for }n,m\in \mathbb{N} 
\]
(see \cite[(5.7.2)]{mt-cambridge}), we have that  
\[
    (\D^*)^2B_\eta^{(N)}=\sum_{n,m=1}^N
     \langle B_\eta h_n,h_m\rangle_{\mathcal{H}} 
     \{(\D^*h_n)(\D^*h_m)-\delta_{nm}\}.
\]
By the assertion (i), letting $N\to\infty$, we obtain the
identity \eqref{t.q.eta.1} and the convergence of the
series in $L^p(\mu)$ for any $p\in(1,\infty)$. 
\end{proof}

Recall that the Wiener chaos $\mathcal{C}_2$ of order two 
is the closure in $L^2(\mu)$ of the subspace spanned by
linear combinations of $(\D^*h_n)(\D^*h_m)-\delta_{nm}$ with
$n,m\in \mathbb{N}$, where $\{h_n\}_{n=1}^\infty$ is an ONB
of $\mathcal{H}$. 
It is easily seen that $\mathcal{C}_2$ is independent of the
choice of $\{h_n\}_{n=1}^\infty$.
Although the following fact is widely known
(\cite{nualart}), we give a proof via
Theorem~\ref{t.q.eta}. 

\begin{theorem}\label{t.w.chaos}
$\mathcal{C}_2=\{\q_\eta \mid \eta\in\stwo\}$.
\end{theorem}

For the proof, we prepare a lemma, which will be used
repeatedly in this monograph.
We say that $\kappa$ and $\kappa^\prime$ are equivalent in
$\ltwo$ if $\|\kappa-\kappa^\prime\|_2=0$, and denote it as
$\kappa\eqltwo\kappa^\prime$.
For $x=(x^1,\dots,x^d),y=(y^1,\dots,y^d)\in \mathbb{R}$,
$x\otimes y$ denotes the matrix given by
\[
    x\otimes y=\bigl(y^ix^j\bigr)_{1\le i,j\le d}
    =\begin{pmatrix}
    y^1x^1 & \dots & y^1x^d \\
    \vdots & & \\
    y^dx^1 & \dots & y^d x^d
    \end{pmatrix}
    \in \mathbb{R}^{d\times d}.
\]

\begin{lemma}\label{l.b.kappa}
For $\kappa,\kappa^\prime\in\ltwo$, 
$B_\kappa=B_{\kappa^\prime}$ if and only if 
$\kappa\eqltwo\kappa^\prime$.
The mapping $\ltwo\ni\kappa\mapsto B_\kappa\in\htwo$ is 
bijective, where the injectivity means that 
$\kappa_1\eqltwo\kappa_2$ for
$\kappa_1,\kappa_2\in \mathcal{L}_2$ with
$B_{\kappa_1}=B_{\kappa_2}$.
Finally, if $\eta\in\stwo$, then
$B_\eta\in \mathcal{S}(\htwo)$, and the restricted mapping 
$\stwo\ni\eta\mapsto B_\eta\in \mathcal{S}(\htwo)$ is also
bijective.
\end{lemma}

By the lemma, each $B\in\htwo$ has a unique
$\kappa\in\ltwo$ up to equivalence in $\ltwo$ such that
$B=B_\kappa$.
We call this $\kappa$ the {\it kernel} for $B$.

\begin{proof}
The first assertion follows from the identity 
\eqref{eq.hs}.
The assertion also implies the injectivity of the mapping
$\kappa\mapsto B_\kappa$.
To see the surjectivity of the mapping, let $B\in\htwo$. 
Take an ONB $\{h_n\}_{n=1}^\infty$ of $\mathcal{H}$.
For $N\in \mathbb{N}$, define $\kappa_N\in\ltwo$ by 
\[
    \kappa_N(t,s)=\sum_{n,m=1}^N 
     \langle Bh_n,h_m \rangle_{\mathcal{H}} 
     [h_n^\prime(s)\otimes h_m^\prime(t)]
    \quad\text{for } (t,s)\in[0,T]^2.
\]
Since 
\[
    \|\kappa_N-\kappa_M\|_2^2
    =\sum_{M<n\vee m\le N}      
     \langle Bh_n,h_m \rangle_{\mathcal{H}}^2
    \quad\text{for $N,M\in \mathbb{N}$ with
    $M<N$}
\]
and 
\[
    \sum_{n,m=1}^\infty 
      \langle Bh_n,h_m\rangle_{\mathcal{H}}^2
    =\sum_{n=1}^\infty \|Bh_n\|_{\mathcal{H}}^2
     =\|B\|_{\htwo}^2<\infty, 
\]
there is a $\kappa\in \ltwo$ with 
$\|\kappa_N-\kappa\|_2\to 0$ as $N\to\infty$. 
Hence 
\[
    \|B_{\kappa_N}-B_\kappa\|_{\htwo}
     =\|\kappa_N-\kappa\|_2\to0
    \quad\text{as }N\to\infty.
\]
Furthermore, if we denote by $\pi_N$ the orthogonal
projection  of $\mathcal{H}$ onto the subspace spanned by
$h_1,\dots,h_N$, then $B_{\kappa_N}=\pi_N B\pi_N$,
which implies that
\[
    \|B_{\kappa_N}-B\|_{\htwo}\to0
    \quad\text{as }N\to\infty.
\]
Therefore $B=B_\kappa$.
Thus the surjectivity of the mapping follows.

For $\kappa\in\ltwo$, define 
$\kappa^*\in\ltwo$ by 
$\kappa^*(t,s)=\kappa(s,t)^\dagger$
for $(t,s)\in[0,T]^2$.
Since $B_\kappa^*=B_{\kappa^*}$, 
$B_\eta\in \mathcal{S}(\htwo)$ 
if $\eta\in\stwo$.
The restricted mapping $\eta\mapsto B_\eta$ inherits the
injectivity from the original one.
Let $B\in \mathcal{S}(\htwo)$.
By the surjectivity of the original mapping, there is an
$\kappa\in\ltwo$ with $B_\kappa=B$.
Then $B_{\kappa^*}=B_\kappa^*=B$.
Set $\eta=\frac12\{\kappa+\kappa^*\}$.
Then $\eta\in\stwo$ and $B_\eta=B$.
Thus the restricted mapping is surjective.
\end{proof}

\begin{proof}[Proof of Theorem~\ref{t.w.chaos}]
By Theorem~\ref{t.q.eta}, $\q_\eta\in \mathcal{C}_2$ for
every $\eta\in\stwo$.

To see the converse, let $\q\in \mathcal{C}_2$.
Since 
\begin{align*}
    & \int_{\mathcal{W}}
     \{(\D^*h_n)(\D^*h_m)-\delta_{nm}\}
     \{(\D^*h_p)(\D^*h_q)-\delta_{pq}\}d\mu
    \\
    & 
    =\begin{cases}
       2 & \text{if }n=m=p=q,
       \\
       1 & 
       \text{if $n\ne m$, $p\ne q$, and $\{n,m\}=\{p,q\}$},
       \\
       0 & \text{otherwise},
     \end{cases}
\end{align*}
there exist $a_{nm}\in \mathbb{R}$, $1\le m\le n<\infty$
such that $\sum\limits_{1\le m\le n<\infty} a_{nm}^2<\infty$
and 
\[
    \q=\sum_{1\le m\le n<\infty} a_{nm}
       \{(\D^*h_n)(\D^*h_m)-\delta_{nm}\},
\]
where the series converges in $L^2(\mu)$.

Define $b_{nm}\in \mathbb{R}$ for $n,m\in \mathbb{N}$ by  
\[
    b_{nm}
    =\begin{cases}
       \frac12 a_{nm} & \text{if }m<n, \\
       a_{nn}   & \text{if } m=n, \\
       \frac12 a_{mn} & \text{if }m>n.
    \end{cases}
\]
Then, the series for $\q$ is rewritten as
\[
    \q=\sum_{n,m=1}^\infty b_{nm}
       \{(\D^*h_n)(\D^*h_m)-\delta_{nm}\},
\]
where the series converges in $L^2(\mu)$.
Define $B\in \mathcal{S}(\htwo)$ by 
\[
    B=\sum_{n,m=1}^\infty b_{nm} h_n\otimes h_m.
\]
By the same approximation argument for $B$ as we have used
to show Theorem~\ref{t.q.eta}(ii), we see that
$\q=(\D^*)^2B$. 
Due to Lemma~\ref{l.b.kappa}, there is an $\eta\in\stwo$
with the property that $B_\eta=2B$.
By Theorem~\ref{t.q.eta}, we have that
$\q=\frac12(\D^*)^2B_\eta=\q_\eta$. 
\end{proof}

We shall give an example of quadratic Wiener functionals,
which looks slightly different from the definition of
quadratic forms. 
For $\kappa\in\ltwo$ and $x\in \mathbb{R}^d$, define the
Wiener functionals
$\mathfrak{h}(\kappa;x)$ and $\mathfrak{h}(\kappa)$ by 
\[
    \mathfrak{h}(\kappa;x)
      =\frac12\int_0^T\biggl\langle x,
        \int_0^T\kappa(t,s)d\theta(s)
       \biggr\rangle^2 dt
    \quad\text{and}\quad
    \mathfrak{h}(\kappa)
      =\frac12\int_0^T\biggl|\int_0^T
        \kappa(t,s)d\theta(s)\biggr|^2 dt.
\]
If $d=1$ and $\kappa(t,s)=\sqrt{2}\one_{[0,t)}(s)$ for 
$(t,s)\in[0,T]$, where $\one_S$ is the indicator function of
the set $S$, then
$\mathfrak{h}(\kappa;1)=\mathfrak{h}(\kappa)
 =\int_0^T\theta(t)^2 dt$, which relates to the
harmonic oscillator $-\frac12\frac{d^2}{dx^2}+x^2$,
one of the fundamental Schr\"odinger operators
(\cite[\S\S5.8.1]{mt-cambridge}).
On account of this, we call $\mathfrak{h}(\kappa;x)$ and
$\mathfrak{h}(\kappa)$ 
{\it quadratic Wiener functionals of harmonic-oscillator
 type}.
In the stochastic approach to the KdV equation
(\cite{ikeda-t,st-spa}), 
$\mathfrak{h}(\kappa;x)$ with $\kappa$ of the form 
$\kappa(t,s)
    =\text{\rm diag}[
        e^{(t-s)p_1},\dots, e^{(t-s)p_d}]$ 
for $(t,s)\in[0,T]^2$,
where $p_1,\dots,p_d\in \mathbb{R}$ and 
$\text{\rm diag}[a_1,\dots,a_d]$ is the $d$-dimensional
diagonal matrix whose $i$th entry is $a_i$,
plays a fundamental role.
See also Section~\ref{sec.kdv}. 

\begin{proposition}\label{p.h.osc}
Let $\kappa\in\ltwo$ and $x\in \mathbb{R}^d$.
Define $c(\kappa;x),c(\kappa)\in\stwo$ by
\[
    c(\kappa;x)(t,s)
    =\int_0^T \bigl[(\kappa(u,s)^\dagger x)
          \otimes (\kappa(u,t)^\dagger x)\bigr] du
    \quad\text{and}\quad
    c(\kappa)(t,s)
    =\int_0^T \kappa(u,t)^\dagger \kappa(u,s)du
\]
for $(t,s)\in[0,T]^2$.
Then it holds that
\[
   \mathfrak{h}(\kappa;x)
      =\q_{c(\kappa;x)}
       +\frac12\int_0^T\int_0^T 
           |\kappa(t,s)^\dagger x|^2 ds dt
\]
and
\begin{equation}\label{p.h.osc.1}
    \mathfrak{h}(\kappa)
      =\q_{c(\kappa)}+\frac12\|\kappa\|_2^2.
\end{equation}
\end{proposition}

\begin{proof}
It holds that
\[
   \langle\D\mathfrak{h}(\kappa;x),g
      \rangle_{\mathcal{H}}
   = \int_0^T \biggl\langle x,\int_0^T \kappa(t,s)
           d\theta(s) \biggr\rangle
         \biggl\langle x,\int_0^T \kappa(t,u)
         g^\prime(u) du \biggr\rangle dt
\]
and 
\begin{align*}
   \langle \D^2\mathfrak{h}(\kappa;x)),h\otimes g
      \rangle_\htwo
    &  = \int_0^T \biggl\langle x,\int_0^T \kappa(t,s)
         h^\prime(s)ds \biggr\rangle
        \biggl\langle x,\int_0^T \kappa(t,u)
         g^\prime(u)du \biggr\rangle dt
   \\
     & =\langle B_{c(\kappa;x)}, h\otimes g\rangle_\htwo
\end{align*}
for $h,g\in \mathcal{H}$.
Hence we have that
\[
    \int_{\mathcal{W}}\D \mathfrak{h}(\kappa;x) d\mu=0
    \quad\text{and}\quad
    \D^2 \mathfrak{h}(\kappa;x)=B_{c(\kappa;x)}.
\]
Rewriting 
$\bigl\langle x,\int_0^T\kappa(t,s)d\theta(s)
 \bigr\rangle$
as 
$\int_0^T\langle\kappa(t,s)^\dagger x,d\theta(s)\rangle$, 
we see that
\[
    \int_{\mathcal{W}} \mathfrak{h}(\kappa;x)d\mu
    =\frac12\int_0^T\int_0^T|\kappa(t,s)^\dagger x|^2 dsdt.
\]
By Lemma~\ref{l.D^3=0} and Theorem~\ref{t.q.eta}, we obtain
the identity for $\mathfrak{h}(\kappa;x)$. 

Let $e_1,\dots,e_d$ be an orthonormal basis of
$\mathbb{R}^d$.
Notice that    
\[
    \mathfrak{h}(\kappa)
    =\sum_{i=1}^d \mathfrak{h}(\kappa;e_i).
\]
Since 
\[
    \sum_{i=1}^d c(\kappa;e_i)=c(\kappa)
    ~~\text{and}~~
    \sum_{i=1}^d \int_0^T\int_0^T
      |\kappa(t,s)^\dagger e_i|^2 dsdt=\|\kappa\|_2^2,
\]
the identity for $\mathfrak{h}(\kappa)$ follows from those
for $\mathfrak{h}(\kappa;e_i)$s.
\end{proof}

Note that $B_{c(\kappa)}=B_\kappa^*B_\kappa$.
Hence $B_{c(\kappa)}$ is of trace class and non-negative
definite,  i.e., 
$\langle B_{c(\kappa)}h,h\rangle_{\mathcal{H}}\ge0$ for any
$h\in \mathcal{H}$. 
By \eqref{p.h.osc.1}, $\mathfrak{h}(\kappa)$ is equal up to
a constant to $\q_\eta$ whose $B_\eta$ is of trace class and
non-negative definite.  
The converse assertion holds.

\begin{proposition}\label{p.h.osc.inv}
Let $\eta\in\stwo$ and assume that $B_\eta$ is of trace
class and non-negative definite.
Then there exists a $\kappa\in\stwo$ such that
\[
    \q_\eta=\mathfrak{h}(\kappa)-\frac12\tr B_\eta.
\]
\end{proposition}

\begin{proof}
Due to the square root lemma (Lemma~\ref{l.sq.root}), there
exists a continuous and linear
$C:\mathcal{H}\to \mathcal{H}$ such that $C^*=C$ and
$C^2=B_\eta$.
Since $B_\eta$ is of trace class, we have that
\begin{equation}\label{p.h.osc.inv.21}
    \|C\|_{\htwo}^2
    =\sum_{n=1}^\infty \|Ch_n\|_{\mathcal{H}}^2
    =\sum_{n=1}^\infty \langle B_\eta h_n,
                h_n\rangle_{\mathcal{H}}
    =\tr B_\eta,
\end{equation}
where $\{h_n\}_{n=1}^\infty$ is an ONB of $\mathcal{H}$.
Hence $C\in \mathcal{S}(\htwo)$.
By Lemma~\ref{l.b.kappa}, there is a $\kappa\in\stwo$ with
$C=B_\kappa$.
It then follows from \eqref{p.h.osc.inv.21} that
\[
    \|\kappa\|_2^2=\|B_\kappa\|_{\htwo}^2
    =\|C\|_{\htwo}^2=\tr B_\eta.
\]
Since $c(\kappa)=B_\kappa^* B_\kappa=C^2=B_\eta$, by
\eqref{p.h.osc.1}, we obtain the desired identity.
\end{proof}

\section{Series expansions due to L\'evy and Kac}
\label{sec.levy.kac}
In this section, we apply Theorem~\ref{t.q.eta} so to
revisit series expansions of quadratic Wiener functionals
due to L\'evy (\cite{levy}) and Kac (\cite{kac}).

\begin{example}\label{e.levy}
Let $d=2$ and $T=2\pi$.
The stochastic area $\mathfrak{s}$ introduced by L\'evy is 
given as
\[
    \mathfrak{s}
     =\frac12\int_0^{2\pi} \langle J\theta(t),
    d\theta(t)\rangle,
    \quad\text{where }
    J=\begin{pmatrix} 0 & -1 \\ 1 & 0 \end{pmatrix}.
\]
Defining $\eta\in \stwo$ by
\[
    \eta(t,s)=\frac12\{\one_{[0,t)}(s)
     -\one_{(t,2\pi]}(s)\}J
    \quad\text{for }(t,s)\in[0,2\pi]^2,
\]
we have that 
\[
    \mathfrak{s}=\q_\eta.
\]
L\'evy (\cite{levy}) achieved the series expansion in
$L^2(\mu)$ that 
\begin{equation}\label{e.levy.1}
    \mathfrak{s}=\sum_{n=1}^\infty \frac1{n}
     \bigl\{\xi_n(\eta_n^\prime-\eta^\prime\sqrt{2})
       -\eta_n(\xi_n^\prime-\xi^\prime\sqrt{2})
     \bigr\},
\end{equation}
where 
$\{\xi_n,\xi_n^\prime,\xi^\prime,\eta_n,\eta_n^\prime,
   \eta^\prime \mid  n\in \mathbb{N}\}$
is a family of IID random variables obeying the normal
distribution $N(0,1)$. 
He reasoned as follows.
\begin{enumerate}
\item
Develop the coordinate process
$\{\theta(t)=(\theta^1(t),\theta^2(t))\}_{t\in[0,2\pi]}$ 
in $\mathcal{W}$ with respect to the uniform convergence
norm as 
\begin{align*}
    & \theta^1(t)=\frac{\xi^\prime t}{\sqrt{2\pi}}
      +\sum_{n=1}^\infty \frac1{n\sqrt{\pi}}\Bigl\{
       \xi_n\bigl(\cos(nt)-1\bigr)
       +\xi_n^\prime \sin(nt)\Bigr\},
    \\
    & \theta^2(t)=\frac{\eta^\prime t}{\sqrt{2\pi}}
      +\sum_{n=1}^\infty \frac1{n\sqrt{\pi}}\Bigl\{
       \eta_n\bigl(\cos(nt)-1\bigr)
       +\eta_n^\prime \sin(nt)\Bigr\}
    \quad\text{for } t\in[0,2\pi].
\end{align*}

\item
Differentiate $\theta^1(t)$ and $\theta^2(t)$ formally as
\begin{align*}
    & d\theta^1(t)=\biggl(
       \frac{\xi^\prime}{\sqrt{2\pi}}
       +\sum_{n=1}^\infty \frac1{\sqrt{\pi}}\Bigl\{
       -\xi_n \sin(nt)+\xi_n^\prime \cos(nt)\Bigr\}
       \biggr)dt,
    \\
    & d\theta^2(t)=\biggl(
       \frac{\eta^\prime t}{\sqrt{2\pi}}
       +\sum_{n=1}^\infty \frac1{\sqrt{\pi}}\Bigl\{
         -\eta_n\sin(nt)+\eta_n^\prime \cos(nt)\Bigr\}
       \biggr)dt.
\end{align*}

\item
Substitute these into the definition of $\mathfrak{s}$ and
evaluate integrals of products of trigonometric functions
over $[0,2\pi]$.
\end{enumerate}
He said that the formal differentiation is justified by the
stochastic calculus.

Applying Theorem~\ref{t.q.eta}, we show \eqref{e.levy.1}.
To do so, let $e_1,e_2$ be an ONB of $\mathbb{R}^2$ and
define $h_{h;i}\in \mathcal{H}$ for 
$n\in \mathbb{N}\cup\{0\}$ and $i=1,2$, by 
\[
    h_{0;i}(t)=\frac{t}{\sqrt{2\pi}}e_i,\quad
    h_{2n;i}(t)=\frac{\cos(nt)-1}{n\sqrt{\pi}}e_i,
    \quad\text{and}\quad
    h_{2n-1;i}(t)=\frac{\sin(nt)}{n\sqrt{\pi}}e_i
\]
for $t\in[0,2\pi]$ and $n\in \mathbb{N}$.
$\{h_{n;i} \mid  n\in \mathbb{N}\cup\{0\},i=1,2\}$ is an ONB of
$\mathcal{H}$.
Notice that 
\[
    (B_\eta h)^\prime(t)
    =J\biggl\{h(t)-\frac12 h(2\pi)\biggr\}
    \quad\text{for }t\in[0,2\pi].
\]
By a straightforward computation, we know that
all 
$\langle B_\eta h_{n;i},h_{m;i}\rangle_{\mathcal{H}}$ 
for $n,m\in \mathbb{N}\cup\{0\}$ and $i=1,2$ vanish, 
\begin{align*}
    & \langle B_\eta h_{0;1},h_{2n;2}\rangle_{\mathcal{H}}
      =-\langle B_\eta h_{2n;1},h_{0;2}\rangle_{\mathcal{H}}
      =\frac{\sqrt{2}}{n},
    \\
    & \langle B_\eta h_{2n;1},h_{2m-1;2}\rangle_{\mathcal{H}}
      =-\langle B_\eta h_{2n-1;1},h_{2m;2}\rangle_{\mathcal{H}}
      =\frac1{n}\delta_{nm}
     \quad\text{for }n,m\in \mathbb{N},
\end{align*} 
and other $\langle B_\eta h_{n;1},h_{m;2}\rangle_{\mathcal{H}}$s
also vanish.
On account of the self-adjointness of $B_\eta$, by \eqref{t.q.eta.1},
we obtain \eqref{e.levy.1} with 
\begin{align*}
    & \xi_n=\D^*h_{2n;1}, \quad
      \xi_n^\prime=\D^*h_{2n-1;1},\quad
      \xi^\prime=\D^*h_{0;1},
    \\
    & \eta_n=\D^*h_{2n;2},\quad
      \eta_n^\prime=\D^*h_{2n-1;2},\quad
      \eta^\prime=\D^*h_{0;2}
      \quad\text{for }n\in \mathbb{N}.
\end{align*}
\end{example}

\begin{example}\label{e.kac}
Let $d=1$ and $T=1$.
Define $\mathfrak{h}:\mathcal{W}\to \mathbb{R}$ by 
\[
    \mathfrak{h}=\int_0^1 \theta(t)^2 dt.
\]
If we set
\[
    \kappa(t,s)=\one_{[0,t)}(s)
    \quad\text{for }(t,s)\in[0,1]^2,
\]
then 
$\mathfrak{h}=2 \mathfrak{h}(\kappa)$.
Since $\|\kappa\|_2^2=\frac12$, by
Proposition~\ref{p.h.osc}, we have that
\begin{equation}\label{e.kac.1}
    \mathfrak{h}=2\q_{c(\kappa)}+\frac12.
\end{equation}
Kac (\cite{kac,kac-pisa}) achieved the series expansion of
$\mathfrak{h}$ in $L^2(\mu)$ such that 
\begin{equation}\label{e.kac.2}
    \mathfrak{h}=\sum_{n=0}^\infty 
      \frac1{(n+\frac12)^2\pi^2} H_n^2,
\end{equation}
where $\{H_n \mid n\in \mathbb{N}\cup\{0\}\}$ is a family of IID
random variables obeying the normal distribution $N(0,1)$.
He obtained the series expansion in the following way.
\begin{enumerate}
\item
Put 
\[
    \varphi_n(t)=\sqrt{2}
    \sin\biggl(\Bigl(n+\frac12\Bigr)\pi t\biggr)
    \quad\text{for }t\in[0,1]\text{ and }
    n\in \mathbb{N}\cup\{0\}.
\]
Then $\{\varphi_n\}_{n=0}^\infty$ is an ONB of
$L^2([0,1];\mathbb{R})$,
Here, for $a,b\in \mathbb{R}$ with $a<b$ and a real
separable Hilbert space $E$, $L^p([a,b];E)$ denotes the
$L^p$-space of  $E$-valued functions defined on $[a,b]$ with
respect to the one-dimensional Lebesgue measure. 

\item
Let 
\[
    L_n=\int_0^1\theta(t)\varphi_n(t)dt
    \quad\text{for }n\in \mathbb{N}\cup\{0\}.
\]
By the Plancherel theorem, it holds that
\[
    \mathfrak{h}=\sum_{n=0}^\infty L_n^2.
\]

\item
Define $C:L^2([0,1];\mathbb{R})\to L^2([0,1];\mathbb{R})$ by  
\[
    (Cf)(t)=\int_0^1 (t\wedge s)f(s)ds
    \quad\text{for }f\in L^2([0,1];\mathbb{R})
      \text{ and }t\in[0,1].
\]
Then it holds that
\[
    C\varphi_n=\frac1{(n+\frac12)^2\pi^2}\varphi_n
    \quad\text{for }n\in \mathbb{N}\cup\{0\}.
\]

\item
Observe that
\[
    \int_{\mathcal{W}} L_nL_m d\mu
    =\int_0^1 \varphi_n(t) (C\varphi_m)(t)dt
    =\frac1{(n+\frac12)^2\pi^2} \delta_{nm}
    \quad\text{for }n,m\in \mathbb{N}\cup\{0\}.
\]
Hence $\{L_n\mid n\in \mathbb{N}\cup\{0\}\}$ is a family of
independent random variables, each $L_n$ obeying the normal
distribution $N\bigl(0,\{(n+\frac12)^2\pi^2\}^{-1}\bigr)$.
If we set 
\[
    H_n=\Bigl(n+\frac12\Bigr)\pi L_n
    \quad\text{for }n\in \mathbb{N}\cup\{0\},
\]
then \eqref{e.kac.2} holds.
\end{enumerate}

We apply Theorem~\ref{t.q.eta} to show \eqref{e.kac.2}.
Notice that 
\[
    c(\kappa)(t,s)
    =\int_0^1 \kappa(u,t)\kappa(u,s) du
    =1-(t\vee s)
    \quad\text{for }(t,s)\in[0,1]^2.
\]
Hence it holds that
\[
    (B_{c(\kappa)}h)^\prime(t)=\int_t^1 h(s)ds
    \quad\text{for }h\in \mathcal{H}
    \text{ and }t\in[0,1].
\] 
This implies that $B_{c(\kappa)}h=0$ if and only if $h=0$.
Let $\lambda\ne0$ be an eigenvalue of $B_{c(\kappa)}$ and 
$h\in \mathcal{H},\ne0$ be an associated eigenvector. 
The equation $B_{c(\kappa)}h=\lambda h$ is equivalent to
the ODE
\[
    \lambda h^{\prime\prime}=-h,
    \quad h(0)=0,~~h^\prime(1)=0.
\]
Since $B_{c(\kappa)}=B_\kappa^* B_\kappa$, we have that
\[
    \lambda\langle h,h\rangle_{\mathcal{H}}
    =\langle B_{c(\kappa)}h,h\rangle_{\mathcal{H}}
    =\|B_\kappa h\|_{\mathcal{H}}^2.
\]
Hence $\lambda>0$.
Solving the ODE, we we see that
\[
    \lambda=\frac1{(n+\frac12)^2\pi^2}
    \quad\text{and}\quad
    h=c\varphi_n
    \quad\text{for some $n\in \mathbb{N}\cup\{0\}$ and 
    $c\in \mathbb{R},\ne0$.}
\]
If we put 
\[
    h_n=\frac1{(n+\frac12)\pi}\varphi_n
    \quad\text{for }n\in \mathbb{N}\cup\{0\},
\]
then $\{h_n\}_{n=0}^\infty$ is an ONB of $\mathcal{H}$ and
we have that
\[
    B_{c(\kappa)}=\sum_{n=0}^\infty
     \frac1{(n+\frac12)^2\pi^2} h_n\otimes h_n.
\]
Since
\[
    \sum_{n=0}^\infty \frac1{(2n+1)^2}=\frac{\pi^2}8,
\]
by \eqref{t.q.eta.1}, we obtain that
\[
    \q_{c(\kappa)}=\frac12\sum_{n=0}^\infty
    \frac1{(n+\frac12)^2\pi^2} (\D^*h_n)^2 
    -\frac14.
\]
Plugging this into \eqref{e.kac.1} we arrive at
\eqref{e.kac.2} with $H_n=\D^*h_n$ for 
$n\in \mathbb{N}\cup\{0\}$.
\end{example}

\chapter{From transformations to quadratic forms}
\label{chap.transf}

\begin{quote}{\small
It is shown that transformations of order one give rise to
exponentially integrable quadratic forms via change of
variables formulas on the Wiener space, and such
transformations have inverse transformations.  
As an application, linear transformations are investigated, and 
the resulting change of variables formula is compared with the one
achieved by Cameron and Martin in 1940s. 
}
\end{quote}

\section{Transformations of order one}
\label{sec.transf}

The Wiener functional
$F_\kappa\in \mathbb{D}^\infty(\mathcal{H})$ and the function
$\eta(\kappa)\in\stwo$ for $\kappa\in\ltwo$ were defined in
Introduction as 
\begin{equation}\label{eq:f.kappa}
    \langle F_\kappa,h\rangle_{\mathcal{H}}
    =\int_0^T \biggl\langle 
         \int_0^T \kappa(t,s)d\theta(s),h^\prime(t)
         \biggr\rangle dt
       \quad\text{for }h\in \mathcal{H}    
\end{equation}
and 
\begin{equation}\label{eq:eta.kappa}
   \eta(\kappa)(t,s)
      =-\biggl\{\kappa(t,s)+\kappa(s,t)^\dagger
       +\int_0^T \kappa(u,t)^\dagger\kappa(u,s)du
       \biggr\}
   \quad\text{for }(t,s)\in[0,T]^2. 
\end{equation}
There we call $\iota+F_\kappa$ a transformation of order
one. 
This is because each component of
$\int_0^T\kappa(t,s)d\theta(s)$ is in the Wiener chaos of
order one.
The aim of this section is to show that a transformation of
order one gives rise to an exponentially integrable
quadratic form.

Put  
\[
    \Lambda(B)
    =\sup_{\|h\|_{\mathcal{H}}=1}\langle Bh,h
         \rangle_{\mathcal{H}}
    \quad\text{for }B\in \mathcal{S}(\htwo).
\]
If we develop $B\in \mathcal{S}(\htwo)$ as 
$B=\sum\limits_{n=1}^\infty a_n h_n\otimes h_n$
with a sequence $\{a_n\}_{n=1}^\infty\subset \mathbb{R}$ of 
square summable real numbers and
an ONB $\{h_n\}_{n=1}^\infty$ of $\mathcal{H}$, 
then we have that 
\[
    \Lambda(B)=\sup_{n\in \mathbb{N}} a_n.
\]
Since $\lim\limits_{n\to\infty} a_n=0$, 
we know that 
\[
    \Lambda(B)\ge0.
\]

\begin{theorem}\label{t.transf}
Let $\kappa\in\ltwo$ and assume that
$\Lambda(B_{\eta(\kappa)})<1$.
Then it holds that
\begin{equation}\label{eq:transf}
    |\dettwo(I+B_\kappa)| 
    \int_{\mathcal{W}} f(\iota+F_\kappa)
       e^{\q_{\eta(\kappa)}} d\mu
    =e^{\frac12 \|\kappa\|_2^2} \int_{\mathcal{W}} f d\mu
    \quad\text{for every }f\in C_b(\mathcal{W}),
\end{equation}
where $\dettwo$ stands for the regularized determinant 
(\cite{GGK}), and $I$ denotes the identity 
mapping of $\mathcal{H}$.
\end{theorem}

\begin{remark}\label{r.transf}
(i) 
It holds that
\begin{equation}\label{r.transf.1}
    B_{\eta(\kappa)}
    =-(B_\kappa+B_\kappa^*+B_\kappa^*B_\kappa).
\end{equation}
Hence we have that
\[
    \inf_{\|h\|_{\mathcal{H}}=1} 
     \|(I+B_\kappa)h\|_{\mathcal{H}}^2
    =\inf_{\|h\|_{\mathcal{H}}=1} \bigl\{
       1-\langle B_{\eta(\kappa)}h,h
           \rangle_{\mathcal{H}} \bigr\}
    =1-\Lambda(B_{\eta(\kappa)}). 
\]
If $\Lambda(B_{\eta(\kappa)})<1$, then, by
Lemma~\ref{l.cont.inv}, $I+B_\kappa$ has a continuous
inverse, which implies that 
 $\dettwo(I+B_\kappa)\ne0$
(\cite[Theorem\,XII.1.1]{GGK}). 
Conversely, if $\dettwo(I+B_\kappa)\ne0$, then $I+B_\kappa$
has a continuous inverse.
By the inequality due to the boundedness of
$(I+B_\kappa)^{-1}$ that
\[
    \|h\|_{\mathcal{H}}^2
    \le \|(I+B_\kappa)^{-1}\|_\op^2
        \|(I+B_\kappa)h\|_{\mathcal{H}}^2
    \quad\text{for }h\in \mathcal{H},
\]
where $\|\cdot\|_\op$ stands for the operator norm, we see
that  
\[
    \Lambda(B_{\eta(\kappa)})\le 1-\|(I+B_\kappa)\|_\op^{-2}<1.
\]
Thus we obtain that
\begin{equation}\label{r.transf.2}
    \Lambda(B_{\eta(\kappa)})<1
    \quad\text{if and only if}\quad
    \dettwo(I+B_\kappa)\ne0.
\end{equation}

Thanks to this equivalence, $\Lambda(B_{\eta(\kappa)})<1$
if $\|B_\kappa\|_\op<1$.
It also follows from \eqref{r.transf.1} that if 
$\kappa(t,s)^\dagger=-\kappa(s,t)$ for $(t,s)\in[0,T]^2$,
that is, $B_\kappa=-B_\kappa^*$, then 
$\Lambda(B_{\eta(\kappa)})\le0<1$.
\\
(ii)
Take $b,c\in \mathbb{R}$ with $b^2+c^2>0$ and orthonormal
$h_1,h_2\in \mathcal{H}$. 
Define $\kappa_1,\kappa_2\in\ltwo$ by 
\begin{align*}
    & \kappa_1(t,s)
      =b\{h_1^\prime(s)\otimes h_1^\prime(t)
        +h_2^\prime(s)\otimes h_2^\prime(t)\}, 
    \\
    & \kappa_2(t,s)
      =c\{h_1^\prime(s)\otimes h_2^\prime(t)
       -h_2^\prime(s)\otimes h_1^\prime(t)\}
      \quad\text{for }(t,s)\in[0,T]^2.
\end{align*}
By a direct computation, we see that 
\begin{align*}
    & \|\kappa_1-\kappa_2\|_2^2=2(b^2+c^2)>0, 
    \\
    &  \eta(\kappa_1)(t,s)
       =-(2b+b^2)
        \{h_1^\prime(s)\otimes h_1^\prime(t)
          +h_2^\prime(s)\otimes h_2^\prime(t)\}, 
    \\
    & \eta(\kappa_2)(t,s)
      =-c^2\{h_1^\prime(s)\otimes h_1^\prime(t)
             +h_2^\prime(s)\otimes h_2^\prime(t)\}
    \quad\text{for }(t,s)\in[0,T]^2.
\end{align*}

Suppose that $1+c^2=(1+b)^2$.
Then $\eta(\kappa_1)=\eta(\kappa_2)$.
Furthermore, it holds that
\[
    \langle B_{\eta(\kappa_i)}h,h\rangle_{\mathcal{H}}
    =-c^2\{\langle h_1,h\rangle_{\mathcal{H}}^2
           +\langle h_2,h\rangle_{\mathcal{H}}^2\} 
    \quad\text{for }h\in \mathcal{H}
     \text{ and }i=1,2.
\]
Hence 
$\Lambda(B_{\eta(\kappa_i)})\le0<1$ for $i=1,2$.
Thus the mapping
$\ltwo\ni\kappa
 \mapsto\eta(\kappa)\in\stwo$
is not injective, even if the assumption in the theorem is
fulfilled.
\\
(iii) 
By Lemma~\ref{l.q.eta.int} below, 
$e^{\q_{\eta(\kappa)}}\in L^1(\mu)$ if and only if  
$\Lambda(B_{\eta(\kappa)})<1$.
\end{remark}

To prove the theorem, we prepare lemmas.

\begin{lemma}\label{l.D=0}
Let $E$ be a real separable Hilbert space and 
$\Phi\in \mathbb{D}^\infty(E)$.
If $\D\Phi=0$, then $\Phi=\int_{\mathcal{W}}\Phi d\mu$. 
\end{lemma}

\begin{proof}
By \cite[Proposition\,5.2.9]{mt-cambridge}, $\Phi$ is a
constant function.
Hence $\Phi=\int_{\mathcal{W}}\Phi d\mu$. 
\end{proof}

\begin{lemma}\label{l.f.kappa}
Let $\kappa\in\ltwo$.
Then it holds that
\[
    F_\kappa=\D^*B_\kappa.
\]
Furthermore, for any ONB $\{h_n\}_{n=1}^\infty$ of
$\mathcal{H}$, it holds that
\[
    F_\kappa=\sum_{n,m=1}^\infty 
     \langle B_\eta h_n,h_m\rangle_{\mathcal{H}}
     (\D^*h_n)h_m,
\]
where the series converges in
$L^p(\mu;\mathcal{H})$ for every $p\in(1,\infty)$. 
\end{lemma}

\begin{proof}
It follows from \eqref{eq:f.kappa} that
\begin{align*}
    \langle \D F_\kappa,h\otimes g
          \rangle_{\htwo}
    & =\langle \D(\langle F_\kappa,g\rangle_{\mathcal{H}}),
       h\rangle_{\mathcal{H}}
      =\int_0^T \biggl\langle 
         \int_0^T \kappa(t,s)h^\prime(s)ds,g^\prime(t)
         \biggr\rangle dt
    \\
    & =\langle B_\kappa h,g\rangle_{\mathcal{H}}
      =\langle B_\kappa, h\otimes g\rangle_{\htwo}
      \quad\text{for }h,g\in \mathcal{H}.
\end{align*}
Thus we see that
\begin{equation}\label{l.f.kappa.21}
    \D F_\kappa=B_\kappa.
\end{equation}

For $h,g\in \mathcal{H}$ and 
$\Phi\in \mathbb{D}^\infty(\mathbb{R})$, it holds that 
\begin{align*}
   & \int_{\mathcal{W}} \langle \D(\D^*B_\kappa),
            h\otimes g\rangle_{\htwo}
          \Phi d\mu 
     =\int_{\mathcal{W}} \langle B_\kappa,
           \D(\D^*(\Phi h))\otimes g
            \rangle_{\htwo} d\mu
   \\
   & =\int_{\mathcal{W}} \langle 
          B_\kappa(\D(\D^*(\Phi h))),g\rangle_{\mathcal{H}}
          d\mu
     =\int_{\mathcal{W}} \langle \Phi h,
          \D(\D^*(B_\kappa^*g))\rangle_{\mathcal{H}}
          d\mu
     =\int_{\mathcal{W}} \langle h,
          B_\kappa^*g \rangle_{\mathcal{H}}\Phi 
          d\mu
   \\
   & =\int_{\mathcal{W}} \langle B_\kappa,h\otimes g
          \rangle_{\htwo} \Phi d\mu,
\end{align*}
where to see the fourth equality we have used the identity 
(\cite[(5.1.9)]{mt-cambridge}) that
$D(D^*h)=h$ for $h\in \mathcal{H}$.
Hence 
\begin{equation}\label{l.f.kappa.22}
    \D(\D^*B_\kappa)=B_\kappa.
\end{equation}

By \eqref{l.f.kappa.21} and \eqref{l.f.kappa.22},
$\D(F_\kappa-\D^*B_\kappa)=0$.
By Lemma~\ref{l.D=0}, it holds that 
\[
    F_\kappa-\D^*B_\kappa
    =\int_{\mathcal{W}}(F_\kappa-\D^*B_\kappa)d\mu.
\]
Due to \eqref{eq:f.kappa}, we know that
\[
    \int_{\mathcal{W}}F_\kappa d\mu=0.
\]
If we think of $h\in \mathcal{H}$ as a constant Wiener
functional in $\mathbb{D}^\infty(\mathcal{H})$, then 
$\D h=0$.
Hence it holds that
\[
    \biggl\langle \int_{\mathcal{W}} (\D^* B_\kappa)d\mu,
      h\biggr\rangle_{\mathcal{H}} 
    =\int_{\mathcal{W}} \langle \D^* B_\kappa,
            h\rangle_{\mathcal{H}} d\mu
    =\int_{\mathcal{W}} \langle B_\kappa,
               \D h\rangle_{\htwo}
       d\mu 
    =0
    \quad\text{for }h\in \mathcal{H}.
\]
Thus we see that
\[
    \int_{\mathcal{W}}(\D^*B_\kappa)d\mu=0.
\]
Therefore we obtain that
\[
    F_\kappa-\D^*B_\kappa=0,
\]
that is, $F_\kappa=\D^*B_\kappa$.

To see the second identity, put
\[
    B_\kappa^{(N)}=\sum_{n,m=1}^N \langle B_\kappa h_n,
         h_m\rangle_{\mathcal{H}} h_n\otimes h_m
    \quad\text{for }N\in \mathbb{N}.
\]
Then $B_\kappa^{(N)}\to B_\kappa$ in 
$\htwo$ as $N\to\infty$.
By the continuity of $\D^*$, 
$\D^*B_\kappa^{(N)}\to \D^*B_\kappa$ in
$L^p(\mu;\mathcal{H})$ for every $p\in(1,\infty)$.
Since $(\D^*)(h_n\otimes h_m)=(\D^*h_n)h_m$
(\cite[(5.7.3)]{mt-cambridge}), we have that 
\[
    \D^*B_\kappa^{(N)}=\sum_{n,m=1}^N \langle B_\kappa h_n,
     h_m\rangle_{\mathcal{H}} (\D^*h_n)h_m.
\]
Thus we obtain the second identity and the convergence of
the series in any $L^p(\mu;\mathcal{H})$.
\end{proof}

\begin{remark}\label{r.D^2F=0}
By \eqref{l.f.kappa.21}, $\D^2F_\kappa=0$ for any
$\kappa\in\ltwo$.
Conversely, for $F\in \mathbb{D}^\infty(\mathcal{H})$ with 
$\D^2F=0$, there are $\kappa\in\ltwo$ and $h\in \mathcal{H}$
such that $F=F_\kappa+h$. 
In fact, by Lemmas~\ref{l.b.kappa} and \ref{l.D=0}, such an
$F$ has a $\kappa\in\ltwo$ with $\D F=B_\kappa$.
By \eqref{l.f.kappa.21}, 
$\D(F-F_\kappa)=B_\kappa-B_\kappa=0$.
Since $\int_{\mathcal{W}}F_\kappa d\mu=0$, due to 
Lemma~\ref{l.D=0} again, setting 
$h=\int_{\mathcal{W}}F d\mu$, we obtain that
$F-F_\kappa=h$.
\end{remark}

\begin{lemma}\label{l.q.eta.int}
For $\eta\in\stwo$, the following holds.
\\
{\rm(i)} 
If $\Lambda(B_\eta)<1$, then it holds that
\[
    \int_{\mathcal{W}} e^{\q_\eta} d\mu
    \le \exp\biggl(\frac12\biggl\{\frac12
          +\frac{\Lambda(B_\eta)}{
           3(1-\Lambda(B_\eta))^3}
         \biggr\}\|\eta\|_2^2 \biggr).
\]
{\rm(ii)} 
$e^{\q_\eta}\in L^1(\mu)$ if and only if
$\Lambda(B_\eta)<1$.
\end{lemma}

In the above inequality, remember that $\Lambda(B_\eta)\ge0$.

\begin{proof}
(i) 
Since $B_\eta\in \mathcal{S}(\htwo)$
(Lemma~\ref{l.b.kappa}), there is an ONB
$\{h_n\}_{n=1}^\infty$ of $\mathcal{H}$ such that 
\begin{equation}\label{eq.b.eta.develop}
    B_\eta=\sum_{n=1}^\infty a_n h_n\otimes h_n, 
\end{equation}
where $a_n\in \mathbb{N}$ for $n\in \mathbb{N}$ and  
$\sum\limits_{n=1}^\infty a_n^2<\infty$.
By Theorem~\ref{t.q.eta}, there is an increasing sequence
$\{N_m\}_{m=1}^\infty\subset \mathbb{N}$ such that
\[
    \frac12 \sum_{n=1}^{N_m}a_n\{(\D^*h_n)^2-1\}
    \to \q_\eta
    \text{ as $m\to\infty$\quad $\mu$-a.s.}
\]
Since $\sup\limits_{n\in \mathbb{N}}a_n=\Lambda(B_\eta)<1$ and 
$\{\D^*h_n\}_{n=1}^\infty$ is an IID sequence of
random variables obeying the normal distribution $N(0,1)$,
by Fatou's lemma, we have that 
\begin{align}
    \int_{\mathcal{W}} e^{\q_\eta} d\mu
    & \le \liminf_{m\to\infty} \int_{\mathcal{W}}
         \exp\biggl(\frac12\sum_{n=1}^{N_m} 
           a_n\{(\D^*h_n)^2-1\}\biggr) d\mu
 \label{l.q.eta.int.21}
    \\
    &   = \liminf_{m\to\infty} \prod_{n=1}^{N_m}
          \biggl(\int_{\mathbb{R}} 
            e^{\frac12 a_n\{x^2-1\}} 
            \frac1{\sqrt{2\pi}} e^{-\frac12 x^2} dx
          \biggr)
 \nonumber
    \\
    &   =\biggl\{\limsup_{m\to\infty}
          \prod_{n=1}^{N_m} (1-a_n)e^{a_n}
            \biggr\}^{-\frac12}.
 \nonumber
\end{align}

The Taylor expansion of $\log(1-a)$ about $a=0$ implies that 
\[
    \log(1-a)+a+\frac{a^2}2
    =-\int_0^a\int_0^b\int_0^c \frac2{(1-u)^3} du dc db
    \ge -\frac{0\vee a}{3\{1-(0\vee a)\}^3}\,a^2
\]
for $a<1$.
Since $a_n\le\Lambda(B_\eta)<1$ for 
$n\in \mathbb{N}$, we obtain that  
\begin{align*}
    \prod_{n=1}^{N_m} (1-a_n)e^{a_n}
    & =\exp\biggl(\sum_{n=1}^{N_m} \biggl\{
        \log(1-a_n)+a_n+\frac{a_n^2}2\biggr\}\biggr)
       \exp\biggl(-\frac12\sum_{n=1}^{N_m} a_n^2\biggr)
    \\
    & \ge \exp\biggl(
          -\frac{\Lambda(B_\eta)}{
            3(1-\Lambda(B_\eta))^3}
         \sum_{n=1}^{N_m} a_n^2\biggr)
       \exp\biggl(-\frac12\sum_{n=1}^{N_m} a_n^2\biggr).
\end{align*}
Thus we have that
\[
    \limsup_{m\to\infty} 
     \prod_{n=1}^{N_m} (1-a_n)e^{a_n}
     \ge \exp\biggl(-\biggl\{\frac12
       +\frac{\Lambda(B_\eta)}{
          3(1-\Lambda(B_\eta))^3}
        \biggr\} \|B_\eta\|_{\htwo}^2\biggr).
\]
Since $\|B_\eta\|_{\htwo}=\|\eta\|_2$,
plugging this lower estimation into \eqref{l.q.eta.int.21},
we obtain the desired upper estimation. 
\\
(ii) 
By (i), it suffices to show that $e^{\q_\eta}\notin
L^1(\mu)$ if $\Lambda(B_\eta)\ge1$.
We continue to use the above development of $B_\eta$ and 
suppose that 
$\sup\limits_{n\in \mathbb{N}} a_n=\Lambda(B_\eta)\ge1$.
Since $\lim\limits_{n\to\infty}a_n=0$, there is an 
$n_0\in \mathbb{N}$ such that $a_{n_0}\ge1$.
Define $\eta^\prime\in \stwo$ as
$B_{\eta^\prime}=\sum\limits_{n\ne n_0} a_n h_n\otimes h_n$.
By Theorem~\ref{t.q.eta}, we see that
\[
    \q_{\eta^\prime}
    =\frac12 \sum_{n\ne n_0} a_n\{(\D^*h_n)^2-1\}
    \quad\text{and}\quad
    \q_\eta
    =\frac12 a_{n_0}\{(\D^*h_{n_0})^2-1\}
     +\q_{\eta^\prime}.
\]
The first identity yields that $\D^*h_{n_0}$ and
$\q_{\eta^\prime}$ are independent.
Since $\D^*h_{n_0}$ obeys the normal distribution $N(0,1)$,
we have that
\[
    \int_{\mathcal{W}} e^{\q_\eta} d\mu
    =\biggl(\int_{\mathbb{R}}
       e^{\frac12 a_{n_0}\{x^2-1\}} \frac1{\sqrt{2\pi}}
       e^{-\frac12 x^2} dx\biggr)
      \int_{\mathcal{W}} e^{\q_{\eta^\prime}} 
       d\mu
    =\infty.
\]
Thus $e^{\q_\eta}\notin L^1(\mu)$.
\end{proof}

\begin{remark}\label{r.q.eta.int}
Let $\eta\in\stwo$ and assume that $\Lambda(B_\eta)<1$.
Take a $p\in(1,\infty)$ with $p\Lambda(B_\eta)<1$.
Since $\Lambda(B_{p\eta})=p\Lambda(B_\eta)<1$, by the above
lemma, $e^{p\q_\eta}\in L^1(\mu)$.
Thus the second assertion of the above lemma is sharpen as
\[
    \Lambda(B_\eta)<1
    \quad\text{if and only if}\quad
    e^{\q_\eta}\in 
    L^{1+}(\mu)\equiv \bigcup_{p\in(1,\infty)}L^p(\mu).
\]
\end{remark}

We identify $\mathcal{H}$ with its dual space
$\mathcal{H}^*$, and use the continuous and dense inclusions
that 
\[
    \mathcal{W}^*\subset \mathcal{H}^*=\mathcal{H} 
    \subset \mathcal{W}.
\]

\begin{proof}[Proof of Theorem~\ref{t.transf}]
Take an ONB $\{\ell_n\}_{n=1}^\infty$ of $\mathcal{H}$ such
that $\ell_n\in \mathcal{W}^*$ for $n\in \mathbb{N}$.
For $N\in \mathbb{N}$, denote by $\pi^{(N)}$ the orthogonal 
projection of $\mathcal{H}$ onto the subspace spanned by
$\ell_1,\dots,\ell_N$.
Define $\Pi^{(N)}, \kappa^{(N)}\in\ltwo$ by
\[
    \Pi^{(N)}(t,s) =\sum_{n=1}^N 
        \ell_n^\prime(s)\otimes \ell_n^\prime(t)
\]
and
\[
    \kappa^{(N)}(t,s) =\int_0^T \int_0^T
       \Pi^{(N)}(t,u) \kappa(u,v)\Pi^{(N)}(v,s)
       dudv
\]
for $(t,s)\in[0,T]^2$.
Then it holds that
\[
    B_{\kappa^{(N)}}=\pi^{(N)}B_\kappa \pi^{(N)}.
\]
By this and \eqref{r.transf.1}, we see that
\[
    \|B_{\kappa^{(N)}}-B_\kappa\|_{\htwo}\to0
    \quad\text{and}\quad
    \|B_{\eta(\kappa^{(N)})}-B_{\eta(\kappa)}
        \|_{\htwo}\to0
    \quad\text{as }N\to\infty.
\]
Due to Theorem~\ref{t.q.eta}, Lemma~\ref{l.f.kappa}, and the 
continuity of $\D^*$, these yield that
\[
    \q_{\eta(\kappa^{(N)})}\to \q_{\eta(\kappa)}
    ~~\text{in $L^p(\mu)$}
    \quad\text{and}\quad
    F_{\kappa^{(N)}}\to F_\kappa
   ~~\text{in $L^p(\mu;\mathcal{H})$}
   \quad\text{for any }p\in(1,\infty).
\]
Furthermore, since
\[
    |\Lambda(B)
     -\Lambda(B^\prime)|
    \le \|B-B^\prime\|_{\htwo}
    \quad\text{for }
    B,B^\prime\in \mathcal{S}(\htwo),
\]
there is an $N_0\in\mathbb{N}$ such that
\begin{equation}\label{t.transf.21}
    \sup_{N\ge N_0}
       \Lambda(B_{\eta(\kappa^{(N)})})
    <1.
\end{equation}
Take a $p\in(1,\infty)$ such that
\[
    \alpha_p\equiv  \sup_{N\ge N_0}
     \Lambda(
         B_{p \eta(\kappa^{(N)})})
     =p \sup_{N\ge N_0} 
       \Lambda(
          B_{\eta(\kappa^{(N)})})<1.
\]
By Lemma~\ref{l.q.eta.int}, we know that
\begin{align*}
    & \sup_{N\ge N_0} \int_{\mathcal{W}} 
        \exp(p\q_{\eta(\kappa^{(N)})})
        d\mu 
      =\sup_{N\ge N_0} \int_{\mathcal{W}} 
         \exp(\q_{p\eta(\kappa^{(N)})})
         d\mu 
    \\
    & \le \exp\biggl(\frac12\biggl\{\frac12
      +\frac{0\vee \alpha_p}{
         3\{1-(0\vee \alpha_p)\}^3}
      \biggr\} p^2
      \sup_{N\ge N_0}\|\eta(\kappa^{(N)})\|_2^2\biggr)
    <\infty.
\end{align*}
Hence the family 
$\{\exp(\q_{\eta(\kappa^{(N)})}) \mid N\ge N_0\}$ 
is uniform integrable. 
Thus, on account of the continuity of the mapping 
$\htwo\ni A\mapsto\dettwo(I+A)$ 
(\cite[Theorem~XI.2.2]{GGK}),
\eqref{eq:transf} follows once we have shown the
identity 
\begin{equation}\label{t.transf.22}
   |\dettwo(I+B_{\kappa^{(N)}})| \int_{\mathcal{W}} 
        f(\iota+F_{\kappa^{(N)}})  
        \exp(\q_{\eta(\kappa^{(N)})})
        d\mu 
      =e^{\frac12 \|\kappa^{(N)}\|_2^2} \int_{\mathcal{W}} f d\mu
\end{equation}
for every $f\in C_b(\mathcal{W})$ and $N\ge N_0$.

Let $N\ge N_0$ and set 
\[
    A_{n,m}^{(N)}
    =\langle B_\kappa \ell_n,\ell_m\rangle_{\mathcal{H}}
    \quad\text{for }1\le n,m\le N
    \quad\text{and}\quad
    A^{(N)}=\bigl(A_{n,m}^{(N)}\bigr)_{1\le n,m\le N}
    \in \mathbb{R}^{N\times N}.
\]
Since $B_{\kappa^{(N)}}=\pi^{(N)}B_\kappa \pi^{(N)}$, it
holds that
\[
    B_{\kappa^{(N)}}=\sum_{n,m=1}^N
      A_{nm}^{(N)}\ell_n\otimes \ell_m.
\]
By \eqref{r.transf.1}, we have that 
\begin{equation}\label{t.transf.23}
    A^{(N)}+A^{(N)\dagger}+A^{(N)}A^{(N)\dagger}
     =\Bigl(-\langle B_{\eta(\kappa^{(N)})}\ell_n,
         \ell_m\rangle_{\mathcal{H}}\Bigr)_{1\le n,m\le N}.
\end{equation}

If we set $\varepsilon_N=1-\Lambda(B_{\kappa^{(N)}})$, then
by \eqref{t.transf.21}, $\varepsilon_N>0$ and 
\[
    |(I_N+A^{(N)\dagger})x|^2\ge \varepsilon_N|x|^2
    \quad\text{for every }x\in \mathbb{R}^N,
\]
where $I_N$ is the
$N$-dimensional identity matrix.
Thus the linear mapping 
$\mathbb{R}^N\ni x\mapsto 
 (I_N+A^{(N)\dagger})x\in \mathbb{R}^N$ is bijective.
Applying the change of variables formula for this mapping,
we obtain that
\begin{align}
    & |\det(I_N+A^{(N)\dagger})|
      \int_{\mathbb{R}^N} \varphi(x+A^{(N)\dagger}x)
      \frac1{\sqrt{2\pi}^N}
      e^{-\frac12 |x+A^{(N)\dagger}x|^2} dx
\label{t.transf.24}
    \\
    & =\int_{\mathbb{R}^N} \varphi(x) 
      \frac1{\sqrt{2\pi}^N} e^{-\frac12 |x|^2} dx
    \quad\text{for every }\varphi\in C_b(\mathbb{R}^N).
 \nonumber
\end{align}

By \eqref{t.transf.23} and \eqref{r.transf.1}, we see that 
\begin{align*}
    |x+A^{(N)\dagger}x|^2-|x|^2
     =&  -\sum_{n,m=1}^N 
        \langle B_{\eta(\kappa^{(N)})}\ell_n,
            \ell_m\rangle_{\mathcal{H}} 
        \{x^nx^m-\delta_{nm}\}
    \\
    & +2\sum_{n=1}^N \langle B_{\kappa^{(N)}}\ell_n,
       \ell_n\rangle_{\mathcal{H}}
      +\|B_{\kappa^{(N)}}\|_{\htwo}^2
    \quad\text{for }x=(x^1,\dots,x^N)\in \mathbb{R}^N.
\end{align*}
Recalling that $\dettwo(I_N+M)=\det(I_N+M)e^{-\tr M}$
for $M\in \mathbb{R}^{N\times N}$, we obtain  that
\begin{align}
    & \det(I_N+A^{(N)\dagger})e^{-\frac12 |x+A^{(N)\dagger}x|^2}
 \label{t.transf.25}
    \\
    & =\dettwo(I_N+A^{(N)}) 
     \exp\biggl(\frac12 \sum_{n,m=1}^N
       \langle B_{\eta(\kappa^{(N)})}\ell_n,
         \ell_m\rangle_{\mathcal{H}}
      \{x^nx^m-\delta_{nm}\}\biggr)
 \nonumber
    \\
    & \quad
      \times
      e^{-\frac12 |x|^2}e^{-\frac12 \|\kappa^{(N)}\|_2^2}
    \quad\text{for every }
    x=(x^1,\dots,x^N)\in \mathbb{R}^N.
 \nonumber
\end{align}

Since $B_{\kappa^{(N)}}=\pi^{(N)}B_\kappa \pi^{(N)}$ and 
$B_{\eta(\kappa^{(N)})}
 =\pi^{(N)}B_{\eta(\kappa^{(N)})} \pi^{(N)}$, 
by Theorem~\ref{t.q.eta} and Lemma~\ref{l.f.kappa}, we have
that 
\begin{align*}
    F_{\kappa^{(N)}} & =\sum_{n,m=1}^N 
       \langle B_\kappa \ell_n,\ell_m\rangle_{\mathcal{H}}
       (\D^*\ell_n)\ell_m,
    \\
    \q_{\eta(\kappa^{(N)})} & =\frac12 \sum_{n,m=1}^N 
       \langle B_{\eta(\kappa^{(N)})} \ell_n,
              \ell_m\rangle_{\mathcal{H}}
       \{(\D^*\ell_n)(\D^*\ell_m)-\delta_{nm}\}.
\end{align*}
Put 
$\boldsymbol{\ell}^{(N)}=(\ell_1,\dots,\ell_N):
  \mathcal{W}\to \mathbb{R}^N$ and
$\D^*\boldsymbol{\ell}^{(N)}
 =(\D^*\ell_1,\dots,\D^*\ell_N)$.
By the above expression of $F_{\kappa^{(N)}}$, we know that
\[
    \ell_m(F_{\kappa^{(N)}})
    =\langle \ell_m,F_{\kappa^{(N)}}\rangle_{\mathcal{H}}
    =\sum_{n=1}^N\langle  B_\kappa\ell_n,
    \ell_m\rangle_{\mathcal{H}}(\D^*\ell_n)
    \quad\text{for }1\le m\le N.
\]
Hence we have that
\[
    \boldsymbol{\ell}^{(N)}(F_{\kappa^{(N)}})
     =A^{(N)\dagger}\D^*\boldsymbol{\ell}^{(N)}.
\]
Since $\D^*\boldsymbol{\ell}^{(N)}$ obeys the normal
distribution $N(0,I_N)$ on $\mathbb{R}^N$,  
due to \eqref{t.transf.25}, 
the above expression of $\q_{\eta(\kappa^{(N)})}$, 
and the fact that $\D^*\ell_n=\ell_n$ for
$1\le n\le N$ (\cite[(5.1.5)]{mt-cambridge}),
the identity \eqref{t.transf.24} is rewritten as 
\[
    |\dettwo(I+B_{\kappa^{(N)}})| \int_{\mathcal{W}} 
     [\varphi\circ \boldsymbol{\ell}^{(N)}]
     (\iota+F_{\kappa^{(N)}}) 
     \exp(\q_{\eta(\kappa^{(N)})})
     d\mu
    =e^{\frac12 \|\kappa^{(N)}\|_2^2}
     \int_{\mathcal{W}} 
     [\varphi\circ \boldsymbol{\ell}^{(N)}]
     d\mu.
\]
Thus \eqref{t.transf.22} holds for 
$f=\varphi\circ \boldsymbol{\ell}^{(N)}$.
Since both $F_{\kappa^{(N)}}$ and
$\q_{\eta(\kappa^{(N)})}$ depend only on
$\D^*\boldsymbol{\ell}^{(N)}$, due to the splitting property
of the Wiener measure, this implies that \eqref{t.transf.22}
holds for every $f\in C_b(\mathcal{W})$. 
\end{proof}

\section{Inverse transformation of order one}
\label{sec:inv.transf}
The aim of this section is to show that every
transformation of order one $\iota+F_\kappa$ for 
$\kappa\in\ltwo$ with $\Lambda(B_{\eta(\kappa)})<1$ has an
inverse transformation.

If $\kappa\in \mathcal{L}_2$ and 
${\det}_2(I+B_\kappa)\ne0$, then $I+B_\kappa$ has a
continuous inverse
(\cite[Theorem\,XII.1.1]{GGK}). 
Due to the identity 
$(I+B_\kappa)^{-1}-I=-(I+B_\kappa)^{-1}B_\kappa$,
we see that 
$(I+B_\kappa)^{-1}-I\in \mathcal{H}^{\otimes 2}$
(\cite[Theorem\,IV.7.1]{GGK}).
Applying Lemma~\ref{l.b.kappa}, we obtain a kernel 
$\widehat{\kappa}\in \mathcal{L}_2$ for $(I+B_\kappa)^{-1}-I$;
\[
    B_{\widehat{\kappa}}=(I+B_\kappa)^{-1}-I.
\]
For $\kappa\in \mathcal{L}_2$ with
$\Lambda(B_{\eta(\kappa)})<1$, by \eqref{r.transf.2},
${\det}_2(I+B_\kappa)\ne0$ and hence $\widehat{\kappa}$ is
well-defined. 
Since
\[
    \int_{\mathcal{W}}
      \|F_\kappa-F_{\kappa^\prime}\|_{\mathcal{H}}^2 d\mu
     =\|\kappa-\kappa^\prime\|_2^2, 
\]
$F_\kappa=F_{\kappa^\prime}$ for 
$\kappa,\kappa^\prime \in \mathcal{L}_2$ with 
$\kappa\eqltwo\kappa^\prime$
By the injectivity described in Lemma~\ref{l.b.kappa}, 
$\widehat{\kappa}$ is determined uniquely by $\kappa$, and
so is $F_{\widehat{\kappa}}$ up to null sets.
The transformation $\iota+F_{\widehat{\kappa}}$ is an
inverse transformation of $\iota+F_\kappa$ as follows.

\begin{theorem}\label{t.inv.transf}
Let $\kappa\in \ltwo$ and assume that
$\Lambda(B_{\eta(\kappa)})<1$.
Then 
\begin{equation}\label{t.inv.transf.1}
    (\iota+F_\kappa)\circ(\iota+F_{\widehat{\kappa}})
    =(\iota+F_{\widehat{\kappa}})\circ(\iota+F_\kappa)
    =\iota
\end{equation}
and it holds that
\begin{equation}\label{eq:inv.transf}
    |\dettwo(I+B_\kappa)| \int_{\mathcal{W}} f
      e^{\q_{\eta(\kappa)}+\D^*h} d\mu
    =e^{\frac12 \|\kappa\|_2^2} \int_{\mathcal{W}} 
     f(\iota+F_{\widehat{\kappa}}) 
     e^{\D^*[(I+B_\kappa^*)^{-1}h]} d\mu
\end{equation}
for every $f\in C_b(\mathcal{W})$ and $h\in \mathcal{H}$.
In particular, 
\[
    |\dettwo(I+B_\kappa)| \int_{\mathcal{W}} 
    e^{\q_{\eta(\kappa)}+\D^*h} d\mu
    =e^{\frac12 \|\kappa\|_2^2} 
        e^{\frac12\|(I+B_\kappa^*)^{-1}h\|_{\mathcal{H}}^2} 
    \quad\text{for every }h\in \mathcal{H}.
\]
\end{theorem}

\begin{remark}\label{r.inv.transf}
(i) 
Precisely speaking, \eqref{t.inv.transf.1} means that
\[
    w+F_{\widehat{\kappa}}(w)+F_\kappa(w+F_{\widehat{\kappa}}(w))
   =w+F_\kappa(w)+F_{\widehat{\kappa}}(w+F_\kappa(w))
   =w
\]
for $\mu$-almost every $w\in \mathcal{W}$.
Since both $F_\kappa$ and $F_{\widehat{\kappa}}$ are
determined up to null sets, to determine terms 
$F_\kappa(w+F_{\widehat{\kappa}}(w))$
and $F_{\widehat{\kappa}}(w+F_\kappa(w))$,  we need nice
modifications of $F_\kappa$ and $F_{\widehat{\kappa}}$.
Such modifications will be obtained in Lemma~\ref{l.h.inv}
below. 
\\
(ii)
If $\dettwo(I+B_\kappa)\ne0$, then so does 
$I+B_\kappa^*$.
Hence $I+B_\kappa^*$ has a continuous inverse.
\end{remark}

For the proof, we prepare lemmas.
A measurable set $X\subset \mathcal{W}$ is said to be
$\mathcal{H}$-invariant if $X+h\equiv\{w+h\mid w\in X\}$
coincides with $X$ for every $h\in \mathcal{H}$.
As is easily seen, 
$X$ is $\mathcal{H}$-invariant if and only if 
$X+h\subset X$ for every $h\in \mathcal{H}$.

\begin{lemma}\label{l.h.inv}
Let $\kappa\in\ltwo$.
Then there is an $\mathcal{H}$-invariant set $X_\kappa$ such
that $\mu(X_\kappa)=1$ and 
\begin{equation}\label{l.h.inv.1}
    F_\kappa(w+h)=F_\kappa(w)+B_\kappa h
    \quad\text{for every }w\in X_\kappa
     \text{ and }h\in \mathcal{H}.
\end{equation}
\end{lemma}

This lemma asserts the existence of both
$\mathcal{H}$-invariant set and nice modification of
$F_\kappa$.

\begin{proof}
Let $\{h_n\}_{n=1}^\infty$ be an ONB of
$\mathcal{H}$.
There are $\mathcal{H}$-invariant sets $X^{(n)}$ for 
$n\in \mathbb{N}$ such that
\[
    \mu(X^{(n)})=1
    \quad\text{and}\quad
    (D^*h_n)(w+h)=(D^*h_n)(w)
     +\langle h_n,h\rangle_{\mathcal{H}}
\]
for any $n\in \mathbb{N}$, $w\in X^{(n)}$, and
$h\in\mathcal{H}$
(\cite[Lemma\,5.7.7]{mt-cambridge}). 
Put
\[
    X_\kappa=
     \biggl\{w\in \bigcap_{n=1}^\infty X^{(n)} \,\bigg|\,
     \lim_{N_1,N_2\to\infty} \biggl\|
      \sum_{N_1<n\vee m\le N_2} 
      \langle B_\kappa h_n, h_m\rangle_{\mathcal{H}}
      \bigl((\D^*h_n)(w)\bigr)h_m\biggr\|_{\mathcal{H}}=0
    \biggr\}.
\]

For $N\in \mathbb{N}$ and $h\in \mathcal{H}$, we know that
\[
    \sum_{n,m=1}^N 
     \langle B_\kappa h_n,h_m\rangle_{\mathcal{H}}
     \langle h,h_n\rangle_{\mathcal{H}} h_m
     =(\pi_N B_\kappa \pi_N) h, 
\]
where $\pi_N$ is the orthogonal projection of $\mathcal{H}$
onto the subspace spanned by $h_1,\dots,h_N$, and hence the
sum converges to $B_\kappa h$ in $\mathcal{H}$ as
$N\to\infty$.  
Thus $X_\kappa$ is $\mathcal{H}$-invariant.
By Lemma~\ref{l.f.kappa}, $\mu(X_\kappa)=1$.
On account of the same lemma, if we define the modification
of $F_\kappa$ by
\[
    F_\kappa(w)
    =\begin{cases}
       \displaystyle
       \lim_{N\to\infty} \sum_{n,m=1}^N
          \langle B_\kappa h_n, h_m\rangle_{\mathcal{H}}
          \bigl((\D^*h_n)(w)\bigr)h_m
          & \text{for }w\in X_\kappa,
       \\
       0 & \text{for }w\notin X_\kappa,
     \end{cases}
\]
then it satisfies \eqref{l.h.inv.1}.  
\end{proof}

\begin{lemma}\label{l.d*a}
Let $A\in\htwo$ and 
$g\in \mathcal{H}$.
Then it holds that
\[
    \langle \D^*A,g\rangle_{\mathcal{H}}=\D^*(A^*g).
\]
\end{lemma}

\begin{proof}
For any $\Phi\in \mathbb{D}^\infty(\mathbb{R})$, we have
that
\begin{align*}
    \int_{\mathcal{W}} 
       \langle \D^*A,g\rangle_{\mathcal{H}}
       \Phi d\mu
    & =\int_{\mathcal{W}} 
       \langle A,\D\Phi\otimes g
          \rangle_{\htwo} d\mu
      =\int_{\mathcal{W}} 
       \langle A(\D\Phi),g\rangle_{\mathcal{H}} d\mu
    \\
    & =\int_{\mathcal{W}} 
       \langle \D\Phi,A^*g\rangle_{\mathcal{H}} d\mu
      =\int_{\mathcal{W}} (\D^*(A^*g))\Phi d\mu.
\end{align*}
This implies the desired identity.
\end{proof}

\begin{lemma}\label{l.d*h.f.hat.kappa}
Let $\kappa\in\ltwo$ and assume that
$\Lambda(B_{\eta(\kappa)})<1$.
For each $h\in \mathcal{H}$, it holds that
\[
    (\D^*h)(\iota+F_{\widehat{\kappa}})
    =\D^*[(I+B_\kappa^*)^{-1}h].
\]
\end{lemma}

\begin{proof}
Let $h\in \mathcal{H}$ and $X$ be an $\mathcal{H}$-invariant
set relative to $\D^*h$ 
(\cite[Lemma\,5.7.7]{mt-cambridge}),
that is, $\mu(X)=1$ and  
\[
    (\D^*h)(w+g)
    =(\D^*h)(w)+\langle h,g\rangle_{\mathcal{H}}
    \quad\text{for every }w\in X
    \text{ and }g\in \mathcal{H}.
\]
By this identity, Lemmas~\ref{l.f.kappa} and \ref{l.d*a},
and the equalities
\[
    (I+B_\kappa^*)(I+B_{\widehat{\kappa}}^*)
    =\{(I+B_{\widehat{\kappa}})(I+B_\kappa)\}^*
    =I,
\]
we have that
\begin{align*}
    (\D^*h)(\iota+F_{\widehat{\kappa}})
    & =\D^*h+\langle h,
          F_{\widehat{\kappa}}\rangle_{\mathcal{H}}
      =\D^*h+\langle h,
         \D^*B_{\widehat{\kappa}}\rangle_{\mathcal{H}}
      =\D^*[(I+B_{\widehat{\kappa}}^*)h]
    \\
    &  =\D^*[(I+B_\kappa^*)^{-1}h].
\qedhere
\end{align*}
\end{proof}

\begin{proof}[Proof of Theorem~\ref{t.inv.transf}]
By Lemma~\ref{l.h.inv} for $\kappa$ and $\widehat{\kappa}$,
we have that 
\[
   \bigl((\iota+F_\kappa)\circ
        (\iota+F_{\widehat{\kappa}})\bigr)(w)
      =w+F_\kappa(w)+F_{\widehat{\kappa}}(w)
       +B_\kappa F_{\widehat{\kappa}}(w)
\]
and
\[
    \bigl((\iota+F_{\widehat{\kappa}})
        \circ(\iota+F_\kappa)\bigr)(w)
      =w+F_\kappa(w)+F_{\widehat{\kappa}}(w)
       +B_{\widehat{\kappa}}F_\kappa(w)
\]
for any $w\in X_\kappa\cap X_{\widehat{\kappa}}$.
By Lemmas~\ref{l.f.kappa} and \ref{l.d*a} and the identity
\[
    (I+B_\kappa)(I+B_{\widehat{\kappa}})=I
     =(I+B_{\widehat{\kappa}})(I+B_\kappa), 
\]
we see that 
\begin{align*}
    \langle
       F_\kappa+F_{\widehat{\kappa}}+B_\kappa F_{\widehat{\kappa}},
       g\rangle_{\mathcal{H}}
    & =\langle \D^* B_\kappa,g\rangle_{\mathcal{H}}
       +\langle \D^* B_{\widehat{\kappa}},
                 g\rangle_{\mathcal{H}}
       +\langle \D^* B_{\widehat{\kappa}},
                 B_\kappa^* g\rangle_{\mathcal{H}}
    \\
    & =\D^*\Bigl[B_\kappa^*g+B_{\widehat{\kappa}}^*g
        +B_{\widehat{\kappa}}^*(B_\kappa^*g)\Bigr]
    \\
    &=\D^*\bigl[\bigl(
        \{(I+B_\kappa)(I+B_{\widehat{\kappa}})\}^*
         -I\bigr)g\bigr]
    \\
    & =0
\end{align*}
and 
\[
   \langle
       F_\kappa+F_{\widehat{\kappa}}+B_{\widehat{\kappa}}F_\kappa,
       g\rangle_{\mathcal{H}}
      =\D^*\bigl[\bigl(
        \{(I+B_{\widehat{\kappa}})(I+B_\kappa)\}^*
         -I\bigr)g\bigr]
      =0
\]
for any $g\in \mathcal{H}$.
Thus \eqref{t.inv.transf.1} holds.

To show \eqref{eq:inv.transf}, notice that \eqref{eq:transf} 
continues to hold for any bounded and measurable 
$\phi:\mathcal{W}\to \mathbb{R}$, that is,
\[
    |\dettwo(I+B_\kappa)| \int_{\mathcal{W}}
     \phi(\iota+F_\kappa) 
     e^{\q_{\eta(\kappa)}} d\mu
    =e^{\frac12 \|\kappa\|_2^2} \int_{\mathcal{W}} \phi d\mu.
\]
Substituting 
$\phi_N=[(N\wedge e^{\D^*h}) f](\iota+F_{\widehat{\kappa}})$
with $N\in \mathbb{N}$ for $\phi$, 
and applying \eqref{t.inv.transf.1} and
Lemma~\ref{l.d*h.f.hat.kappa}, we obtain that 
\[
    |\dettwo(I+B_\kappa)| \int_{\mathcal{W}}
       (N\wedge e^{\D^*h}) f e^{\q_{\eta(\kappa)}} d\mu
    =e^{\frac12\|\kappa\|_2^2} \int_{\mathcal{W}}
       f(\iota+F_{\widehat{\kappa}}) 
       (N\wedge e^{\D^*[(I+B_\kappa^*)^{-1}h]}) d\mu.
\]
Recall that 
$e^{\D^*g}\in \bigcap\limits_{p\in(1,\infty)}L^p(\mu)$ for
any $g\in \mathcal{H}$.
Hence $e^{\D^*[(I+B_\kappa^*)^{-1}h]}\in L^1(\mu)$.
Moreover, in conjunction with Remark~\ref{r.q.eta.int}, we
know that
$e^{\D^*h}e^{\q_{\eta(\kappa)}} \in L^1(\mu)$.
Thanks to the dominated convergence theorem, letting
$N\to\infty$ in the above identity, we arrive at
\eqref{eq:inv.transf}. 

Since $\D^*g$ obeys the normal distribution with mean $0$
and variance $\|g\|_{\mathcal{H}}^2$ for $g\in \mathcal{H}$,
the last assertion follows from \eqref{eq:inv.transf}.
\end{proof}

Looking at the identity \eqref{eq:inv.transf} from the point
of view of $\widehat{\kappa}$ and exchanging $\kappa$ and
$\widehat{\kappa}$, we obtain the following.    

\begin{corollary}\label{c.inv.transf}
Let $\kappa\in \mathcal{L}_2$ and assume that
${\det}_2(I+B_\kappa)\ne0$.
Then $\Lambda(B_{\eta(\widehat{\kappa})})<1$ and the
distribution $\mu\circ(\iota+F_\kappa)^{-1}$ of 
$\iota+F_\kappa$ under $\mu$ has the Radon-Nikodym
derivative with respect to $\mu$ of the form  
$|\dettwo(I+B_{\widehat{\kappa}})|
     e^{-\frac12 \|\widehat{\kappa}\|_2^2} 
     e^{\q_{\eta(\widehat{\kappa})}}$,
that is,
\[
    \frac{d(\mu\circ(\iota+F_\kappa)^{-1})}{d\mu}
    =|\dettwo(I+B_{\widehat{\kappa}})|
     e^{-\frac12 \|\widehat{\kappa}\|_2^2} 
     e^{\q_{\eta(\widehat{\kappa})}}.
\]
\end{corollary}

\begin{proof}
Set $\kappa^\prime=\widehat{\kappa}$.
Since $I+B_{\kappa^\prime}=(I+B_\kappa)^{-1}$, 
$\dettwo(I+B_{\kappa^\prime})\ne0$.
By \eqref{r.transf.2}, 
$\Lambda(B_{\eta(\kappa^\prime)})<1$.
Applying Theorem~\ref{t.inv.transf} to $\kappa^\prime$, 
and noticing that $\widehat{\kappa^\prime}=\kappa$, 
we see that 
\[
    |\dettwo(I+B_{\widehat{\kappa}})|
     e^{-\frac12 \|\widehat{\kappa}\|_2^2}
    \int_{\mathcal{W}} f
      e^{\q_{\eta(\widehat{\kappa})}}
      d\mu
    =\int_{\mathcal{W}} 
       f(\iota+F_\kappa)d\mu
    \quad\text{for every }f\in C_b(\mathcal{W}).
\]
Thus the desired expression of the Radon-Nikodym derivative
is obtained.
\end{proof}

\section{Linear transformations}
\label{sec:lin.transf}
In this section, we apply Theorem~\ref{t.transf} to
linear transformations initially investigated by Cameron and 
Martin (\cite{CM}). 
For this purpose, define
$\mathbb{F}_\phi:\mathcal{W}\to \mathcal{H}$ for
$\phi\in\ltwo$ by
\[
    \langle \mathbb{F}_\phi, h\rangle_{\mathcal{H}}
    =\int_0^T \biggl\langle 
       \int_0^T \phi(t,s)\theta(s)ds, h^\prime(t)
       \biggr\rangle dt
    \quad\text{for } h\in \mathcal{H}.
\]
Observe that 
\[
    (\mathbb{F}_\phi(w))(t)
    =\int_0^t\biggl(\int_0^T \phi(u,s)w(s)ds
       \biggr)du
    \quad\text{for }w\in \mathcal{W}
     \text{ and }t\in[0,T].
\]
Thus $\iota+\mathbb{F}_\phi:\mathcal{W}\to \mathcal{W}$
is a continuous linear transformation.
Set $\kappa_\phi\in\ltwo$ so that 
\[
    \kappa_\phi(t,s)=\int_s^T \phi(t,u)du
    \quad\text{for } (t,s)\in[0,T]^2.
\]
Applying It\^o's formula, we see that
\begin{equation}\label{eq:kappa.phi}
    \int_0^T \kappa_\phi(t,s)d\theta(s)
    =\int_0^T \phi(t,s)\theta(s)ds
    \quad\text{for } t\in[0,T].
\end{equation}
Hence $\mathbb{F}_\phi=F_{\kappa_\phi}$.
The aim of this section is to establish the change of
variables formula for the transformation
$\iota+\mathbb{F}_\phi$.
The comparison between our result and Cameron and Martin's
will be given in Remark~\ref{r.transf.cm} below.

\begin{theorem}\label{t.transf.cm}
Let $\phi\in\ltwo$ and assume that
$\Lambda(B_{\eta(\kappa_\phi)})<1$. 
Then $B_{\kappa_\phi}$ is of trace class and it holds that
\begin{equation}\label{t.transf.cm.1}
    |\det(I+B_{\kappa_\phi})|\int_{\mathcal{W}} 
      f(\iota+\mathbb{F}_\phi) e^{\Psi_\phi}
      d\mu
    =\int_{\mathcal{W}} f d\mu
    \quad\text{for every }f\in C_b(\mathcal{W}),
\end{equation}
where 
\[
    \Psi_\phi 
    = -\int_0^T\biggl\langle \int_0^T 
         \phi(t,s)^\dagger d\theta(t),
       \theta(s)\biggr\rangle ds
      -\frac12 \int_0^T\biggl|\int_0^T
         \phi(t,s)\theta(s)ds \biggr|^2dt.
\]
\end{theorem}

For the proof, we prepare lemmas.

\begin{lemma}\label{l.d*f.kappa}
For each $\kappa\in\ltwo$, it holds that 
\[
    \q_{\eta(\kappa)}
     =-\D^*F_\kappa-\frac12 \|F_\kappa\|_{\mathcal{H}}^2
    +\frac12\|\kappa\|_2^2.
\]
\end{lemma}

\begin{proof}
Let $n\in \mathbb{N}$ and $h_1,\dots,h_n\in \mathcal{H}$.
Take an arbitrary infinitely differentiable 
$\varphi:\mathbb{R}^n\to \mathbb{R}$, whose derivatives of
all orders are bounded.
Put 
$\D^*\boldsymbol{h}=(\D^*h_1,\dots,\D^*h_n)$.
Since $\D\D^*h_i=h_i$ for $1\le i\le n$ 
(\cite[(5.1.9)]{mt-cambridge}), 
$\D^*\boldsymbol{h}\in \mathbb{D}^\infty(\mathbb{R}^n)$.
By Lemma~\ref{l.f.kappa}, the chain rule for $\D$ 
(\cite[Corollary\,5.3.2]{mt-cambridge}), and
the symmetry of the Hessian matrix 
$\Bigl(
  \dfrac{\partial^2\varphi}{\partial x^i \partial x^j}
 \Bigr)_{1\le i,j\le n}$, 
we see that
\begin{align*}
    & \int_{\mathcal{W}} (\D^*F_\kappa)
          \varphi(\D^*\boldsymbol{h}) d\mu
      =\int_{\mathcal{W}} \langle B_\kappa,
          \D^2\bigl(\varphi(\D^*\boldsymbol{h})\bigr) 
          \rangle_{\htwo} d\mu
    \\
    & =\sum_{i,j=1}^n \int_{\mathcal{W}} \langle B_\kappa,
         h_i\otimes h_j\rangle_{\htwo} 
         \frac{\partial^2\varphi}{
                \partial x^i \partial x^j}(
           \D^*\boldsymbol{h}) d\mu
    \\
    & =\frac12 \sum_{i,j=1}^n \int_{\mathcal{W}} 
        \langle (B_\kappa+B_\kappa^*),
         h_i\otimes h_j\rangle_{\htwo} 
         \frac{\partial^2\varphi}{
                \partial x^i \partial x^j}(
           \D^*\boldsymbol{h}) d\mu
    \\
    & =\frac12
         \int_{\mathcal{W}} 
           \bigl((\D^*)^2(B_\kappa+B_\kappa^*)\bigr)
           \varphi(\D^*\boldsymbol{h}) d\mu.
\end{align*}
Thus we obtain that 
\begin{equation}\label{l.d*f.kappa.21}
    \D^*F_\kappa=\frac12(\D^*)^2(B_\kappa+B_\kappa^*).
\end{equation}

Since $\D F_\kappa=B_\kappa$ (see \eqref{l.f.kappa.21}), we see that 
\[
    \biggl\langle \D\biggl(\frac12 
       \|F_\kappa\|_{\mathcal{H}}^2\biggr),g
      \biggr\rangle_{\mathcal{H}}
      =\langle F_\kappa,B_\kappa g\rangle_{\mathcal{H}}
\]
and
\[
    \biggl\langle \D^2\biggl(\frac12
         \|F_\kappa\|_{\mathcal{H}}^2\biggr),
        h\otimes g \biggr\rangle_{\htwo}
    =\langle B_\kappa h,B_\kappa g\rangle_{\mathcal{H}}
      =\langle B_\kappa^*B_\kappa, h\otimes g
       \rangle_{\htwo}
\]
for $h,g\in \mathcal{H}$.
Hence it holds that
\[
    \int_{\mathcal{W}} \D\biggl(
       \frac12\|F_\kappa\|_{\mathcal{H}}^2\biggr) d\mu=0
    \quad\text{and}\quad
    \D^2\biggl(\frac12\|F_\kappa\|_{\mathcal{H}}^2\biggr)
    =B_\kappa^*B_\kappa. 
\]
Furthermore, by the isometry for It\^o integral, we have
that 
\[
    \int_{\mathcal{W}} \frac12
        \|F_\kappa\|_{\mathcal{H}}^2 d\mu
    =\frac12\|\kappa\|_2^2.
\]
Then, by Lemma~\ref{l.D^3=0}, we obtain that  
\begin{equation}\label{l.d*f.kappa.22}
    \frac12 \|F_\kappa\|_{\mathcal{H}}^2
    =\frac12(\D^*)^2(B_\kappa^*B_\kappa)
     +\frac12 \|\kappa\|_2^2.
\end{equation}

Applying Theorem~\ref{t.q.eta} to $\q_{\eta(\kappa)}$
with the help of \eqref{r.transf.1}, we see that
\[
    \q_{\eta(\kappa)}
    =-\frac12(\D^*)^2\{B_\kappa+B_\kappa^*
     +B_\kappa^*B_\kappa\}.
\]
Plugging \eqref{l.d*f.kappa.21} and \eqref{l.d*f.kappa.22}
into this equality, we obtain the desired expression of
$\q_{\eta(\kappa)}$.
\end{proof}

\begin{remark}\label{r.(d*)2}
Let $B\in\htwo$ and $\kappa\in\ltwo$ be its kernel,
i.e., $B=B_\kappa$.
By \eqref{l.d*f.kappa.21}, we have that 
\[
    (\D^*)^2B
    =(\D^*)^2\Bigl(\frac12\{B_\kappa+B_\kappa^*\}\Bigr).
\]
Hence we have that 
\[
    \{(\D^*)^2B \mid B\in \htwo\}
    =\{(\D^*)^2B \mid B\in \mathcal{S}(\htwo)\}
    =\{\q_\eta \mid \eta\in\stwo\}.
\]
\end{remark}

\begin{lemma}\label{l.transf.cm}
Let $\phi\in\ltwo$ and assume that
$\Lambda(B_{\eta(\kappa_\phi)})<1$.
Then it holds that 
\begin{equation}\label{l.transf.cm.1}
    |\dettwo(I+B_{\kappa_\phi})|\int_{\mathcal{W}} 
      f(\iota+\mathbb{F}_\phi) 
      e^{\widetilde{\Psi}_\phi} d\mu
    =\int_{\mathcal{W}} f d\mu
    \quad\text{for every }f\in C_b(\mathcal{W}),
\end{equation}
where
\[
    \widetilde{\Psi}_\phi 
    = \Psi_\phi
    +\int_0^T\biggl(\int_0^s 
       \tr\phi(t,s) dt\biggr)ds.
\]
\end{lemma}

\begin{proof}
By Theorem~\ref{t.transf} and Lemma~\ref{l.d*f.kappa},
what to show is the identity
\begin{equation}\label{l.transf.cm.21}
    \D^*F_{\kappa_\phi}=\int_0^T\biggl\langle\int_0^T
       \phi(t,s)^\dagger d\theta(t),\theta(s)
       \biggr\rangle ds
    -\int_0^T\biggl(\int_0^s 
        \tr\phi(t,s) dt\biggr)ds.
\end{equation}
Take $\phi_m,\phi\in\ltwo$ for $m\in \mathbb{N}$ with 
$\|\phi_m-\phi\|_2\to0$ as $m\to\infty$.
Then 
\[
    \|B_{\kappa_{\phi_m}}-B_{\kappa_\phi}
  \|_{\htwo}
  =\|\kappa_{\phi_m}-\kappa_\phi\|_2\to0.
\]
By Lemma~\ref{l.f.kappa} and the continuity of $\D^*$, 
$F_{\kappa_{\phi_m}}\to F_{\kappa_\phi}$ in
$L^p(\mu;\mathcal{H})$ for any $p\in(1,\infty)$.
Furthermore, it is easily seen that the right hand side of
\eqref{l.transf.cm.21} for $\phi_m$ converges to that for
$\phi$ in probability. 
Thus, to show \eqref{l.transf.cm.21}, we may and will assume
that $\phi\in\ltwo$ is of the form 
\[
    \phi(t,s)=\sum_{n=0}^{N-1} 
       \one_{[T_{N;n},T_{N;n+1})}(t)
       \phi(T_{N;n},s)
    \quad\text{for } (t,s)\in[0,T]^2
    \quad\text{with some }N\in \mathbb{N},
\]
where $T_{N;n}=\frac{nT}{N}$.

Let $e_1,\dots,e_d$ be an ONB of $\mathbb{R}^d$.
For $1\le i\le d$, $0\le n\le N-1$, and $s\in[0,T]$, define 
$h_{n;i},k_{s;i}\in \mathcal{H}$ by 
$h_{n;i}^\prime=\one_{[T_{N;n},T_{N;n+1})}e_i$ and
$k_{s;i}^\prime=\one_{[0,s)}e_i$.
By \eqref{eq:kappa.phi}, it holds that
\[
    F_{\kappa_\phi}=\sum_{n=0}^{N-1} \sum_{i,j=1}^d
    \biggl(\int_0^T \phi_j^i(T_{N;n},s)
    \theta^j(s)ds\biggr) h_{n;i}.
\]
By \cite[(5.1.5)]{mt-cambridge}, we see that
\[
    \D^*h_{n;i}=\theta^i(T_{N;n+1})-\theta^i(T_{N;n})
    \quad\text{and}\quad
    \D\theta^i(s)=k_{s;i}.
\]
Due to the product rule for $\D^*$ 
(\cite[Theorem\,5.2.8]{mt-cambridge}), we
obtain that 
\begin{align*}
    \D^*F_{\kappa_\phi}
    = & \sum_{n=0}^{N-1} \sum_{i,j=1}^d
    \biggl(\int_0^T \phi_j^i(T_{N;n},s)
    \theta^j(s)ds\biggr)
    \{\theta^i(T_{N;n+1})-\theta^i(T_{N;n})\}
    \\
    & -\sum_{n=0}^{N-1} \sum_{i,j=1}^d
       \biggl\langle \int_0^T 
       \phi_j^i(T_{N;n},s) k_{s;j} ds,
       h_{n;i}\biggr\rangle_{\mathcal{H}}.
\end{align*}
The first term of the right hand side of the equality
coincides with 
\begin{align*}
    & \int_0^T \biggl[\sum_{j=1}^d \biggl(\sum_{i=1}^d
       \biggl\{\sum_{n=0}^{N-1}
         \phi_j^i(T_{N;n},s)
         \{\theta^i(T_{N;n+1})-\theta^i(T_{N;n})\}
       \biggr\}\biggr)\theta^j(s)\biggr]ds
    \\
    & =\int_0^T \biggl\langle \int_0^T \phi(t,s)^\dagger
        d\theta(t),\theta(s)\biggr\rangle ds.
\end{align*}
Since 
\[
    \langle k_{s;j},h_{n;i}\rangle_{\mathcal{H}}
    =\bigl\{(T_{N;n+1}\wedge s)
         -(T_{N;n}\wedge s)\bigr\}\delta_{ji},
\]
the second term of the right hand side coincides with
\[
    \int_0^T \biggl[\sum_{n=0}^{N-1} \biggl(
         \sum_{i=1}^d
         \phi_i^i(T_{N;n},s)\biggr)
         \bigl\{(T_{N;n+1}\wedge s)
           -(T_{N;n}\wedge s)\bigr\}\biggr] ds
    =\int_0^T \biggl(\int_0^s \tr\phi(t,s) dt
      \biggr)ds.
\]
Thus \eqref{l.transf.cm.21} holds.
\end{proof}

\begin{lemma}\label{l.kappa.phi.trace}
$B_{\kappa_\phi}$ is of trace class.
\end{lemma}

\begin{proof}
Define $\psi\in\ltwo$ by 
$\psi(t,s)=\one_{[0,t)}(s)I_d$
for $(t,s)\in[0,T]^2$.
It holds that 
\[
    \kappa_\phi(t,s)
    =\int_0^T \phi(t,u)\psi(u,s)du
    \quad\text{for }(t,s)\in[0,T]^2.
\]
Hence we know that
\begin{equation}\label{l.kappa.phi.trace.21}
    B_{\kappa_\phi}=B_\phi B_\psi,
\end{equation}
which implies that $B_{\kappa_\phi}$ is of trace class.
\end{proof}

\begin{proof}[Proof of Theorem~\ref{t.transf.cm}]
By Lemma~\ref{l.kappa.phi.trace}, it suffices to show
\eqref{t.transf.cm.1}.
Suppose that the sequence 
$\{\phi_m\}_{m=1}^\infty \subset \ltwo$ satisfies that 
$\|\phi_m-\phi\|_2\to0$ as $m\to\infty$.
Denote by $\|\cdot\|_1$ the trace class norm.
By \cite[Lemma~IV.7.2]{GGK} and
\eqref{l.kappa.phi.trace.21}, we have that  
\[
    \|B_{\kappa_{\phi_m}}-B_{\kappa_\phi}\|_1
    \le \|B_\psi\|_{\htwo}
        \|B_{\phi_m}-B_\phi\|_{\htwo}
    =\|B_\psi\|_{\htwo}
     \|\phi_m-\phi\|_2
    \to0 \quad\text{as }m\to\infty.
\]
Thanks to the continuity of the mapping $B\mapsto\det(I+B)$
with respect to the trace class norm
(\cite[Corollary~II.4.2]{GGK}), 
$\det(I+B_{\kappa_{\phi_m}})$ converges to
$\det(I+B_{\kappa_\phi})$.
It is easily seen that $\Psi_{\phi_m}$ converges to
$\Psi_\phi$ in probability.
Furthermore,
\[
    |\Lambda(B_{\kappa_{\phi_m}})
      -\Lambda(B_{\kappa_\phi})|
    \le \|B_{\kappa_{\phi_m}}
         -B_{\kappa_\phi}\|_{\htwo}\to0
    \quad\text{as }m\to\infty.
\]
Thus, it suffices to show \eqref{t.transf.cm.1} for $\phi$,
which is continuous on $[0,T]^2$.

If $\phi$ is continuous, so is $\kappa_\phi$.
By Lemmas~\ref{l.kappa.phi.trace} and \ref{l.trace.detail},
we obtain that 
\[
    \tr B_{\kappa_\phi}
    =\int_0^T \tr\kappa_\phi(t,t) dt
    =\int_0^T \biggl(\int_t^T
        \tr\phi(t,s) ds\biggr)dt
    =\int_0^T \biggl(\int_0^s
         \tr\phi(t,s) dt\biggr)ds.
\]
With the help of the identity
$\dettwo(I+A)=\det(I+A)e^{-\tr A}$ for
$A\in\htwo$ of trace class,
plugging this into \eqref{l.transf.cm.1}, we obtain
\eqref{t.transf.cm.1}.
\end{proof}

\begin{remark}\label{r.transf.cm}
We shall compare Theorem~\ref{t.transf.cm} with the change
of variables formula achieved by Cameron and Martin
(\cite{CM}). 
To do so, we first recall their formula.
Let $d=1$ and $T=1$.
They considered the linear transformation of the form
\[
    \mathbb{T}:\mathcal{W}\ni w\mapsto
       w+\int_0^1 K(\cdot,s)w(s)ds,
\]
where $K\in\ltwo$ with $K(0,\cdot)=0$.
Reminding that their ``Wiener measure'' is the distribution
of the mapping $w\mapsto \frac1{\sqrt{2}}w$ under $\mu$, we
rewrite their change of variables formula in our setting as
follows.  
Under several assumptions, which we do not restate here, it
holds that
\begin{equation}\label{eq:cm.1}
    |D|\int_S (f\circ \mathbb{T}) e^{-\Phi} d\mu
    =\int_{\mathbb{T}(S)} f d\mu
    \quad\text{for any Borel set $S\subset\mathcal{W}$ and  
               $f\in C_b(\mathcal{W})$,}
\end{equation}
for every Borel subset $S$ of $\mathcal{W}$ and 
$f\in C_b(\mathcal{W})$, where 
$\mathbb{T}(S)=\{\mathbb{T} w \mid w\in S\}$,
\begin{equation}\label{eq:cm.2}
    D=1+\sum_{n=1}^\infty \frac1{n!} \int_0^1\cdots\int_0^1
      \det\Bigl(\bigl(K(s_i,s_j)\bigr)_{1\le i,j\le n}\Bigr)
      ds_1\cdots ds_n,
\end{equation}
and
\begin{align*}
    \Phi(w)
    = & \frac12\int_0^1\biggl[\frac{d}{dt} 
          \int_0^1 K(t,s)w(s)ds\biggr]^2dt
     +\int_0^1\biggl[\int_0^1 
        \frac{\partial K}{\partial t}(t,s) w(s) ds
        \biggr]dw(t)
   \\
    & +\frac12\int_0^1 J(s) d\{w(s)^2\}
    \quad\text{with }
    J(s)=\lim_{t\downarrow s}K(t,s)
         -\lim_{t\uparrow s}K(t,s).
\end{align*}
They said that both ``$dw(t)$'' and ``$d\{w(s)^2\}$'' can be defined 
as Stieltjes integrals.   

To compare \eqref{eq:cm.1} with \eqref{t.transf.cm.1},
we assume that $K$ is continuous on $[0,1]^2$ and of the
form 
\[
    K(t,s)=\int_0^t \phi(u,s)du
    \quad\text{for }(t,s)\in[0,1]^2
\]
for some $\phi\in\ltwo$ with 
$\Lambda(B_{\kappa_\phi})<1$.
Then 
$\mathbb{T}=\iota+\mathbb{F}_\phi
           =\iota+F_{\kappa_\phi}$, and 
Theorem~\ref{t.transf.cm} is applicable to this $\phi$.
Combining the following observations (i)--(v), we can
conclude that \eqref{eq:cm.1} follows from
\eqref{t.transf.cm.1}. 
\begin{enumerate}
\item
The first term of $-\Phi$ is equal to the second one of 
$\Psi_\phi$.

\item
Since $\dfrac{\partial K}{\partial t}=\phi$, exchanging the
order of the double integral with respect to $ds$ and
``$dw(t)$'', we rewrite the second term of $-\Phi$ as 
\[
    -\int_0^1\biggl[\int_0^1\phi(t,s)dw(t)\biggr]
       w(s)ds.
\]
Since $d=1$, $M^\dagger=M$ for $M\in \mathbb{R}^{1\times 1}$
and the multiplications is commutative.
Hence this term coincides with the first term of
$\Psi_\phi$ by regarding ``$dw(t)$'' as the It\^o integral
$d\theta(t)$. 

\item
$J=0$ and the third term of $-\Phi$ vanishes.

\item
For $\sigma\in\ltwo$, define the linear operator
$\mathcal{L}_\sigma$ on $L^2([0,1];\mathbb{R})$ by 
\[
    \mathcal{L}_\sigma f=\int_0^1 \sigma(\cdot,s)f(s)ds
    \quad\text{for }f\in L^2([0,1];\mathbb{R}).
\]
We identify $\mathcal{L}_\sigma$ and $B_\sigma$ via the
correspondence between  the spaces $L^2([0,1];\mathbb{R})$ 
and $\mathcal{H}$ such that
\[
    L^2([0,1];\mathbb{R})\ni f\leftrightarrow
    \int_0^\bullet f(t)dt \in \mathcal{H}.
\]
Since 
\[
    K(t,s)=\int_0^1 \psi(t,u)\phi(u,s)du
    \quad\text{for }(t,s)\in[0,1]^2, 
\]
where 
$\psi(t,s)=\one_{[0,t)}(s)$ for $(t,s)\in[0,1]^2$.
$\mathcal{L}_K=\mathcal{L}_\psi \mathcal{L}_\phi$, 
and it is of trace class.
Due to the continuity of $K$, $\det(I+\mathcal{L}_K)$ admits
the Fredholm representation as described in
\eqref{eq:cm.2} 
(\cite[Theorem~VI.1.1]{GGK}).
Hence 
\[
    D=\det(I+\mathcal{L}_K)
     =\det(I+\mathcal{L}_\psi \mathcal{L}_\phi).
\]

\quad
If $A_1,A_2\in \htwo$ are both of trace class,
then $\det(I+A_1A_2)=\det(I+A_2A_1)$ 
(\cite[IV.(5.9)]{GGK}).
This identity is extended to $\htwo$ by a
standard approximation argument.
Thus we have that
\[
    \det(I+B_\psi B_\phi)=\det(I+B_\phi B_\psi).
\]
Hence by \eqref{l.kappa.phi.trace.21} and the identification 
between $\mathcal{L}_\sigma$ and $B_\sigma$, we obtain that
\[
    D=\det(I+B_{\kappa_\phi}).
\]

\item
Let $S$ be a Borel subset of $\mathcal{W}$ and 
$f\in C_b(\mathcal{W})$.
Put
$G=\one_S\circ
    (\iota+F_{\widehat{\kappa_\phi}})$.
By Theorem~\ref{t.inv.transf}, it holds that 
\[
     f(\iota+F_{\kappa_\phi})
       \one_S
     =\bigl(f G\bigr)(\iota+F_{\kappa_\phi})
    \quad\text{and}\quad
    f G = f \one_{\mathbb{T}(S)}
    \quad\text{$\mu$-a.s.}
\]
\end{enumerate}
\end{remark}

\begin{remark}\label{r.gen.cv}
We compare \eqref{eq:transf} with the identity obtained from
the general change of variables formula described in
\cite[Theorem~5.6.1]{mt-cambridge}.
The change of variables formula there asserts that 
$F\in \mathbb{D}^\infty(\mathcal{H})$ having an
$r\in(\frac12,\infty)$ with 
\[
    e^{-D^*F+r\|DF\|_{\mathcal{H}^{\otimes 2}}^2}
    \in L^{1+}(\mu)
\]
satisfies that 
\[
    \int_{\mathcal{W}} f(\iota+F)
      \dettwo(I+DF)e^{-D^*F-\frac12\|F\|_{\mathcal{H}}^2} d\mu
    =\int_{\mathcal{W}} f d\mu
    \quad\text{for every }f\in C_b(\mathcal{W}).
\]

Let $\kappa\in\ltwo$.
Applying this change of variables formula to $F_\kappa$ with
the help of \eqref{l.f.kappa.21} and
Lemma~\ref{l.d*f.kappa}, we obtain that $F_\kappa$ with 
\[
    e^{-D^*F_\kappa}\in L^{1+}(\mu)
\]
satisfies that  
\[
    \dettwo(I+B_\kappa)\int_{\mathcal{W}} f(\iota+F_\kappa)
      e^{\q_{\eta(\kappa)}} d\mu
    =e^{\frac12\|\kappa\|_2^2} \int_{\mathcal{W}} f d\mu
    \quad\text{for every }f\in C_b(\mathcal{W}).
\]
Substituting $f=1$ into this, we see that 
${\det}_2(I+B_\kappa)>0$.
Thus the identity \eqref{eq:transf} is rediscovered.

Define $s(\kappa)\in\stwo$ as 
\[
    s(\kappa)(t,s)=-\{\kappa(t,s)+\kappa(s,t)^\dagger\}
    \quad\text{for }(t,s)\in[0,T]^2.
\]
By \eqref{l.d*f.kappa.21}, we have that
\[
    -D^*F_\kappa=\q_{s(\kappa)}.
\]
By Remark~\ref{r.q.eta.int}, we see that
\[
    e^{-D^*F_\kappa}\in L^{1+}(\mu)
    \quad\text{if and only if}\quad
    \Lambda(B_{s(\kappa)})<1.  
\]
Thus the general change of variables formula implies 
\eqref{eq:transf} with $\dettwo(I+B_\kappa)>0$ if and only if
$\Lambda(B_{s(\kappa)})<1$.

Assume that $\Lambda(B_{s(\kappa)})<2$.
Then $\Lambda(B_{\eta(\kappa)})<1$ and 
$\dettwo(I+B_\kappa)>0$.
In fact, it holds that
\[
    \inf_{\|h\|_{\mathcal{H}}=1}
      \langle B_{a\kappa} h,h\rangle_{\mathcal{H}}
    =-\frac{a}2 \Lambda(B_{s(\kappa)})>-a
    \quad\text{for }a\in[0,1].
\]
Hence 
\[
    \inf_{\|h\|_{\mathcal{H}}=1}
       \|(I+B_{a\kappa})h\|_{\mathcal{H}}
    \ge \inf_{\|h\|_{\mathcal{H}}=1}
        \langle (I+B_{a\kappa})h,h\rangle_{\mathcal{H}}
    >0
    \quad\text{for }a\in[0,1].
\]
By Lemma~\ref{l.cont.inv} and \eqref{r.transf.2},
$\dettwo(I+B_{a\kappa})\ne0$ and 
$\Lambda(B_{\eta(a\kappa)})<1$ for $a\in[0,1]$.
In particular, $\Lambda(B_{\eta(\kappa)})<1$.
Furthermore, since the mapping 
$[0,1]\ni a\mapsto 
 \dettwo(I+B_{a\kappa})=\dettwo(I+aB_\kappa)
 \in \mathbb{R}$ is
continuous and $\dettwo(I+B_{a\kappa})=1$ for $a=0$, 
$\dettwo(I+B_{a\kappa})>0$ for any $a\in[0,1]$. 
Hence $\dettwo(I+B_\kappa)>0$.
Thus Theorem~\ref{t.transf} covers the cases obtained by
using the general change of variables formula. 

In general, even if $\Lambda(B_{\eta(\kappa)})<1$ and 
$\dettwo(I+B_{\kappa})>0$, $\Lambda(B_{s(\kappa)})$ may not
be less than $2$.
For example, take orthonormal $h_1,h_2\in \mathcal{H}$ and 
$b_1,b_2<-1$.
Define $\kappa\in\stwo$ by 
\[
    \kappa(t,s)=b_1[h_1^\prime(s)\otimes h_1^\prime(t)]
     +b_2[h_2^\prime(s)\otimes h_2^\prime(t)]
    \quad\text{for }(t,s)\in[0,T]^2.
\]
We have that
\[
    s(\kappa)(t,s)
      =-2b_1[h_1^\prime(s)\otimes h_1^\prime(t)]
       -2b_2[h_2^\prime(s)\otimes h_2^\prime(t)]
\]
and
\[
    \eta(\kappa)(t,s)
       =-(2b_1+b_1^2)[h_1^\prime(s)\otimes h_1^\prime(t)]
        -(2b_2+b_2^2)[h_2^\prime(s)\otimes h_2^\prime(t)]
\]
for $(t,s)\in[0,T]^2$.
Hence we see that
\[
    \Lambda(B_{s(\kappa)})=(-2b_1)\vee(-2 b_2)>2
\]
and 
\[
    \Lambda(B_{\eta(\kappa)})
      =(-2b_1-b_1^2)\vee (-2b_2-b_2^2)
      =\{1-(1+b_1)^2\}\vee \{1-(1+b_2)^2\}
      <1.
\]
Furthermore, it holds that
\[
    \dettwo(I+B_\kappa)=(1+b_1)(1+b_2)e^{-(b_1+b_2)}>0.
\]

Thus, our condition that $\Lambda(B_{\eta(\kappa)})<1$
covers a wider class of transformations of order one than
the class obtained via the general change of variables
formula.
\end{remark}

\chapter{From quadratic forms to transformations}
\label{chap.twoways}

\begin{quote}{\small
Since $\lambda\q_\eta=\q_{\lambda\eta}$ for $\lambda\in \mathbb{R}$, 
investing the Laplace transformation
$\mathbb{R}\ni\lambda\mapsto
 \int_{\mathcal{W}} fe^{\lambda\q_\eta} d\mu$
with $\eta\in\stwo$ and $f\in  C_b(\mathcal{W})$
comes down to the case when $\lambda=1$:
what to be probed is the Wiener integral
$\int_{\mathcal{W}} fe^{\q_\eta} d\mu$.
In order to apply Theorem~\ref{t.transf} to the evaluation of
such Wiener integrals, the surjectivity of the mapping
$\ltwo\ni\kappa\mapsto\eta(\kappa)\in\stwo$ will be shown
in two ways.
We will also see its injectivity by restricting the domain.
The achieved evaluation is applied to 
quadratic Wiener functionals of harmonic-oscillator type.
Furthermore, the correspondence with the Girsanov transformation 
is investigated.
}
\end{quote}

\section{Surjectivity --- finding $\kappa$ with $\eta(\kappa)=\eta$}
\label{sec.twoways}
In this section, we present the two ways to find
$\kappa\in\ltwo$ with $\eta=\eta(\kappa)$ for given
$\eta\in\stwo$.

\begin{theorem}\label{t.tways}
Let $\eta\in\stwo$.
The following four conditions are equivalent.
\\
{\rm(i)}
$e^{\q_\eta}\in L^1(\mu)$.
\\
{\rm(ii)}
$\Lambda(B_\eta)<1$.
\\
{\rm(iii)}
There is a $\tau\in\stwo$ satisfying that
$\dettwo(I+B_\tau)\ne0$ and 
\begin{equation}\label{t.tways.1}
    \eta(t,s)=-2\tau(t,s)-\int_0^T \tau(t,u)\tau(u,s)du
    \quad\text{for a.e.}~(t,s)\in[0,T]^2.
\end{equation}
{\rm(iv)}
There is a $\rho\in\stwo$ satisfying that
\begin{equation}\label{t.tways.2}
    \eta(t,s)=\rho(t,s)
      -\int_{t\vee s}^T \rho(t,u)\rho(u,s)du
    \quad\text{for a.e.}~(t,s)\in[0,T]^2.
\end{equation}
If these conditions hold, then 
\[
    \eta(\tau)\eqltwo\eta
    \quad\text{and}\quad
    \eta(\kappa_A(\rho))\eqltwo\eta,
\]
where $\kappa_A(\rho)\in\ltwo$ is given as
$\kappa_A(\rho)(t,s)=-\one_{[0,t)}(s)\rho(t,s)$ for $(t,s)\in[0,T]^2$,
and it holds that
\begin{equation} \label{t.tways.3}
    \int_{\mathcal{W}} f e^{\q_\eta+\D^*h} d\mu
      = \{\dettwo(I-B_\eta)\}^{-\frac12}
        \int_{\mathcal{W}} f(\iota+F_{\widehat{\tau}})
          e^{\D^*[(I+B_\tau)^{-1}h]} d\mu
\end{equation}
and
\begin{equation}\label{t.tways.4}
    \int_{\mathcal{W}} f e^{\q_\eta+\D^*h} d\mu
      =e^{\frac14\|\rho\|_2^2} 
       \int_{\mathcal{W}}
         f(\iota+F_{\widehat{\kappa_A(\rho)}})
         e^{\D^*[(I+B_{\kappa_A^*(\rho)})^{-1}h]} d\mu
\end{equation}
for every $f\in C_b(\mathcal{W})$ and $h\in \mathcal{H}$,
where $\kappa_A^*(\rho)(t,s)=-\one_{(t,T]}(s)\rho(t,s)$ for
$(t,s)\in[0,T]^2$.
\end{theorem}

\begin{remark}\label{r.tways}
(i)
The equivalence between the conditions (i) and (ii) has been
already seen in Lemma~\ref{l.q.eta.int}.
\\
(ii)
Let $\eta\in\stwo$ satisfy $\Lambda(B_\eta)<1$.
Due to the theorem, we have such $\tau,\rho\in\stwo$ 
as described in the conditions (iii) and (iv).
Suppose that $\tau\eqltwo\kappa_A(\rho)$.
Since $\tau\in\stwo$, we have that
\[
    -\one_{[0,t)}(s) \rho(t,s)
    =\tau(t,s)=\tau(s,t)^\dagger
    =-\one_{(t,T]}(s) \rho(t,s)
    \quad\text{for a.e.}~(t,s)\in[0,T]^2.
\]
This implies that $\rho\eqltwo0$ and hence 
$\eta\eqltwo\eta(\kappa_A(\rho))\eqltwo0$.
Thus if $\|\eta\|_2\ne0$, then
$\|\tau-\kappa_A(\rho)\|_2\ne0$, that is, we have two types
of $\kappa$ with $\eta(\kappa)=\eta$.
\end{remark}

The proof is broken into several steps, each being a lemma.
On the way, we will see that $\dettwo(I-B_\eta)\ne0$ and 
$I+B_{\kappa_A^*(\rho)}$ has a continuous inverse if the
conditions are fulfilled.

\begin{lemma}\label{l.tw.(iii)to(ii)}
Suppose that the condition {\rm(iii)} in the theorem holds. 
Then $\eta(\tau)=\eta$, the condition {\rm(ii)} is fulfilled, 
and the identity \eqref{t.tways.3} holds.
\end{lemma}

\begin{proof}
The equation \eqref{t.tways.1} is rewritten as
\[
    I-B_\eta=(I+B_\tau)^2.
\]
Since $\tau\in\stwo$, this implies that
$I-B_\eta=I-B_{\eta(\tau)}$.
By Lemma~\ref{l.b.kappa}, 
$\eta\eqltwo\eta(\tau)$.
Furthermore, due to \eqref{r.transf.2} and the assumption 
that $\dettwo(I+B_\tau)\ne0$, we obtain that
$\Lambda(B_\eta)=\Lambda(B_{\eta(\tau)})<1$.

Recall that 
\[
    \det\bigl((I+B)^2\bigr)
    =(\det(I+B))^2
    \quad\text{and}\quad
    \dettwo(I+B)=\det(I+B)e^{-\tr B}
\]
for $B\in \mathcal{S}(\htwo)$ of trace class
(see \cite[IV(5.10)]{GGK} for the first equality).
Hence the identity 
\[
    \dettwo\bigl((I+B)^2\bigr)
    =\{\dettwo(I+B)\}^2  e^{-\|B\|_{\htwo}^2}
\]
holds for such a $B$, and extends to $\mathcal{S}(\htwo)$ by a
standard continuation argument.
Since $\|B_\tau\|_{\htwo}=\|\tau\|_2$, we know that
\[
    \dettwo(I-B_\eta)
    =\dettwo((I+B_\tau)^2)
    =|\dettwo(I+B_\tau)|^2e^{-\frac12\|\tau\|_2^2}.
\]
Plugging this into \eqref{eq:inv.transf} with $\kappa=\tau$,
we obtain \eqref{t.tways.3}.
\end{proof}

\begin{lemma}\label{l.dettwo}
For every $\rho\in\stwo$, 
$\dettwo(I+B_{\kappa_A(\rho)})
 =\dettwo(I+B_{\kappa_A^*(\rho)})=1$ 
and 
$\Lambda(B_{\eta(\kappa_A(\rho))})<1$.
\end{lemma}

\begin{proof}
Note that $B_{\kappa_A^*(\rho)}=B_{\kappa_A(\rho)}^*$.
Thus
$\dettwo(I+B_{\kappa_A^*(\rho)})=\dettwo(I+B_{\kappa_A(\rho)})$. 
Due to \eqref{r.transf.2}, the proof completes once we have
shown that $\dettwo(I+B_{\kappa_A(\rho)})=1$.
To show this identity, by the continuity of the mapping 
$\htwo\ni A\mapsto\dettwo(I+A)$, we may and will assume that 
$\kappa_A(\rho)$ is of the form 
\[
    \kappa_A(\rho)(t,s)
    =\sum_{0\le m<n\le N-1}
       \one_{[T_{N;n},T_{N;n+1})}(t)
       \one_{[T_{N;m},T_{N;m+1})}(s)
     K_{nm}
\]
for $(t,s)\in[0,T]^2$ for some $N\in \mathbb{N}$ and 
$K_{nm}\in \mathbb{R}^{d\times d}$, where
$T_{N;n}=\frac{nT}{N}$ for $0\le n\le N$.

Take an orthonormal basis $e_1,\dots,e_d$ of
$\mathbb{R}^d$. 
Define $h_{n;i}\in \mathcal{H}$ by 
\[
    h_{n;i}^\prime=\one_{[T_{N;n},T_{N;n+1})}e_i
\]
for $0\le n\le N-1$ and $1\le i\le d$.
Then $\{h_{n;i}\mid 0\le n\le N-1, 1\le i\le d\}$ is an
orthogonal system.
We have that
\[
    B_{\kappa_A(\rho)} 
    =\sum_{0\le m<n\le N-1} \sum_{i,j=1}^d K_{nm}^{ij}
      h_{m;j}\otimes h_{n;i},
    \quad\text{where }
    K_{nm}=\bigl(K_{nm}^{ij}\bigr)_{1\le i,j\le d}.
\]
Due to this expression, we know that $B_{\kappa_A(\rho)}$ is
of trace class. 
Define $0\le n(p)\le N-1$ and $1\le i(p)\le d$ 
for $1\le p\le dN$ so that $p=n(p)d+i(p)$.
Setting $\tilde{h}_p=h_{n(p);i(p)}$ for $1\le p\le dN$, we 
have that 
\[
    B_{\kappa_A(\rho)}
    =\sum_{1\le p<q\le dN} a^{pq} 
       \tilde{h}_p\otimes \tilde{h}_q,
    \quad\text{where }
    a^{pq}=\begin{cases}
            K_{n(q)n(p)}^{i(q)i(p)} &
            \text{if } n(p)<n(q),
            \\
            0 & \text{if } n(p)\ge n(q).
          \end{cases}
\]
This implies that 
\[
    \det(I+B_{\kappa_A(\rho)})=1
    \quad\text{and}\quad
    \tr B_{\kappa_A(\rho)}=0.
\]
Thus 
$\dettwo(I+B_{\kappa_A(\rho)})
 =\det(I+B_{\kappa_A(\rho)})e^{-\tr B_{\kappa_A(\rho)}}=1$.
\end{proof}

\begin{lemma}\label{l.tw.(iv)to(ii)}
Suppose that the condition {\rm(iv)} in the theorem holds. 
Then $\eta(\kappa_A(\rho))=\eta$, 
the condition {\rm(ii)} is fulfilled, and 
the identity \eqref{t.tways.4} holds.
\end{lemma}

\begin{proof}
Notice that 
\begin{equation}\label{eq.eta.kappa.a.rho}
    \eta(\kappa_A(\rho))(t,s)
    =\rho(t,s)-\int_{t\vee s}^T \rho(t,u)\rho(u,s)du
    \quad\text{for $(t,s)\in[0,T]^2$ with $t\ne s$}.
\end{equation}
Hence $\eta(\kappa_A(\rho))\eqltwo\eta$.
By Lemma~\ref{l.dettwo}, we see that
$\Lambda(B_\eta)=\Lambda(B_{\eta(\kappa_A(\rho))})<1$.

Notice that $\|\kappa_A(\rho)\|_2^2=\frac12\|\rho\|_2^2$.
By Lemma~\ref{l.dettwo} and the identity \eqref{eq:inv.transf} 
with $\kappa=\kappa_A(\rho)$, we obtain \eqref{t.tways.4}
\end{proof}

\begin{lemma}\label{l.tw.(ii)to(iii)}
If the condition {\rm(ii)} holds, then so does {\rm(iii)}.
\end{lemma}

\begin{proof}
Let $\mathcal{S}_+(\mathcal{H})$ be the totality of 
self-adjoint, continuous, non-negative definite, and linear
operators of $\mathcal{H}$ to $\mathcal{H}$.
Since $\Lambda(B_\eta)<1$, it holds that
\[
    \langle(I-B_\eta)h,h\rangle_{\mathcal{H}}
    \ge
    (1-\Lambda(B_\eta))\|h\|_{\mathcal{H}}^2
    \quad\text{for }h\in \mathcal{H}.
\]
Hence $I-B_\eta\in \mathcal{S}_+(\mathcal{H})$.
Due to the square root lemma (Lemma~\ref{l.sq.root}), there
is a unique $C_\eta\in \mathcal{S}_+(\mathcal{H})$ with
\[
    C_\eta^2=I-B_\eta.
\]
Since $C_\eta$ is non-negative definite, it holds that
\[
    \|(I+C_\eta)h\|_{\mathcal{H}}
    \ge\langle (I+C_\eta)h,h\rangle_{\mathcal{H}}
    \ge1
    \quad\text{for every }h\in \mathcal{H}
    \text{ with }\|h\|_{\mathcal{H}}=1.
\]
By Lemma~\ref{l.cont.inv}, $I+C_\eta$ has a continuous
inverse.
Using this inverse, we have that
\[
    C_\eta-I=-(I+C_\eta)^{-1}B_\eta \in \mathcal{S}(\htwo).
\]
By virtue of Lemma~\ref{l.b.kappa}, there is a unique 
$\kappa_S(\eta)\in \stwo$ with
\[
    B_{\kappa_S(\eta)}=C_\eta-I.
\]

Set $\tau=\kappa_S(\eta)$.
By definition, we have that
\[
    (I+B_\tau)^2=C_\eta^2=I-B_\eta.
\]
This means that \eqref{t.tways.1} holds.

Since 
\[
    \inf_{\|h\|_{\mathcal{H}}=1}\|C_\eta h\|_{\mathcal{H}}^2
    =\inf_{\|h\|_{\mathcal{H}}=1}
      \langle(I-B_\eta)h,h\rangle_{\mathcal{H}}
    =1-\Lambda(B_\eta)>0,
\]
by Lemma~\ref{l.cont.inv}, $C_\eta=I+B_\tau$ has a
continuous inverse.
Thus $\dettwo(I+B_\tau)\ne0$ (\cite[Theorem\,XII.1.2]{GGK}).
\end{proof}

\begin{remark}\label{r.sqrt}
Since $B_\eta\in \mathcal{S}(\htwo)$, it has an ONB 
$\{h_n\}_{n=1}^\infty$ developing it in $\htwo$ as
$B_\eta=\sum\limits_{n=1}^\infty a_n h_n\otimes h_n$,
where $a_n=\langle B_\eta h_n,h_n\rangle_{\mathcal{H}}$ for
$n\in \mathbb{N}$.
Then we have that
$B_{\kappa_S(\eta)}=\sum\limits_{n=1}^\infty 
 \{(1-a_n)^{\frac12}-1\} h_n\otimes h_n$.
\end{remark}

\begin{lemma}\label{l.tw.(ii)to(iv)}
If the condition {\rm(ii)} holds, then so does {\rm(iv)}.
\end{lemma}

\begin{proof}
As was seen in the Proof of Lemma~\ref{l.tw.(ii)to(iii)}, 
the square root operator
$C_\eta\in \mathcal{S}_+(\mathcal{H})$ 
of $I-B_\eta$ has a continuous inverse. 
Hence so does $I-B_\eta$.
Define $\varphi\in\stwo$ as $\varphi=\widehat{-\eta}$, that
is, $B_\varphi=(I-B_\eta)^{-1}-I$.
Rewrite this equality as
\begin{equation}\label{l.tw.(ii)to(iv).21}
    (I-B_{-\varphi})^{-1}=I+B_{-\eta}.
\end{equation}

Let 
$a=\inf\limits_{\|h\|_{\mathcal{H}}=1}\langle B_\eta h,
 h\rangle_{\mathcal{H}}$.
Then $-\infty\le -\|B_\eta\|_\op\le a\le \Lambda(B_\eta)<1$ and  
\[
    \langle (I-B_\eta)g,g\rangle_{\mathcal{H}}
    \le (1-a)\|g\|_{\mathcal{H}}^2
    \quad\text{for any }g\in \mathcal{H}.
\]
Substituting $g=C_\eta^{-1}h$ into this inequality, we
obtain that  
\[
    (1-a)^{-1}
    \le \langle (I-B_\eta)^{-1}h,h\rangle_{\mathcal{H}}
    \quad\text{for $h\in \mathcal{H}$ with 
    $\|h\|_{\mathcal{H}}=1$}.
\]
Combined with \eqref{l.tw.(ii)to(iv).21}, this implies that
\[
    \Lambda(B_{-\varphi})\le 1-(1-a)^{-1}<1.
\]

For $\xi\in(0,T]$, define the subspace $\mathcal{H}_\xi$ of
$\mathcal{H}$ by
\[
    \mathcal{H}_\xi
    =\{h\in \mathcal{H}\mid h^\prime=\one_{[0,\xi)}h^\prime\}.
\]
Letting $\pi_\xi$ be the orthogonal projection of
$\mathcal{H}$ onto $\mathcal{H}_\xi$, we have that
\begin{equation}\label{l.tw.(ii)to(iv).22}
    \Lambda(\pi_\xi B_{-\varphi} \pi_\xi)
    \le \Lambda(B_{-\varphi})<1.
\end{equation}

For $\psi\in\stwo$, let $\mathcal{L}_\psi$ be the 
continuous operator of $L^2([0,T];\mathbb{R}^d)$ to itself
such that
\[
    \mathcal{L}_\psi f
    =\int_0^T \psi(\cdot,s)f(s)ds
    \quad\text{for }f\in L^2([0,T];\mathbb{R}^d).
\]
Identify 
$\mathcal{H}$ and $L^2([0,T];\mathbb{R}^d)$ via the
correspondence 
\[
    \mathcal{H}\ni h\leftrightarrow h^\prime\in
    L^2([0,T];\mathbb{R}^d).
\]
Under this correspondence, $B_\psi$ and $\mathcal{L}_\psi$
are also identified, and the identity
\eqref{l.tw.(ii)to(iv).21} is rewritten as 
\[
    (I-\mathcal{L}_{-\varphi})^{-1}=I+\mathcal{L}_{-\eta}.
\]
Thus $-\eta$ is the resolvent kernel of
$\mathcal{L}_{-\varphi}$.

Let $\xi\in(0,T]$. 
Under the identification between
$\mathcal{H}_\xi$ and $L^2([0,\xi];\mathbb{R}^d)$, 
\eqref{l.tw.(ii)to(iv).22} implies that
\[
    \sup_{\|f\|_{L^2([0,\xi];\mathbb{R}^d)}=1}
    \int_0^\xi\biggl\langle \int_0^\xi (-\varphi)(t,s)
      f(s)ds,f(t)\biggr\rangle dt <1,
\]
where $\|\cdot\|_{L^2([0,\xi];\mathbb{R}^d)}$ is the
$L^2$-norm in $L^2([0,\xi];\mathbb{R}^d)$. 
Hence the equation
\[
    g(t)-\int_0^\xi (-\varphi)(t,s)g(s)ds=f(t)
    \quad\text{for } t\in[0,\xi]
\]
is solvable in $L^2([0,\xi];\mathbb{R}^d)$.
Thanks to the special factorization of
$\mathcal{L}_{-\varphi}$ due to Gohberg-Krein (\cite{GK})
(see Proposition~\ref{p.sigma.gamma}), there exists a
$\nu\in\stwo$ such that 
\[
    -\eta(t,s)=\nu(t,s)+\int_t^T \nu(t,u)\nu(u,s)du
    \quad\text{for }0\le s<t\le T.
\]
Setting $\rho=-\nu$, we see that \eqref{t.tways.2} holds.
\end{proof}

Let $\kappa_S(\eta)$ be as in
Lemma~\ref{l.tw.(ii)to(iii)} and $\kappa_A(\rho)$ be as in 
Theorem~\ref{t.tways}.
The kernel $\kappa_S(\eta)$ was constructed with the help of
the square root lemma.
The stochastic process
$\{(F_{\kappa_A(\rho)})^\prime(t)
 =\int_0^t \rho(t,s)d\theta(s)\}_{t\in[0,T]}$ is adapted.
On account of these, we call $\iota+F_{\kappa_S(\eta)}$ 
a {\it square root transformation} of order one, and
$\iota+F_{\kappa_A(\rho)}$ 
an {\it adapted transformation} of order one.

\section{Bijectivity}
\label{sec.bijectivity}

As was seen in Remark~\ref{r.transf}(ii), the mapping
$\ltwo\ni\kappa
 \mapsto\eta(\kappa)\in\stwo$ 
is not injective. 
We have the following two types of bijectivity by restricting the
domain of the mapping.
For $\mathcal{L}\subset\ltwo$ and $\mathcal{S}\subset\stwo$,
we say that a mapping 
$A:\mathcal{L}\ni \kappa\mapsto A(\kappa)\in \mathcal{S}$
is bijective if $\kappa_1\eqltwo\kappa_2$ for 
$\kappa_1,\kappa_2\in \mathcal{L}$ with
$A(\kappa_1)\eqltwo A(\kappa_2)$ and 
each $\eta\in \mathcal{S}$ has a $\kappa\in \mathcal{L}$
with $A(\kappa)\eqltwo\eta$.

\begin{proposition}\label{p.bij}
Put
\begin{align*}
    & \widehat{\stwo}=\{\eta\in\stwo\mid
       \Lambda(B_\eta)<1\},
    \\
    & \mathscr{S}_{2,+}=\bigl\{\kappa\in\stwo\mid 
      I+B_\kappa\ge0\bigr\},
    \\
    & \widehat{\mathscr{S}_{2,+}}
      =\bigl\{\kappa\in\mathscr{S}_{2,+}\mid 
        \Lambda(B_{\eta(\kappa)})<1\bigr\},
\end{align*}
and
\[
    \widehat{\ltwo}=\bigl\{\kappa\in\ltwo\mid 
      \kappa(t,s)=\one_{[0,t)}(s)\kappa(t,s)
      \text{ for }(t,s)\in[0,T]^2\bigr\},
\]
where ``$I+B_\kappa\ge0$'' means that the operator is 
non-negative definite.
\\
{\rm(i)}
If $\kappa_1,\kappa_2\in \mathscr{S}_{2,+}$ and 
$\eta(\kappa_1)\eqltwo\eta(\kappa_2)$, then 
$\kappa_1\eqltwo\kappa_2$.
Moreover, the mapping 
$\widehat{\mathscr{S}_{2,+}}\ni\kappa\mapsto
 \eta(\kappa)\in\widehat{\stwo}$ 
is bijective.
\\
{\rm (ii)}
If $\rho_1,\rho_2\in \stwo$ and 
$\eta(\kappa_A(\rho_1))\eqltwo\eta(\kappa_A(\rho_2))$, 
then $\rho_1\eqltwo\rho_2$.
Moreover, the image $\eta\bigl(\widehat{\ltwo}\bigr)$
of $\widehat{\ltwo}$ through $\eta$ is included in 
$\widehat{\stwo}$, and the mapping 
$\widehat{\ltwo}\ni\kappa\mapsto 
 \eta(\kappa)\in\widehat{\stwo}$ 
is bijective.
\end{proposition}

\begin{proof}
(i)
Let $\kappa_1,\kappa_2\in \mathscr{S}_{2,+}$ and assume that  
$\eta(\kappa_1)\eqltwo\eta(\kappa_2)$.
By \eqref{r.transf.1}, we have that
\[
    (I+B_{\kappa_1})^2
    =I-B_{\eta(\kappa_1)}=I-B_{\eta(\kappa_2)}
    =(I+B_{\kappa_2})^2.
\]
Since 
$I+B_{\kappa_1}$ and $I+B_{\kappa_2}$ are both in 
$\mathcal{S}_+(\mathcal{H})$, due to the uniqueness of
square root operator (Lemma~\ref{l.sq.root}),
$I+B_{\kappa_1}=I+B_{\kappa_2}$. 
By Lemma~\ref{l.b.kappa}, $\kappa_1\eqltwo\kappa_2$.

The first assertion shows the injectivity of the mapping
$\widehat{\mathscr{S}_{2,+}}\ni\kappa
 \mapsto\eta(\kappa)\in\widehat{\stwo}$. 
To see the surjectivity, let $\eta\in\widehat{\stwo}$.
As was seen in the proof of Lemma~\ref{l.tw.(ii)to(iii)},
$\kappa_S(\eta)\in\mathscr{S}_{2,+}$ and
$\eta(\kappa_S(\eta))\eqltwo\eta$.
Thus the surjectivity has been shown.

\smallskip\noindent
(ii)
Notice that
$\widehat{\ltwo}
 =\{\kappa_A(\rho)\mid \rho\in\stwo\}$.
By Lemma~\ref{l.dettwo},
$\eta\bigl(\widehat{\ltwo}\bigr)
 \subset\widehat{\stwo}$.
Due to Theorem~\ref{t.tways},
the mapping 
$\widehat{\ltwo}\ni\kappa
 \mapsto\eta(\kappa)\in\widehat{\stwo}$
is surjective.
Thus it suffices to show the first assertion.

Let $\rho_1,\rho_2\in\stwo$ and assume that
$\eta(\kappa_A(\rho_1))\eqltwo\eta(\kappa_A(\rho_2))$.
By \eqref{eq.eta.kappa.a.rho}, it holds that
\begin{equation}\label{p.bij.21}
    \rho_1(t,s)-\rho_2(t,s)
    =\int_{t\vee s}^T \bigl\{
        \rho_1(t,u)\rho_1(u,s)
        -\rho_2(t,u)\rho_2(u,s)\bigr\} du
\end{equation}
for a.e.~$(t,s)\in[0,T]^2$.
Take an $N\in \mathbb{N}$ such that
\[
    \iint_{[0,a+\frac{T}{N}]\times[0,T]}
     |\rho_i(t,s)|^2 dtds<\frac15
    \quad
    \text{for any $a\in[0,T_1]$ and $i=1,2$},
\]
where $T_n=T-\frac{nT}{N}$ for $1\le n\le N$.
Put 
\[
    \gamma(t,s)=\rho_1(t,s)-\rho_2(t,s)
    \quad\text{for }(t,s)\in[0,T]^2.
\]
By \eqref{p.bij.21} and the Schwarz inequality, we have that  
\begin{align*}
    |\gamma(t,s)|^2
    \le 2\biggl\{ 
    & \biggl(\int_{t\vee s}^T |\rho_1(u,s)|^2 du\biggr)
      \biggl(\int_{t\vee s}^T |\gamma(t,u)|^2 du\biggr)
    \\
    &  +\biggl(\int_{t\vee s}^T |\rho_2(t,u)|^2 du\biggr)
       \biggl(\int_{t\vee s}^T |\gamma(u,s)|^2 du\biggr)
     \biggr\}
    \quad\text{for a.e.~}(t,s)\in[0,T]^2.
\end{align*}
Since $|\rho_2(t,u)|=|\rho_2(u,t)|$, this implies that
\[
    \iint_{t\vee s\ge T_1} |\gamma(t,s)|^2 dtds
    \le \frac45 
    \iint_{t\vee s\ge T_1} |\gamma(t,s)|^2 dtds,
\]
where $\iint_{t\vee s\ge a} \dots dtds$ is the integral over
the set $\{(t,s)\in[0,T]^2\mid t\vee s\ge a\}$. 
Hence we obtain that 
\[
    \iint_{t\vee s\ge T_1} |\gamma(t,s)|^2 dtds=0.
\]
In other words, $\gamma(t,s)$ vanishes 
for a.e.~$(t,s)\in[0,T]^2$ with $t\vee s\ge T_1$.

By this vanishing, we know that \eqref{p.bij.21} holds with
$T_1$ for $T$.
In exactly the same manner as above, only this time with
$T_1$ for $T$, we see that $\gamma(t,s)=0$ 
for a.e.~$(t,s)\in[0,T]^2$ with $T_2\le t\vee s\le T_1$.
Hence so does for a.e.~$(t,s)\in[0,T]^2$ with 
$t\vee s\ge T_2$.
Continue the argument successively to have the vanishing of
$\gamma(t,s)$ for a.e.~$(t,s)\in[0,T]^2$.
Thus we know that $\rho_1\eqltwo\rho_2$.
\end{proof}

\section{Quadratic Wiener functional of harmonic-oscillator
  type} 
\label{sec.elementary}
We give an application of Theorem~\ref{t.tways} to
$\mathfrak{h}(\kappa;x), \mathfrak{h}(\kappa)
 \in \mathbb{D}^\infty$ 
described in Proposition~\ref{p.h.osc}.  

\begin{proposition}\label{p.h.osc.lap}
Let $\kappa\in\ltwo$ and $x\in \mathbb{R}^d$.
Set $c(\kappa;x),c(\kappa)\in\stwo$ as in
Proposition~\ref{p.h.osc}, and define
$c^\prime(\kappa;x),c^\prime(\kappa)\in\stwo$ by 
\[
    c^\prime(\kappa;x)=\kappa_S(-c(\kappa;x))
    \quad\text{and}\quad
    c^\prime(\kappa)=\kappa_S(-c(\kappa)).
\]
Then $B_{c(\kappa;x)}$ is of trace class and it holds that 
\[
  \int_{\mathcal{W}} f e^{-\mathfrak{h}(\kappa;x)+\D^*h}
         d\mu
     =\{\det(I+B_{c(\kappa;x)})\}^{-\frac12}
       \int_{\mathcal{W}} f\Bigl(\iota
         +F_{\widehat{c^\prime(\kappa;x)}}\Bigr)
       e^{\D^*[(I+B_{c^\prime(\kappa;x)})^{-1}h]} d\mu
\]
and
\[
  \int_{\mathcal{W}} f e^{-\mathfrak{h}(\kappa)+\D^*h}
         d\mu
     =\{\det(I+B_\kappa^*B_\kappa)\}^{-\frac12}
      \int_{\mathcal{W}} f\Bigl(\iota
         +F_{\widehat{c^\prime(\kappa)}}\Bigr)
       e^{\D^*[(I+B_{c^\prime(\kappa)})^{-1}h]} d\mu 
\]
for every $f\in C_b(\mathcal{W})$ and $h\in \mathcal{H}$.
In particular, 
\[
   \int_{\mathcal{W}} e^{-\mathfrak{h}(\kappa;x)+\D^*h}
         d\mu
    =\{\det(I+B_{c(\kappa;x)})\}^{-\frac12}
       e^{\frac12\langle (I+B_{c(\kappa;x)})^{-1}h,
             h\rangle_{\mathcal{H}}}
\]
and 
\[
  \int_{\mathcal{W}} e^{-\mathfrak{h}(\kappa)+\D^*h} d\mu
    =\{\det(I+B_\kappa^*B_\kappa)\}^{-\frac12}
       e^{\frac12\langle (I+B_\kappa^*B_\kappa)^{-1}h,
             h\rangle_{\mathcal{H}}}
\]
for any $h\in \mathcal{H}$.
\end{proposition}

\begin{proof}
By Proposition~\ref{p.h.osc}, it holds that
\begin{equation}\label{p.h.osc.21}
    -\mathfrak{h}(\kappa;x)
    =\q_{-c(\kappa;x)}
     -\frac12 \int_0^T\int_0^T|\kappa(t,s)^\dagger x|^2 
       dsdt.
\end{equation}
Since
\begin{equation}\label{p.h.osc.22}
    \langle B_{-c(\kappa;x)}h,h\rangle_{\mathcal{H}}
    =-\int_0^T\biggl(\int_0^T \langle 
       \kappa(u,s)^\dagger x,h^\prime(s)\rangle ds
       \biggr)^2 du
    \quad\text{for } h\in \mathcal{H},
\end{equation}
$\Lambda(B_{-c(\kappa;x)})\le0<1$.
By Theorem~\ref{t.tways} and \eqref{p.h.osc.21}, we
have that
\begin{align}
\label{p.h.osc.23}
    \int_{\mathcal{W}} f e^{-\mathfrak{h}(\kappa;x)+\D^*h}
        d\mu 
    = & \exp\biggl(-\frac12\int_0^T\int_0^T
           |\kappa(t,s)^\dagger x|^2 dsdt\biggr)
     \{\dettwo(I+B_{c(\kappa;x)})\}^{-\frac12}
    \\
    & 
     \times
     \int_{\mathcal{W}}
       f\Bigl(\iota+F_{\widehat{c^\prime(\kappa;x)}}\Bigr) 
       e^{\D^*[(I+B_{c^\prime(\kappa;x)})^{-1}h]} d\mu
 \nonumber
\end{align}
for every $f\in C_b(\mathcal{W})$ and $h\in \mathcal{H}$.

For $A\in \htwo$, denote by 
$|A|\in \mathcal{S}_+(\mathcal{H})$ the square root of
$A^*A$.
By definition, $A$ is of trace class if 
$\sum\limits_{n=1}^\infty \langle |A|h_n,h_n
  \rangle_{\mathcal{H}}<\infty$
for some ONB $\{h_n\}_{n=1}^\infty$ of $\mathcal{H}$.
Since 
$B_{c(\kappa;x)}\in \mathcal{S}(\htwo)$,
$B_{c(\kappa;x)}^*B_{c(\kappa;x)}=B_{c(\kappa;x)}^2$.
By \eqref{p.h.osc.22}, 
$B_{c(\kappa;x)}\in \mathcal{S}_+(\mathcal{H})$.
Due to the uniqueness of square root, we obtain that
$B_{c(\kappa;x)}=|B_{c(\kappa;x)}|$.
By \eqref{p.h.osc.22} again, we have that
\[
    \sum_{n=1}^\infty \langle B_{c(\kappa;x)}h_n,h_n
       \rangle_{\mathcal{H}}
    =\int_0^T\int_0^T|\kappa(t,s)^\dagger x|^2 dsdt
    \quad\text{for any ONB $\{h_n\}_{n=1}^\infty$ of
      $\mathcal{H}$.}
\]
Hence $B_{c(\kappa;x)}$ is of trace class and 
\[
    \tr B_{c(\kappa;x)}
    =\int_0^T\int_0^T|\kappa(t,s)^\dagger x|^2 dsdt.
\]
Thus we have that 
\[
    \dettwo(I+B_{c(\kappa;x)})
    =\det(I+B_{c(\kappa;x)})
     \exp\biggl(-\int_0^T\int_0^T|\kappa(t,s)^\dagger x|^2 
       dsdt\biggr).
\]
Plugging this into \eqref{p.h.osc.23}, we obtain the first
identity.

By Proposition~\ref{p.h.osc}, we have that
\[
    -\mathfrak{h}(\kappa)
    =\q_{-c(\kappa)}-\frac12 \|\kappa\|_2^2.
\]
Recall that $B_{c(\kappa)}=B_\kappa^*B_\kappa$.
Hence $\Lambda(B_{-c(\kappa)})\le0<1$, $B_{c(\kappa)}$ is of
trace class, and 
$\tr B_{c(\kappa)}
 =\|B_\kappa\|_{\htwo}^2=\|\kappa\|_2^2$.
By Theorem~\ref{t.tways}, it holds that
\[
    \int_{\mathcal{W}} f e^{-\mathfrak{h}(\kappa)+\D^*h}
       d\mu 
    = e^{-\frac12\|\kappa\|_2^2}
     \{\dettwo(I+B_{c(\kappa)})\}^{-\frac12}
     \int_{\mathcal{W}} 
      f(\iota+F_{\widehat{c^\prime(\kappa)}})
       e^{\D^*[(I+B_{c^\prime(\kappa)})^{-1}h]} d\mu
\]
for every $f\in C_b(\mathcal{W})$.
Remembering that $\dettwo(I+B)=\det(I+B)e^{-\tr B}$ for 
$B\in\htwo$ of trace class, we obtain the second identity. 

The third and fourth identities follows from the first and
second ones, respectively.
\end{proof}

In Proposition~\ref{p.h.osc.inv}, we showed that 
$\mathfrak{h}(\kappa)$ relates to a quadratic form $\q_\eta$
with non-negative definite $B_\eta\in \mathcal{S}(\htwo)$.
We shall see that $\mathfrak{h}(\eta)$ for $\eta\in\stwo$
arises as the square of the $\mathcal{H}$-norm of the
$\mathcal{H}$-derivative of $\q_\eta$.

\begin{proposition}\label{p.h.eta}
Let $\eta\in\stwo$.
It holds that
\[
    \mathfrak{h}(\eta)
    =\frac12 \|\D\q_\eta\|_{\mathcal{H}}^2.
\]
Furthermore, it holds that
\[
    \int_{\mathcal{W}} f 
      e^{-\frac12\|\D\q_\eta\|_{\mathcal{H}}^2+\D^*h} d\mu
     =\{\dettwo(I+B_\eta^2)\}^{-\frac12}
       \int_{\mathcal{W}} 
        f(\iota+F_{\widehat{c^\prime(\eta)}})
        e^{\D^*[(I+B_{c^\prime(\eta)})^{-1}h]} d\mu
\]
and
\[
    \int_{\mathcal{W}} 
      e^{-\frac12\|\D\q_\eta\|_{\mathcal{H}}^2+\D^*h} d\mu
      = \{\dettwo(I+B_\eta^2)\}^{-\frac12}
        e^{\frac12\langle (I+B_\eta^2)^{-1}h,
                    h\rangle_{\mathcal{H}}}
\]
for every $f\in C_b(\mathcal{W})$ and $h\in\mathcal{H}$.
\end{proposition}

\begin{proof}
The second assertion follows from the first one and
Proposition~\ref{p.h.osc.lap}.
As was seen in the proofs of Theorem~\ref{t.q.eta},
$\D^2\q_\eta=B_\eta$ and $\int_{\mathcal{W}}\D\q_\eta d\mu=0$.
Recalling that $\D F_\eta=B_\eta$ (see \eqref{l.f.kappa.21}), we have
that
\[
    \D(\D \q_\eta-F_\eta)=B_\eta-B_\eta=0.
\]
Furthermore, $\int_{\mathcal{W}} F_\eta d\mu=0$, and hence 
$\int_{\mathcal{W}} (\D\q_\eta-F_\eta) d\mu=0$.
By Lemma~\ref{l.D=0}, we obtain that 
\[
    \D \q_\eta=F_\eta.
\]

By the identity $\D F_\eta=B_\eta$, we see that 
\[
    \biggl\langle 
       \D\biggl(\frac12 \|\D\q_\eta\|_{\mathcal{H}}^2
       \biggr),g\biggr\rangle_{\mathcal{H}}
      =\langle F_\eta,B_\eta g\rangle_{\mathcal{H}}
\]
and
\[
    \biggl\langle 
       \D^2\biggl(\frac12 \|\D\q_\eta\|_{\mathcal{H}}^2
       \biggr),h\otimes g\biggr\rangle_{\htwo}
      =\langle B_\eta h,B_\eta g\rangle_{\mathcal{H}}
      =\langle B_{c(\eta)}, h\otimes g\rangle_{\htwo}
\]
for every $h,g\in \mathcal{H}$.
Noticing that 
\[
    \int_{\mathcal{W}}\|\D\q_\eta\|_{\mathcal{H}}^2 d\mu
    =\int_{\mathcal{W}}\|F_\eta\|_{\mathcal{H}}^2 d\mu
    =\|\eta\|_2^2,
\]
and applying Lemma~\ref{l.D^3=0}, Theorem~\ref{t.q.eta}, and 
Proposition~\ref{p.h.osc}, we obtain that
\[
    \frac12 \|\D\q_\eta\|_{\mathcal{H}}^2
    =\frac12 (\D^*)^2B_{c(\eta)}+\frac12\|\eta\|_2^2
    =\q_{c(\eta)}+\frac12\|\eta\|_2^2
    =\mathfrak{h}(\eta).
\qedhere
\]
\end{proof}

\begin{remark}
In \cite{st-ecp}, the author investigated
$\bigl\|\D\bigl(\int_0^T|\theta(t)|^2 dt\bigr)
 \bigr\|_{\mathcal{H}}^2$ 
when $d=1$.
The above proposition is its extension to general 
$\q_\eta$s. 
Heuristically speaking, the derivative of quadratic form is
of order one, and the square of its $\mathcal{H}$-norm must
be of order two, i.e., quadratic.
\end{remark}

\section{Derivation of the Girsanov theorem}
\label{sec.girsanov}
Let $\rho\in \stwo$.
Set 
$\kappa=\kappa_A(\rho)$ and 
$\eta=\eta(\kappa)$.
Using Lemma~\ref{l.d*f.kappa}, we shall see that the
Girsanov theorem is derived from the transformation
$\iota+F_\kappa$.

Since 
$\|\kappa\|_2^2=\frac12\|\rho\|_2^2$, 
by \eqref{eq:transf} and Lemma~\ref{l.dettwo},
we have that
\begin{equation}\label{eq:girsanov}
    \int_{\mathcal{W}} f(\iota+F_\kappa)
     e^{\q_\eta} d\mu
    =e^{\frac14\|\rho\|_2^2} \int_{\mathcal{W}} f d\mu
    \quad\text{for every }f\in C_b(\mathcal{W}).
\end{equation}
Setting  
\[
    u(t)=-\int_0^t \rho(t,s)d\theta(s)
    \quad\text{for }t\in[0,T],
\]
we have that 
\[
    (F_\kappa)^\prime(t)=u(t)
    \quad\text{for }t\in[0,T].
\]
Since the stochastic process $\{u(t)\}_{t\in[0,T]}$ is
adapted, by \cite[Theorem\,5.3.3]{mt-cambridge},
we know that
\[
    \D^*F_\kappa=\int_0^T\langle u(t),d\theta(t)\rangle.
\]
Note that
$\|F_\kappa\|_{\mathcal{H}}^2=\int_0^T|u(t)|^2dt$.
By Lemma~\ref{l.d*f.kappa}, the above identity 
\eqref{eq:girsanov} is rewritten as 
\[
    \int_{\mathcal{W}} 
     f\Bigl(\theta+\int_0^\bullet u(t)dt\Bigr)
     \exp\biggl(
      -\int_0^T\langle u(t),d\theta(t)\rangle
      -\frac12 \int_0^T|u(t)|^2dt\biggr) d\mu
    =\int_{\mathcal{W}} f d\mu
\]
for every $f\in C_b(\mathcal{W})$.
That is, the distribution of $\theta+\int_0^\bullet u(t)dt$
under the probability measure 
\[
    \exp\biggl(
      -\int_0^T\langle u(t),d\theta(t)\rangle
      -\frac12 \int_0^T|u(t)|^2dt\biggr) d\mu
\]
coincides with the Wiener measure $\mu$. 
Thus the Girsanov theorem follows.

The Novikov condition is one of well-known sufficient
conditions for the Girsanov theorem to hold.
The condition is that 
$\exp(\frac12\int_0^T|u(t)|^2dt)\in L^1(\mu)$.
We show that such exponential integrability holds 
if and only if $\|B_\kappa\|_\op<1$.

To see this, note that
\[
    \frac12 \int_0^T|u(t)|^2dt
    =\frac12\int_0^T \biggl|\int_0^T 
        \kappa(t,s)d\theta(s)\biggr|^2 dt.
\]
In conjunction with Proposition~\ref{p.h.osc}, we have that 
\[
    \frac12 \int_0^T|u(t)|^2dt
    =\mathfrak{q}_{c(\kappa)}+\frac12\|\kappa\|_2^2.
\]
Since $B_{c(\kappa)}=B_\kappa^* B_\kappa$, we have that
\[
    \Lambda(B_{c(\kappa)})
    =\sup_{\|h\|_{\mathcal{H}}=1}
      \|B_\kappa h\|_{\mathcal{H}}^2
    =\|B_\kappa\|_\op^2.
\]
Due to Lemma~\ref{l.q.eta.int} for
$\mathfrak{q}_{c(\kappa)}$, we obtain the desired
equivalence. 

The Kazamaki condition is another well-known sufficient
condition for the Girsanov theorem to hold.
The condition is that 
$\bigl\{\frac12\int_0^r\langle u(t),d\theta(t)\rangle
  \bigr\}_{r\in[0,T]}$ is uniformly integrable.
We show that such uniform integrability holds if and only if 
$\Lambda(B_\rho)<2$.

To see this, for each $r\in[0,T]$, define 
$\beta_r\in \mathcal{S}_2$ by 
\[
    \beta_r(t,s)=\frac12
    \boldsymbol{1}_{[0,r)}(t)\boldsymbol{1}_{[0,r)}(s)
    \rho(t,s)
    \quad\text{for }(t,s)\in[0,T]^2.
\]
Then it holds that
\[
    \frac12\int_0^r\langle u(t),d\theta(t)\rangle
     =\mathfrak{q}_{\beta_r}.
\]
If 
$\frac12\int_0^T\langle u(t),d\theta(t)\rangle$ is
integrable, then, by Lemma~\ref{l.q.eta.int},
$\Lambda(B_{\beta_T})<1$. 
Thus $\Lambda(B_\rho)=2\Lambda(B_{\beta_T})<2$.
Conversely, suppose that $\Lambda(B_\rho)<2$, and hence
$\Lambda(B_{\beta_T})<1$.
Take a $p\in(1,\infty)$ with $p\Lambda(B_{\beta_T})<1$.
Since 
$\Lambda(B_{p\beta_r})\le \Lambda(B_{p\beta_T})
 =p\Lambda(B_{\beta_T}))<1$ 
and $\|\beta_r\|_2\le\|\beta_T\|_2$ for $r\in[0,T]$, by 
Lemma~\ref{l.q.eta.int}, we have that
\[
    \int_{\mathcal{W}} \exp[p \mathfrak{q}_{\beta_r}] d\mu
    \le \exp\biggl[\frac12\biggl\{\frac12
      +\frac{p\Lambda(B_{\beta_T})}{
         3(1-p\Lambda(B_{\beta_T}))^3}
      \biggr\} \|\beta_T\|_2^2 \biggr].
\]
Thus 
$\bigl\{\frac12\int_0^r\langle u(t),d\theta(t)\rangle
  \bigr\}_{r\in[0,T]}$ is uniformly integrable.

Since $B_\rho=B_\kappa+B_\kappa^*$, 
$\Lambda(B_\rho)\le 2\Lambda(B_\kappa)$.
Hence $\Lambda(B_\rho)<2$ if $\|B_\kappa\|_{\text{\rm op}}<1$.
Thus the Novikov condition implies the Kazamaki one.
Such an implication is widely known
(cf. \cite[Proposition~VIII(1.15)]{RY}).
We have rediscovered it via $\iota+F_\kappa$. 

It is natural to ask which adapted $\mathbb{R}^d$-valued
stochastic process $\{u(t)\}_{t\in[0,T]}$ with the property 
that
$\int_{\mathcal{W}} \bigl(\int_0^T |u(t)|^2 dt 
 \bigr)d\mu <\infty$
is obtained via an adapted transformation of order one as
above.
To see this, first assume that each component of $u(t)$
belongs to the Wiener chaos of order one for every
$t\in[0,T]$. 
Then there is a 
$\kappa_t:[0,T]\to \mathbb{R}^{d\times d}$ such that
\[
    \int_0^T|\kappa_t(s)|^2 ds<\infty
    \quad\text{and}\quad
    u(t)=\int_0^T \kappa_t(s)d\theta(s).
\]
Suppose in addition that the mapping 
$(t,s)\mapsto \kappa_t(s)$ is measurable.
It then holds that
\[
    \int_0^T\int_0^T |\kappa_t(s)|^2 dsdt
    =\int_0^T \biggl(\int_{\mathcal{W}} |u(t)|^2 d\mu 
      \biggr)dt<\infty.
\]
Define $\rho\in\stwo$ as
\[
    \rho(t,s)=\one_{[0,t)}(s)\kappa_t(s)
    +\one_{(t,T]}(s)\kappa_s(t)^\dagger
    \quad\text{for }(t,s)\in[0,T]^2.
\]
Since $\{u(t)\}_{t\in[0,T]}$ is adapted, we have that
\[
    u(t)=\int_0^t \kappa_t(s)d\theta(s)
        =\int_0^t \rho(t,s)d\theta(s).
\]
Thus $\{u(t)\}_{t\in[0,T]}$ corresponds to the adapted
transformation $\iota+F_{\kappa_A(\rho)}$ of order one.

\chapter{Linear adapted transformations}
\label{chap.lin.adapted}

\begin{quote}
{\small
We investigate adapted transformations which are also
continuous linear operators on $\mathcal{W}$.
We shall show the equivalence between the exponential
integrability of quadratic from and the solvability of the
associated Riccati ODE. 
The equivalence gives an explicit expression of Laplace
transformations of quadratic forms via ODEs.
The expression is applied to stochastic representations of
analytic or algebraic objects: (i)~heat kernels on two-step
nilpotent Lie groups, (ii)~Euler, Bernoulli, and Eulerian 
polynomials, and (iii)~reflectionless potentials and soliton
solution to the KdV equation.
}
\end{quote}

\section{Linear adapted transformation}
\label{sec.lin.adapted}
Let $C([0,T];\mathbb{R}^{d\times d})$ be the space of
$\mathbb{R}^{d\times d}$-valued continuous functions on
$[0,T]$. 
Define the Wiener functionals 
$G_\phi:\mathcal{W}\to \mathcal{H}$ and 
$\p_\phi:\mathcal{W}\to \mathbb{R}$ 
for $\phi\in C([0,T];\mathbb{R}^{d\times d})$ by
\[
    \langle G_\phi,h\rangle_{\mathcal{H}}
      =-\int_0^T\langle \phi(t)\theta(t),h^\prime(t)
           \rangle dt
      \quad\text{for }h\in \mathcal{H}
    \quad\text{and}\quad
    \p_\phi=\int_0^T\langle\phi(t)\theta(t),
        d\theta(t)\rangle.
\]
$G_\phi$ is also thought of as a continuous linear mapping
of $\mathcal{W}$ to $\mathcal{W}$ such that
\[
    (G_\phi(w))(t)=-\int_0^t \phi(s)w(s)ds
    \quad\text{for $w\in \mathcal{W}$ and $t\in[0,T]$}.
\]
When we think of $G_\phi$ as a continuous linear mapping, we
write simply as $G_\phi w$ for $G_\phi(w)$. 
Define $\rho_\phi\in\stwo$ by
\[
    \rho_\phi(t,s)
    =\one_{[0,t)}(s)\phi(t)
     +\one_{(t,T]}(s)\phi(s)^\dagger
    \quad\text{for } (t,s)\in[0,T]^2.
\]
It is easily seen that
\[
     G_\phi=F_{\kappa_A(\rho_\phi)}
     \quad\text{and}\quad
     \p_\phi=\q_{\rho_\phi}.
\]
Hence we call $\iota+G_\phi$ a 
{\it linear adapted transformation} of order one.
In addition, define 
$\sigma(\phi)\in C([0,T];\mathbb{R}^{d\times d})$ by
\begin{equation}\label{eq.sigma.chi}
    \sigma(\phi)(t)
    =\phi(t)-\int_t^T \phi(u)^\dagger\phi(u)du
    \quad\text{for }t\in[0,T].    
\end{equation}

Theorem~\ref{t.tways} yields the following.

\begin{theorem}\label{t.lap.lin}
For each $\chi\in C([0,T];\mathbb{R}^{d\times d})$, 
$e^{\p_{\sigma(\chi)}}\in L^1(\mu)$ and it holds that
\begin{equation}\label{t.lap.lin.3}
    \int_{\mathcal{W}} f e^{\p_{\sigma(\chi)}} d\mu
    =e^{\frac12\int_0^T\tr[(\chi-\sigma(\chi))(t)]dt}
     \int_{\mathcal{W}} 
        f(\iota+F_{\widehat{\kappa_A(\rho_\chi)}}) d\mu
\end{equation}
and 
\begin{equation}\label{t.lap.lin.4}
    \int_{\mathcal{W}} f 
        \exp\biggl(\int_0^T \langle\chi(t)\theta(t),
            d\theta(t)\rangle 
           -\frac12 \int_0^T |\chi(t)\theta(t)|^2 dt
       \biggr) d\mu
      =     \int_{\mathcal{W}} 
        f(\iota+F_{\widehat{\kappa_A(\rho_\chi)}}) d\mu
\end{equation}
for every $f\in C_b(\mathcal{W})$.
\end{theorem}

\begin{proof}
It holds that
\[
    \eta(\kappa_A(\rho_\chi))(t,s)
    =\rho_\chi(t,s)
       -\int_t^T \rho_\chi(t,u)\rho_\chi(u,s)du
    =\chi(t)-\int_t^T \chi(u)^\dagger\chi(u)du
    =\sigma(\chi)(t)
\]
for $0\le s<t\le T$.
Hence we have that
\[
    \eta(\kappa_A(\rho_\chi))\eqltwo \rho_{\sigma(\chi)}.
\]
Since $\p_{\sigma(\chi)}=\q_{\rho_{\sigma(\chi)}}$, by
Lemmas~\ref{l.dettwo} and \ref{l.q.eta.int}, we have that 
$e^{\p_{\sigma(\chi)}}\in L^1(\mu)$.

Since $\rho_\chi\in\stwo$, by \eqref{eq.sigma.chi}, we have
that 
\begin{align*}
    \|\rho_\chi\|_2^2
    & =2\int_0^T \biggl(\int_t^T |\rho_\chi(t,s)|^2 
        ds\biggr) dt
      =2\int_0^T \biggl(\int_t^T |\chi(s)|^2 ds\biggr)dt
    \\
    & =2\int_0^T \tr[(\chi-\sigma(\chi))(t)]dt.
\end{align*}
Plugging these into \eqref{t.tways.4} with 
$\eta=\rho_{\sigma(\chi)}$ and $\rho=\rho_\chi$, we obtain
\eqref{t.lap.lin.3}. 

By It\^o's formula and the definition of
$\sigma(\chi)$, we know that 
\begin{align*}
    & \int_0^T \biggl\langle \biggl(\int_t^T
      \chi(u)^\dagger \chi(u)du\biggr)\theta(t),
      d\theta(t)\biggr\rangle
    \\
    & = \frac12 \int_0^T |\chi(t)\theta(t)|^2 dt
        -\frac12 \int_0^T \tr\biggl[\int_t^T
          \chi(u)^\dagger \chi(u)du\biggr] dt
    \\
    & = \frac12 \int_0^T |\chi(t)\theta(t)|^2 dt
        -\frac12 \int_0^T \tr\bigl[
          (\chi-\sigma(\chi))(t)\bigr] dt.
\end{align*}
Hence we have that
\[
    \p_{\sigma(\chi)}
    =\int_0^T \langle \chi(t)\theta(t),d\theta(t)
       \rangle
     -\frac12 \int_0^T |\chi(t)\theta(t)|^2 dt
     +\frac12 \int_0^T \tr\bigl[
          (\chi-\sigma(\chi))(t)\bigr] dt.
\]
Plugging this into \eqref{t.lap.lin.3}, we obtain
\eqref{t.lap.lin.4}. 
\end{proof}

By this theorem, the evaluation of the Laplace
transformation of quadratic form $\p_\sigma$ comes down to 
finding $\chi$ with $\sigma(\chi)=\sigma$.
Before getting into this problem, we present an application
of Theorem~\ref{t.lap.lin} to the integral
$e^{\p_{\sigma(\chi)}}$ with respect to the pinned measure.
On the way, we also see that the transformation
$\iota+F_{\widehat{\kappa_A(\rho_\chi)}}$ has a more
explicit expression.

We make a brief review on pinned measures.
For details, see \cite[\S5.4]{mt-cambridge}.
Put
\[
    \mathbb{D}^{-\infty,\infty}
    =\bigcup_{k\in \mathbb{N}\cup\{0\}}
     \bigcap_{p\in(1,\infty)} 
     (\mathbb{D}^{k,p})^\prime
    \quad\text{and}\quad
    \mathbb{D}^{\infty,1+}
     =\bigcap_{k\in \mathbb{N}\cup\{0\}}
      \bigcup_{p\in(1,\infty)} \mathbb{D}^{k,p},
\]
where $(\mathbb{D}^{k,p})^\prime$ is the dual space of 
$\mathbb{D}^{k,p}$.
For $\xi\in \mathbb{D}^{-\infty,\infty}$, we denote by $\xi[\Psi]$ its
value at $\Psi\in \mathbb{D}^{\infty,1+}$.
Each $\Phi\in \mathbb{D}^\infty$ can be regarded as an element of
$\mathbb{D}^{\infty,1+}$ via the $L^2$-inner product:
\[
    \Phi[\Psi]=\int_{\mathcal{W}} \Psi\Phi d\mu
    \quad\text{for }\Psi\in \mathbb{D}^{\infty,1+}\subset L^{1+}(\mu).
\]
Let $N\in \mathbb{N}$,
$\mathscr{S}(\mathbb{R}^N)$ be the space of rapidly
decreasing functions on $\mathbb{R}^N$, and 
$\mathscr{S}^\prime(\mathbb{R}^N)$ the space of
tempered distributions on $\mathbb{R}^N$.
Assume that 
$\Phi=(\Phi^1,\dots,\Phi^N)
 \in \mathbb{D}^\infty(\mathbb{R}^N)$ is {\it non-degenerate}, 
that is,  
\[
    \Bigl\{\det\bigl[\bigl(\langle \D\Phi^i,\D\Phi^j
    \rangle_{\mathcal{H}}\bigr)_{1\le i,j\le N}\bigr]
    \Bigr\}^{-1}
    \in \bigcap\limits_{p\in(1,\infty)} L^p(\mu).
\]
The mapping 
$\mathscr{S}(\mathbb{R}^N)\ni u
 \mapsto u(\Phi)\in \mathbb{D}^{-\infty,\infty}$ can be 
extended to a continuous mapping 
$\mathfrak{C}:\mathscr{S}^\prime(\mathbb{R}^N)
 \to \mathbb{D}^{-\infty,\infty}$.
Inheriting the notation for $u\in\mathscr{S}(\mathbb{R}^N)$,
we denote as $u(\Phi)$ for the value $\mathfrak{C}(u)$ at 
$u\in \mathscr{S}^\prime(\mathbb{R}^N)$, and write the value of
$u(\Phi)$ at $\Psi\in \mathbb{D}^{\infty,1+}$ as
\[
    \int_{\mathcal{W}} \Psi u(\Phi)d\mu.
\]
We call $u(\Phi)$ the {\it pull-back} of $u$ by $\Phi$.
Furthermore, we Observe that
\begin{equation}\label{eq.e^a}
    fe^{\mathfrak{a}}\in \mathbb{D}^{\infty,1+}
    \quad\text{for }f\in \mathbb{D}^\infty
    \text{ and }\mathfrak{a}\in \mathbb{D}^\infty
    \text{ with }e^{\mathfrak{a}}\in L^{1+}(\mu).
\end{equation}
In fact, by the chain rule for $\D$ 
(\cite[Theorem\,5.2.8 and Corollary\,5.3.2]{mt-cambridge}),
we can express the higher $\mathcal{H}$-derivatives of 
$e^{\mathfrak{a}}$ as
\[
    \D^k e^{\mathfrak{a}}=e^{\mathfrak{a}}
     \sum_{\substack{j_1,\dots,j_k\ge0 \\ j_1+\dots+j_k=k}}
      c_{j_1,\dots,j_k} (\D^{j_1} \mathfrak{a} \otimes
      \cdots \otimes \D^{j_k} \mathfrak{a})
    \quad\text{with } c_{j_1,\dots,j_k}\in \mathbb{R}.    
\]
Hence 
$e^{\mathfrak{a}}\in\mathbb{D}^{\infty,1+}$, and so does 
$fe^{\mathfrak{a}}$.
Letting $\delta_z$ be the Dirac measure on $\mathbb{R}^N$ 
concentrated at $z\in \mathbb{R}^N$, we know that
$\delta_z(\Phi)\in \bigcap\limits_{p\in(1,\infty)}
 (\mathbb{D}^{2N+2,p})^\prime$ 
and $\delta_z(\Phi)$ determines a finite measure on
$\mathcal{W}$ 
(\cite[Theorem\,5.4.15]{mt-cambridge}).
The measure is also written as $\delta_z(\Phi)d\mu$, and 
called the {\it pinned measure} at $z$ by $\Phi$ 
(cf.\cite[Example\,5.4.17]{mt-cambridge}). 

Let $C^k([0,T];\mathbb{R}^{d\times d})$ for $k\in \mathbb{N}$ be
the space of $\mathbb{R}^{d\times d}$-valued $k$-times
continuously differentiable functions on $[0,T]$.
Define the function $g_V:\mathbb{R}^d\to \mathbb{R}$ for symmetric and
positive definite $V\in \mathbb{R}^{d\times d}$ as
\[
    g_V(x)=\frac1{\sqrt{(2\pi)^d \det V}}
           e^{-\frac12\langle V^{-1}x,x\rangle}
    \quad\text{for } x\in \mathbb{R}^d.
\]

\begin{theorem}\label{t.alpha}
Let $\chi\in C([0,T];\mathbb{R}^{d\times d})$.
Assume that there is an 
$\alpha\in C^1([0,T];\mathbb{R}^{d\times d})$ with the
property that
\begin{equation}\label{t.alpha.0}
    \alpha(T)=I_d,
    \quad
    \det \alpha(t)\ne0
    ~~\text{for every }t\in[0,T],
    \quad\text{and}\quad 
    \chi=\alpha^\prime \alpha^{-1}.
\end{equation}
Then the following holds.
\\
{\rm(i)}
Define 
$\boldsymbol{\xi}_\alpha:\mathcal{W}\to \mathcal{W}$ by  
\[
    (\boldsymbol{\xi}_\alpha(\cdot))(t)
    =\alpha(t)\int_0^t \alpha(s)^{-1} d\theta(s)
    \quad\text{for }t\in[0,T],
\]
where $(\boldsymbol{\xi}_\alpha(w))(t)$ is the evaluation of 
$\boldsymbol{\xi}_\alpha(w)\in \mathcal{W}$ at $t\in[0,T]$
for $w\in \mathcal{W}$.
Then it holds that
\begin{equation}\label{t.alpha.1}
    \int_{\mathcal{W}} f e^{\p_{\sigma(\chi)}+\D^*h} d\mu
    =\biggl(\frac{e^{-\int_0^T\tr[\sigma(\chi)(t)]dt}}{
           \det \alpha(0)}\biggr)^{\frac12}
     \int_{\mathcal{W}} f(\boldsymbol{\xi}_\alpha) 
       e^{\int_0^T \langle h^\prime(t),
             d\xi_\alpha(t)\rangle}
       d\mu
\end{equation}
for every $f\in C_b(\mathcal{W})$ and $h\in \mathcal{H}$,
where $\{\xi_\alpha(t)\}_{t\in[0,T]}$ is the stochastic
process given as
$\xi_\alpha(t)=(\boldsymbol{\xi}_\alpha(\cdot))(t)$ 
and $d\xi_\alpha(t)$ stands for its It\^o integral.
\\
{\rm(ii)}
Let $t\in(0,T]$ and put
\[
    v_t(\alpha)=\int_0^t (\alpha(t)\alpha(s)^{-1})
      (\alpha(t)\alpha(s)^{-1})^\dagger ds.
\]
Then it holds that
\begin{equation}\label{t.alpha.2}
      \int_{\mathcal{W}} e^{\p_{\sigma(\chi)}}
       \delta_y(\theta(t))d\mu
     =\biggl(\frac{e^{-\int_0^T\tr[\sigma(\chi)(t)]dt}}{
          \det \alpha(0)}\biggr)^{\frac12}
      g_{v_t(\alpha)}(y) 
     \quad\text{for }y\in \mathbb{R}^d.
\end{equation}
\end{theorem}

\begin{remark}\label{r.alpha}
(i) Due to \eqref{t.alpha.0} and continuity of $\alpha$, 
$\det \alpha(t)>0$ for any $t\in[0,T]$.
Hence $\{\det \alpha(0)\}^{\frac12}$ is a real number.
\\
(ii) 
By It\^o's formula and the relation between $\alpha$ and $\chi$, 
we have that
\begin{equation}\label{eq:xi.alpha}
    d\xi_\alpha(t)
    =d\theta(t)+\chi(t)\xi_\alpha(t)dt.
\end{equation}
Hence the stochastic process
$\{\xi_\alpha(t)\}_{t\in[0,T]}$ is an It\^o process, and 
its It\^o integral is well defined.
\end{remark}

\begin{proof}
(i)
Define the continuous linear operator 
$A_\alpha:\mathcal{W}\to \mathcal{W}$ by 
\[
    (A_\alpha w)(t)=-\alpha(t)\int_0^t 
     (\alpha^{-1})^\prime(s) w(s)ds
     \quad\text{for }t\in[0,T]\text{ and }
     w\in \mathcal{W}.
\]
We shall show that $\iota+A_\alpha$ is the continuous
inverse of $\iota+G_\chi$.
To do so, let $h\in \mathcal{H}$.
By the integration by parts on $[0,T]$, we have that
\[
    (A_\alpha h)(t)=-h(t)+\alpha(t)\int_0^t 
        \alpha(s)^{-1} h^\prime(s)ds
    \quad\text{for } t\in[0,T].
\]
Hence we obtain that
\[
    \bigl((\iota+A_\alpha)h\bigr)(t)
    =\alpha(t)\int_0^t \alpha(s)^{-1}h^\prime(s)ds
    \quad\text{for } t\in[0,T].
\]
Since 
\[
    (G_\chi g)(t)
     =-\int_0^t \alpha^\prime(s)\alpha(s)^{-1} g(s)ds
    \quad\text{for }g\in \mathcal{H}
    \text{ and }t\in[0,T],
\]
we see that
\begin{align*}
    \bigl(\bigl[(\iota+A_\alpha)\circ
      (\iota+G_\chi)\bigr]h\bigr)(t)
    &  =\alpha(t)\int_0^t \alpha(s)^{-1}
       \bigl\{h^\prime(s)
        -\alpha^\prime(s)\alpha^{-1}(s)h(s)\bigr\}
       ds
    \\
    & =\alpha(t)\int_0^t (\alpha^{-1}h)^\prime(s)ds
      =h(t)
\end{align*}
and 
\begin{align*}
    \bigl(\bigl[(\iota+G_\chi)\circ
      (\iota+A_\alpha)\bigr]h\bigr)(t)
    & =\alpha(t)\int_0^t \alpha^{-1}(s)h^\prime(s)ds
       -\int_0^t \alpha^\prime(s)\biggl(
          \int_0^s \alpha^{-1}(u)h^\prime(u)du
          \biggr) ds
    \\
    & =\int_0^t h^\prime(u)du
      =h(t)
\end{align*}
for $t\in[0,T]$.
Thus it holds that
\[
    \bigl((\iota+A_\alpha)\circ(\iota+G_\chi)\bigr)h
     =\bigl((\iota+G_\chi)\circ(\iota+A_\alpha)\bigr)h=h
    \quad\text{for every }h\in \mathcal{H}.
\]
Due to the continuity of $A_\alpha$ and $G_\chi$ on
$\mathcal{W}$, these equalities continue to hold on
$\mathcal{W}$, i.e., we have that
\[
    \bigl((\iota+A_\alpha)\circ(\iota+G_\chi)\bigr)w
     =\bigl((\iota+G_\chi)\circ(\iota+A_\alpha)\bigr)w=w
    \quad\text{for every }w\in \mathcal{W}.
\]

By Theorem~\ref{t.inv.transf}, it holds that
\[
    (\iota+F_{\widehat{\kappa_A(\rho_\chi)}})\circ
      (\iota+G_\chi)
    =(\iota+G_\chi)\circ
      (\iota+F_{\widehat{\kappa_A(\rho_\chi)}})
    =\iota
    \quad\text{$\mu$-a.s.}
\]
Combining this with the above identity, we know that
\[
    \iota+F_{\widehat{\kappa_A(\rho_\chi)}}
    =\iota+A_\alpha
    \quad\text{$\mu$-a.s.}
\]
By It\^o's formula, we know that
\[
    \int_0^t (\alpha^{-1})^\prime(s)\theta(s)ds
    =\alpha(t)^{-1}\theta(t)
     -\int_0^t \alpha(s)^{-1} d\theta(s)
    \quad\text{for }t\in[0,T].
\]
We then have that
\begin{equation}\label{t.alpha.21}
    \iota+F_{\widehat{\kappa_A(\rho_\chi)}}
    =\boldsymbol{\xi}_\alpha.
\end{equation}

Since
\[
    (\det \alpha)^\prime
    =(\tr[\alpha^\prime \alpha^{-1}])\det \alpha 
    =(\tr \chi)\det \alpha,
    \quad 
    \det \alpha(T)=1,
\]
we have that
\begin{equation}\label{t.alpha.22}
    \det \alpha(0)=e^{-\int_0^T \tr \chi(t) dt}.
\end{equation}

If $h\in \mathcal{H}$ is of the form
\[
    h^\prime=\sum_{n=0}^{N-1} \one_{[t_n,t_{n+1})} e_n
\]
with $0=t_0<t_1<\dots<t_N=T$ and 
$e_0,\dots,e_{N-1}\in \mathbb{R}^d$, then 
\[
    \D^*h=\sum_{n=0}^{N-1} \langle e_n,
      \theta(t_{n+1})-\theta(t_n)\rangle.
\]
This implies that
\[
    (\D^*h)(\boldsymbol{\xi}_\alpha)
    =\sum_{n=0}^{N-1} \langle e_n,
      \xi_\alpha(t_{n+1})-\xi_\alpha(t_n)\rangle
    =\int_0^T \langle h^\prime(t),d\xi_\alpha(t)\rangle.
\]
Hence, by the standard limiting procedure, we obtain that 
\[
    (\D^*h)(\boldsymbol{\xi}_\alpha)
    =\int_0^T \langle h^\prime(t),d\xi_\alpha(t)\rangle
    \quad\text{for each }h\in \mathcal{H}.
\]
Plugging this, \eqref{t.alpha.21}, and \eqref{t.alpha.22}
into \eqref{t.lap.lin.3}, we arrive at \eqref{t.alpha.1}. 

\noindent
(ii) 
Let $t\in(0,T]$.
Note that $v_t(\alpha)$ is symmetric and positive
definite.
Hence $\xi_\alpha(t)=(\boldsymbol{\xi}_\alpha(\cdot))(t)$ obeys 
the normal distribution on $\mathbb{R}^d$ with mean $0$ and
covariance matrix $v_t(\alpha)$.
By \eqref{t.alpha.1} and
\cite[Theorem\,5.4.11]{mt-cambridge}, we have that 
\begin{align*}
    & \int_{\mathbb{R}^d} \varphi(y)
        \biggl(\int_{\mathcal{W}} e^{\p_{\sigma(\chi)}}
         \delta_y(\theta(t)) d\mu \biggr) dy
      =\int_{\mathcal{W}} \varphi(\theta(t))
       e^{\p_{\sigma(\chi)}} d\mu
    \\
    & =\biggl(\frac{e^{-\int_0^T\tr[\sigma(\chi)(t)]dt}}{
         \det \alpha(0)}\biggr)^{\frac12}
       \int_{\mathcal{W}} 
         \varphi(\xi_\alpha(t)) d\mu
      =\biggl(\frac{e^{-\int_0^T\tr[\sigma(\chi)(t)]dt}}{
         \det \alpha(0)}\biggr)^{\frac12}
       \int_{\mathbb{R}^d} \varphi(y) g_{v_t(\alpha)}(y)dy
\end{align*}
for every $\varphi\in \mathscr{S}(\mathbb{R}^d)$.
Thus \eqref{t.alpha.2} holds.
\end{proof}

\begin{remark}\label{r.xi.alpha}
As was seen in the proof of Theorem~\ref{t.alpha}, the
identity \eqref{t.alpha.2} is an immediate consequence of 
\eqref{t.alpha.1}.
With a little effort, we can extend \eqref{t.alpha.2} to
more general integrals with respect to pinned measures.
To see this, we continue to work under the same assumption 
as in Theorem~\ref{t.alpha}.
Let $N\in \mathbb{N}$ and take
$k_0\in \mathcal{H}$ and
$\boldsymbol{k}=(k_1,\dots,k_N)\in \mathcal{H}^N$ with
$\det C(\boldsymbol{k})\ne0$, where 
$C(\boldsymbol{k})=\bigl(
  \langle k_i,k_j\rangle_{\mathcal{H}}
  \bigr)_{1\le i,j\le N}$.
Put $\D^*\boldsymbol{k}=(\D^*k_1,\dots,\D^*k_N)$.
Define $\kappa[\alpha]\in\ltwo$ by
\[
    \kappa[\alpha](t,s)
    =\one_{[0,t)}(s) \alpha^\prime(t)\biggl\{
       \alpha(s)^{-1}+\biggl(\int_0^s 
       \alpha(u)^{-1}\bigl(\alpha(u)^{-1}\bigr)^\dagger
       du\biggr)\alpha^\prime(s)^\dagger
      \biggr\}
     \quad\text{for }(t,s)\in[0,T]^2.
\]
Set
\[
    \Gamma[\alpha]=\Bigl(\bigl\langle(I+B_{\kappa[\alpha]}
         +B_{\kappa[\alpha]}^*)k_p,k_q 
        \bigr\rangle_{\mathcal{H}}
      \Bigr)_{0\le p,q\le N} 
    \in \mathbb{R}^{(N+1)\times(N+1)}.
\]
Denote by $(x^0,x^1,\dots,x^N)$ the coordinate system on
$\mathbb{R}^{N+1}$ and let $P_\alpha(dx^0dx^1\dots dx^N)$ be
the Gaussian distribution on $\mathbb{R}^{N+1}$ obeying
$N(0,\Gamma[\alpha])$.
We have that 
\begin{align*}
   & \biggl(\int_{\mathcal{W}} e^{\q_{\sigma(\chi)}+\D^*k_0}
       \delta_{(x^1,\dots,x^N)}(\D^*\boldsymbol{k}) 
       d\mu\biggr)
      dx^1\cdots dx^N
   \\
   & =\biggl(\frac{e^{-\int_0^T \tr[\sigma(\chi)(t)]dt}}{
        \det \alpha(0)}\biggr)^{\frac12}
       \int_{\mathbb{R}} e^{x^0} P_\alpha(dx^0dx^1\cdots dx^N).
\end{align*}

\begin{proof}
Recall that $\D^*\boldsymbol{k}$ obeys the normal
distribution on $\mathbb{R}^N$ with mean $0$ and covariance
$C(\boldsymbol{k})$.
Since 
$\D^*h(\boldsymbol{\xi}_\alpha) 
 =\int_0^T\langle h^\prime(t),d\xi_\alpha(t)\rangle$ 
for $h\in \mathcal{H}$ as was seen in the proof of
Theorem~\ref{t.alpha}, by
\cite[Theorem\,5.4.11]{mt-cambridge}, it follows from
\eqref{t.alpha.1} that
\begin{align*}
    & \int_{\mathbb{R}^N} \varphi(x^1,\dots,x^N)
       \biggl(\int_{\mathcal{W}} e^{\q_{\sigma(\chi)}+\D^*k_0}
        \delta_{(x^1,\dots,x^N)}(\D^*\boldsymbol{k})d\mu
       \biggr) dx^1\cdots dx^N
    \\
    & =\int_{\mathcal{W}} \varphi(\D^*\boldsymbol{k})
          e^{\q_{\sigma(\chi)}+\D^*k_0} d\mu
    \\
    & =\biggl(\frac{e^{-\int_0^T \tr[\sigma(\chi)(t)]dt}}{
            \det \alpha(0)}\biggr)^{\frac12}
       \int_{\mathcal{W}} \varphi\biggl(
          \int_0^T\langle k_1^\prime(t),
               d\xi_\alpha(t)\rangle, \dots,
          \int_0^T\langle k_N^\prime(t),
               d\xi_\alpha(t)\rangle\biggr)
          e^{\int_0^T\langle k_0^\prime(t),
               d\xi_\alpha(t)\rangle} d\mu
\end{align*}
for any $\varphi\in \mathscr{S}(\mathbb{R}^N)$.
Note that 
$\Bigl(\int_0^T\langle k_p^\prime(t),
  d\xi_\alpha(t)\rangle\Bigr)_{0\le p\le N}$ 
obeys the normal distribution on $\mathbb{R}^{N+1}$ with
mean $0$. 
Thus it suffices to show that
\begin{equation}\label{c.alpha.21}
    \int_{\mathcal{W}}
       \biggl(\int_0^T\langle h^\prime(t),
           d\xi_\alpha(t)\rangle\biggr)
       \biggl(\int_0^T\langle g^\prime(t),
           d\xi_\alpha(t)\rangle\biggr) d\mu
    =\bigl\langle
       (I+B_{\kappa[\alpha]}+B_{\kappa[\alpha]}^*)h,g
      \bigr\rangle_{\mathcal{H}}
\end{equation}
for $h,g\in \mathcal{H}$.

Let $h,g\in \mathcal{H}$.
Using the expression of the stochastic differential of
$\{\xi_\alpha(t)\}_{t\in[0,T]}$ as 
\[
    d\xi_\alpha(t)=d\theta(t)
    + \alpha^\prime(t)\biggl(\int_0^t \alpha(u)^{-1}
        d\theta(u)\biggr) dt,
\]
we know that
\[
    \int_0^t \langle h^\prime(s),
           d\xi_\alpha(s)\rangle
    =\int_0^t \langle h^\prime(s),d\theta(s)\rangle 
     +\int_0^t \biggl\langle h^\prime(s),
        \alpha^\prime(s)\biggl(\int_0^s \alpha(u)^{-1}
         d\theta(u)\biggr) \biggr\rangle ds
\]
for $t\in[0,T]$.
Due to It\^o's formula, we have that
\begin{align*}
    & \biggl(\int_0^T \langle h^\prime(s),
           d\xi_\alpha(s)\rangle \biggr)
      \biggl(\int_0^T \langle g^\prime(s),
           d\xi_\alpha(s)\rangle \biggr)
    \\
    & =\int_0^T \biggl\langle
         \biggl(\int_0^t \langle h^\prime(s),
           d\xi_\alpha(s)\rangle \biggr)g^\prime(t)
         +\biggl(\int_0^t \langle g^\prime(s),
           d\xi_\alpha(s)\rangle \biggr)h^\prime(t),
       d\theta(t) \biggr\rangle
    \\
    & \hphantom{=}
      +\int_0^T \biggl\{
         \biggl(\int_0^t \langle h^\prime(s),
              d\xi_\alpha(s)\rangle \biggr)
         \biggl\langle g^\prime(t),
           \alpha^\prime(t) \int_0^t \alpha(u)^{-1} 
            d\theta(u) \biggr\rangle
    \\
    & \hphantom{= \int_0^T \biggl\{}
      +\biggl(\int_0^t \langle g^\prime(s),
              d\xi_\alpha(s)\rangle \biggr)
         \biggl\langle h^\prime(t),
           \alpha^\prime(t) \int_0^t \alpha(u)^{-1} 
            d\theta(u) \biggr\rangle \biggr\}dt
    \\
    & \hphantom{=}
      +\int_0^T \langle h^\prime(t),
         g^\prime(t)\rangle dt.
\end{align*}
The integration of the first term with respect to $\mu$
vanishes. 
The last term is exactly 
$\langle h,g\rangle_{\mathcal{H}}$.
Thus what to be computed is the Wiener integral of the
middle term.

To compute the middle term, rewrite as
\begin{align*}
    & \biggl(\int_0^t \langle h^\prime(s),
              d\xi_\alpha(s)\rangle \biggr)
         \biggl\langle g^\prime(t),
           \alpha^\prime(t) \int_0^t \alpha(u)^{-1} 
            d\theta(u) \biggr\rangle
    \\
    & 
    = \biggl(\int_0^t \langle h^\prime(s),
              d\theta(s)\rangle \biggr)
         \biggl\langle g^\prime(t),
           \alpha^\prime(t) \int_0^t \alpha(u)^{-1} 
            d\theta(u) \biggr\rangle
    \\
    & \hphantom{=}
      +\biggl(\int_0^t \biggl\langle h^\prime(s),
          \alpha^\prime(s)\int_0^s \alpha(u)^{-1}
          d\theta(u)\biggr\rangle ds \biggr)
         \biggl\langle g^\prime(t),
           \alpha^\prime(t) \int_0^t \alpha(u)^{-1} 
            d\theta(u) \biggr\rangle.
\end{align*}
Thanks to the isometry for It\^o integral, we have that
\begin{align*}
    & \int_{\mathcal{W}}
         \biggl(\int_0^t \langle h^\prime(s),
              d\theta(s)\rangle \biggr)
         \biggl\langle g^\prime(t),
           \alpha^\prime(t) \int_0^t \alpha(u)^{-1} 
            d\theta(u) \biggr\rangle d\mu
    \\
    &
    = \biggl\langle \alpha^\prime(t)^\dagger g^\prime(t),
       \int_0^t \alpha(u)^{-1}h^\prime(u) du
      \biggr\rangle
    =\biggl\langle \int_0^T \one_{[0,t)}(s)
         \alpha^\prime(t) \alpha^{-1}(s) h^\prime(s) ds,
        g^\prime(t) \biggr\rangle
\end{align*}
and 
\begin{align*}
    & \int_{\mathcal{W}} 
       \biggl(\int_0^t \biggl\langle h^\prime(s),
          \alpha^\prime(s)\int_0^s \alpha(u)^{-1}
          d\theta(u)\biggr\rangle ds \biggr)
         \biggl\langle g^\prime(t),
           \alpha^\prime(t) \int_0^t \alpha(u)^{-1} 
            d\theta(u) \biggr\rangle d\mu
    \\
    & = \int_0^t \biggl\langle 
          \alpha^\prime(s)^\dagger h^\prime(s),
          \biggl( \int_0^s \alpha(u)^{-1} 
           (\alpha(u)^{-1})^\dagger du \biggr)
          \alpha^\prime(t)^\dagger g^\prime(t)
          \biggr\rangle ds
    \\
    & =\biggl\langle \int_0^T \one_{[0,t)}(s)
        \alpha^\prime(t) 
          \biggl( \int_0^s \alpha(u)^{-1} 
           (\alpha(u)^{-1})^\dagger du \biggr)
          \alpha^\prime(s)^\dagger h^\prime(s) ds,
        g^\prime(t)\biggr\rangle.
\end{align*}
Hence we obtain that
\[
    \int_{\mathcal{W}} 
      \biggl(\int_0^T 
         \biggl(\int_0^t \langle h^\prime(s),
              d\xi_\alpha(s)\rangle \biggr)
         \biggl\langle g^\prime(t),
           \alpha^\prime(t) \int_0^t \alpha(u)^{-1} 
            d\theta(u) \biggr\rangle dt\biggr)d\mu
    =\langle B_{\kappa[\alpha]}h,g\rangle_{\mathcal{H}}.    
\]
Switching $h$ and $g$, we have that
\[
    \int_{\mathcal{W}} 
      \biggl(\int_0^T 
         \biggl(\int_0^t \langle g^\prime(s),
              d\xi_\alpha(s)\rangle \biggr)
         \biggl\langle h^\prime(t),
           \alpha^\prime(t) \int_0^t \alpha(u)^{-1} 
            d\theta(u) \biggr\rangle dt\biggr)d\mu
    =\langle B_{\kappa[\alpha]}^* h,g\rangle_{\mathcal{H}}.    
\]
Thus we arrive at \eqref{c.alpha.21}.
\end{proof}
\end{remark}

The assumption \eqref{t.alpha.0} is equivalent to that the
unique solution $\alpha$ to the ODE
\begin{equation}\label{eq.ode.alpha}
    \alpha^\prime=\chi \alpha,
    \quad \alpha(T)=I_d    
\end{equation}
satisfies that $\det \alpha(t)\ne0$ for every $t\in[0,T]$.
A na\"{\i}ve sufficient condition for this to hold is given
as follows. 

\begin{proposition}\label{c.lap.lin}
Let $\chi\in C([0,T];\mathbb{R}^{d\times d})$ and assume
that 
\begin{equation}\label{c.lap.lin.0}
    T\sqrt{d}\|\chi\|_\infty e^{T\|\chi\|_\infty }<1,
\end{equation}
where $\|\chi\|_\infty=\sup_{t\in[0,T]}|\chi(t)|$.
Then the solution 
$\alpha\in C^1([0,T];\mathbb{R}^{d\times d})$ to the ODE
\eqref{eq.ode.alpha} satisfies \eqref{t.alpha.0}.
In particular, the identities \eqref{t.alpha.1} and 
\eqref{t.alpha.2} hold.
\end{proposition}

\begin{proof}
It holds that
\[
    |\alpha(T-t)|
    =\biggl|I_d-\int_0^t \chi(T-s) \alpha(T-s) ds
      \biggr|
    \le \sqrt{d}+\|\chi\|_\infty \int_0^t
       |\alpha(T-s)| ds
    \quad\text{for }t\in[0,T].
\]
By Gronwall's inequality, we see  that 
\[
    \|\alpha\|_\infty\le\sqrt{d}e^{T\|\chi\|_\infty}.
\]
Hence we have that
\[
    |I_d-\alpha(t)|=\biggl|\int_t^T 
      \chi(s)\alpha(s)ds\biggr|
    \le T\sqrt{d}\|\chi\|_\infty e^{T\|\chi\|_\infty }<1.
\]
Thus $\det \alpha(t)\ne0$ for any $t\in[0,T]$.
\end{proof}

We proceed to finding 
$\chi\in C([0,T];\mathbb{R}^{d\times d})$ with 
$\sigma(\chi)=\sigma$ for given 
$\sigma\in C([0,T];\mathbb{R}^{d\times d})$.
Such a $\chi$ is constructed by using the Riccati ODE as
follows.

\begin{theorem}\label{t.riccati}
Let $\sigma\in C([0,T];\mathbb{R}^{d\times d})$.
The following three conditions are equivalent.
\\
{\rm(i)}
$e^{\p_\sigma}\in L^1(\mu)$.
\\
{\rm(ii)}
There exists a $\chi\in C([0,T];\mathbb{R}^{d\times d})$
with $\sigma(\chi)=\sigma$.
\\
{\rm(iii)}
There exists an 
$\bR\in C^1([0,T];\mathbb{R}^{d\times d})$
obeying the $\mathbb{R}^{d\times d}$-valued Riccati ODE
\begin{equation}\label{t.riccati.1}
    \bR^\prime=-\bR^2-\sigma^\dagger \bR -\bR\sigma
        -\sigma^\dagger\sigma,
    \quad \bR(T)=0.
\end{equation}

If these conditions hold, then
$\chi=\bR+\sigma$ and it holds that
\begin{equation}\label{t.riccati.2-}
    \int_{\mathcal{W}} f e^{\mathfrak{p}_\sigma} d\mu
    =e^{\int_0^T \tr \bR(t) dt/2}
     \int_{\mathcal{W}}
       f(\iota+F_{\widehat{\kappa_A(\rho_\chi)}})  d\mu
    \quad\text{for every }f\in C_b(\mathcal{W}).
\end{equation}
\end{theorem}

It should be noted that the Riccati ODE \eqref{t.riccati.1}
is solved backward from t=T.
As will be seen in the proof of the theorem, 
this is because the condition $\sigma=\sigma(\chi)$ gives
the backward equation that
\[
    \sigma(t)=\chi(t)-\int_t^T \chi(u)^\dagger\chi(u)du
    \quad\text{for }t\in[0,T].
\]
For the proof, we prepare an elementary lemma on Lebesgue
integrals.

\begin{lemma}\label{l.leb.const}
Let $t\in(0,T]$.
If $f\in L^2([0,t];\mathbb{R})$ satisfies that
\[
    \iint_{[0,t]^2}|f(s_1)-f(s_2)|^2 ds_1ds_2=0,
\]
then it holds that
\[
    f(u)=\frac1{t}\int_0^t f(s)ds
    \quad\text{a.e.~$u\in[0,t]$.}
\]
\end{lemma}

\begin{proof}
It suffices to show that the essential supremum $M$ of 
$f$ on $[0,t]$ is equal to its essential infimum $m$ on $[0,t]$.
In fact, if so, then $f=M$ a.e.~on $[0,t]$ and 
$\int_0^t f(s)ds=tM$.

We prove the equality $M=m$ by contradiction.
To do so, suppose that $M>m$.
Let 
$A_1=\{s\in[0,t]\mid f(s)\ge M-\frac{M-m}3\}$ and 
$A_2=\{s\in[0,t]\mid f(s)\le m+\frac{M-m}3\}$.
By the definition of $M$ and $m$, the Lebesgue measures $\lambda(A_1)$
and $\lambda(A_2)$ of $A_1$ and $A_2$, respectively, are both
positive. 
It then holds that
\begin{align*}
    0 & =\iint_{[0,t]^2}|f(s_1)-f(s_2)|^2 ds_1ds_2
        \ge \iint_{A_1\times A_2}|f(s_1)-f(s_2)|^2 ds_1ds_2
    \\
    & \ge \Bigl(\frac{M-m}3\Bigr)^2
          \lambda(A_1)\lambda(A_2)
      >0,
\end{align*}
which is a contradiction.
\end{proof}

\begin{proof}[Proof of Theorem~\ref{t.riccati}]
The implication of (ii) to (i) have already seen in
Theorem~\ref{t.lap.lin}. 

To see the implication of (ii) to (iii), let $\chi$ be as in 
(ii) and set $\bR=\chi-\sigma$.
By \eqref{eq.sigma.chi}, it holds that
\[
    \bR(t)=\int_t^T \chi(s)^\dagger\chi(s)ds
    \quad\text{for }t\in[0,T].
\]
Hence $\bR\in C^1([0,T];\mathbb{R}^{d\times d})$.
This identity also yields that $\bR^\dagger=\bR$, and hence  
$\chi^\dagger=\bR+\sigma^\dagger$.
Then the identity is rewritten as
\begin{equation}\label{t.riccati.22}
    \bR(t)=\int_t^T (\bR(s)+\sigma(s)^\dagger)
      (\bR(s)+\sigma(s))ds
    \quad\text{for }t\in[0,T].
\end{equation}
This is exactly the integral expression of the Riccati ODE
\eqref{t.riccati.1}. 

To show that (iii) implies (ii), let 
$\bR\in C^1([0,T];\mathbb{R}^{d\times d})$ be the solution
to the Riccati ODE \eqref{t.riccati.1}. 
Taking the transpose of the ODE, we see
that $\bR^\dagger$ obeys the same ODE.
Hence $\bR^\dagger=\bR$. 
Set $\chi=\bR+\sigma$.
Then $\chi\in C([0,T];\mathbb{R}^{d\times d})$ and 
$\chi^\dagger=\bR+\sigma^\dagger$.
Plugging these into the integral expression
\eqref{t.riccati.22} of the Riccati ODE, we have that
\[
    \chi(t)-\sigma(t)=\int_t^T \chi(u)^\dagger \chi(u) du
    \quad\text{for }t\in[0,T].
\]
Thus $\sigma(\chi)=\sigma$, i.e., (ii) holds.

We show that (i) implies (ii).
Since $\p_\sigma=\q_{\rho_\sigma}$, by
Theorem~\ref{t.tways}, there exists a  
$\rho\in \mathcal{S}_2$ 
with   
$\eta(\kappa_A(\rho))\eqltwo\rho_\sigma$.
Hence we have that
\begin{equation}\label{t.sigma.23}
    \int_0^T\biggl(\int_0^t \biggl| 
      \sigma(t)-\rho(t,s)
         +\int_t^T\rho(t,u)\rho(u,s)du
      \biggr|^2 ds \biggr)dt=0.
\end{equation}

Putting
\[
    \gamma(t,s_1,s_2)
    =\rho(t,s_1)-\rho(t,s_2)
    \quad\text{for }(t,s_1,s_2)\in[0,T]^3,
\]
we show that
\begin{equation}\label{t.sigma.24}
    \int_0^T \biggl( \iint_{[0,t]^2}
      |\gamma(t,s_1,s_2)|^2 ds_1ds_2\biggr)dt
    =0.
\end{equation}
To see this, observe that \eqref{t.sigma.23} yields that 
\begin{equation}\label{t.sigma.25}
    \int_0^T\biggl(\iint_{[0,t]^2} \biggl|
        \gamma(t,s_1,s_2)
        -\int_t^T \rho(t,u)\gamma(u,s_1,s_2)du
      \biggr|^2 ds_1ds_2 \biggr)dt
    =0.
\end{equation}
Take an $N\in \mathbb{N}$ such that
\[
    \iint_{[a,a+\frac{T}{N}]\times[0,T]}|\rho(t,s)|^2 dt ds
    <\frac12
    \quad\text{for any }a\in[0,T_1],
\]
where $T_n=T-\frac{nT}{N}$ for $0\le n\le T$.
By \eqref{t.sigma.25} and the Schwarz inequality, we have
that
\begin{align*}
    & \int_{T_1}^T \biggl(\iint_{[0,t]^2}
         |\gamma(t,s_1,s_2)|^2 ds_1 ds_2 \biggr)dt
    \\
    & \le \int_{T_1}^T 
          \biggl(\int_t^T
           |\rho(t,u)|^2 du\biggr)
          \biggl(\iint_{[0,t]^2} \biggl(\int_t^T
           |\gamma(v,s_1,s_2)|^2 dv \biggr) ds_1 ds_2 
         \biggr)dt
    \\
    & \le \frac12 \int_{T_1}^T \biggl(\iint_{[0,v]^2}
         |\gamma(v,s_1,s_2)|^2 ds_1 ds_2 \biggr)dv.
\end{align*}
Thus it holds that
\[
    \int_{T_1}^T \biggl(\iint_{[0,t]^2}
        |\gamma(t,s_1,s_2)|^2 ds_1 ds_2 \biggr)dt=0.
\]
This implies that \eqref{t.sigma.25} holds with $T_1$
for $T$. 
In exactly the same manner as above, only this time with
$T_1$ for $T$, we have that  
\[
    \int_{T_2}^{T_1} \biggl(\iint_{[0,t]^2}
        |\gamma(t,s_1,s_2)|^2 ds_1 ds_2 \biggr)dt=0.
\]
Continuing this argument successively, we obtain that
\[
    \int_{T_n}^{T_{n-1}} 
      \biggl(\iint_{[0,t]^2} |\gamma(t,s_1,s_2)|^2 
        ds_1 ds_2 \biggr)dt=0
    \quad\text{for }1\le n\le N.
\]
Thus \eqref{t.sigma.24} holds.

Define $\chi_0:[0,T]\to \mathbb{R}^{d\times d}$ as 
$\chi_0(0)=0$ and 
\[
    \chi_0(t)
    =\frac1{t}\int_0^t \rho(t,s)ds 
     \quad\text{for }t\in(0,T].
\]
By \eqref{t.sigma.24}, applying Lemma~\ref{l.leb.const} to
each component of $\rho$, we see that
\[
    \int_0^T \biggl(\int_0^t |\rho(t,s)-\chi_0(t)|^2 ds
       \biggr)dt=0.
\]
Since $\rho\in \mathcal{S}_2$, switching the order of
integration and exchanging $t$ and $s$, we can rewrite this 
equation as 
\[
    \int_0^T \biggl(\int_t^T |\rho(t,s)-\chi_0(s)^\dagger|^2
        ds \biggr)dt=0.
\]
Plugging these into \eqref{t.sigma.23}, we have that
\[
    \int_0^T \biggl|\sigma(t)-\chi_0(t)
      +\int_t^T \chi_0(u)^\dagger \chi_0(u) du\biggr|^2 dt
    =0.
\]
Define $\chi\in C([0,T];\mathbb{R}^{d\times d})$ by 
\[
    \chi(t)=\sigma(t)
      +\int_t^T \chi_0(u)^\dagger \chi_0(u) du
    \quad\text{for }t\in[0,T].
\]
By the above identity, we obtain that $\sigma(\chi)=\sigma$.
Thus (iii) holds.

We show the second assertion.
As was seen in the proof of the equivalence of (ii) and
(iii), we have that $\chi=\bR+\sigma$.
Then it holds that
\[
    \chi-\sigma(\chi)
    =\chi-\sigma=\bR.
\]
Plugging this into \eqref{t.lap.lin.3}, we obtain
\eqref{t.riccati.2-}.  
\end{proof}

\begin{remark}\label{r.riccati}
For $\eta\in\stwo$, it holds that
\[
    \langle B_\eta h,h\rangle_{\mathcal{H}}
    =2\int_0^T \biggl\langle \int_0^t \eta(t,s)h^\prime(s) ds,
       h^\prime(t)\biggr\rangle dt.
\]
In particular, we have that
\[
    \langle B_{\rho_\sigma} h,h\rangle_{\mathcal{H}}
    =2\int_0^T \langle
      \sigma(t)h(t),h^\prime(t)\rangle dt
    \quad\text{for }\sigma\in C([0,T];\mathbb{R}^{d\times d}).
\]
Thus, we know that
\[
    \Lambda(B_{\rho_\sigma})
     =2\sup_{\|h\|_{\mathcal{H}}=1} \int_0^T \langle
      \sigma(t)h(t),h^\prime(t)\rangle dt.
\]

Recall that $\p_\sigma=\q_{\rho_\sigma}$.
Due to Lemma~\ref{l.q.eta.int}, in order that
$e^{\p_\sigma}\in L^1(\mu)$, it is necessary and sufficient
that  
\begin{equation}\label{r.riccati.1}
    \sup_{\|h\|_{\mathcal{H}}=1} \int_0^T \langle
      \sigma(t)h(t),h^\prime(t)\rangle dt
    <\frac12.
\end{equation}
Thus, by Theorem~\ref{t.riccati}, if this condition
holds, the Riccati ODE \eqref{t.riccati.1} has a
$C^1$-solution.

Extending the notation $\Lambda(B)$ for 
$B\in \mathcal{S}_2(\htwo)$, we set 
$\Lambda(M)=\sup\limits_{x\in \mathbb{R}^d,|x|=1}
 \langle Mx,x\rangle$ for $M\in \mathbb{R}^{d\times d}$.
For $h\in \mathcal{H}$ with $\|h\|_{\mathcal{H}}=1$,
it holds that
\[
    \biggl|\int_0^T \langle
      \sigma(t)h(t),h^\prime(t)\rangle dt\biggr|
    \le \biggl(\int_0^T |\sigma(t)h(t)|^2 dt
          \biggr)^{\frac12}
    \le \biggl(\int_0^T t 
          \Lambda(\sigma(t)^\dagger \sigma(t)) dt
          \biggr)^{\frac12}.
\]
Hence, a sufficient condition for \eqref{r.riccati.1} to
hold is that
\[
    \int_0^T t\Lambda(\sigma(t)^\dagger \sigma(t)) dt
    <\frac14.
\]
\end{remark}

If $\sigma\in C^1([0,T];\mathbb{R}^{d\times d})$, then the
Riccati ODE in Theorem~\ref{t.riccati} can be replaced by a 
second order linear ODE as follows.

\begin{corollary}\label{c.2nd.ode}
Let $\sigma\in C^1([0,T];\mathbb{R}^{d\times d})$ and 
$\bS\in C^2([0,T];\mathbb{R}^{d\times d})$ be solution to
the second order ODE
\begin{equation}\label{c.2nd.ode.1}
    \bS^{\prime\prime}-2\sigma_A \bS^\prime-\sigma^\prime \bS=0,
    \quad \bS(T)=I_d,~~\bS^\prime(T)=\sigma(T),
\end{equation}
where $\sigma_A=\frac12(\sigma-\sigma^\dagger)$.
Then the equivalent conditions {\rm(i)}, {\rm(ii)}, and 
{\rm(iii)} in Theorem~\ref{t.riccati} is also equivalent to
the condition that
\\[5pt]
{\rm(iv)} $\det \bS(t)\ne0$ for any $t\in[0,T]$. 
\\[5pt]
If these conditions hold,
then $\sigma=\sigma(\chi)$ with $\chi=\bS^\prime \bS^{-1}$
and it holds that
\begin{equation}\label{c.2nd.ode.2}
    \int_{\mathcal{W}} f e^{\p_\sigma+\D^*h} d\mu
    =\biggl(\frac{e^{-\int_0^T \tr\sigma_S(t) dt}}{
       \det \bS(0)}\biggr)^{\frac12}
     \int_{\mathcal{W}} 
         f(\boldsymbol{\xi}_{\bS}) 
     e^{\int_0^T\langle h^\prime(t),d\xi_{\bS}(t)\rangle}
     d\mu
\end{equation}
for every $f\in C_b(\mathcal{W})$ and $h\in \mathcal{H}$,
where $\sigma_S=\frac12(\sigma+\sigma^\dagger)$, and
\begin{equation}\label{c.2nd.ode.3}
    \int_{\mathcal{W}} e^{\p_\sigma} 
     \delta_y(\theta(t)) d\mu
    =\biggl(\frac{e^{-\int_0^T \tr\sigma_S(t) dt}}{
       \det \bS(0)}\biggr)^{\frac12}
     g_{v_t(\bS)}(y)
    \quad\text{for $t\in(0,T]$ and $y\in \mathbb{R}^d$}.
\end{equation}
\end{corollary}

\begin{proof}
Take a $\tau\in[0,T)$ with the property that 
$\det\bS(t)\ne0$ for any $t\in[\tau,T]$.
Since $\det\bS(T)=1$, such a $\tau$ exists.
Define 
$\bR_\tau\in C^1([\tau,T];\mathbb{R}^{d\times d})$ as
\[
    \bR_\tau(t)=\bS^\prime(t)\bS(t)^{-1}-\sigma(t)
    \quad\text{for }t\in[\tau,T].
\]
By \eqref{c.2nd.ode.1}, we have that
\[
    \bR_\tau^\prime 
    =-\bR_\tau^2-\sigma^\dagger \bR_\tau
          -\bR_\tau\sigma-\sigma^\dagger\sigma,
    \quad
    \bR_\tau(T)=0.
\]
Thus $\bR_\tau$ obeys the Riccati ODE \eqref{t.riccati.1} on 
$[\tau,T]$.

Assume that $\det\bS(t)\ne0$ for any $t\in[0,T]$.
Setting $\tau=0$ in the above observation, we know that
$\bR_0\in  C^1([0,T];\mathbb{R}^{d\times d})$ solves the
Riccati ODE \eqref{t.riccati.1}.
Thus the condition (iii) in Theorem~\ref{t.riccati} holds.

Assume that the condition (iii) in Theorem~\ref{t.riccati}
holds.
Hence there is an 
$\bR\in  C^1([0,T];\mathbb{R}^{d\times d})$ solving the
Riccati ODE \eqref{t.riccati.1}.
Set
\[
    \tau_0=\inf\Bigl\{\tau\in[0,T] \,\Big|\, \det\bS(t)\ne0 
           \text{ for any }t\in[\tau,T]\Bigr\}.
\]
Since $\det\bS(T)=1$, $\tau_0<T$.
Let $\tau\in(\tau_0,T]$.
By the observation of the first paragraph and the uniqueness
of solution to the Riccati ODE, we have that
$\bR_\tau(t)=\bR(t)$ for $t\in[\tau,T]$.
Recall that 
$(\det\bS)^\prime(t)
 =[\tr(\bS^\prime\bS^{-1})(t)] \det\bS(t)$ for
$t\in[\tau,T]$.
Hence we have that
\[
    \det \bS(\tau)
    =e^{-\int_\tau^T \tr[(\bS^\prime\bS^{-1})(t)]dt}
    =e^{-\int_\tau^T\tr[\bR_\tau(t)+\sigma(t)]dt}
    =e^{-\int_\tau^T\tr[\bR(t)+\sigma(t)]dt}.
\]
Letting $\tau\searrow\tau_0$, we obtain that
\begin{equation}\label{c.2nd.ode.21}
    \det\bS(\tau_0)=e^{-\int_{\tau_0}^T\tr[\bR(t)+\sigma(t)]dt}
    \ne0.
\end{equation}
This implies that $\tau_0=0$.
In fact, if $\tau_0>0$, then by the continuity of $\bS$ at
$t=\tau_0$ and \eqref{c.2nd.ode.21}, there is an
$\varepsilon>0$ with $\tau_0\le \tau_0-\varepsilon$, which
is a contradiction.
By the continuity of $\bS$ at $t=0$ and \eqref{c.2nd.ode.21}
with $\tau_0=0$, we have that $\det\bS(t)\ne0$ for any
$t\in[0,T]$.

Assume that $\det\bS(t)\ne0$ for any $t\in[0,T]$.
Then $\bR_0$ obeys the Riccati ODE \eqref{t.riccati.1}.
By Theorem~\ref{t.riccati}, $\sigma(\chi)=\sigma$ with
$\chi=R_0+\sigma=\bS^\prime\bS^{-1}$.
Due to this identity for $\chi$ and Theorem~\ref{t.alpha}, 
we obtain \eqref{c.2nd.ode.2} and \eqref{c.2nd.ode.3}.
\end{proof}

We give an example of Corollary~\ref{c.2nd.ode}, which will
be used in the next two sections.
In the example, we use the matrix functions 
$\e,\ch,\snh,\sh,\tnh:
 \mathbb{C}^{d\times d}\to \mathbb{C}^{d\times d}$
given as
\begin{align*}
    & \e[M]=\sum_{n=0}^\infty \frac1{n!} M^n,
      \quad
      \ch[M]=\sum_{n=0}^\infty \frac1{(2n)!} M^{2n},
      \quad
      \snh[M]=\sum_{n=0}^\infty \frac1{(2n+1)!} M^{2n+1},
    \\
    & 
      \sh[M]=\sum_{n=0}^\infty \frac1{(2n+1)!} M^{2n},
      \quad
      \tnh[M^\prime]=\sh[M^\prime](\ch[M^\prime])^{-1}
\end{align*}
for $M,M^\prime\in \mathbb{C}^{d\times d}$
with $\det\ch[M^\prime]\ne0$.

\begin{example}\label{e.sAD}
Take $A,D\in \mathbb{R}^{d\times d}$ with $A^\dagger=-A$ and
$D^\dagger=D$.
Define the quadratic Wiener functional 
$\mathfrak{s}_{A,D}:\mathcal{W}\to \mathbb{R}$ as
\[
    \mathfrak{s}_{A,D}
    =\int_0^T \langle A\theta(t),d\theta(t)\rangle
     +\frac12\langle D\theta(T),\theta(T)\rangle.
\]
Since 
\[
    \frac12\langle D\theta(T),\theta(T)\rangle
    =\int_0^T \langle D\theta(t),d\theta(t)\rangle
     +\frac{T}2\tr D,
\]
if we put $\sigma\in C^1([0,T];\mathbb{R}^{d\times d})$ as
$\sigma\equiv A+D$, then it holds that
\begin{equation}\label{e.sAD.1}
    \mathfrak{s}_{A,D}
    =\p_\sigma+\frac12\int_0^T \tr\sigma_S(t)dt
\end{equation}

The ODE \eqref{c.2nd.ode.1} for this $\sigma$ reads as
\[
    \bS^{\prime\prime}-2A \bS^\prime=0,
    \quad \bS(T)=I_d,~\bS^\prime(T)=A+D.
\]
Thinking of this as the ODE for $\bS^\prime$, we obtain that
\[
    \bS^\prime(s)=\e[2(s-T)A](A+D)
    \quad\text{for }s\in[0,T].
\]
Observing that
\[
    \frac{d}{ds}\Bigl(s \e[sA]\sh[sA]\Bigr)
    =\e[2sA],
\]
we have that
\begin{align*}
    I_d-\bS(t)
    & =\int_t^T \bS^\prime(s) ds
      =(s-T)\e[(s-T)A]\sh[(s-T)A](A+D)\Bigr|_{s=t}^{s=T}
    \\
    & =-(t-T)\e[(t-T)A]\sh[(t-T)A](A+D)
    \\
    & =-\e[(t-T)A]\snh[(t-T)A]
       -(t-T)\e[(t-T)A]\sh[(t-T)A]D.
\end{align*}
Rewriting $I_d$ as $\e[(t-T)A]\e[-(t-T)A]$, we obtain that
\[
    \bS(t)=\e[(t-T)A]\Bigl\{
      \ch[(t-T)A]+(t-T)\sh[(t-T)A]D\Bigr\}.
\]

Since $A$ is skew-symmetric, it has a unitary matrix
$U\in \mathbb{C}^{d\times d}$ and non-zero
$\lambda_1,\dots,\lambda_k\in \mathbb{R}$ such that
\begin{equation}\label{e.sAD.1+}
    A= U^*\text{\rm diag}\bigl[
       \kyosu\lambda_1,-\kyosu\lambda_1,\dots,
       \kyosu\lambda_k,-\kyosu\lambda_k,
       \underbrace{0,\dots,0}_{d-2k} \bigr] U.    
\end{equation}
For every $a\in \mathbb{R}$, this expression implies that
\[
    \e[a A]=U^*\text{\rm diag}\bigl[
       e^{\kyosu\lambda_1 a},e^{-\kyosu\lambda_1 a},
       \dots,
       e^{\kyosu\lambda_k a},e^{-\kyosu\lambda_k a},
       1,\dots,1\bigr] U
\]
and hence 
\[
    \det\e[a A]=1.
\]

Define $T_0>0$ so that
\[
    T_0=\inf\Bigl\{t\ge0 \,\Big|\,
        \det\Bigl(\ch[t A]-t \sh[t A]D\Bigr)
        =0 \Bigr\}.
\]
Due to the above expression of $\bS$, we
conclude that 
\[
    \det\bS(t)\ne0\text{ for any }t\in[0,T]
    \quad\text{if and only if}\quad
    T<T_0.
\]

When $d=2$, $A\ne0$, and $D$ is diagonal, we have an
explicit expression $T_0$.
In fact, then there are $a,b,c\in \mathbb{R}$ with $a\ne0$
such that $A=aJ$ and 
$D=\begin{pmatrix} b & 0 \\ 0 & c\end{pmatrix}$,
where $J=\begin{pmatrix} 0 & -1 \\ 1 & 0 \end{pmatrix}$
as in Example~\ref{e.levy}.
By a direct computation, we know that
\[
    \ch[\omega J]
    =\begin{pmatrix} \cos\omega & 0 
     \\ 0 & \cos\omega\end{pmatrix},
    \quad
    \omega\sh[\omega J]
    =\begin{pmatrix} \sin\omega & 0 
     \\ 0 & \sin\omega\end{pmatrix}
    \quad\text{for }\omega\in \mathbb{R}.
\]
Hence it holds that
\begin{align*}
    \det\Bigl(\ch[t A]-t \sh[t A]D\Bigr)
    & =\Bigl(\cos(at)+\frac{b}{a}\sin(at)\Bigr)
       \Bigl(\cos(at)+\frac{c}{a}\sin(at)\Bigr)   
    \\
    & =\Bigl(\cos(|a|t)+\frac{b}{|a|}\sin(|a|t)\Bigr)
       \Bigl(\cos(|a|t)+\frac{c}{|a|}\sin(|a|t)\Bigr).
\end{align*}
The function 
$\tan_{[0,\pi)}:[0,\pi)\to \mathbb{R}\cup\{\infty\}$ 
defined as $\tan_{[0,\pi)}\omega=\tan\omega$ if
$\omega\ne\frac{\pi}2$ and $=\infty$ if
$\omega=\frac{\pi}2$ is bijective.
Denoting its inverse function as $\tan_{[0,\pi)}^{-1}$, 
we obtain that
\[
    T_0=\frac1{|a|}
        \min\biggl\{
         \tan_{[0,\pi)}^{-1}\Bigl(-\frac{|a|}{b}\Bigr),
         \tan_{[0,\pi)}^{-1}\Bigl(-\frac{|a|}{c}\Bigr)
        \biggr\},
\]
where $-\frac{|a|}{0}=\infty$.

Assume that $T<T_0$, equivalently $\det\bS(t)\ne0$ for any 
$t\in[0,T]$.
Then Corollary~\ref{c.2nd.ode} is applicable to $\p_\sigma$.
Due to \eqref{c.2nd.ode.2}, \eqref{c.2nd.ode.3}, and 
\eqref{e.sAD.1}, it holds that
\begin{equation}\label{e.sAD.2}
    \int_{\mathcal{W}} f e^{\mathfrak{s}_{A,D}+\D^*h} d\mu
      =\Bigl\{\det\Bigl(\ch[TA]-T\sh[TA]D\Bigr)
       \Bigr\}^{-\frac12}
       \int_{\mathcal{W}} f(\boldsymbol{\xi}_{\bS})
         e^{\int_0^T\langle h^\prime(t),
               d\xi_{\bS}(t)\rangle}
         d\mu
\end{equation}
and 
\begin{equation}\label{e.sAD.3}
   \int_{\mathcal{W}} e^{\mathfrak{s}_{A,D}} 
          \delta_x(\theta(t))d\mu
      =\Bigl\{\det\Bigl(\ch[TA]-T\sh[TA]D\Bigr)
       \Bigr\}^{-\frac12}
       g_{v_t(\bS)}(x)
\end{equation}
for any $f\in C_b(\mathcal{W})$, $h\in \mathcal{H}$,
$t\in(0,T]$, and $x\in \mathbb{R}^d$.

When $t=T$, the LHS of \eqref{e.sAD.3} transforms as
\[
    \int_{\mathcal{W}} e^{\mathfrak{s}_{A,D}} 
          \delta_x(\theta(T))d\mu
    =e^{\frac12\langle Dx,x\rangle}
     \int_{\mathcal{W}} e^{\mathfrak{s}_{A,0}} 
          \delta_x(\theta(T))d\mu.
\]
Hence, to investigate $\mathfrak{s}_{A,D}$ under
$\delta_x(\theta(T))d\mu$, we may assume that $D=0$.
Continue to assume that $T<T_0$.
We then know that
\[
    \bS(t)=\e[(t-T)A]\ch[(t-T)A]
    \quad\text{and}\quad
    \det\ch[(t-T)A]\ne0
    \quad\text{for } t\in[0,T].
\]
It holds that
\begin{align*}
    \bS(t)\bS(s)^{-1}
    & =\e[(t-T)A]\ch[(t-T)A]
       \ch[(s-T)A]^{-1}\e[-(s-T)A]
    \\
    & =\ch[(t-T)A](\ch[(s-T)A])^{-1}\e[(t-s)A]
      \quad\text{for } 0\le s\le t\le T.
\end{align*}
Since 
$(\e[a A])^\dagger=\e[-a A]$ and
$(\ch[a A])^\dagger=\ch[a A]$ for $a\in \mathbb{R}$,
we have that
\[
    (\bS(t)\bS(s)^{-1})(\bS(t)\bS(s)^{-1})^\dagger
    =(\ch[(t-T)A])^2(\ch[(s-T)A])^{-2}
    \quad\text{for }0\le s\le t\le T.    
\]
Noticing that
\begin{equation}\label{eq.tnh}
    \frac{d}{ds}\Bigl((s-T)\tnh[(s-T)A]\Bigr)
    =(\ch[(s-T)A])^{-2},
\end{equation}
we have that
\[
    v_T(\bS)
    =(s-T)\tnh[(s-T)A]\Bigr|_{s=0}^{s=T}
    =T \tnh[TA].
\]
Thus we obtain that
\[
    \det\bS(0) \det v_T(\bS)
    =\det(\ch[TA]) \det(T\tnh[TA])
    =T^d \det(\sh[TA]).
\]
Plugging these into \eqref{e.sAD.3} with $t=T$ and $D=0$, we
obtain that
\begin{equation}\label{e.sAD.4}
    \int_{\mathcal{W}} 
      e^{\int_0^T \langle A\theta(t),d\theta(t)\rangle}
        \delta_x(\theta(T))d\mu
      =\Bigl\{(2\pi T)^d\det(\sh[TA])\Bigr\}^{-\frac12}
        e^{-\frac1{2T}\langle
        (\tnh[TA])^{-1}x,x \rangle}.
\end{equation}
\end{example}

Using the similar argument as in this example, we give
another example of Corollary~\ref{c.2nd.ode}.

\begin{example}\label{e.h.osc}
Take $D_1,D_2\in \mathbb{R}^{d\times d}$ with
$D_1^\dagger=D_1$ and $D_2^\dagger=D_2$.
Define the quadratic Wiener functional
$\mathfrak{h}_{D_1,D_2}:\mathcal{W} \to \mathbb{R}$ as
\[
    \mathfrak{h}_{D_1,D_2}
    =\frac12\langle D_1\theta(T),\theta(T)\rangle 
     +\frac12\int_0^T \langle D_2\theta(t),
     \theta(t)\rangle dt.
\]
If we set $\sigma\in C([0,T];\mathbb{R}^{d\times d})$ so
that $\sigma(t)=D_1+(t-T)D_2$ for $t\in[0,T]$, then
\[
    \mathfrak{h}_{D_1,D_2} 
    =\p_\sigma+\frac12\int_0^T \tr\sigma_S(t)dt.
\]

The ODE \eqref{c.2nd.ode.1} for this $\sigma$ reads as
\[
    \bS^{\prime\prime}+D_2\bS=0,
    \quad \bS(T)=I_d,~\bS^\prime(T)=D_1.
\]
The symmetric $D_2$ has an orthogonal matrix 
$V\in \mathbb{R}^{d\times d}$ and
$\lambda_1,\dots,\lambda_d\in \mathbb{R}$ such that
$V^\dagger D_2V=\text{\rm diag}\bigl[
 \lambda_1,\dots,\lambda_d\bigr]$.
Set 
$\sqrt{D_2}=V \text{\rm diag}\bigl[
 \sqrt{\lambda_1},\dots,\sqrt{\lambda_d}\bigr] V^\dagger$.
Then we have that
\[
    \bS(t)=\ch[\kyosu(t-T)\sqrt{D_2}]
     +(t-T)\sh[\kyosu(t-T)\sqrt{D_2}] D_1
    \quad\text{for } t\in[0,T].
\]
Setting 
\[
    T_1=\inf\Bigl\{t\ge0 \,\Big|\,
        \det\Bigl(\ch[\kyosu t \sqrt{D_2}]
         -t \sh[\kyosu t \sqrt{D_2}] D_1
        \Bigr)=0 \Bigr\},
\]
we conclude that
\[
    \text{
    $\det\bS(t)\ne0$ for any $t\in[0,T]$ if and only if 
    $T<T_1$.}
\]

Suppose that $T<T_1$, equivalently $\det\bS(t)\ne0$ for
every $t\in[0,T]$. 
By Corollary~\ref{c.2nd.ode}, we have that
\[
    \int_{\mathcal{W}} 
      f e^{\mathfrak{h}_{D_1,D_2}+\D^*h} d\mu
      =\Bigl\{
         \det\Bigl(\ch[\kyosu T\sqrt{D_2}]
           -T\sh[\kyosu T\sqrt{D_2}] D_1\Bigr)
         \Bigr\}^{-\frac12}
      \int_{\mathcal{W}} 
         f(\boldsymbol{\xi}_\bS) d\mu
\]
and
\[
    \int_{\mathcal{W}} e^{\mathfrak{h}_{D_1,D_2}} 
        \delta_x(\theta(t))d\mu
      =\Bigl\{
         \det\Bigl(\ch[\kyosu T\sqrt{D_2}]
           -T\sh[\kyosu T\sqrt{D_2}] D_1\Bigr)
         \Bigr\}^{-\frac12}
       g_{v_t(\bS)}(x)
\]
for every $f\in C_b(\mathcal{W})$, $h\in \mathcal{H}$,
$t\in(0,T]$, and $x\in \mathbb{R}^d$.

When $t=T$, the term
$e^{\frac12\langle D_1 \theta(T),\theta(T)\rangle}$ in the
LHS of the second identity is factored out as 
$e^{\frac12\langle D_1 x,x\rangle}$.
Thus, to integrate $e^{\mathfrak{h}_{D_1,D_2}}$ under
$\delta_x(\theta(T))d\mu$, we may assume that $D_1=0$. 
Since 
$(\ch[z \sqrt{D_2}])^\dagger=\ch[z \sqrt{D_2}]$ 
for $z\in \mathbb{C}$,
we have that
\[
    (\bS(t)\bS(s)^{-1})(\bS(t)\bS(s)^{-1})^\dagger
    =(\ch[\kyosu(t-T)\sqrt{D_2}])^2
     (\ch[\kyosu(s-T)\sqrt{D_2}])^{-2}.
\]
By \eqref{eq.tnh}, we see that
\[
    v_T(\bS)
    =(s-T)\tnh[\kyosu(s-T)\sqrt{D_2}]
          \Bigr|_{s=0}^{s=t}
    =T\tnh[\kyosu T\sqrt{D_2}].
\]
Thus we obtain that
\[
    \int_{\mathcal{W}}
        e^{\frac12\int_0^T \langle D_2\theta(t),
              \theta(t)\rangle dt}
        \delta_x(\theta(T))d\mu
    =\Bigl\{(2\pi T)^d 
           \det \sh[\kyosu T\sqrt{D_2}]\Bigr\}^{-\frac12}
     e^{-\frac1{2T}\langle
      (\tnh[\kyosu T\sqrt{D_2}])^{-1} x,x \rangle}.
\]
\end{example}

\section{Feynman-Kac density function}
\label{sec.FK}
Define the quadratic Wiener functional 
$\mathfrak{a}_{\phi,\psi}^x$ for
$\phi\in C^1([0,T];\mathbb{R}^{d\times d})$,
$\psi\in C([0,T];\mathbb{R}^{d\times d})$, and 
$x\in \mathbb{R}^d$ as
\[
    \mathfrak{a}_{\phi,\psi}^x
    =\int_0^T \langle\phi(t)(x+\theta(t)),
         d\theta(t)\rangle
     +\frac12\int_0^T\langle \psi(t)(x+\theta(t)),
           x+\theta(t)\rangle dt.
\]
As an application of Corollary~\ref{c.2nd.ode}, we give
an explicit representation of the density function of the
distribution of $x+\theta(T)$ under
$e^{\mathfrak{a}_{\phi,\psi}^x}d\mu$.
Remember that $\mathfrak{a}_{\phi,\psi}^x$ relates to the 
Feynman-Kac formula (cf.\cite[\S3.5]{mt-cambridge}).
For example, if $\phi$ and $\psi$ are both constant functions, then
the density function is exactly the heat kernel for the Laplacian 
with scalar and vector potentials. 

\begin{theorem}\label{t.FK}
Let 
$\bS,\bU,\bV\in C^2([0,T];\mathbb{R}^{d\times d})$ be
the solutions to the ODE on $\mathbb{R}^{d\times d}$ 
\begin{equation}\label{t.FK.1}
    \Phi^{\prime\prime}-2\phi_A\Phi^\prime
    +(\psi_S-\phi^\prime)\Phi=0
\end{equation}
with the terminal condition that
\[
    \bS(T)=I_d,~ \bS^\prime(T)=\phi(T),~
    \bU(T)=I_d,~ \bU^\prime(T)=0,~
    \bV(T)=0, \text{ and }\bV^\prime(T)=I_d.
\]
Assume that $\det\bS(t)\ne0$ for every $t\in[0,T]$ and 
$\det\bV(0)\ne0$.
For $x,y\in \mathbb{R}^d$, put
\begin{align*}
    & \mathfrak{d}(x,y) 
    =\frac12\Bigl\{ 
      \bigl\langle\phi_S(T)y,y\bigr\rangle
      -\bigl\langle \phi_S(0)x,x \bigr\rangle
    \\
    & \hphantom{\mathfrak{d}(x,y)=\frac12\Bigl\{}
      +\bigl\langle \bU^\prime(0)y,x\bigr\rangle
      +\bigl\langle \bV(0)^{-1}(x-\bU(0)y),
           \bV^\prime(0)^\dagger x -y\bigr\rangle 
     \Bigr\}
\end{align*}
and 
\[    
   p_T(x,y)
   =\biggl(
     \frac{e^{-\int_0^T \tr\phi_S(t) dt}}{
         (2\pi)^d \det\bS(0) \det v_T(\bS)}
     \biggr)^{\frac12}
      e^{\mathfrak{d}(x,y)}.
\]
Then it holds that
\begin{equation}\label{t.FK.2}
    \int_{\mathcal{W}} \varphi(x+\theta(T))
    e^{\mathfrak{a}_{\phi,\psi}^x} d\mu
    =\int_{\mathbb{R}^d} \varphi(y)p_T(x,y)dy
    \quad\text{for every }\varphi\in C_b(\mathbb{R}^d).
\end{equation}
\end{theorem}

For the proof, we prepare a lemma.
To state it, define 
$\sigma\in C^1([0,T];\mathbb{R}^{d\times d})$ by 
\begin{equation}\label{eq.sigma.phi.psi}
    \sigma(t)=\phi(t)+\int_t^T \psi_S(s)ds
    \quad\text{for }t\in[0,T].
\end{equation}
Since $\sigma_A=\phi_A$,
$\sigma^\prime=\phi^\prime-\psi_S$, and
$\sigma(T)=\phi(T)$,
the ODE \eqref{t.FK.1} is rewritten with $\sigma$ as
\begin{equation}\label{eq.ode.sigma}
    \Phi^{\prime\prime}-2\sigma_A\Phi^\prime
    -\sigma^\prime\Phi=0
\end{equation}
and the terminal condition for $\bS$ is as
$\bS^\prime(T)=\sigma(T)$.
Applying It\^o's formula, we see that 
\begin{align*}
    &
    \frac12 \int_0^T \langle\psi(t)(x+\theta(t)),
         x+\theta(t) \rangle dt
      = \int_0^T \biggl\langle
          \biggl(\int_t^T \psi_S(s)ds\biggr)
          \{x+\theta(t)\},d\theta(t)\biggr\rangle 
    \\
    &  \hphantom{\frac12 \int_0^T \langle\psi(t)(x+\theta(t)),x+\theta(t)}
      +\frac12\biggl\langle 
          \biggl(\int_0^T \psi_S(s)ds\biggr)x,x
          \biggr\rangle
        +\frac12\int_0^T\biggl(\int_t^T \tr\psi_S(s)ds
            \biggr) dt. 
\end{align*}
Hence, if we put
\[
    \mathfrak{p}_\sigma^x
    =\int_0^T \langle \sigma(t)(x+\theta(t)),
      d\theta(t)\rangle
    =\mathfrak{p}_\sigma
     +\int_0^T \langle \sigma(t)x,d\theta(t)\rangle,
\]
then it holds that
\begin{equation}\label{eq.a.p}
    \mathfrak{a}_{\phi,\psi}^x
    =\mathfrak{p}_\sigma^x
     +\frac12 \biggl\langle \biggl(
       \int_0^T \psi_S(s)ds\biggr)x,x \biggr\rangle
     +\frac12 \int_0^T \biggl(\int_t^T 
        \text{\rm tr}\psi_S(s)ds\biggr) dt.
\end{equation}
Thus the investigation of $\mathfrak{a}_{\phi,\psi}^x$ is
reduced to that of $\mathfrak{p}_\sigma^x$.

\begin{lemma}\label{l.FK}
Assume the same assumption as in Theorem~\ref{t.FK}.
Put
\begin{align*}
    & \mathfrak{d}^{(\sigma)}(x,y)
      =\frac12\Bigl\{
      \bigl\langle \sigma_S(T)y,y\bigr\rangle
      -\bigl\langle \sigma_S(0)x,x\bigr\rangle
    \\
    & \hphantom{\mathfrak{d}^{(\sigma)}(x,y)=\frac12\Bigl\{}
      +\bigl\langle \bU^\prime(0)y,x\bigr\rangle
      +\bigl\langle \bV(0)^{-1}(x-\bU(0)y),
           \bV^\prime(0)^\dagger x -y\bigr\rangle 
      \Bigr\}
\end{align*}
and
\[
    p_T^{(\sigma)}(x,y)
    =\biggl(\frac{e^{-\int_0^T \tr\sigma_S(t) dt}}{
       (2\pi)^d \det \bS(0) \det v_T(\bS)}
     \biggr)^{\frac12}
     e^{\mathfrak{d}^{(\sigma)}(x,y)}
    \quad\text{for }x,y\in \mathbb{R}^d.
\]
Then it holds that
\begin{equation}\label{l.FK.1}
   \int_{\mathcal{W}} \varphi(x+\theta(T))
    e^{\p_\sigma^x} d\mu
    =\int_{\mathbb{R}^d} \varphi(y)p_T^{(\sigma)}(x,y)dy
    \quad\text{for every }\varphi\in C_b(\mathbb{R}^d).
\end{equation}
\end{lemma}

\begin{proof}
Let $x,y\in \mathbb{R}^d$.
Define $\beta\in C^2([0,T];\mathbb{R}^d)$ by 
\[
    \beta(t)=\bU(t)y+\bV(t)\bV(0)^{-1}(x-\bU(0)y)
    \quad\text{for }t\in[0,T].
\]
It follows from \eqref{eq.sigma.phi.psi} that $\beta$ obeys 
the following ODE on $\mathbb{R}^d$:
\[
    \beta^{\prime\prime}-2\sigma_A \beta^\prime
    -\sigma^\prime \beta=0,
    \quad \beta(0)=x,~~\beta(T)=y.
\]
Put $h=\beta-x$.
Define the Wiener functional 
$\mathfrak{s}[x,y]:\mathcal{W}\to \mathbb{R}$ by
\begin{align*}
    \mathfrak{s}[x,y] 
    & =\p_\sigma^x(\cdot+h)
        -\D^*h-\frac12\|h\|_{\mathcal{H}}^2
    \\
    & =\int_0^T \Bigl\langle \sigma(t)(\theta(t)+\beta(t)),
        d\theta(t)+\beta^\prime(t) dt\Bigr\rangle
      -\int_0^T \langle \beta^\prime(t),d\theta(t) \rangle
      -\frac12\int_0^T|\beta^\prime(t)|^2 dt.
\end{align*}
Applying the Cameron-Martin theorem, we have that
\begin{equation}\label{l.FK.21}
    \int_{\mathcal{W}} \varphi(x+\theta(T))
      e^{\mathfrak{p}_\sigma^x} d\mu
    =\int_{\mathcal{W}} \varphi(y+\theta(T))
     e^{\mathfrak{s}[x,y]}d\mu
    \quad\text{for }\varphi\in C_b(\mathbb{R}^d).
\end{equation}

By It\^o's formula, we know that
\[
    \int_0^T\langle \sigma(t)\theta(t),\beta^\prime(t)
       \rangle dt
    =\int_0^T \biggl\langle \int_t^T 
       \sigma(s)^\dagger \beta^\prime(s)ds,d\theta(t)
       \biggr\rangle.
\]
Hence we have that
\begin{align*}
    \mathfrak{s}[x,y] 
    = & \p_\sigma
        +\int_0^T \biggl\langle \sigma(t)\beta(t)
         +\int_t^T \sigma(s)^\dagger \beta^\prime(s)ds
         -\beta^\prime(t), d\theta(t)\biggr\rangle
    \\
     & +\int_0^T \langle\sigma(t) \beta(t),
           \beta^\prime(t)\rangle dt
      -\frac12\int_0^T|\beta^\prime(t)|^2 dt.
\end{align*}
Since
\[
    \biggl(\sigma \beta
      +\int_\bullet^T \sigma(s)^\dagger \beta^\prime(s)ds 
      -\beta^\prime \biggr)^\prime
    =\sigma^\prime \beta+2\sigma_A \beta^\prime
     -\beta^{\prime\prime}=0,
\]
we obtain that
\[
    \int_0^T \biggl\langle \sigma(t)\beta(t)
         +\int_t^T \sigma(s)^\dagger \beta^\prime(s)ds
         -\beta^\prime(t), d\theta(t)\biggr\rangle
    =\langle\sigma(T)y-\beta^\prime(T),
     \theta(T)\rangle.
\]
Furthermore, we have that 
\begin{align*}
    \langle \sigma \beta,\beta^\prime\rangle
      -\frac12|\beta^\prime|^2
    &  =\langle\sigma_S\beta,\beta^\prime\rangle
      -\frac12\Bigl\{
       \langle \beta,2\sigma_A \beta^\prime\rangle
       +|\beta^\prime|^2\Bigr\}
    \\
    & =\langle\sigma_S\beta,\beta^\prime\rangle
       -\frac12\Bigl\{
         \langle \beta,\beta^{\prime\prime}
            -\sigma^\prime \beta \rangle
         +|\beta^\prime|^2\Bigr\}
    \\
    &  =\langle\sigma_S\beta,\beta^\prime\rangle
       +\frac12\langle\sigma_S^\prime \beta,\beta\rangle
       -\frac12\Bigl\{
         \langle \beta,\beta^{\prime\prime}\rangle
         +|\beta^\prime|^2\Bigr\}
   \\
   & =\frac12\Bigl\{
        \langle\sigma_S\beta,\beta\rangle
        -\langle\beta,\beta^\prime\rangle
      \Bigr\}^\prime.
\end{align*}
Thus we see that
\begin{align*}
    \mathfrak{s}[x,y]
    =& 
     \p_\sigma
     +\bigl\langle\sigma(T)y-\beta^\prime(T),
           \theta(T)\bigr\rangle
    \\
    & +\frac12\Bigl\{
        \bigl\langle \sigma_S(T)y,y\bigr\rangle
        -\bigl\langle \sigma_S(0)x,x\bigr\rangle
        -\bigl\langle \beta^\prime(T),y\bigr\rangle
        +\bigl\langle \beta^\prime(0),x\bigr\rangle
      \Bigr\}.
\end{align*}
Since
\[
    \beta^\prime(0)=\bU^\prime(0)y
      +\bV^\prime(0)\bV(0)^{-1}(x-\bU(0)y)
    \quad\text{and}\quad
    \beta^\prime(T)=\bV(0)^{-1}(x-\bU(0)y),
\]
we obtain that
\[
    \mathfrak{s}[x,y]
     =\p_\sigma
     +\langle\sigma(T)y-\beta^\prime(T),
           \theta(T)\rangle
     +\mathfrak{d}^{(\sigma)}(x,y).
\]

Combining this with  \eqref{l.FK.21} and Corollary~\ref{c.2nd.ode}, 
we see that
\begin{align*}
    \int_{\mathcal{W}} e^{\p_\sigma^x} 
         \delta_y(x+\theta(T)) d\mu
    & =\int_{\mathcal{W}} 
        e^{\p_\sigma+
           \langle\sigma(T)y-\beta^\prime(T),
           \theta(T)\rangle
          +\mathfrak{d}^{(\sigma)}(x,y)} 
         \delta_0(\theta(T)) d\mu
     \\
     & =e^{\mathfrak{d}^{(\sigma)}(x,y)} 
        \biggl(\frac{e^{-\int_0^T \tr\sigma_S(t)dt}}{
              \det\bS(0)}\biggr)^{\frac12}
        g_{v_T(\bS)}(0).
\end{align*}
Thus we obtain \eqref{l.FK.1}. 
\end{proof}

\begin{proof}[Proof of Theorem~\ref{t.FK}]
Notice that 
\[
    \mathfrak{d}^{(\sigma)}(x,y)
      =\mathfrak{d}(x,y)
       -\frac12 \biggl\langle \biggl(
        \int_0^T \psi_S(s)ds\biggr)x,x \biggr\rangle
\]
and
\[
    \text{\rm tr}\sigma_S(t)
      =\text{\rm tr}\phi_S(t)
       +\int_t^T \text{\rm tr}\psi_S(s) ds
\]
for $t\in[0,T]$.
In conjunction with these and \eqref{eq.a.p}, \eqref{l.FK.1} yields
\eqref{t.FK.2}.
\end{proof}

\begin{remark}\label{r.FK}
While Corollary~\ref{c.2nd.ode} is a special case of
Theorem~\ref{t.FK} with $x=0$ and $\psi=0$, the corollary is
indispensable to show the theorem as have been seen.
Due to the expression \eqref{eq.a.p}, the corollary extends to 
$\mathfrak{a}_{\phi,\psi}^0$ as follows.
With $\bS$ as in Theorem~\ref{t.FK}, it holds that
\[
    \int_{\mathcal{W}} f 
       e^{\mathfrak{a}_{\phi,\psi}^0+\D^*h} d\mu
    =\biggl(\frac{e^{-\int_0^T\tr\phi_S(t)dt}}{
        \det\bS(0)}\biggr)^{\frac12}
     \int_{\mathcal{W}} f(\boldsymbol{\xi}_{\bS})
       e^{\int_0^T\langle h^\prime(t),d\xi_{\bS}(t)\rangle}
      d\mu
\]
for every $f\in C_b(\mathcal{W})$.
\end{remark}

We give two examples illustrating Theorem~\ref{t.FK}.

\begin{example}\label{e.FK}
Let $d=1$, $a\in \mathbb{R}$, and $\lambda>0$.
Set $\phi\equiv a$ and $\psi\equiv -\lambda^2$.
Since $\phi_A=0$, $\psi_S=-\lambda^2$, and $\phi^\prime=0$,
the ODE \eqref{t.FK.1} reads as 
\[
    \Phi^{\prime\prime}-\lambda^2\Phi=0.
\]
Solving this ODE with the given terminal conditions, 
we know that 
\[
    \bU(t)=\cosh(\lambda (t-T)),\quad
    \bV(t)=\frac1{\lambda}\sinh(\lambda (t-T)),
\]
and
\[
    \bS(t)=\cosh(\lambda (t-T))
         +\frac{a}{\lambda}\sinh(\lambda (t-T))
\]
for $t\in[0,T]$.
Hence $\bV(0)\ne0$.
Furthermore, $\bS(t)\ne0$ for every $t\in[0,T]$ if and only if 
$a<\lambda/\tanh(\lambda T)$.

Assume that $a<\lambda/\tanh(\lambda T)$.
Then the assumption in Theorem~\ref{t.FK} is fulfilled.
Since $\phi_S\equiv a$, by a direct computation, we obtain that
\[
    \mathfrak{d}(x,y)
      =\frac{a}2(y^2-x^2)
       -\frac{\lambda}2 \coth(\lambda T)
        \bigl\{
         x^2-2\text{\rm sech}(\lambda T)xy
         +y^2\bigr\}.
\]
Rewriting $\bS(t)$ as
\[
    \bS(t)
    =\frac12\biggl\{
     \Bigl(1-\frac{a}{\lambda}\Bigr)e^{\lambda(T-t)}
    +\Bigl(1+\frac{a}{\lambda}\Bigr)e^{-\lambda(T-t)}\biggr\}
    \quad\text{for }t\in[0,T],
\]
and using the identity that
\[
    \frac{d}{dt} (pe^{2\lambda t}+q)^{-1}
    =-2p\lambda (pe^{\lambda t}+q e^{-\lambda t})^{-2}
    \quad\text{for any }p,q\in \mathbb{R},
\]
we see that
\[
    v_T(\bS)=\frac{\sinh(\lambda T)}{\lambda}
            \Bigl(\cosh(\lambda T)
              +\frac{a}{\lambda} \sinh(-\lambda T)
              \Bigr)^{-1}.
\]
Thus $p_T(x,y)$ in Theorem~\ref{t.FK} is represented as
\[
    p_T(x,y)=\frac1{\sqrt{2\pi T}}
       \biggl(\frac{\lambda T}{\sinh(\lambda T)}
         \biggr)^{1/2}
    \exp\Bigl[\frac{a}2(y^2-x^2)
       -\frac{\lambda}2 \coth(\lambda T)
        \bigl\{x^2-2\text{\rm sech}(\lambda T)xy
         +y^2\bigr\} \Bigr].
\]

When $a=0$, this $p_T(x,y)$ is exactly the heat kernel
for the Schr\"odinger operator for harmonic oscillator
(see \cite[Theorem\,5.8.2]{mt-cambridge}).
\end{example}

\begin{example}\label{e.FK2}
Let $d=1$, $a\in \mathbb{R}$, and $\lambda>0$.
Set $\phi\equiv a$ and $\psi\equiv \lambda^2$.
Then the ODE \eqref{t.FK.1} reads as 
\[
    \Phi^{\prime\prime}+\lambda^2\Phi=0.
\]
Solving this ODE with the given terminal conditions, 
we know that 
\[
    \bU(t)=\cos(\lambda (t-T)),\quad
    \bV(t)=\frac1{\lambda}\sin(\lambda (t-T)),
\]
and
\[
    \bS(t)=\cos(\lambda (t-T))
         +\frac{a}{\lambda}\sin(\lambda (t-T))
\]
for $t\in[0,T]$.
Hence $\bV(0)\ne0$ if and only if 
$\lambda T\notin\{n\pi\mid n\in \mathbb{N}\}$.
Furthermore, $\bS(t)\ne0$ for every $t\in[0,T]$ if
and only if either 
(i)~$a>0$ and $\lambda T<\text{\rm arctan}(\lambda/a)$,
(ii)~$a<0$ and 
$\lambda T<(\pi/2)+\text{\rm arctan}(|a|/\lambda)$,
or
(iii)~$a=0$ and $\lambda T<\pi/2$.

Assume that one of the above conditions (i), (ii), and (ii)
holds.
Then $\lambda T<\pi$, and the assumption in
Theorem~\ref{t.FK} is fulfilled. 
A similar computation as in Example~\ref{e.FK} with 
$\sqrt{-1}\lambda$ for $\lambda$ yields that 
\[
    \mathfrak{d}(x,y)
      =\frac{a}2 (y^2-x^2)
        -\frac{\lambda}2 \cot(\lambda T)
       \bigl\{x^2
       -2\text{\rm sec}(\lambda T) xy+y^2\bigr\}
\]
and
\[
    v_T(\bS)=\frac{\sin(\lambda T)}{\lambda}
            \Bigl(\cos(\lambda T)
              +\frac{a}{\lambda} \sin(-\lambda T)
              \Bigr)^{-1}.
\]
Thus $p_T(x,y)$ in Theorem~\ref{t.FK} is represented as
\[
    p_T(x,y)=\frac1{\sqrt{2\pi T}}
       \biggl(\frac{\lambda T}{\sin(\lambda T)}
         \biggr)^{1/2}
    \exp\Bigl[\frac{a}2(y^2-x^2)
      -\frac{\lambda}2 \cot(\lambda T)
      \bigl\{x^2-2\text{\rm sec}(\lambda T)xy
         +y^2\bigr\} \Bigr].
\]
\end{example}

\begin{remark}\label{r.FK.1dim}
In the above two examples, $\bV(0)\ne0$ whenever
$\bS(t)\ne0$ for any $t\in[0,T]$.
This always happens when $d=1$.
In fact, both $\bS$ and $\bV$ obey the ODE
\[
    \Phi^{\prime\prime}-(\psi_S-\phi^\prime)\Phi=0.
\]
The Sturm-Picone comparison theorem 
(\cite[10$\cdot$30,10$\cdot$31]{ince})
implies that, if $t_0\in(0,T]$ is a zero of $\bV$, then
$\bS$ has a zero in $(t_0,T)$.
Thus, if $\bS(t)\ne0$ for any $t\in[0,T]$, then
$\bV(t)\ne0$ for $t\in(0,T]$.
\end{remark}

In general, we have a na\"{\i}ve sufficient condition for
the assumption to be satisfied. 

\begin{proposition}\label{p.FK}
Let $\sigma$ be as in Corollary~\ref{c.2nd.ode} or
(\ref{eq.sigma.phi.psi}).
Put
\[
    \delta=2\|\sigma_A\|_\infty+\|\sigma^\prime\|_\infty
    \quad\text{and}\quad
    \delta^\prime=|\sigma(T)|+\delta.
\]
Then the following holds.
\\
{\rm(i)}
$\det\bV(T)\ne0$ if 
\[
    \frac{T}2\sqrt{d} e^{(1+\delta)T}\delta<1.
\]
{\rm(ii)}
$\det\bS(t)\ne0$ for every $t\in[0,T]$ if 
\begin{equation}\label{p.FK.1}
    \frac{T}2\bigl\{2+T(\sqrt{d}+\delta^\prime)
       e^{(1+\delta)T}\bigr\}\delta^\prime
    <1.
\end{equation}
\end{proposition}

\begin{proof}
Let
$\alpha,\gamma\in C([0,T];\mathbb{R}^{d\times d})$,
$C_0,C_1\in \mathbb{R}^{d\times d}$, and 
$W\in C^2([0,T];\mathbb{R}^{d\times d})$ be the solution to
the ODE on $\mathbb{R}^{d\times d}$
\[
    W^{\prime\prime}+\alpha W^\prime+\gamma W=0,
    \quad W(0)=C_0,~W^\prime(0)=C_1.
\]
Put $\Delta=\|\alpha\|_\infty+\|\gamma\|_\infty$.
The ODE implies that
\[
    |W(t)|+|W^\prime(t)|
    \le |C_0|+|C_1|+(1+\Delta)\int_0^t
    \{|W(s)|+|W^\prime(s)|\} ds
    \quad\text{for }t\in[0,T].
\]
By Gronwall's inequality, we have that
\[
    |W(t)|+|W^\prime(t)|
    \le (|C_0|+|C_1|)e^{(1+\Delta)T}
    \quad\text{for }t\in[0,T].
\]
Hence it holds that
\begin{align*}
    |W(t)-C_0-tC_1|
    & \le \int_0^t \biggl(\int_0^s 
      |\alpha(u)W^\prime(u)+\gamma(u)W(u)|du\biggr) ds
    \\
    & \le \frac{t^2}2 (|C_0|+|C_1|)e^{(1+\Delta)T}\Delta
    \quad\text{for }t\in[0,T].
\end{align*}
This yields two inequalities that
\[
    \biggl|\frac1{T}W(T)-C_1\biggr| 
      \le \frac1{T}|C_0|+\frac{T}2
        (|C_0|+|C_1|)e^{(1+\Delta)T}\Delta
\]
and 
\[
    |W(t)-C_0| \le T|C_1|+\frac{T^2}2
        (|C_0|+|C_1|)e^{(1+\Delta)T}\Delta
\]
for $t\in[0,T]$.
Applying the first inequality to $\bV$ and the second one to
$\bS(T-\cdot)$, we obtain that 
\[
    \biggl|\frac1{T}\bV(T)-I_d\biggr|
      \le \frac{T}2 \sqrt{d} e^{(1+\delta)T}\delta
\]
and
\[
    |\bS(t)-I_d|
      \le T|\sigma(T)|+\frac{T^2}2
          \bigl(\sqrt{d}+|\sigma(T)|\bigr)
          e^{(1+\delta)T}\delta
      \le \frac{T}2\bigl\{
           2+T(\sqrt{d}+\delta^\prime)
           e^{(1+\delta^\prime)T}\bigr\}\delta^\prime.
\]
Plugging the assumptions on $\delta$ and $\delta^\prime$, we
complete the proof.
\end{proof}

\begin{remark}\label{r.S}
If $\phi$ and $\psi$ are both constant functions, then $\bS$
in Theorem~\ref{t.FK} has a concrete expression.
To see this, take $C,D\in \mathbb{R}^{d\times d}$ with
$D^\dagger=D$.
Let $\phi\equiv C$ and $\psi\equiv D$, and set $A=C-C^\dagger$.
The ODE \eqref{t.FK.1} for $\bS$ reads as
\[
    \bS^{\prime\prime}-A\bS^\prime+D\bS=0,
    \quad \bS(T)=I_d,~\bS^\prime(T)=C.
\]
The function
$\hat{\bS}\in C^2([0,T];\mathbb{R}^{d\times d})$ given as
$\hat{\bS}=\bS(T-\cdot)$ obeys the ODE
\[
    \hat{\bS}^{\prime\prime}+A\hat{\bS}^\prime
    +D\hat{\bS}=0,
    \quad \hat{\bS}(0)=I_d,~\hat{\bS}^\prime(0)=-C.
\]

Let $\lambda_1,\dots,\lambda_{2d}$ be the eigenvalues
counted with multiplicity of the matrix
\[
    \begin{pmatrix} 0 & I_d \\ -D & -A \end{pmatrix}
    \in \mathbb{R}^{2d\times 2d}.
\]
Define 
$r_1,\dots,r_{2d}\in C^\infty(\mathbb{R};\mathbb{C})$ by the
system of ODEs 
\[
    \left\{
      \begin{aligned}     
         & r_1^\prime=\lambda_1 r_1,
         \\
         & r_j^\prime=r_{j-1}+\lambda_j r_j
           \quad\text{for }2\le j\le 2d,
         \\
         & r_1(0)=1,~~
           r_2(0)=\dots=r_{2d}(0)=0,
      \end{aligned}
    \right.
\]
and $Q_1,\dots,Q_{2d}\in \mathbb{C}^{d\times d}$
successively as 
$Q_1=I_d$, $Q_2=-C$, and
\begin{align*}
    Q_{3+n}
    = & -\sum_{j=1}^{n+1} \bigl\{r_j^{(n)}(0)D
         +r_j^{(n+1)}(0)A+r_j^{(n+2)}(0)I_d\bigr\}Q_j
    \\
    & 
      -\bigl\{r_{n+2}^{(n+1)}(0)A+r_{n+2}^{(n+2)}(0)I_d
       \bigr\}Q_{n+2}
      \quad\text{for }0\le n\le 2d-3.
\end{align*}
By Proposition~\ref{p.ode.U}, it holds that
\[
    \hat{\bS}=\sum_{j=1}^{2d} r_j Q_j
    \quad\text{and hence }\quad
    \bS=\sum_{j=1}^{2d} r_j(T-\cdot) Q_j.
\]

If $\lambda_1,\dots,\lambda_{2d}$ are different from each
other, then we have another expression of $\bS$.
In fact, define 
$\widehat{U}_n=(\widehat{U}_n^{pq})_{1\le p,q\le d}
 \in \mathbb{R}^{d\times d}$ for $0\le n\le 2d-1$
and 
$\widehat{Q}_j=(\widehat{Q}_j^{pq})_{1\le p,q\le d}
 \in \mathbb{C}^{d\times d}$ 
for $1\le j\le 2d$ as
\[
    \widehat{U}_0=I_d,\quad
    \widehat{U}_1=-C,\quad
    \widehat{U}_n=-(D\widehat{U}_{n-2}+A\widehat{U}_{n-1})
    \quad\text{for }3\le n\le 2d-1
\]
and 
\[
    \begin{pmatrix} \widehat{Q}_1^{pq} \\ \vdots \\ \vdots
      \\  \widehat{Q}_{2d}^{pq} \end{pmatrix}
    =\begin{pmatrix} 1 & \cdots & \cdots & 1 \\
       \lambda_1 & \cdots & \cdots & \lambda_{2d} \\
       \vdots & & & \vdots \\
      \lambda_1^{2d-1} & \cdots & \cdots &
      \lambda_{2d}^{2d-1}
     \end{pmatrix}^{-1}
     \begin{pmatrix} \widehat{U}_0^{pq} \\ \vdots \\
      \vdots \\ \widehat{U}_{2d-1}^{pq}
     \end{pmatrix}
     \quad\text{for }1\le p,q\le d.
\]
By Proposition~\ref{p.ode.U.nu}, we have that
\[
    \hat{\bS}(t)=\sum_{j=1}^{2d} e^{\lambda_j t} \widehat{Q}_j
    \quad\text{and hence}\quad
    \bS(t)=\sum_{j=1}^{2d} e^{\lambda_j(T-t)} \widehat{Q}_j
    \quad\text{for }t\in[0,T].
\]
\end{remark}

We present an application of Theorem~\ref{t.lap.lin} and  
Remark~\ref{r.FK} to Ornstein-Uhlenbeck processes.
Let $\chi\in C([0,T];\mathbb{R}^{d\times d})$ and
$\xi^{[\chi]}=\{\xi^{[\chi]}(t)\}_{t\in[0,T]}$ be the unique
solution to the stochastic differential equation (SDE in
short)
\[
    d\xi^{[\chi]}(t)=d\theta(t)+\chi(t)\xi^{[\chi]}(t) dt,
    \quad \xi^{[\chi]}(0)=0.    
\]

\begin{corollary}\label{c.ode.ou}
Let $\chi$ and $\xi^{[\chi]}$ be as above.
\\
{\rm(i)}
Let $\gamma\in C([0,T];\mathbb{R}^{d\times d})$ and 
$\beta \in C^1([0,T];\mathbb{R}^{d\times d})$.
Assume that $\gamma^\dagger=\gamma$, and the solution 
$\alpha\in C^1([0,T];\mathbb{R}^{d\times d})$ to the ODE
\[
    \alpha^\prime=\chi \alpha,\quad \alpha(T)=I_d
\]
satisfies that $\det \alpha(t)\ne0$ for any $t\in[0,T]$.
Then it holds that
\begin{align}
 \label{c.ode.ou.1}
    & \int_{\mathcal{W}} f(\xi^{[\chi]})
      \exp\biggl(\frac12\int_0^T 
        \langle \gamma(t)\xi^{[\chi]}(t),
             \xi^{[\chi]}(t)\rangle dt
       +\frac12 \langle \beta(T)\xi^{[\chi]}(T),
             \xi^{[\chi]}(T)\rangle \biggr)
    \\
    & =e^{\frac12\int_0^T\tr[\beta_S(t)]dt}
       \int_{\mathcal{W}} f 
         e^{\mathfrak{a}_{\chi+\beta_S,\gamma+\beta_S^\prime-\chi^\dagger\chi}}
         d\mu
 \nonumber
\end{align}
for every non-negative $f\in C_b(\mathcal{W})$,
where the identity may hold as $\infty=\infty$.
\\
{\rm(ii)}
If $\chi$ is of $C^1$-class and the solution 
$\bS\in C^2([0,T];\mathbb{R}^{d\times d})$  
to the ODE 
\[
    \bS^{\prime\prime}-2\chi_A\bS^\prime
    +(\gamma-\chi^\dagger\chi-\chi^\prime)\bS=0,
    \quad \bS(T)=I_d,~\bS^\prime(T)=\chi(T)+\beta_S(T)
\]
satisfies that $\det\bS(t)\ne0$ for every $t\in[0,T]$, then
it holds that 
\begin{align*}
    & \int_{\mathcal{W}} f(\xi^{[\chi]})
        \exp\biggl(\frac12 \int_0^T \langle 
             \gamma(t)\xi^{[\chi]}(t),\xi^{[\chi]}(t) \rangle dt
           +\frac12\langle \beta(T)\xi^{[\chi]}(T),
              \xi^{[\chi]}(T)\rangle\biggr) d\mu
    \\
    & =\biggl(\frac{e^{-\int_0^T\tr\chi_S(t)dt}}{
         \det \bS(0)}\biggr)^{\frac12}
       \int_{\mathcal{W}} f(\xi_{\bS})d\mu
    \quad\text{for every }f\in C_b(\mathcal{W}).
\end{align*}
\end{corollary}

This expression plays a key role in the stochastic approach
to the KdV equation (see \cite{ikeda-t,st-spa} or
Lemma~\ref{l.reflpot.cv0} below).

\begin{proof}
(i)
By \eqref{eq:xi.alpha}, 
$(\boldsymbol{\xi}_\alpha(\cdot))(t)=\xi^{[\chi]}(t)$
for $t\in[0,T]$.
Put 
\[
    \varphi=\exp\biggl(\frac12 \int_0^T \langle 
         \gamma(t)\theta(t),\theta(t) \rangle dt
           +\frac12\langle \beta(T)\theta(T),
              \theta(T)\rangle\biggr)
\]
and $\varphi_N=\varphi\wedge N$ for $N\in \mathbb{N}$.
By \eqref{t.lap.lin.4} and \eqref{t.alpha.21}, we obtain
that 
\begin{equation}\label{c.ode.ou.21}
    \int_{\mathcal{W}} f 
     \varphi_N \exp\biggl( \int_0^T
     \langle \chi(t)\theta(t),d\theta(t)\rangle 
     -\frac12 \int_0^T |\chi(t)\theta(t)|^2 dt
     \biggr)d\mu
    =\int_{\mathcal{W}} (f\varphi_N)(\xi^{[\chi]})d\mu
\end{equation}
for non-negative $f\in C_b(\mathcal{W})$.

Since
\begin{align*}
    \frac12\langle \beta(T)\theta(T),\theta(T)\rangle
    & =\frac12\langle \beta_S(T)\theta(T),\theta(T)\rangle
    \\
    & =\frac12 \int_0^T \langle \beta_S^\prime(t)\theta(t),
       \theta(t)\rangle dt
     +\int_0^T \langle \beta_S(t)\theta(t),d\theta(t)
       \rangle 
     +\frac12\int_0^T \tr[\beta_S(t)]dt,
\end{align*}
we have that
\begin{align*}
    & \varphi_N \exp\biggl( \int_0^T
       \langle \chi(t)\theta(t),d\theta(t)\rangle 
       -\frac12 \int_0^T |\chi(t)\theta(t)|^2 dt
       \biggr)
    \\
    & \nearrow \exp\biggl(
      \mathfrak{a}_{\chi+\beta_S,\gamma+\beta_S^\prime-\chi^\dagger\chi}^0
      +\frac12 \int_0^T \tr[\beta_S(t)]dt\biggr)
    \quad\text{as }N\to\infty.
\end{align*}
Thanks to the monotone convergence theorem, letting
$N\to\infty$ in \eqref{c.ode.ou.21}, we obtain
\eqref{c.ode.ou.1}.
\\
(ii)
Set $\phi=\chi+\beta_S$ and
$\psi=\gamma+\beta_S^\prime-\chi^\dagger\chi$. 
Then $\phi_S=\chi_S+\beta_S$, $\phi_A=\chi_A$, and 
$\psi_S-\phi^\prime=\gamma-\chi^\dagger\chi-\chi^\prime$. 
Applying Remark~\ref{r.FK}, we obtain the desired identity
from \eqref{c.ode.ou.1}.
\end{proof}

As another application of Corollary~\ref{c.2nd.ode} and 
Proposition~\ref{p.FK}, we give an explicit expression of
the characteristic function of $\p_\sigma$ as follows.

\begin{theorem}\label{t.ch.fn}
Let $\sigma\in C^1([0,T];\mathbb{R}^{d\times d})$.
For $\zeta\in \mathbb{C}$, denote by 
$\bS_\zeta\in C^2([0,T];\mathbb{C}^{d\times d})$ the unique
solution to the ODE on $\mathbb{C}^{d\times d}$
\[
    \bS_\zeta^{\prime\prime}
    -2\zeta\sigma_A \bS_\zeta^\prime
    -\zeta\sigma^\prime \bS_\zeta=0,
    \quad \bS_\zeta(T)=I_d,~~
    \bS_\zeta^\prime(T)=\zeta\sigma(T).
\]
Set 
$\Omega_\sigma=\{\zeta\in \mathbb{C}\mid 
 |\text{\rm Re}\zeta|\|B_{\rho_\sigma}\|_\op<1 \}$.
Then, for every $\zeta\in \Omega_\sigma$, 
$\det\bS_{\zeta}(0)\ne0$ and it holds that
\begin{equation}\label{t.ch.fn.1}
    \biggl(\int_{\mathcal{W}} e^{\zeta\p_\sigma} 
      d\mu\biggr)^2
    =\frac{e^{-\zeta\int_0^T\tr\sigma_S(t)dt}}{
         \det\bS_\zeta(0)}.      
\end{equation}
In particular, for every $\lambda\in \mathbb{R}$, 
$\det\bS_{\kyosu\lambda}(0)\ne0$ and it holds that
\[
    \biggl(\int_{\mathcal{W}} e^{\kyosu\lambda\p_\sigma} 
      d\mu\biggr)^2
    =\frac{e^{-\kyosu\lambda\int_0^T\tr\sigma_S(t)dt}}{
         \det\bS_{\kyosu\lambda}(0)}.      
\]
\end{theorem}

To show this, we prepare a lemma on holomorphy.

\begin{lemma}\label{l.hol}
Let $\eta\in\stwo$.
Set 
$\mathcal{D}_\eta=\bigl\{\zeta\in \mathbb{C}\mid 
  |\text{\rm Re}\zeta|\|B_\eta\|_\op<1\bigr\}$.
For any $\zeta\in \mathcal{D}_\eta$, 
$|e^{\zeta\q_\eta}|\in L^1(\mu)$.
Furthermore, the function 
\[
    g:\mathcal{D}_\eta\ni\zeta\mapsto
      \int_{\mathcal{W}} e^{\zeta\q_\eta}d\mu \in \mathbb{C} 
\]
is holomorphic.
\end{lemma}

\begin{proof}
Take a $\delta>0$ with $\delta \|B_\eta\|_\op<1$ and set 
$\mathcal{D}^{(\delta)}
 =\{\zeta\in \mathbb{C}\mid |\text{\rm Re}\zeta|<\delta\}$.
It suffices to show that
$|e^{\zeta\q_\eta}|\in L^1(\mu)$ for any
$\zeta\in \mathcal{D}^{(\delta)}$ and 
$g$ is holomorphic on $\mathcal{D}^{(\delta)}$.

If $\zeta\in \mathcal{D}^{(\delta)}$, then 
$|e^{\zeta\q_\eta}|\le e^{\delta|\q_\eta|}
    \le e^{\delta \q_\eta}+e^{-\delta \q_\eta}$.
Since 
$\Lambda(B_{\pm\delta\eta})\le\delta\|B_\eta\|_\op<1$, 
by Lemma~\ref{l.q.eta.int},  
$e^{\pm\delta \q_\eta}\in L^1(\mu)$.
Hence $|e^{\zeta\q_\eta}|\in L^1(\mu)$.

Choose an $\varepsilon>0$ such that 
$(\delta+\varepsilon)\|B_\eta\|_\op<1$.
Observe that
\[
    \biggl|\frac{d}{d\zeta}
       e^{\zeta\q_\eta}\biggr|
    =|\q_\eta e^{\zeta\q_\eta}|
    \le \frac1{\varepsilon} 
        e^{(\delta+\varepsilon)|\q_\eta|}
\]
for any $\zeta\in \mathcal{D}^{(\delta)}$.
By the similar argument as above, we have that
$e^{(\delta+\varepsilon)|\q_\eta|}\in L^1(\mu)$.
Applying the dominated convergence theorem of  
differentiation type, we see that $g$ is holomorphic on
$\mathcal{D}^{(\delta)}$. 
\end{proof}

\begin{proof}[Proof of Theorem~\ref{t.ch.fn}]
Since the function 
$\mathbb{C}\ni\zeta\mapsto
 \bS_\zeta(0)\in \mathbb{C}^{d\times d}$ is holomorphic,
so is the function
$\mathbb{C}\ni\zeta\mapsto
 \det \bS_\zeta(0)\in \mathbb{C}$.
By Lemma~\ref{l.hol}, the function 
$g:\Omega_\sigma\to \mathbb{C}$ given as
\[
    g(\zeta)=(\det\bS_\zeta(0))
    e^{\zeta\int_0^T\tr\sigma_S(t)dt}
     \biggl(\int_{\mathcal{W}} e^{\zeta\p_\sigma}
     d\mu\biggr)^2
    \quad\text{for }\zeta\in\Omega_\sigma
\]
is holomorphic.

Take a $\delta^\prime>0$ satisfying \eqref{p.FK.1}.
For $\lambda\in \mathbb{R}$ with 
$|\lambda|(2\|\sigma_A\|_\infty
 +\|\sigma^\prime\|_\infty+|\sigma(T)|)<\delta^\prime$, 
by Proposition~\ref{p.FK} and Corollary~\ref{c.2nd.ode},
$g(\lambda)=1$. 
Hence $g=1$ on $\Omega_\sigma$.
This implies that $\det\bS_\zeta(0)\ne0$ for every 
$\zeta\in\Omega_\sigma$ and \eqref{t.ch.fn.1} holds.
\end{proof}

\begin{remark}\label{r.ode.history}
The evaluation of 
$\int_{\mathcal{W}} e^{\mathfrak{a}_{\phi,\psi}^0} f d\mu$ by using 
solutions to ODEs was first made by Cameron and Martin
(\cite{CM0,CM}) when $d=1$. 
In their case, $\phi=0$ and hence
$\mathfrak{a}_{\phi,\psi}^0$ is the weighted $L^2$-integral
$\int_0^T \psi(t)|\theta(t)|^2 dt$ of sample path.
The corresponding ODE is the Sturm-Liouville equation 
\[
    f^{\prime\prime}+\psi f=0,\quad f(T)=1,~f^\prime(T)=0.
\]
If $\psi\equiv1$, then it corresponds to the harmonic
oscillator (\cite{CM,kac} and 
\cite[Subsection~5.8.1]{mt-cambridge}). 
When $d=2$, 
$\phi=\begin{pmatrix} 0 & -1 \\ 1 & 0 \end{pmatrix}$, and 
$\psi=0$, $\mathfrak{a}_{\phi,\psi}^0$ is L\'evy's
stochastic area and the evaluation is know as 
L\'evy's stochastic area formula (\cite{levy,yor} and
\cite[Subsection~5.8.2]{mt-cambridge}).  
Such an evaluation was extended to general dimensions 
by the author (\cite{st-ptrf}) with the additional
assumption that $\phi^\dagger=-\phi$.
The extension was made by using Girsanov transformations.
In \cite{st-kjm-2024}, another evaluation with the help of
change of variables formula on the Wiener space was achieved
by the author under an assumption of smallness of $\phi$ and
$\psi$.
The evaluation for general $\phi$, $\psi$, and $x$ was
achieved by the author recently 
(\cite{st-kjm-2026}).
\end{remark}

\section{Heat kernel on two-step nilpotent Lie group}
\label{sec.heat.kernel}
In this section, we apply Corollary~\ref{c.2nd.ode} to
compute heat kernels on two-step nilpotent Lie groups.
We start with introducing the Lie group on which we work.
Let 
\[
    \as(d)=\{A\in \mathbb{R}^{d\times d}\mid 
       A^\dagger=-A\}
    \quad\text{and}\quad
    \mathbb{G}(d)=\mathbb{R}^d\times \as(d).
\]
$\mathbb{G}(d)$ is the Lie group with the product
\[
    (x,A)\cdot(y,B)
    =\biggl(x+y,A+B+\frac12
      \bigl(x^i y^j -x^j y^i 
         \Bigr)_{1\le i,j\le d}\biggr)
\quad\text{for }(x,A),(y,B)\in \mathbb{G}(d).
\]

Put $N(d)=\frac12 d(d-1)$ and denote by 
$a=(a^{i j})_{1\le i<j\le d}$ 
the coordinate system on $\mathbb{R}^{N(d)}$ obtained by identifying
$\mathbb{R}^{N(d)}$ with the space of $d\times d$ matrices 
whose $(i,j)$-entries for $i\ge j$ all vanish.
Define 
$\A:\mathbb{R}^{N(d)}\ni a\mapsto 
 \A(a)=\bigl(\A(a)^{i j}
  \bigr)_{1\le i,j\le d}\in \as(d)$ 
by
\[
    \A(a)^{i j}
    =\begin{cases}
        a^{i j} & \text{if }i<j,
        \\
        -a^{j i} & \text{if }i>j,
        \\
        0 & \text{if }i=j
     \end{cases}
\]
for $a=(a^{i j})_{1\le i<j\le d}$.
The global coordinate system 
$(x,a)=\bigl((x^i)_{1\le i\le d},
      (a^{ij})_{1\le i<j\le d}\bigr)$
on $\mathbb{G}(d)$ is given via the mapping
\[
    \mathbb{R}^d\times \mathbb{R}^{N(d)}
    \ni (x,a)\mapsto (x,\A(a))\in \mathbb{G}(d).
\]

For $(z,C)\in \mathbb{G}(d)$, let $\mathcal{T}_{(z,C)}$ be the tangent
space of $\mathbb{G}(d)$ at $(z,C)$, and 
$(\varphi_*)_{(z,C)}:\mathcal{T}_{(z,C)}\to
 \mathcal{T}_{\varphi((z,C))}$
be the differential of smooth function $\varphi$ on $\mathbb{G}(d)$ 
at $(z,C)$.
Denote by $L(x,A):\mathbb{G}(d)\to \mathbb{G}(d)$ the left action by
$(x,A)\in \mathbb{G}(d)$: 
$L(x,A)[(y,B)]=(x,A)\cdot(y,B)$ for $(y,B)\in \mathbb{G}(d)$.
It is easily seen that
\begin{equation}\label{eq.laction.x}
    (L(x,A)_*)_{(0,0)}\Bigl(\frac{\partial}{\partial x^i}\Bigr)_0
    =\Bigl(\frac{\partial}{\partial x^i}\Bigr)_x
     +\frac12\sum_{j<i} x^j
         \Bigl(\frac{\partial}{\partial a^{ji}}\Bigr)_A
     -\frac12\sum_{j>i} x^j
         \Bigl(\frac{\partial}{\partial a^{ij}}\Bigr)_A
\end{equation}
and
\begin{equation}\label{eq.laction.A}
    (L(x,A)_*)_{(0,0)}\Bigl(\frac{\partial}{\partial a^{ij}}\Bigr)_0
    =\Bigl(\frac{\partial}{\partial a^{ij}}\Bigr)_A.
\end{equation}
Define the left-invariant $C^\infty$-vector fields $V_1,\dots,V_d$ on  
$\mathbb{G}(d)$ as
\[
    V_i=\frac{\partial}{\partial x^i}
     +\frac12\sum_{j=1}^{i-1} x^j 
         \frac{\partial}{\partial a^{j i}}
     -\frac12\sum_{j=i+1}^d x^j 
         \frac{\partial}{\partial a^{i j}}
    \quad\text{for }1\le i\le d.
\]
Their Lie brackets enjoy the property that
\[
    [V_i,V_j]=\frac{\partial}{\partial a^{ij}}
    \quad\text{for }1\le i<j\le d
    \quad\text{and}\quad
    [V_i,[V_j,V_k]]=0
    \quad\text{for }1\le i,j,k\le d.
\]
Thus, in conjunction with \eqref{eq.laction.x} and
\eqref{eq.laction.A}, we see that $\mathbb{G}(d)$ is a two-step
nilpotent Lie group.

Put
\[
    \sa^{ij}(t)=\frac12\int_0^t
      \{\theta^i(s)d\theta^j(s)
        -\theta^j(s)d\theta^i(s)\}
    \quad\text{for }1\le i<j\le d,~t\in[0,T]
\]
and 
\[
    \sa(t)=\bigl(\sa^{i j}(t)\bigr)_{1\le i<j\le d}
    \quad \text{for }t\in[0,T].
\]
The stochastic process
$\{\boldsymbol{\Xi}(t)\}_{t\in[0,T]}$ defined as
\[
    \boldsymbol{\Xi}(t)
    =\bigl(\theta(t),\A(\sa(t))\bigr)
    \quad\text{for }t\in[0,T]
\]
is the unique solution to the following SDE on $\mathbb{G}(d)$:
\[
    d\boldsymbol{\Xi}(t)
    =\sum_{i=1}^d V_i(\boldsymbol{\Xi}(t))\circ
     d\theta^i(t),
    \quad \boldsymbol{\Xi}(0)=(0,0), 
\]
where $\circ d\theta^i(t)$ stands for the Stratonovich
integral with respect to $\{\theta^i(t)\}_{t\in[0,T]}$. 
Due to the left invariance of $V_i$s, the solution to
the above SDE with the initial condition 
$\boldsymbol{\Xi}(0)=(x,A)$ is exactly the stochastic
process 
$\bigl\{(x,A)\cdot\boldsymbol{\Xi}(t)\bigr\}_{t\in[0,T]}$. 
Thus the heat kernel associated with the second order
differential operator 
\[
    \frac12\sum_{i=1}^d V_i^2,
\]
i.e., the transition density function of the diffusion
process generated by the differential operator is computed
by using the density function of $\boldsymbol{\Xi}(t)$.  
On account of the arbitrariness of $T$, it suffices to know
an expression of the density function only when $t=T$.

\begin{theorem}\label{t.as(d)}
Put 
\[
    p_T(x,a)
    =\frac{1}{(2\pi)^{N(d)}\sqrt{2\pi T}^d} 
     \int_{\mathbb{R}^{N(d)}} 
      \frac{e^{\kyosu\langle a,b\rangle}}{
          \det(\sh[\kyosu\frac{T}2\A(b)])}
      \exp\biggl(-\frac1{2T}\biggl\langle
        \Bigl(\tnh\Bigl[\kyosu\frac{T}2\A(b)\Bigr]
          \Bigr)^{-1}x,x \biggr\rangle \biggr) db
\]
for $(x,a)\in \mathbb{R}^d\times \mathbb{R}^{N(d)}$,
where $\langle\cdot,\cdot\rangle$ stands for the 
Euclidean inner product in $\mathbb{R}^{N(d)}$.
Then it holds that
\begin{equation}\label{t.as(d).1}
     \int_{\mathcal{W}} f(\boldsymbol{\Xi}(T)) d\mu
     =\int_{\mathbb{R}^d\times \mathbb{R}^{N(d)}}
       f(x,\A(a))p_T(x,a)dx da
\end{equation}
for every $f\in C_b(\mathbb{R}^d\times\as(d))$.
\end{theorem}

\begin{remark}\label{r.as(d)}
(i)
By \eqref{e.sAD.1+}, there are $k\in \mathbb{N}, \le \frac{d}2$,
$\lambda_1,\dots,\lambda_k\in \mathbb{R}\setminus\{0\}$, and
a unitary matrix $U\in \mathbb{C}^{d\times d}$ such that 
\[
    \ch[\kyosu A] 
    = U^* \text{\rm diag}[\cosh(\lambda_1),
         \cosh(\lambda_1),\dots,\cosh(\lambda_k),
         \cosh(\lambda_k),\underbrace{1,\dots,1}_{d-2k}]
        U
\]
and 
\[
    \sh[\kyosu A] 
    = U^* \text{\rm diag}\biggl[
         \frac{\sinh(\lambda_1)}{\lambda_1},
         \frac{\sinh(\lambda_1)}{\lambda_1},
         \dots,
         \frac{\sinh(\lambda_k)}{\lambda_k},
         \frac{\sinh(\lambda_k)}{\lambda_k},
         \underbrace{1,\dots,1}_{d-2k}\biggr] U.
\]
Hence we obtain that
\[
    \det(\ch[\kyosu A])=\prod_{i=1}^k 
     (\cosh(\lambda_i))^2 \ne0,
    \quad
    \det(\sh[\kyosu A])=\prod_{i=1}^k 
     \biggl(\frac{\sinh(\lambda_i)}{\lambda_i}
      \biggr)^2 \ne0,
\]
and 
\[
    \det(\tnh[\kyosu A])\ne0.
\]
(ii)
The identity \eqref{t.as(d).1} was shown by Inahama and the
author in \cite{inahama-t}.  
\end{remark}

For the proof of Theorem~\ref{t.as(d)}, we prepare a lemma.

\begin{lemma}\label{l.as(d)}
For every $A\in\as(d)$, it holds that
\begin{align}
\label{l.as(d).1}
    & \int_{\mathcal{W}}
        e^{\kyosu  \int_0^T \langle A\theta(t),
                d\theta(t)\rangle}
         \delta_x(\theta(T))d\mu
    \\
    & =\biggl(\frac1{(2\pi T)^d \det(\sh[\kyosu TA])}
        \bigg)^{\frac12}
        \exp\biggl(-\frac1{2T}\bigl\langle
         (\tnh[\kyosu TA])^{-1}x,x 
         \bigr\rangle\biggr)
    \quad\text{for }x\in \mathbb{R}^d.
 \nonumber
\end{align}
\end{lemma}

\begin{proof}
Let $A\in\as(d)$.
Take a $\delta^\prime>0$ satisfying \eqref{p.FK.1}.
For $\lambda\in \mathbb{R}$ with 
$3|\lambda||A|<\delta^\prime$, set
$\sigma\in C^1([0,T];\mathbb{R}^{d\times d})$ 
so that $\sigma\equiv \lambda A$.
Since 
$|\sigma(T)|+2\|\sigma_A\|_\infty
 +\|\sigma^\prime\|_\infty
 =3|\lambda||A|<\delta^\prime$,
by Proposition~\ref{p.FK}, the solution 
$\bS\in C^2([0,T];\mathbb{R}^{d\times d})$ to the ODE 
\[
    \bS^{\prime\prime}-2\lambda A\bS^\prime=0,
    \quad \bS(T)=I_d,~~\bS^\prime(T)=\lambda A
\]
satisfies that $\det \bS(t)\ne0$ for any $t\in[0,T]$.
Due to \eqref{e.sAD.4} in Example~\ref{e.sAD}, we have that
\begin{align}
    & \int_{\mathcal{W}} 
      e^{\lambda \int_0^T \langle A\theta(t),d\theta(t)\rangle}
        \delta_x(\theta(T))d\mu
\label{l.as(d).21}
    \\
    & =\biggl(\frac1{(2\pi T)^d\det(\sh[\lambda TA])}
          \biggr)^{\frac12}
      \exp\biggl(-\frac1{2T}\bigl\langle
        (\tnh[\lambda TA])^{-1}x,x 
        \bigr\rangle\biggr)
    \quad\text{for }x\in \mathbb{R}^d.
 \nonumber
\end{align}

Let 
$\Omega_A=\{\zeta\in \mathbb{C}\mid 
 3|\text{Re}\zeta||A|<\delta^\prime\}$.
By the same argument as in Lemma~\ref{l.hol}, we know that 
the mapping 
\[
    \Omega_A\ni\zeta\mapsto
    \int_{\mathcal{W}} 
      e^{\zeta \int_0^T \langle A\theta(t),
             d\theta(t)\rangle}
        \delta_x(\theta(T)) d\mu
\]
is holomorphic.
Since $\det(\sh[\kyosu aTA])\ne0$ for any $a\in \mathbb{R}$
as was seen in Remark~\ref{r.as(d)}, there is a domain
$\Omega^\prime\subset \mathbb{C}$ such that
$\{\kyosu\lambda\mid \lambda\in \mathbb{R}\}
 \subset \Omega^\prime$
and 
$\det(\sh[\zeta TA])\ne0$ for $\zeta\in\Omega^\prime$.
Hence \eqref{l.as(d).21} extends holomorphically to
$\Omega_A\cap\Omega^\prime$, that is, it holds that
\begin{align*}
    & \int_{\mathcal{W}} 
      e^{\zeta \int_0^T \langle A\theta(t),
             d\theta(t)\rangle}
        \delta_x(\theta(T)) d\mu
    \\
    & =\biggl(\frac1{
          (2\pi T)^d\det(\sh[\zeta TA])}\biggr)^{\frac12}
       \exp\biggl(-\frac1{2T}\bigl\langle
        (\tnh[\zeta TA])^{-1}x,x 
        \bigr\rangle\biggr)
    \quad\text{for }\zeta\in 
    \Omega_A\cap\Omega^\prime.
\end{align*}
Substituting $\zeta=\kyosu$, we obtain \eqref{l.as(d).1}.
\end{proof}

\begin{proof}[Proof of Theorem~\ref{t.as(d)}]
Let 
$f\in \mathscr{S}(\mathbb{R}^d\times\mathbb{R}^{N(d)})$.
Using the Fourier analysis and exchanging the order of
integration, we rewrite as 
\begin{align*}
    & \int_{\mathcal{W}} f(\boldsymbol{\Xi}(T)) d\mu
      =\int_{\mathbb{R}^d} 
       \biggl(\int_{\mathcal{W}} f(x,\sa(T)) 
         \delta_x(\theta(T)) d\mu\biggr) dx
    \\
    & =\int_{\mathbb{R}^d\times \mathbb{R}^{N(d)}}
             \biggl(\int_{\mathbb{R}^{N(d)}}
            \frac{e^{\kyosu\langle b,a\rangle}}{
                         (2\pi)^{N(d)}}
          \biggl(\int_{\mathcal{W}} 
              e^{-\kyosu\langle b,\sa(T)\rangle}
              \delta_x(\theta(T)) d\mu\biggr)
         db\biggr)     
       f(x,a) dxda.
\end{align*}
For 
$b=(b^{i j})_{1\le i<j\le d}
 \in \mathbb{R}^{N(d)}$, 
it holds that
\begin{align*}
    -\langle b,\sa(T)\rangle
    & =-\sum_{i<j} b^{i j} 
      \int_0^T \{\theta^i(t) d\theta^j(t)
                -\theta^j(t)d\theta^i(t) \}
    \\
    & 
      =\sum_{i,j=1}^d \A(b)^{ji}
      \int_0^T \theta^i(t) d\theta^j(t)
    =\int_0^T \langle \A(b)\theta(t),d\theta(t)\rangle.
\end{align*}
Plugging this and Lemma~\ref{l.as(d)} into the above
identity, we obtain that
\[
    \int_{\mathcal{W}} f(\boldsymbol{\Xi}(T)) d\mu
    =\int_{\mathbb{R}^d\times \mathbb{R}^{N(d)}}
       f(x,a)p_T(x,a) dxda.
\]
Thus we arrive at \eqref{t.as(d).1}.
\end{proof}

\section{Euler, Bernoulli, and Eulerian polynomials}
\label{sec.euler}

In this section, we apply Example~\ref{e.sAD} 
to representing stochastically Euler, Bernoulli, and 
Eulerian polynomials 
In this section, we use the conditional expectation
$\mathbb{E}[\cdot|\theta(t)=x]$ given $\theta(t)=x$ instead
of pinned measures $\delta_x(\theta(t))d\mu$.
They relate to each other as
\[
    \mathbb{E}[\Psi|\theta(t)=x]
    =(2\pi t)^{\frac{d}2} \,e^{\frac1{2t}|x|^2}
     \int_{\mathcal{W}} \Psi \delta_x(\theta(t))d\mu
    \quad\text{for }
    \Psi\in\mathbb{D}^{\infty,1+}.
\]

Let $d=2$ and $T=1$.
For $a,b\in \mathbb{R}$, define
$\mathfrak{s}_{a,b}:\mathcal{W}\to \mathbb{R}$ as
\[
    \mathfrak{s}_{a,b}
    =\frac{a}2 \int_0^1 \langle J\theta(t),d\theta(t)
       \rangle+\frac{b}2|\theta(1)|^2,
    \quad\text{where }
    J=\begin{pmatrix} 0 & -1 \\ 1 & 0 \end{pmatrix}.
\]
We write simply as $\mathfrak{s}_1$ for
$\mathfrak{s}_{1,0}$.

\begin{proposition}\label{p.sa}
Let $a,b\in \mathbb{R}$.
Define $c_{a,b},v_{a,b}:[0,1]\to \mathbb{R}$ by 
\[
    c_{a,b}(t)=\cos\frac{a(1-t)}2
       -\frac{2b}{a} \sin\frac{a(1-t)}2
\]
and
\[
    v_{a,b}(t)= c_{a,b}(t)^2 \int_0^t c_{a,b}(s)^{-2} ds
\]    
for $t\in[0,1]$.
There exists an $\varepsilon>0$ such that, if
$|a|+|b|<\varepsilon$, then $c_{a,b}(t)>0$ for every
$t\in[0,1]$ and it holds that
\begin{equation}\label{p.sa.1}
    \int_{\mathcal{W}} e^{\mathfrak{s}_{a,b}} d\mu
      =\frac1{c_{a,b}(0)} 
\end{equation}
and
\begin{equation}\label{p.sa.2}
   \mathbb{E}[e^{\mathfrak{s}_{a,b}}\mid \theta(t)=x]
      =\frac{t e^{-\frac12\bigl\{
         \frac1{v_{a,b}(t)}-\frac1{t}\bigr\}|x|^2}}{
        c_{a,b}(0)v_{a,b}(t)}
   \quad\text{for $t\in(0,1]$ and $x\in \mathbb{R}^2$}.
\end{equation}
\end{proposition}

\begin{proof}
Let $|a|+|b|<\frac14$.
Since $\bigl|\frac{a(1-t)}2\bigr|<\frac{\pi}3$ for $t\in[0,1]$, we
have that
\begin{equation}\label{p.sa.21}
    \cos\frac{a(1-t)}2
      -\frac{2b}{a}\sin\frac{a(1-t)}2
    \ge \frac12-|b|>\frac14
    \quad\text{for any }t\in[0,1].
\end{equation}

Take a $\delta^\prime>0$ satisfying \eqref{p.FK.1}, and put 
$\varepsilon=\frac14\wedge
  \bigl(\frac{\sqrt{2}}3 \delta^\prime\bigr)$.
Assume that $|a|+|b|<\varepsilon$.
Set $A=\frac{a}2 J$ and $D=bI_2$.
Using the notation in Example~\ref{e.sAD}, we have that 
$\mathfrak{s}_{a,b}=\mathfrak{s}_{A,D}$.
Define the function $\sigma\in C^1([0,T];\mathbb{R}^{2\times 2})$
as $\sigma\equiv A+D$.
Then we see that
\[
    |\sigma(T)|+2\|\sigma_A\|_\infty+\|\sigma^\prime\|_\infty
    \le \frac2{\sqrt{2}}(|a|+|b|)<\delta^\prime.
\]
By Proposition~\ref{p.FK}, the solution $\bS$ to the ODE 
\[
    \bS^{\prime\prime}-2A\bS^\prime=0,
    \quad \bS(1)=I_2,~\bS^\prime(1)=A+D
\]
enjoys that $\det\bS(t)\ne0$ for any $t\in[0,T]$.
Due to \eqref{e.sAD.2} and \eqref{e.sAD.3}, we obtain that
\begin{equation}\label{p.sa.22}
    \int_{\mathcal{W}} e^{\mathfrak{s}_{a,b}} d\mu
      =\Bigl\{\det\Bigl(\ch[A]-b\sh[A]\Bigr)
       \Bigr\}^{-\frac12}
\end{equation}
and
\begin{equation}\label{p.sa.23}
    \int_{\mathcal{W}} e^{\mathfrak{s}_{a,b}} 
          \delta_x(\theta(t))d\mu
      =\Bigl\{\det\Bigl(\ch[A]-b\sh[A]\Bigr)
       \Bigr\}^{-\frac12}
       g_{v_t(\bS)}(x)
\end{equation}
for any $t\in(0,T]$ and $x\in \mathbb{R}^d$. 

Since $\ch[\xi J]=\cos(\xi)I_2$ and 
$\sh[\xi J]=\frac{\sin(\xi)}{\xi}I_2$ 
for $\xi\in \mathbb{R}$, we have that
\[
    \ch[A]-b\sh[A]
    =\biggl\{\cos\Bigl(\frac{a}2\Bigr)
      -\frac{2b}{a} \sin\Bigl(\frac{a}2\Bigr)
     \biggr\} I_2
    =c_{a,b}(0)I_2.
\]
Plugging this into \eqref{p.sa.22} together with the
positivity of $c_{a,b}(0)$, we obtain \eqref{p.sa.1}.

If we put 
$\boldsymbol{o}(\xi)
  =\begin{pmatrix}
       \cos\xi & -\sin\xi \\
       \sin\xi & \cos\xi 
   \end{pmatrix}$
for $\xi\in \mathbb{R}$, then 
$\e[\xi J]=\boldsymbol{o}(\xi)$.
As was seen in Example~\ref{e.sAD}, we know that
\[
    \bS(t)
    =\e[(t-1)A]\Bigl\{\ch[(t-1)A]+b(t-1)\sh[(t-1)A]\Bigr\}
    =c_{a,b}(t) \boldsymbol{o}\Bigl(\frac{a(t-1)}2\Bigr).
\]
Since $\boldsymbol{o}(\xi)$ is an orthogonal matrix, we
see that 
\[
    v_t(\bS)=\int_0^t (\bS(t)\bS(s)^{-1})
      (\bS(t)\bS(s)^{-1})^\dagger ds
     = v_{a,b}(t) I_2.
\]
Plugging this into \eqref{p.sa.23}, we obtain
\eqref{p.sa.2}. 
\end{proof}

Continuing holomorphically the above identities for
$e^{\mathfrak{s}_{a,b}}$, we obtain the following.

\begin{proposition}\label{p.sa.ext}
For $\beta,\gamma,a\in \mathbb{R}$ with $-\beta\gamma<1$ and
$|\gamma|\le\frac12$, it holds that
\begin{equation}\label{p.sa.ext.1}  
    \int_{\mathcal{W}} 
       e^{\kyosu \beta\{\mathfrak{s}_1
          +\frac{\kyosu}2\gamma|\theta(1)|^2\}} 
       d\mu
      =\biggl\{
        \biggl(\frac12+\gamma\biggr)e^{\frac{\beta}2}
        +\biggl(\frac12-\gamma\biggr)e^{-\frac{\beta}2}
       \biggr\}^{-1}
\end{equation}
and
\begin{equation}\label{p.sa.ext.2}  
   \mathbb{E}[e^{\kyosu a \mathfrak{s}_1}\mid \theta(1)=0]
     =\frac{\frac{a}2}{\sinh\frac{a}2}.    
\end{equation}
\end{proposition}

The identities \eqref{p.sa.ext.1} with $\gamma=0$ and
\eqref{p.sa.ext.2} are called 
{\it L\'evy's stochastic area formulas}.

\begin{proof}
Take a $\delta>0$ such that 
$e^{\delta|\mathfrak{s}_1|}\in L^1(\mu)$
(cf. Lemma~\ref{l.q.eta.int}).
For $n\ge3$ and $\alpha>0$, put
\[
    \Omega_{n,\alpha}
    =\Bigl\{ (z,\zeta)\in \mathbb{C}^2 \,\Big|\,
      |\text{\rm Re}z|<\frac{\delta}{n}-\alpha,
      \text{\rm Re}(z\zeta)<1-\frac2{n}-\alpha
     \Bigr\}.
\]
If $(z,\zeta)\in\Omega_{n,\alpha}$, then it holds that
\[
    \bigl|
       e^{z\{\mathfrak{s}_1+\frac{\zeta}2|\theta(1)|^2\}}
      \bigr|
    \le e^{|\text{\rm Re}z||\mathfrak{s}_1|
         +\frac12\text{\rm Re}(z\zeta)|\theta(1)|^2}
    \le 
      e^{\frac{\delta}{n}|\mathfrak{s}_1|
         +(1-\frac2{n})\frac12|\theta(1)|^2}
\]
and 
\[
    \Bigl\{|\mathfrak{s}_1|+\frac12|\theta(1)|^2
        \Bigr\} \bigl|
         e^{z\{\mathfrak{s}_1+\frac{\zeta}2|\theta(1)|^2\}}
       \bigr|
     \le \frac1{\alpha} 
        e^{\frac{\delta}{n}|\mathfrak{s}_1|
         +(1-\frac2{n})\frac12|\theta(1)|^2}.
\]
Applying H\"older's inequality with conjugates $n$ and
$\frac{n}{n-1}$, we know that 
\[
    e^{\frac{\delta}{n}|\mathfrak{s}_1|
     +(1-\frac2{n})\frac12|\theta(1)|^2}\in L^1(\mu).
\]
Hence, by the dominated convergence theorem of differential
type, the function
\[
    f(z,\zeta)=\int_{\mathcal{W}} 
     e^{z\{\mathfrak{s}_1+\frac{\zeta}2|\theta(1)|^2\}}d\mu
\]
is well-defined and holomorphic on $\Omega_{n,\alpha}$.
Thus $f$ is well-defined and holomorphic on 
\[
    \Omega\equiv \bigcup_{n\ge3} 
    \biggl\{ (z,\zeta)\in \mathbb{C}^2 \,\bigg|\,
      |\text{\rm Re}z|<\frac{\delta}{n},
      \text{\rm Re}(z\zeta)<1-\frac2{n} \biggr\}.
\]

Put
\[
    g(z,\zeta)=\cos\frac{z}2-2\zeta \sin\frac{z}2
    \quad\text{for }(z,\zeta)\in \mathbb{C}^2.
\]
Note that $g(0,0)=1$ and 
\begin{equation}\label{p.sa.ext.21}
    g(\kyosu \beta,\kyosu\gamma)
    =\Bigl(\frac12+\gamma\Bigr)e^{\frac{\beta}2}
     +\Bigl(\frac12-\gamma\Bigr)e^{-\frac{\beta}2}
    >0
    \quad\text{for $\beta,\gamma\in \mathbb{R}$ 
    with $|\gamma|\le\frac12$}.
\end{equation}
Take a $\delta^\prime>0$ such that 
\[
    g(z,\zeta)\ne0
    \quad\text{on }
    B(\delta^\prime)\equiv
    \bigl\{(z,\zeta)\in \mathbb{C}^2\mid 
       |z|+|\zeta|<\delta^\prime\bigr\}.
\]
Then there is a domain $\Omega^\prime\subset\mathbb{C}^2$
such that 
\[
    B(\delta^\prime)\cup
    \Bigl\{(\kyosu \beta,\kyosu\gamma) \,\Big|\,
       \beta,\gamma\in \mathbb{R},|\gamma|\le\frac12
    \Bigr\}
    \subset \Omega^\prime
\]
and $1/g$ is holomorphic on $\Omega^\prime$.
Thus both $f$ and $1/g$ are holomorphic on
$\Omega\cap\Omega^\prime$.

Let $\varepsilon>0$ be as in Proposition~\ref{p.sa} and 
put 
$\varepsilon^\prime=\varepsilon\wedge\frac{\delta}3
   \wedge\frac13\wedge\delta^\prime$.
If $a,b\in \mathbb{R}$ and $|a|+|b|<\varepsilon^\prime$,
then 
\[
    |a|+|ab|<\varepsilon,~~
    |\text{\rm Re}a|=|a|<\frac{\delta}3,~~
    \text{\rm Re}(ab)\le|ab|<\frac13=1-\frac23,~~
    |a|+|b|<\delta^\prime.   
\]
Hence $(a,b)$ satisfies the condition in
Proposition~\ref{p.sa} and is in $\Omega\cap\Omega^\prime$.
Applying the same proposition, we see that
\[
    f(a,b)=\int_{\mathcal{W}} e^{\mathfrak{s}_{a,ab}} d\mu
    =\frac1{g(a,b)}.
\]
Due to the holomorphy of $f$ and $1/g$ on
$\Omega\cap\Omega^\prime$, we know that
\[
    f=\frac1{g}
    \quad\text{on }\Omega\cap\Omega^\prime.
\]

Take $\beta,\gamma\in \mathbb{R}$ with $-\beta\gamma<1$ and
$|\gamma|\le\frac12$.
Then 
$(\kyosu \beta,\kyosu\gamma)\in\Omega\cap\Omega^\prime$. 
Due to the above observation, we obtain that
\[
    f(\kyosu \beta,\kyosu\gamma)
    =\frac1{g(\kyosu \beta,\kyosu\gamma)}.
\]
In conjunction with \eqref{p.sa.ext.21}, we obtain
\eqref{p.sa.ext.1}.

Let $|a|<\varepsilon$.
Notice that
\[
    \int_0^1\Bigl( \cos\frac{a(1-s)}2 \Bigr)^{-2} ds
    =\frac{\tan\frac{a}2}{\frac{a}2}.
\]
Substituting this into \eqref{p.sa.2} with $b=0$, $t=1$, and
$x=0$, we obtain that
\[
    \mathbb{E}[e^{a \mathfrak{s}_1}\mid\theta(1)=0]
    =\frac{\frac{a}2}{\sin\frac{a}2}.
\]
Continuing this holomorphically, we obtain
\eqref{p.sa.ext.2}.
\end{proof}

We apply Proposition~\ref{p.sa.ext} to representing 
stochastically Euler, Bernoulli, and Eulerian polynomials.
Euler polynomial $E_n(\xi)$ and Bernoulli polynomial
$B_n(\xi)$ of $\xi\in \mathbb{R}$, where $n\in 
\mathbb{N}\cup\{0\}$, are defined via the exponential
generating functions as  
\[
    \sum_{n=0}^\infty E_n(\xi)\frac{\zeta^n}{n!}
    =\frac{2e^{\zeta \xi}}{e^{\zeta}+1}
    \quad\text{and}\quad
    \sum_{n=0}^\infty B_n(\xi)\frac{\zeta^n}{n!}
    =\frac{\zeta e^{\zeta \xi}}{e^{\zeta}-1}
\]
for $\xi\in \mathbb{R}$ and $\zeta\in \mathbb{R}$ with
$\zeta(\xi-1)<1$. 
It may be interesting to notice that the Euler number 
$e_n$ and the Bernoulli number $b_n$ satisfy that
\[
    e_n=2^n E_n\Bigl(\frac12\Bigr)
    \quad\text{and}\quad
    b_n=B_n(0)
    \quad\text{for }n\in \mathbb{N}\cup\{0\}
\]
For these, see \cite{erdelyi}.
We have the following stochastic representations of $E_n$
and $B_n$.

\begin{theorem}\label{t.euler}
For each $n\in \mathbb{N}\cup\{0\}$, it holds that
\[
    E_n(\xi)
    =\kyosu^n \sum_{k=0}^n \binom{n}{k}
       (1-2\xi)^k \biggl(\int_{\mathcal{W}} \Bigl\{
         \mathfrak{s}_1+\frac{\kyosu}4 |\theta(1)|^2
       \Bigr\}^k d\mu\biggr) 
       \biggl(\int_{\mathcal{W}} \mathfrak{s}_1^{n-k} d\mu
          \biggr)
\]
and
\[
    B_n(\xi)=\kyosu^n \sum_{k=0}^n \binom{n}{k}
       (1-2\xi)^k \biggl(\int_{\mathcal{W}} \Bigl\{
         \mathfrak{s}_1+\frac{\kyosu}4 |\theta(1)|^2
       \Bigr\}^k d\mu\biggr) 
       \mathbb{E}[\mathfrak{s}_1^{n-k} \mid \theta(1)=0].
\]
\end{theorem}

\begin{proof}
We first show the expression of $E_n$.
Rewrite the exponential generating function of the Euler
polynomials as
\[
    \frac{2e^{\zeta \xi}}{e^{\zeta}+1}
    =\frac{e^{-\frac12\zeta(1-2\xi)}}{
       \frac12(e^{\frac{\zeta}2}+e^{-\frac{\zeta}2})}.
\]

Given $\xi\in \mathbb{R}$, choose $\zeta\in \mathbb{R}$ with
$\zeta(\xi-1)<1$ and $-\frac12\zeta(1-2\xi)<1$.
The identity \eqref{p.sa.ext.1} with $\beta=\zeta(1-2\xi)$
and $\gamma=\frac12$ implies that
\begin{equation}\label{t.euler.21}
    e^{-\frac12\zeta(1-2\xi)}
    =\int_{\mathcal{W}} 
       e^{\kyosu \zeta(1-2\xi)\{\mathfrak{s}_1
         +\frac{\kyosu}4|\theta(1)|^2\}} d\mu.
\end{equation}
The identity \eqref{p.sa.ext.1} with $\beta=\zeta$ and
$\gamma=0$ also yields that
\[
    \frac1{\frac12(e^{\frac{\zeta}2}+e^{-\frac{\zeta}2})}
    =\int_{\mathcal{W}} e^{\kyosu\zeta \mathfrak{s}_1} d\mu.
\]

Take $0<\delta<2$ with 
$\int_{\mathcal{W}} e^{\delta\{|\mathfrak{s}_1|+\frac14|\theta(1)|^2\}}
 d\mu<\infty$.
Then it holds that
\begin{align*}
    & \sum_{n=0}^\infty E_n(\xi)\frac{\zeta^n}{n!}
      =\biggl(\int_{\mathcal{W}} 
        e^{\kyosu \zeta(1-2\xi)\{\mathfrak{s}_1
         +\frac{\kyosu}4|\theta(1)|^2\}} d\mu\biggr)
        \biggl(\int_{\mathcal{W}} 
          e^{\kyosu\zeta \mathfrak{s}_1} d\mu \biggr)
    \\
    & =\biggl( \sum_{k=0}^\infty \frac1{k!}
        \{\kyosu\zeta(1-2\xi)\}^k 
        \int_{\mathcal{W}} \biggl\{\mathfrak{s}_1
          +\frac{\kyosu}4|\theta(1)|^2\biggr\}^k d\mu
        \biggr)
       \biggl(\sum_{j=0}^\infty \frac1{j!}
        (\kyosu\zeta)^j
        \int_{\mathcal{W}} \mathfrak{s}_1^j d\mu \biggr)
\end{align*}
for $\zeta\in \mathbb{R}$ with 
$\zeta(\xi-1)<1$ and 
$|\zeta(1-2\xi)|\vee|\zeta|<\delta$.
Transforming this into a series of $\zeta$ about $0$ and comparing 
the coefficients of $\zeta^n$s, we obtain the desired expression
of $E_n$s. 

We next show the expression of $B_n$.
Rewrite the exponential generating function of Bernoulli
polynomials as
\[
    \frac{\zeta e^{\zeta\xi}}{e^{\zeta}-1}
    =e^{-\frac12 \zeta(1-2\xi)} 
     \frac{\frac{\zeta}2}{\sinh\frac{\zeta}2}.
\]
By \eqref{p.sa.ext.2} and \eqref{t.euler.21}, we have that
\[
    \sum_{n=0}^\infty B_n(\xi)\frac{\zeta^n}{n!}
    =\biggl(\int_{\mathcal{W}} 
        e^{\kyosu \zeta(1-2\xi)\{\mathfrak{s}_1
         +\frac{\kyosu}4|\theta(1)|^2\}} d\mu\biggr)
     \mathbb{E}[e^{\kyosu \zeta \mathfrak{s}_1}\mid\theta(1)=0].
\]
Transforming the right hand side of the identity into a
series of $\zeta$, and comparing the coefficients of
$\zeta^n$s, we obtain the desired expression of $B_n$s.
\end{proof}

The similar method works for Eulerian polynomials.
The Eulerian polynomials $P_n(\xi)$ of type A and
$P(B_n;\xi)$ of type B are defined by 
\[
    \frac{P_n(\xi)}{(1-\xi)^{n+1}}
    =\sum_{k=0}^\infty (k+1)^n \xi^k
    \quad\text{and}\quad
    \frac{P(B_n;\xi)}{(1-\xi)^{n+1}}
    =\sum_{k=0}^\infty (2k+1)^n \xi^k
\]
for $n\in \mathbb{N}\cup\{0\}$ and $\xi\in \mathbb{R}$ with
$|\xi|<1$, and their exponential generating functions are
given by 
\[
    \sum_{n=0}^\infty P_n(\xi)\frac{\zeta^n}{n!}
      =\frac{(1-\xi)e^{(1-\xi)\zeta}}{
              1-\xi e^{(1-\xi)\zeta}}
      \quad\text{for $\zeta\in \mathbb{R}$ with
         $|\xi|e^{(1-\xi)\zeta}<1$}
\]
and
\[
    \sum_{n=0}^\infty P(B_n;\xi)\frac{\zeta^n}{n!}
      =\frac{(1-\xi)e^{(1-\xi)\zeta}}{
              1-\xi e^{2(1-\xi)\zeta}}
      \quad\text{for $\zeta\in \mathbb{R}$ with
         $|\xi|e^{2(1-\xi)\zeta}<1$}
\]
(cf.\cite{cohen,hirzebruch}).
We have the following stochastic expressions of $P_n$s and
$P(B_n;\cdot)$s.

\begin{proposition}\label{p.eul.poly}
It holds that
\begin{align*}
    P_n(\xi)
    = & \sum_{k=0}^n \binom{n}{k}
       (1-\xi)^k \biggl(\int_{\mathcal{W}} \Bigl\{
         \kyosu\mathfrak{s}_1+\frac14 |\theta(1)|^2
         \Bigr\}^k d\mu\biggr) 
    \\
    & \hphantom{\sum_{k=0}^n}
      \times 
      \biggl(\int_{\mathcal{W}} \Bigl\{
          \kyosu(1-\xi)\mathfrak{s}_1
          +\frac{1+\xi}4 |\theta(1)|^2 \Bigr\}^{n-k}
          d\mu \biggr)
\end{align*}
and
\[
    P(B_n;\xi)
      =\int_{\mathcal{W}} \Bigl\{
          2\kyosu(1-\xi)\mathfrak{s}_1
          +\frac{1+\xi}2 |\theta(1)|^2 \Bigr\}^n
          d\mu
\]
for $n\in \mathbb{N}\cup\{0\}$ and $\xi\in \mathbb{R}$ with
$|\xi|<1$. 
\end{proposition}

\begin{proof}
We first show the expression of $P_n$.
Rewrite the exponential generating function of $P_n$s as
\[
    \frac{(1-\xi)e^{(1-\xi)\zeta}}{
              1-\xi e^{(1-\xi)\zeta}}
    =e^{\frac12(1-\xi)\zeta}
     \biggl\{\frac1{1-\xi}e^{-\frac12(1-\xi)\zeta}
       -\frac{\xi}{1-\xi}e^{\frac12(1-\xi)\zeta}       
     \biggr\}^{-1}.
\]

Given $\xi\in \mathbb{R}$ with $|\xi|<1$, let 
$\zeta\in \mathbb{R}$ satisfy that 
$|\xi|e^{(1-\xi)\zeta}<1$ and $\frac12(1-\xi)\zeta<1$.
By \eqref{p.sa.ext.1} with $\beta=(1-\xi)\zeta$ and
$\gamma=-\frac12$, we obtain that
\[
    e^{\frac12(1-\xi)\zeta}
    =\int_{\mathcal{W}}
       e^{\kyosu(1-\xi)\zeta\{\mathfrak{s}_1
           -\frac{\kyosu}4|\theta(1)|^2\}} d\mu.
\]

Let $-1<\xi\le 0$ and put
$\gamma=-\dfrac{1+\xi}{2(1-\xi)}$.
Then 
\[
    |\gamma|\le\frac12,
    \quad
    \frac12+\gamma=-\frac{\xi}{1-\xi},
    \quad\text{and}\quad
    \frac12-\gamma=\frac1{1-\xi}.
\]
Take $\zeta\in \mathbb{R}$ with $\frac12(1+\xi)\zeta<1$. 
By \eqref{p.sa.ext.1} with $\beta=(1-\xi)\zeta$ and this
$\gamma$, we have that
\begin{equation}\label{p.eul.poly.21}
    \biggl\{\frac1{1-\xi}e^{-\frac12(1-\xi)\zeta}
       -\frac{\xi}{1-\xi}e^{\frac12(1-\xi)\zeta}       
     \biggr\}^{-1}
    =\int_{\mathcal{W}} 
        e^{\kyosu(1-\xi)\zeta\{\mathfrak{s}_1
           -\frac{\kyosu(1+\xi)}{4(1-\xi)}
            |\theta(1)|^2\}} d\mu.
\end{equation}

Thus we obtain that
\[
    \sum_{n=0}^\infty P_n(\xi)\frac{\zeta^n}{n!}
    =\biggl(\int_{\mathcal{W}}
          e^{(1-\xi)\zeta\{\kyosu\mathfrak{s}_1
            +\frac14|\theta(1)|^2\}} 
        d\mu \biggr)
      \biggl(\int_{\mathcal{W}} 
        e^{\zeta\{\kyosu(1-\xi)\mathfrak{s}_1
           +\frac{(1+\xi)}4 |\theta(1)|^2\}} 
        d\mu \biggr)
\]
for $-1<\xi\le 0$ and $\zeta\in \mathbb{R}$ with
$\max\{(1-x)\zeta,(1+\xi)\zeta\}<1$ and 
$|\xi|e^{(1-\xi)\zeta}<1$.
Transforming the right hand side of the identity into a
series of $\zeta$ about $0$, and comparing the coefficients of
$\zeta^n$s, we obtain the desired expression of $P_n$s.

We next show the expression of $P(B_n;\cdot)$.
Rewrite the corresponding exponential generating function as
\[
    \frac{(1-\xi)e^{(1-\xi)\zeta}}{
              1-\xi e^{2(1-\xi)\zeta}}
    =\biggl\{\frac1{1-\xi}e^{-(1-\xi)\zeta}
       -\frac{\xi}{1-\xi}e^{(1-\xi)\zeta}       
     \biggr\}^{-1}.
\]
By \eqref{p.eul.poly.21}, we see that
\[
    \sum_{n=0}^\infty P(B_n;\xi)\frac{\zeta^n}{n!}
    =\int_{\mathcal{W}} 
        e^{2\zeta\{\kyosu(1-\xi)\mathfrak{s}_1
           +\frac{(1+\xi)}4 |\theta(1)|^2\}} d\mu 
\]
for $-1<\xi\le 0$ and $\zeta\in \mathbb{R}$ with
$\max\{(1-x)\zeta,(1+\xi)\zeta\}<\frac12$ and 
$|\xi|e^{2(1-\xi)\zeta}<1$.
This implies the desired expression of $P(B_n;\xi)$.
\end{proof}

\begin{remark}\label{r.euler}
The subject in this section is taken from the papers
\cite{ikeda-t-spa,ikeda-t-bull-sci} by Ikeda and the
author. 
Therein the proof of Proposition~\ref{p.sa.ext} was based on 
the expression of heat kernels, which was shown by Matsumoto 
(\cite{matsumoto}) with the help of the Van Vleck formula.
The expression was used to handle the term $|\theta(1)|^2$.
As was seen in the proof of Proposition~\ref{p.sa}, which
was used to show Proposition~\ref{p.sa.ext}, the term
is a part of the quadratic form $\p_{\sigma_{a,b}}$ and
could be dealt with in our scheme of evaluating Laplace
transformations of quadratic forms. 
Proposition~\ref{p.sa} is much more na\"{\i}ve and direct than 
using heat kernels.
\end{remark}

\section{KdV equation}
\label{sec.kdv}
In this section, we apply Corollary~\ref{c.ode.ou} to
reflectionless potentials and soliton solutions to the KdV
equation.

Let $n\in \mathbb{N}$, $\widehat{\mathcal{W}}$ be the space
of $\mathbb{R}^n$-valued continuous functions on
$[0,\infty)$ vanishing at $0$, and $\widehat{\mu}$ the
Wiener measure on $\widehat{\mathcal{W}}$.
In what follows, we use the expectation symbol
$\widehat{\mathbb{E}}[\cdots]$ to denote the integration over
$\widehat{\mathcal{W}}$ with respect to $\widehat{\mu}$, that is,
$\widehat{\mathbb{E}}[f]
 =\int_{\widehat{\mathcal{W}}} f d\widehat{\mu}$.
The assertions before this section continue to hold with 
$d=n$, $\widehat{\mathcal{W}}$, and $\widehat{\mu}$
instead of $\mathcal{W}$ and $\mu$:
we regard Wiener functionals on $\mathcal{W}$ as those on 
$\widehat{\mathcal{W}}$ by thinking of  $\mathcal{W}$ as 
a subspace of $\widehat{\mathcal{W}}$ and $\mu$ as 
a restriction of $\widehat{\mu}$ to $\mathcal{W}$.

Put
\[
    \mathcal{SD}_n
    =\Bigl\{(\eta_j,m_j)_{1\le j\le n} \,\Big|\,
      0<\eta_1<\dots<\eta_n,m_1,\dots,m_n>0\Bigr\}.
\]
Its element is called a {\it scattering date} of length $n$.
For 
$\boldsymbol{s}=(\eta_j,m_j)_{1\le j\le n}
 \in \mathcal{SD}_n$,
define $C^\infty$-functions
$G_{\boldsymbol{s}}:\mathbb{R}\to \mathbb{R}^{n\times n}$
and 
$u_{\boldsymbol{s}}:\mathbb{R}\to \mathbb{R}$
as
\[
    G_{\boldsymbol{s}}(x)
      =\biggl(
       \frac{\sqrt{m_im_j} e^{-(\eta_i+\eta_j)x}}{
               \eta_i+\eta_j}\biggr)_{1\le i,j\le n}
    \quad\text{for } x\in \mathbb{R}
\]
and 
\[
    u_{\boldsymbol{s}}
      =-2\Bigl(\frac{d}{dx}\Bigr)^2
       \log\det(I_n+G_{\boldsymbol{s}}).
\]
The function $u_{\boldsymbol{s}}$ is called a 
{\it reflectionless potential} with scattering data
$\boldsymbol{s}$. 
If we set 
\[
    \boldsymbol{s}(t)
    =\bigl(\eta_j,m_j\exp(-2\eta_j^3t)
       \bigr)_{1\le j\le n}
    \quad\text{for }t\in \mathbb{R},
\]
then the function 
$v(x,t)=-u_{\boldsymbol{s}(t)}(x)$ for $x,t\in \mathbb{R}$
is a $n$-{\it soliton} solution to the KdV equation
\begin{equation}\label{eq.kdv}
    \frac{\partial v}{\partial t}
    -\frac32 v \frac{\partial v}{\partial x}
    -\frac14 \frac{\partial^3 v}{\partial x^3}=0
    \quad\text{with }v(\cdot,0)=-u_{\boldsymbol{s}}.
\end{equation}
See \cite{miwa-jimbo-date}.

Define the space $\boldsymbol{\Sigma}_n$ of discrete
measures on $\mathbb{R}^n$ as 
\[
    \boldsymbol{\Sigma}_n
    =\Biggl\{\sum_{j=1}^n c_j^2 \delta_{p_j} \,\Bigg|\,
      p_1,\dots,p_n\in \mathbb{R},\,
      c_1,\dots,c_n>0,\,
      p_i\ne p_j\text{ for }i\ne j \Biggr\},
\]
where $\delta_p$ stands for the Dirac measure on $\mathbb{R}^n$
concentrated at $p$.
In what follows, the letters $x,y,z$ are also used to indicate 
the time parameter of stochastic processes.  
Fix 
\[
    \ba=\sum\limits_{j=1}^n c_j^2\delta_{p_j}
        \in \boldsymbol{\Sigma}_n.
\]
Put
\[
    D_\ba=\text{\rm diag}[p_1,\dots,p_n]
      \in \mathbb{R}^{n\times n}
    \quad\text{and}\quad
    \boldsymbol{c}_\ba=(c_1,\dots,c_n)
      \in \mathbb{R}^n.
\]
The Ornstein-Uhlenbeck process
$\{\xi_\ba(x)\}_{x\in[0,\infty)}$ associated with $\ba$ is given by
\[
    \xi_\ba(x)=e^{xD_\ba }\int_0^x 
         e^{-yD_\ba} d\theta(y)
    \quad\text{for } x\in[0,\infty),
\]
that is, its $i$th component $\xi_\ba^i(x)$ is equal to
$e^{xp_i}\int_0^x e^{-yp_i} d\theta^i(y)$.
$\{\xi_\ba(x)\}_{x\in[0,\infty)}$ is a unique solution
to the SDE
\[
    d\xi_\ba(x)
    =d\theta(x)+D_\ba\xi_\ba(x)dx,
    \quad \xi_\ba(0)=x.
\]
Define $\Psi_\ba:[0,\infty)\to [0,\infty)$ by
\[
    \Psi_\ba(x)=\widehat{\mathbb{E}}\biggl[
      \exp\biggl(-\frac12\int_0^x 
        \langle \boldsymbol{c}_\ba,
             \xi_\ba(y) \rangle^2 dy\biggr)\biggr]
    ~~\text{for }x\in[0,\infty).
\]

\begin{lemma}\label{l.psi.a.smooth}
$\Psi_\ba(x)>0$ for $x>0$ and 
$\Psi_\ba$ is infinitely differentiable on $[0,\infty)$.
\end{lemma}

\begin{proof}
The positivity is obvious.
By virtue of the dominated convergence theorem,
$\Psi_\ba$ is differentiable and its derivative is  
\[
    \Psi_\ba(x)=\widehat{\mathbb{E}}\biggl[
     \biggl(-\frac12\langle \boldsymbol{c}_\ba,
        \xi_\ba(x)\rangle^2\biggr)
      \exp\biggl(-\frac12\int_0^x
        \langle \boldsymbol{c}_\ba,
             \xi_\ba(y) \rangle^2 dy\biggr)\biggr].
\]
To see the higher order differentiability, let 
$C_\nearrow^\infty(\mathbb{R}^n)$ be the space of 
infinitely differentiable functions on $\mathbb{R}^n$ whose 
derivatives of all orders are at most of polynomial growth
order.
For $f\in C_\nearrow^\infty(\mathbb{R}^n)$, set
$\displaystyle\nabla f=\Bigl(
       \frac{\partial f}{\partial x^i}
       \Bigr)_{1\le i\le n}$ 
and define $\mathcal{L} f$ as
\[
    \mathcal{L} f(\xi)
      =\frac12\sum_{i=1}^n
        \Bigl(\Bigl(\frac{\partial}{\partial x^i}\Bigr)^2 f
         \Bigr)(\xi)
      +\langle D_\ba\xi,\nabla f(\xi)
       \rangle
      -\frac12\langle \boldsymbol{c}_\ba,
         \xi\rangle^2 f(\xi) 
    \quad\text{for }\xi\in \mathbb{R}^n.
\]
By It\^o's formula, we have that
\begin{align*}
    & \widehat{\mathbb{E}}\biggl[f(\xi_\ba(x))
      \exp\biggl(-\frac12\int_0^x 
        \langle \boldsymbol{c}_\ba,
             \xi_\ba(y) \rangle^2 dy\biggr)\biggr]
    \\
    & =\int_0^x \widehat{\mathbb{E}}\biggl[
        (\mathcal{L} f)(\xi_\ba(y))
      \exp\biggl(-\frac12\int_0^y 
        \langle \boldsymbol{c}_\ba,
             \xi_\ba(z) \rangle^2 dz\biggr)\biggr] dy
    \quad\text{for every }
    f\in C_\nearrow^\infty(\mathbb{R}^n).
\end{align*}
Applying this identity successively, we obtain the higher
order differentiability.
\end{proof}

Without loss of generality, we assume that
\\[5pt]
\indent
$\text{\bf (H)}_m$
\begin{minipage}[t]{320pt}
there are $0\le m<n$ and $1\le j(1)<\dots <j(m)\le n$ such
that 
\begin{align*}
    & |p_j|\le |p_{j+1}|
      \quad\text{for }1\le j\le n-1, 
    \quad
     \#\{|p_1|,\dots,|p_n|\}=n-m, 
    \\
    & p_{j(\ell)}>0
      ~~\text{and}~~
      p_{j(\ell)+1}=-p_{j(\ell)}
      ~~\text{for }1\le \ell\le m.
\end{align*}
\end{minipage}
\\[5pt]
Denote by $0<r_1<\dots<r_{n-m}$ the roots of the 
algebraic equation for $r$ that
\[
    -1=\sum_{j=1}^n \frac{c_j^2}{p_j^2-r}.
      =\sum_{\ell=1}^m 
          \frac{c_{j(\ell)}^2+c_{j(\ell)+1}^2}{
                      p_{j(\ell)}^2-r}
       +\sum_{j\notin \{j(\ell),j(\ell)+1;1\le \ell\le m\}}
           \frac{c_j^2}{p_j^2-r}.
\]
Define 
$\boldsymbol{s}(\ba)=(\eta_j,m_j)_{1\le j\le n}
 \in \mathcal{SD}_n$ as
\[
    \{\eta_1<\dots<\eta_n\}
      =\{p_{j(1)},\dots,p_{j(m)},\sqrt{r_1},\dots,
         \sqrt{r_{n-m}}\}
\]
and
\[
    m_j 
      =\begin{cases}
        \displaystyle
         2\eta_j\frac{c_{j(\ell)+1}^2}{c_{j(\ell)}^2}
          \prod_{k\ne j} \frac{\eta_k+\eta_j}{
                              \eta_k-\eta_j}
          \prod_{k\ne j,j+1}
            \frac{p_k+\eta_j}{p_k-\eta_j}
         & \text{if }j=j(\ell),
        \\
        \displaystyle
         -2\eta_j 
           \prod_{k\ne j} \frac{\eta_k+\eta_j}{
                              \eta_k-\eta_j}
           \prod_{k=1}^n 
            \frac{p_k+\eta_j}{p_k-\eta_j}
          & \text{if }
            j\notin\{j(\ell);1\le\ell\le m\}.
       \end{cases}
\]

\begin{remark}\label{rem.m_j>0}
It holds that
\begin{equation}\label{r.m_j>0.1}
   \begin{cases}
        |p_k|<\eta_k<|p_{k+1}| 
        & \text{for } k\notin\{j(\ell);1\le \ell\le m\},
        \\
        \eta_{j(\ell)-1}
        < p_{j(\ell)}=-p_{j(\ell)+1}=\eta_{j(\ell)}
        <\eta_{j(\ell)+1}
        & \text{for }1\le\ell\le m.
   \end{cases}
\end{equation}
This implies the positivity of $\eta_j$s and the order that
$\eta_1<\eta_2<\dots<\eta_n$.
This also yields that
\[
    p_k^2-\eta_j^2
    \begin{cases}
      <0 & \text{if (i) $j\notin\{j(\ell)\mid 1\le\ell\le m\}$ and 
                 $k\le j$ or (ii) $j=j(\ell)$ and $k<j(\ell)$},
      \\
      >0 & \text{if (iii) $j\notin\{j(\ell)\mid 1\le\ell\le m\}$ and 
                 $k>j$ or (vi) $j=j(\ell)$ and $k>j(\ell)+1$}.
    \end{cases}
\]
Hence all denominators appearing in $m_j$s do not vanish, and hence
$m_j$s are well-defined.
Furthermore, we have that
\[
    \text{sgn}(m_j)
    =\begin{cases}
       (-1)^{j-1}(-1)^{j-1}=1 & \text{if }j=j(\ell),
       \\
       -(-1)^{j-1}(-1)^j=1 
       & \text{if }j\notin\{j(\ell)\mid 1\le\ell\le m\}.
     \end{cases}
\]
Thus $m_j$s are all positive.
\end{remark}

We shall present a probabilistic representation o
reflectionless potential in terms of the Ornstein-Uhlenbeck
process.

\begin{theorem}\label{t.reflpot}
Let $\ba$ and $\boldsymbol{s}(\ba)=(\eta_j,m_j)_{1\le j\le n}$ be as
above.
Then it holds that
\begin{align}
    4\log\Psi_\ba(x)
    = & -2\log \det\Bigl(
            I_n+G_{\boldsymbol{s}(\ba)}(x)\Bigr)
\label{t.reflpot.1}
    \\
     & +2\log \det\Bigl(
         I_n+G_{\boldsymbol{s}(\ba)}(0)\Bigr)
       -2x\sum_{j=1}^n (p_j+\eta_j)
     \quad\text{for $x\in[0,\infty)$.}
\nonumber
\end{align}
In particular, it holds that
\[
    4 \Bigl(\frac{d}{dx}\Bigr)^2 \log\Psi_\ba
    =u_{\boldsymbol{s}(\ba)}
    \quad\text{on }[0,\infty).
\]
\end{theorem}

\begin{remark}\label{r.reflpot}
Since every reflectionless potential is real analytic on
$\mathbb{R}$, it is determined by values on $[0,\infty)$. 
Thus 
$4\bigl(\frac{d}{dx}\bigr)^2 \log\Psi_\ba$
determines a reflectionless potential uniquely.
It is known that the mapping 
$\boldsymbol{\Sigma}_n\ni\ba\mapsto 
 \boldsymbol{s}(\ba)\in \mathcal{SD}_n$
is surjective (\cite{st-spa}).
Thus every reflectionless potential is of the form
$4\bigl(\frac{d}{dx}\bigr)^2 \log\Psi_\ba$ on $[0,\infty)$. 
\end{remark}

The proof of Theorem~\ref{t.reflpot} is broken into several
steps, each step being a lemma.
During the proof, we also fix $x>0$ in addition to 
$\ba$ and $\boldsymbol{s}(\ba)=(\eta_j,m_j)_{1\le j\le n}$.
Put 
\[
    E_\ba=D_\ba^2
      +\boldsymbol{c}_\ba\otimes \boldsymbol{c}_\ba.
\]
Denote by $\widehat{\mathcal{W}}_x$ the space of
$\mathbb{R}^n$-valued continuous functions on $[0,x]$
vanishing at $0$.
$\widehat{\mathcal{W}}_x$ is a subspace of 
$\widehat{\mathcal{W}}$ as 
$w(x\wedge \cdot)\in \widehat{\mathcal{W}}$ for 
$w\in \widehat{\mathcal{W}}_x$.

\begin{lemma}\label{l.reflpot.cv0}
Let $\phi\in C^\infty([0,x];\mathbb{R}^{n\times n})$ be
a solution to the ODE on $[0,x]$
\[
    \phi^{\prime\prime}-E_\ba\phi=0.
\]
Assume that $\det\phi(y)\ne0$ for any $y\in[0,x]$.
Define $\psi\in C^\infty([0,x];\mathbb{R}^{n\times n})$ as
$\psi=-(\phi^\prime \phi^{-1})(x-\cdot)$.
Denote by $\psi_S$ (resp. $\psi_A$) the symmetric
(resp. anti-symmetric) part of $\psi$.
Then it holds that 
\begin{align}
 \label{l.reflpot.cv0.1}
    & \widehat{\mathbb{E}}\biggl[ f(\xi_\ba)
      \exp\biggl(-\frac12\int_0^x \langle 
          \boldsymbol{c}_\ba,\xi_\ba(y)\rangle^2 dy
        +\frac12\bigl\langle (\psi_S(x)-D_{\ba})\xi_\ba(x),
            \xi_\ba(x)\bigr\rangle
    \\
    & \hphantom{\widehat{\mathbb{E}}\biggl[ f(\xi_\ba)
      \exp\biggl(-\frac12\int_0^x \langle 
          \boldsymbol{c}_\ba,\xi_\ba(y)\rangle^2 dy}
      -\frac12\int_0^x |\psi_A(y)\xi_\ba(y)|^2 dy
      \biggr) \biggr]
 \nonumber
    \\
    & =\biggl(
         \frac{\det\phi(0)e^{-x\tr D_\ba}}{
           \det\phi(x)}\biggr)^{\frac12}
       \widehat{\mathbb{E}}\bigl[ 
          f(\iota+F_{\widehat{\kappa_A(\rho_{\psi_S})}})\bigr]
    \quad\text{for every }f\in \widehat{\mathcal{W}}_x.
 \nonumber
\end{align}
\end{lemma}

\begin{proof}
Define $\chi,\alpha\in C^1([0,x];\mathbb{R}^{n\times n})$ as
$\chi\equiv D_{\ba}$ and $\alpha=\e[(\cdot-x)D_\ba]$.
Then $\alpha$ satisfies that
\[
    \alpha^\prime=\chi \alpha,\quad
    \alpha(x)=I_n, \quad\text{and}\quad
    \det \alpha(y)\ne0~~\text{for }y\in[0,x].
\]
Furthermore, $\xi^{[\chi]}=\xi_{\ba}$, where $\xi^{[\chi]}$ is 
the Ornstein-Uhlenbeck process discussed in
Corollary~\ref{c.ode.ou}. 

Put 
$\gamma=-\boldsymbol{c}_\ba\otimes \boldsymbol{c}_\ba
  +\psi_A^2$ 
and $\beta=\psi_S-D_\ba$.
Then the LHS of \eqref{l.reflpot.cv0.1} is written as
\[
    \widehat{\mathbb{E}}\biggl[ f(\xi^{[\chi]})
     \exp\biggl( \frac12\int_0^x 
        \langle\gamma(y)\xi^{[\chi]}(y),\xi^{[\chi]}(y)\rangle 
        dy
        +\frac12 \langle \beta(x)\xi^{[\chi]}(x),
           \xi^{[\chi]}(x)\rangle \biggr)\biggr].
\]
By Corollary~\ref{c.ode.ou}, this is equal to
\begin{equation}\label{l.reflpot.cv0.21}
    \exp\biggl(\frac12 \int_0^x \tr[\beta(y)]dy\biggr)
    \widehat{\mathbb{E}}\Bigl[ f 
     \exp\bigl(
      \mathfrak{a}_{\chi+\beta_S,\gamma+\beta_S^\prime-\chi^\dagger\chi}^0
     \bigr)\Bigr]
    \quad\text{for }f\in C_b(\mathcal{W}_x)
    \text{ with }f\ge0.
\end{equation}

Note that $\psi^\prime=E_\ba-\psi^2$ and 
$\psi_S^\prime=E_\ba-\psi_S^2-\psi_A^2$.
Then we know that
\[
    \chi+\beta_S=D_\ba+\psi_S-D_\ba=\psi_S
\]
and
\[
    \gamma+\beta_S^\prime-\chi^\dagger\chi
      =-\boldsymbol{c}_\ba\otimes \boldsymbol{c}_\ba
       +\psi_A^2+E_\ba-\psi_S^2-\psi_A^2-D_\ba^2
     =-\psi_S^2.
\]
Hence we have that
\[
    \mathfrak{a}_{\chi+\beta_S,\gamma+\beta_S^\prime-\chi^\dagger\chi}^0    
    =\int_0^x \langle \psi_S(y)\theta(y),d\theta(y)
               \rangle
     -\frac12 \int_0^x |\psi_S(y)\theta(y)|^2 dy.
\]
In conjunction with \eqref{t.lap.lin.4}, this implies that
\begin{equation}\label{l.reflpot.cv0.22}
    \widehat{\mathbb{E}}\Bigl[ f 
     \exp\bigl(
      \mathfrak{a}_{\chi+\beta_S,\gamma+\beta_S^\prime-\chi^\dagger\chi}^0
     \bigr)\Bigr]
    =\widehat{\mathbb{E}}\Bigl[ 
       f\Bigl(\iota+F_{\widehat{\kappa_A(\rho_{\psi_S})}}\Bigr)
       \Bigr].
\end{equation}

Notice that 
\[
    \det\phi(x)=\det\phi(0)
     \exp\biggl(-\int_0^x \tr[\psi(y)]dy\biggr).
\]
Hence we have that
\[
    \exp\biggl(\int_0^x \tr[\beta_S(y)]dy\biggr)
    =\exp\biggl(\int_0^x \tr[\psi(y)]dy
       -x\tr D_\ba\biggr)
    =\frac{\det\phi(0) e^{-x\tr D_\ba}}{\det\phi(x)}.
\]
Plugging this and \eqref{l.reflpot.cv0.22} into 
\eqref{l.reflpot.cv0.21}, we obtain
\eqref{l.reflpot.cv0.1}. 
\end{proof}

\begin{lemma}\label{l.reflpot.ode}
Let 
$\phi_\ba\in C^\infty([0,\infty);\mathbb{R}^{n\times n})$
be the unique solution to the ODE
\begin{equation}\label{eq.reflpot.ode}
    \phi_\ba^{\prime\prime}-E_\ba\phi_\ba=0,
    \quad \phi_\ba(0)=I_n,~~
    \phi_\ba^\prime(0)=-D_\ba.
\end{equation}
Then 
$\det\phi_\ba(y)\ne0$ for every $y\in[0,\infty)$.
\end{lemma}

\begin{proof}
To show the assertion, we employ proof by contradiction.
To do so, suppose that there is a $y\in(0,\infty)$ with 
$\det\phi_\ba(y)=0$.
Take a $v\in \mathbb{R}^n\setminus\{0\}$ satisfying that  
$\phi_\ba(y)v=0$.
Set 
$\widetilde{\phi}_\ba
 =\phi_\ba^\prime-D_\ba\phi_\ba$.
By \eqref{eq.reflpot.ode}, we have that
\[
    \widetilde{\phi}_\ba^\prime
      =(\boldsymbol{c}_\ba\otimes 
        \boldsymbol{c}_\ba)\phi_\ba
        -D_\ba\widetilde{\phi}_\ba.
\]
Rewriting as
$\phi_\ba^\prime
 =\widetilde{\phi}_\ba+D_\ba\phi_\ba$,
we see that
\[
    0=\langle \phi_\ba(y)v,
      \widetilde{\phi}_\ba(y)v \rangle
    =\int_0^y \Bigl\{|\widetilde{\phi}_\ba(z)v|^2
       +\langle \boldsymbol{c}_\ba,
         \phi_\ba(z)v\rangle^2
      \Bigr\}dz.
\]
This implies that
$\widetilde{\phi}_\ba(z)v=0$ for every $z\in[0,y]$.
Hence 
\[
    \phi_\ba^\prime(z) v-D_\ba \phi_\ba(z)v
    =0\quad\text{for every }z\in[0,y].
\]
Combining with $\phi_\ba(0)=I_n$, we have that
\[
    \phi_\ba(z)v=\e[zD_\ba]v 
    \quad\text{for every }z\in[0,y].
\]
Since $\phi_\ba(y)v=0$, this implies that $v=0$, which is a
contradiction. 
\end{proof}

\begin{lemma}\label{l.reflpot.cv}
Let $\phi_\ba$ be as in Lemma~\ref{l.reflpot.ode}.
It holds that 
\begin{equation}\label{l.reflpot.cv.1}
    \Psi_\ba(x)=\biggl(
      \frac{e^{-x \tr D_\ba}}{\det\phi_\ba(x)}\biggr)^{\frac12}.
\end{equation}
\end{lemma}

\begin{proof}
Let $\phi_\ba$ be as in Lemma~\ref{l.reflpot.ode}.
By the lemma, $\det\phi_\ba(y)\ne0$ for every $y\in[0,x]$.
Put $\psi_\ba=-(\phi_\ba^\prime \phi_\ba^{-1})(x-\cdot)$.
It holds that
\[
    \psi_\ba^\prime=E_\ba-\psi_\ba^2,\quad
    \psi_\ba(x)=D_\ba.
\]
Mention that $\psi_\ba^\dagger$ satisfies the same ODE.
Hence $\psi_\ba=\psi_\ba^\dagger$.
Thus we obtain that 
\[
    \psi_\ba=(\psi_\ba)_S 
    \quad\text{and}\quad
    (\psi_\ba)_A=0.
\]
Plugging these into \eqref{l.reflpot.cv0.1} with $f=1$ and
$\psi=\psi_\ba$, we obtain \eqref{l.reflpot.cv.1}.
\end{proof}

\begin{lemma}\label{l.reflpot.E_sigma}
Put $R=\text{\rm diag}[\eta_1,\dots,\eta_n]$.
Define $U=(U_j^i)_{1\le i,j\le n}\in \mathbb{R}^{n\times n}$ by 
\[
    U_j^i
    =\begin{cases}
       \dfrac1{\,|(D_\ba^2-\eta_j^2I_n)^{-1}
              \boldsymbol{c}_\ba|\,}
       \dfrac{c_i}{p_i^2-\eta_j^2}
       & \text{if }j\notin\{j(\ell)\mid 1\le\ell\le m\},
       \\[15pt]
       \dfrac{\delta_{i,j(\ell)+1}c_{j(\ell)}-
              \delta_{i,j(\ell)}c_{j(\ell)+1}}{
            (c_{j(\ell)}^2+c_{j(\ell)+1}^2)^{\frac12}}
       & \text{if }j=j(\ell).
     \end{cases}
\]
Then $U$ is an orthogonal matrix and 
\[
    E_\ba=UR^2U^{-1}.
\]
Furthermore, it holds that
\begin{equation}\label{eq.reflpot.phi_sigma.2}
    \phi_\ba(y)
    =U\bigl\{\ch[yR]
      -\mathfrak{sh}[yR]R^{-1}U^{-1}D_\ba U\bigr\}U^{-1}
    \quad\text{for }y\in[0,\infty).
\end{equation}
\end{lemma}

\begin{proof}
To show the first assertion, put 
\[
    u_\ell=(\underbrace{0,\dots,0}_{j(\ell)-1},
       -c_{j(\ell)+1},c_{j(\ell)},
      \underbrace{0,\dots,0}_{n-j(\ell)-1})^\dagger
    \in \mathbb{R}^n.
\]
Since 
$\langle \boldsymbol{c}_\ba,u_\ell
   \rangle=0$ 
and $p_{j(\ell)+1}^2=p_{j(\ell)}^2$, we see that
\[
    E_\ba u_\ell=D_\ba^2u_\ell=p_{j(\ell)}^2u_\ell.
\]

Noting that 
\[
    \det(D_\ba^2-r_kI_n)=\prod_{i=1}^n (p_i^2-r_k)\ne0,
\]
we set 
\[
    v_k=(D_\ba^2-r_kI_n)^{-1}\boldsymbol{c}_\ba.
\]
Then we have that
\[
    E_\ba v_k=(1+\langle v_k,\boldsymbol{c}_\ba
       \rangle)\boldsymbol{c}_\ba
       +r_k v_k.
\]
Since 
\[
    1+\langle v_k,\boldsymbol{c}_\ba
       \rangle
    =1+\sum_{j=1}^n \frac{c_j^2}{p_j^2-r_k}=0,
\]
it holds that
\[
    E_\ba v_k=r_k v_k.
\]

Thus 
$p_{j(1)}^2,\dots,p_{j(m)}^2$ and $r_1,\dots,r_{n-m}$
are different $n$ positive eigenvalues of $E_\ba$ and 
$u_1,\dots,u_m$ and $v_1,\dots,v_{n-m}$ are the
corresponding eigenvectors.
If $j\notin\{j(\ell)\mid 1\le \ell\le m\}$, then
$\eta_j=\sqrt{r_k}$ for some $1\le k\le n-m$, and 
$\eta_{j(\ell)}=p_{j(\ell)}$ for $1\le \ell\le m$.
Hence $U$ is an orthogonal matrix and $E_{\ba}=UR^2U^{-1}$.

We next show the expression \eqref{eq.reflpot.phi_sigma.2}.
To do so, denote its RHS by $\psi$.
It is easily seen that
\[
    \psi^\prime(y)
      =UR\bigl\{\mathfrak{sh}[yR]
       -\ch[yR]R^{-1}U^{-1}D_\ba U\bigr\}U^{-1}
\]
and
\[
    \psi^{\prime\prime}(y)
      =UR^2\bigl\{\ch[yR]
       -\mathfrak{sh}[yR]R^{-1}U^{-1}D_\ba U\bigr\}U^{-1}
      =E_\ba\psi(y).
\]
Thus $\psi$ obeys the ODE \eqref{eq.reflpot.ode}, and hence 
$\phi_\ba=\psi$.
\end{proof}

\begin{lemma}\label{l.reflpot.H_0}
Assume that $\text{\bf(H)}_0$ holds.
\\
{\rm(i)}
Put 
\begin{align*}
    &
    X=\biggl(\dfrac1{\eta_i+p_j}\biggr)_{1\le i,j\le n},
    \quad
    C(\boldsymbol{c}_\ba)
      =\text{\rm diag}[c_1,\dots,c_n],
    \\
    & V(\boldsymbol{c}_\ba)=\text{\rm diag}\Bigl[
        |(D_\ba^2-\eta_1^2I_n)^{-1}\boldsymbol{c}_\ba|^{-1},
        \dots,
        |(D_\ba^2-\eta_n^2I_n)^{-1}\boldsymbol{c}_\ba|^{-1}
        \Bigr],
    \\
    & \tau(i)=\text{\rm sgn}\biggl(
        \prod_{j=1}^n (p_j-\eta_i)\biggr),
      \quad
       b(i)=\tau(i)\left\{-2\eta_i
       \frac{\prod_{\alpha\ne i}
              (\eta_\alpha^2-\eta_i^2)}{
              \prod_{\beta=1}^n (p_\beta^2-\eta_i^2)}
       \right\}^{\frac12},
\end{align*}
and $B=\text{\rm diag}[b(1),\dots,b(n)]$.
Then it holds that
\begin{equation}\label{l.reflpot.H_0.1}
    U^{-1}\phi_\ba(y)U
    =-\frac12 V(\boldsymbol{c}_\ba)R^{-1}B
      \Bigl(I_n+G_{\boldsymbol{s}(\ba)}(y)\Bigr)
      \e[yR]B^{-1}
      X C(\boldsymbol{c}_\ba)U.
\end{equation}
{\rm(ii)}
The identity \eqref{t.reflpot.1} holds.
\end{lemma}

\begin{remark}
By \eqref{r.m_j>0.1}, we have that
\[
    \frac{\prod_{\alpha\ne i}
            (\eta_\alpha^2-\eta_i^2)}{
          \prod_{\beta=1}^n (p_\beta^2-\eta_i^2)}
    <0
    \quad\text{for }1\le i\le n.
\]
Hence each $b(i)$ is a real number.
\end{remark}

\begin{proof}
(i)
By \eqref{eq.reflpot.phi_sigma.2}, we have that
\[
    \phi_\ba(y)=\frac12UR^{-1}\Bigl\{
      \e[yR](RU^{-1}-U^{-1}D_\ba)
      +\e[-yR](RU^{-1}+U^{-1}D_\ba)\Bigr\}.
\]
Due to Lemma~\ref{l.reflpot.E_sigma}, we know that 
\begin{equation}\label{eq.p264}
    U=C(\boldsymbol{c}_\ba)
      \biggl(\frac1{p_i^2-\eta_j^2}\biggr)_{1\le i,j\le n}
      V(\boldsymbol{c}_\ba).
\end{equation}
By the same lemma, we obtain 
\[
     U^{-1}=U^\dagger 
     =V(\boldsymbol{c}_\ba)\biggl(
        \frac1{p_j^2-\eta_i^2}\biggr)_{1\le i,j\le n}
      C(\boldsymbol{c}_\ba).
\]
By virtue of the commutativity of diagonal matrices, we
obtain that
\[
    RU^{-1}\pm U^{-1}D_\ba
    =V(\boldsymbol{c}_\ba)\biggl(
     \frac1{\pm p_j-\eta_i}\biggr)_{1\le i,j\le n}
     C(\boldsymbol{c}_\ba).
\]

By Cauchy's identity \eqref{eq.cauchy}, we have that 
\begin{equation}\label{l.reflpot.H_0.211}
    \det X
    =\frac{\prod_{1\le i<j\le n}
          (\eta_i-\eta_j)(p_i-p_j)}{
       \prod_{i,j=1}^n (\eta_i+p_j)}
    \ne0.
\end{equation}
If we set 
\[
    Y=\biggl(\frac1{p_j-\eta_i}\biggr)_{1\le i,j\le n},
\]
then, thanks to the commutation
$RV(\boldsymbol{c}_\ba)=V(\boldsymbol{c}_\ba)R$, we obtain
that
\begin{equation}\label{l.reflpot.H_0.21}
    \phi_\ba(y)=-\frac12 UR^{-1}V(\boldsymbol{c}_\ba)
     \bigl\{I_n-\e[-yR]YX^{-1}\e[-yR]\bigr\}
        \e[yR]XC(\boldsymbol{c}_\ba).
\end{equation}

By the definition of the cofactor matrix 
$\widetilde{X}=(\widetilde{X}_{k\ell})_{1\le k,\ell\le n}$
of $X$ and Cauchy's identity \eqref{eq.cauchy}, we have that 
\[
    \widetilde{X}_{k\ell}
    = (-1)^{k+\ell} 
      \frac{\prod_{i<j,i,j\ne \ell}(\eta_i-\eta_j)
               \prod_{i<j,i,j\ne k}(p_i-p_j)}{
       \prod_{i,j=1,i\ne\ell,j\ne k}^n (\eta_i+p_j)}.
\]
Hence $(k,\ell)$th component of $X^{-1}$ is given by
\[
    (X^{-1})_{k\ell}=\frac1{\det X}\widetilde{X}_{k\ell}
     =\frac{\prod_{\alpha\ne\ell}(\eta_\alpha+p_k)
            \prod_{\beta=1}^n(\eta_\ell+p_\beta)}{
        \prod_{\alpha\ne\ell}(\eta_\alpha-\eta_\ell)
        \prod_{\beta\ne k}(p_\beta-p_k)}.
\]
This expression of $X^{-1}$ and Lagrange's
identity \eqref{eq.lagrange} imply that the $(i,j)$th
component of $YX^{-1}$ is represented as 
\begin{align}
    (YX^{-1})_{ij}
    & =\frac{\prod_{\beta=1}^n(\eta_j+p_\beta)}{
        \prod_{\beta=1}^n(p_\beta-\eta_i)
        \prod_{\alpha\ne j}(\eta_\alpha-\eta_j)}
      \prod_{\alpha\ne j}(\eta_\alpha+\eta_i)
\label{l.reflpot.H_0.22}
    \\
    & =\frac{\prod_{\alpha\ne i}(\eta_\alpha+\eta_i)}{
        \prod_{\beta=1}^n(p_\beta-\eta_i)}
     \frac{2\eta_i}{\eta_i+\eta_j}
     \frac{\prod_{\beta=1}^n(\eta_j+p_\beta)}{        
        \prod_{\alpha\ne j}(\eta_\alpha-\eta_j)}.
 \nonumber 
\end{align}

Notice that
\[
    \frac{\prod_{\alpha\ne i}(\eta_\alpha+\eta_i)}{
         \prod_{\beta=1}^n(p_\beta-\eta_i)}
    \frac{\prod_{\beta=1}^n(p_\beta+\eta_i)}{
          \prod_{\alpha\ne i}(\eta_\alpha-\eta_i)}
    =\frac{\prod_{\alpha\ne i}(\eta_\alpha+\eta_i)^2
     \prod_{\beta=1}^n(p_\beta+\eta_i)^2}{
       \prod_{\alpha\ne i}(\eta_\alpha^2-\eta_i^2)
        \prod_{\beta=1}^n(p_\beta^2-\eta_i^2)}
    <0.
\]
This implies that
\[
    \text{\rm sgn}\left[
      \frac{\prod_{\beta=1}^n(p_\beta+\eta_i)}{
          \prod_{\alpha\ne i}(\eta_\alpha-\eta_i)}
       \right]
    =-\tau(i).
\]
Hence we have that
\[
    b(i)m_i^{\frac12}
      =\tau(i)2\eta_i\left|
           \frac{\prod_{\alpha\ne i}
                   (\eta_\alpha+\eta_i)}{
         \prod_{\beta=1}^n(p_\beta-\eta_i)}\right|
      =2\eta_i
        \frac{\prod_{\alpha\ne i}
             (\eta_\alpha+\eta_i)}{
        \prod_{\beta=1}^n(p_\beta-\eta_i)}
\]
and
\[
    m_i^{\frac12}b(i)^{-1}
      =\tau(i)\left|
        \frac{\prod_{\beta=1}^n(p_\beta+\eta_i)}{
           \prod_{\alpha\ne i}(\eta_\alpha-\eta_i)}
         \right|
    =-\frac{\prod_{\beta=1}^n(p_\beta+\eta_i)}{
           \prod_{\alpha\ne i}(\eta_\alpha-\eta_i)}.
\]
Substituting these into \eqref{l.reflpot.H_0.22}, we see that
\[
    YX^{-1}=-B\biggl(\frac{\sqrt{m_im_j}}{\eta_i+\eta_j}
        \biggr)_{1\le i,j\le n} B^{-1}.
\]
Plugging this into \eqref{l.reflpot.H_0.21} and using the
commutativity of diagonal matrices, we arrive at
\eqref{l.reflpot.H_0.1}.

\noindent
(ii)
By \eqref{l.reflpot.H_0.1}, we have that
\[
    \det\phi_\ba(y)
    =\det\bigl(I_n+G_{\boldsymbol{s}(\ba)}(y)\bigr)
     e^{y\tr\,R}
     \det\biggl(-\frac12 V(\boldsymbol{c}_\ba)R^{-1}
        X C(\boldsymbol{c}_\ba)U\biggr).
\]
Substituting $y=0$, we see that
\[
    1=\det\bigl(I_n+G_{\boldsymbol{s}(\ba)}(0)\bigr)
     \det\biggl(-\frac12 V(\boldsymbol{c}_\ba)R^{-1}
        X C(\boldsymbol{c}_\ba)U\biggr).
\]
Hence we have that 
\[
    \det\phi_\ba(y)
    =\frac{\det\bigl(I_n+G_{\boldsymbol{s}(\ba)}(y)\bigr)
     e^{y\tr\,R}}{
     \det\bigl(I_n+G_{\boldsymbol{s}(\ba)}(0)\bigr)}.
\]
Plugging this into \eqref{l.reflpot.cv.1}, we obtain that 
\[
    \Psi_\ba(x)=\biggl(
     \frac{\det\bigl(I_n+G_{\boldsymbol{s}(\ba)}(0)\bigr)}{
     \det\bigl(I_n+G_{\boldsymbol{s}(\ba)}(x)\bigr)
     e^{x\tr\,(D_\ba+R)}}\biggr)^{\frac12}.
\]
Thus \eqref{t.reflpot.1} holds.
\end{proof}

\begin{lemma}\label{l.reflpot.H_m}
If $\text{\bf(H)}_m$ with $m\ge1$ holds, then
\eqref{t.reflpot.1} holds.
\end{lemma}

\begin{proof}
Take an $\varepsilon>0$ with 
$\varepsilon<\min\{|a-b|\mid a,b\in\{|p_1|,\dots,|p_n|,a\ne b\}$.
Put 
\[
    p_j(\varepsilon)
    =\begin{cases}
       p_j-\varepsilon 
       & \text{if $j=j(\ell)+1$ for $1\le \ell\le m$},
       \\
       p_j & \text{otherwise}
     \end{cases}
\]
and
\[
    \ba(\varepsilon)=\sum_{j=1}^n c_j^2
    \delta_{p_j(\varepsilon)}.
\]
Since $|p_j(\varepsilon)|<|p_{j+1}(\varepsilon)|$
for $1\le j\le n-1$,
$\ba(\varepsilon)\in\boldsymbol{\Sigma}_n$
satisfies $\text{\bf(H)}_0$.

Define 
$(\eta_j(\varepsilon),m_j(\varepsilon))_{1\le j\le n}
 \in \mathcal{SD}_n$
as
\[
    (\eta_j(\varepsilon),m_j(\varepsilon))_{1\le j\le n}
    =\boldsymbol{s}(\boldsymbol{a}(\varepsilon)).
\]
By Lemma~\ref{l.reflpot.H_0}, we have that
\begin{align*}
    \log\Psi_{\ba(\varepsilon)}(x)
    =& -\frac12\log\det(
        I_n+G_{\boldsymbol{s}(\ba(\varepsilon))}(x))
    \\
     & +\frac12\log\det(
        I_n+G_{\boldsymbol{s}(\ba(\varepsilon))}(0))
     -\frac{x}2 \sum_{j=1}^n 
      (p_j(\varepsilon)+\eta_j(\varepsilon)).
\end{align*}
It is an easy exercise of It\^o calculus to see that 
\[
    \lim_{\varepsilon\to0}
      \widehat{\mathbb{E}}\biggl[\sup_{y\in[0,x]}
      |\xi_{\ba(\varepsilon)}(y)
        -\xi_\ba(y)|^p\biggr]=0
    \quad\text{for every 
               $x\in[0,\infty)$ and $p\in(1,\infty)$}.
\]
Due to the dominated convergence theorem, we obtain that 
\[
    \lim_{\varepsilon\to0} 
     \Psi_{\ba(\varepsilon)}(x)
    =\Psi_\ba(x) 
    \quad\text{for every }x\in[0,\infty).
\]
Thus, to prove \eqref{t.reflpot.1}, it remains to show that 
\[
    \lim_{\varepsilon\to0}\eta_j(\varepsilon)=\eta_j
    \quad\text{and}\quad
    \lim_{\varepsilon\to0} m_j(\varepsilon)=m_j
    \quad\text{for }1\le j\le n.
\]

We first show that $\eta_j(\varepsilon)$ converges to
$\eta_j$.
If $j=j(\ell)$ for some $1\le \ell\le m$, then it holds that
\begin{equation}\label{l.reflpot.H_m.21-1}
    p_{j(\ell)}=p_{j(\ell)}(\varepsilon)
    <\eta_{j(\ell)}(\varepsilon)
    <|p_{j(\ell)+1}(\varepsilon)|
     =p_{j(\ell)}+\varepsilon.
\end{equation}
Hence $\eta_{j(\ell)}(\varepsilon)$ converges to
$\eta_{j(\ell)}$.

Suppose that  $j\notin\{j(\ell)\mid 1\le \ell\le m\}$.
Let $\hat{\eta}_j$ be an accumulation point of
$\eta_j(\varepsilon)$ as $\varepsilon\to0$.
Taking a subsequence $\{\varepsilon_k\}_{k=1}^\infty$ with 
$\eta_j(\varepsilon_k)\to\hat{\eta}_j$ as $k\to\infty$, and 
letting $\varepsilon\to0$ along the subsequence in the
equation 
\begin{equation}\label{l.reflpot.H_m.21-1+}
    \sum_{i=1}^n \frac{c_i^2}{
        p_i(\varepsilon)^2-\eta_j(\varepsilon)^2}
    =-1,
\end{equation}
we obtain that
\[
    \sum_{i=1}^n \frac{c_i^2}{p_i^2-\hat{\eta}_j^2}
    =-1.
\]
Thus $\hat{\eta}_j=\eta_j$, that is,
$\eta_j(\varepsilon)$ converges to $\eta_j$.

Next we see that $m_j(\varepsilon)$ converges to $m_j$.
To do so, we recall the expression of
$m_j(\varepsilon)$ as
\begin{equation}\label{l.reflpot.H_m.211}
    m_j(\varepsilon)  =-2\eta_j(\varepsilon)
     \prod_{k\ne j} 
      \frac{\eta_k(\varepsilon)+\eta_j(\varepsilon)}{
         \eta_k(\varepsilon)-\eta_j(\varepsilon)}
     \prod_{k=1}^n 
      \frac{p_k(\varepsilon)+\eta_j(\varepsilon)}{
          p_k(\varepsilon)-\eta_j(\varepsilon)}
\end{equation}
and observe that
\begin{equation}\label{l.reflpot.H_m.210}
    -2\eta_j(\varepsilon)
     \prod_{k\ne j} 
      \frac{\eta_k(\varepsilon)+\eta_j(\varepsilon)}{
         \eta_k(\varepsilon)-\eta_j(\varepsilon)}
    \to -2\eta_j  \prod_{k\ne j} 
      \frac{\eta_k+\eta_j}{\eta_k-\eta_j}
    \quad\text{as }\varepsilon\to0.
\end{equation}

If $j\notin\{j(\ell)\mid 1\le \ell\le m\}$, then 
$p_k\ne \eta_j$ for every $1\le k\le n$.
By \eqref{l.reflpot.H_m.211} and \eqref{l.reflpot.H_m.210}, we have
that 
\[
    m_j(\varepsilon) 
    \to -2\eta_j  \prod_{k\ne j} 
      \frac{\eta_k+\eta_j}{\eta_k-\eta_j}
     \prod_{k=1}^n \frac{p_k+\eta_j}{p_k-\eta_j}
     =m_j
    \quad\text{as }\varepsilon\to0.
\]

Let $j=j(\ell)$ for some $1\le \ell\le m$.
If $k\notin\{j(\ell),j(\ell)+1\}$, then $p_k\ne
\eta_{j(\ell)}$, and hence 
\begin{equation}\label{l.reflpot.H_m.22}
    \prod_{k\ne j(\ell),j(\ell)+1}
      \frac{p_k(\varepsilon)+\eta_{j(\ell)}(\varepsilon)}{
          p_k(\varepsilon)-\eta_{j(\ell)}(\varepsilon)}
    \to\prod_{k\ne j(\ell),j(\ell)+1}
      \frac{p_k+\eta_{j(\ell)}}{p_k-\eta_{j(\ell)}}
    \quad\text{as }\varepsilon\to0.
\end{equation}
We rewrite the product for $k=j(\ell),j(\ell)+1$ as
\[
    \prod_{k=j(\ell),j(\ell)+1}
      \frac{p_k(\varepsilon)+\eta_{j(\ell)}(\varepsilon)}{
          p_k(\varepsilon)-\eta_{j(\ell)}(\varepsilon)}
    =\frac{p_{j(\ell)}(\varepsilon)
               +\eta_{j(\ell)}(\varepsilon)}{
        p_{j(\ell)+1}(\varepsilon)
               -\eta_{j(\ell)}(\varepsilon)}
     \times
     \frac{p_{j(\ell)+1}(\varepsilon)
               +\eta_{j(\ell)}(\varepsilon)}{
        p_{j(\ell)}(\varepsilon)
               -\eta_{j(\ell)}(\varepsilon)}.
\]
The first factor satisfies that
\begin{equation}\label{l.reflpot.H_m.23}
    \frac{p_{j(\ell)}(\varepsilon)+\eta_{j(\ell)}(\varepsilon)}{
           p_{j(\ell)+1}(\varepsilon)-\eta_{j(\ell)}(\varepsilon)}
    \to \frac{2p_{j(\ell)}}{-2p_{j(\ell)}}=-1
    \quad\text{as }\varepsilon\to0.
\end{equation}
We compute the limit of the second factor.
Since
\[
    p_{j(\ell)}=|p_{j(\ell)}(\varepsilon)|
    <\eta_{j(\ell)}(\varepsilon)
    <|p_{j(\ell)+1}(\varepsilon)|=p_{j(\ell)}+\varepsilon,
\]
we have that
\[
    |p_{j(\ell)}(\varepsilon)^2-\eta_{j(\ell)}(\varepsilon)^2|
    =O(\varepsilon)
    \quad\text{and}\quad
    |p_{j(\ell)+1}(\varepsilon)^2-\eta_{j(\ell)}(\varepsilon)^2|
    =O(\varepsilon),
\]
where 
``$f(\varepsilon)=O(\varepsilon^p)$'' for $p>0$ means that
$\limsup\limits_{\varepsilon\to0}
 \varepsilon^{-p} |f(\varepsilon)|<\infty$.
Due to these estimations, multiplying the identity
\eqref{l.reflpot.H_m.21-1+} by 
$(p_{j(\ell)}(\varepsilon)^2-\eta_{j(\ell)}(\varepsilon)^2)
 (p_{j(\ell)+1}(\varepsilon)^2-\eta_{j(\ell)}(\varepsilon)^2)$,
we obtain that 
\begin{equation}\label{l.reflpot.H_m.24}
    c_{j(\ell)}^2\{p_{j(\ell)+1}(\varepsilon)^2
            -\eta_{j(\ell)}(\varepsilon)^2\}
    +c_{j(\ell)+1}^2\{p_{j(\ell)}(\varepsilon)^2
            -\eta_{j(\ell)}(\varepsilon)^2\}
    =O(\varepsilon^2).
\end{equation}
Recalling the definition that
$p_{j(\ell)}(\varepsilon)=p_{j(\ell)}$ and 
$p_{j(\ell)+1}(\varepsilon)=-p_{j(\ell)}-\varepsilon$,
we see that
\begin{equation}\label{l.reflpot.H_m.241}
    p_{j(\ell)}(\varepsilon)^2-\eta_{j(\ell)}(\varepsilon)^2
    =-\frac{2p_{j(\ell)}c_{j(\ell)}^2}{
        c_{j(\ell)}^2+c_{j(\ell)+1}^2}\,\varepsilon
     +O(\varepsilon^2).
\end{equation}
Substituting this into \eqref{l.reflpot.H_m.24}, we have that 
\begin{equation}\label{l.reflpot.H_m.242}
    p_{j(\ell)+1}(\varepsilon)^2-\eta_{j(\ell)}(\varepsilon)^2
    =\frac{2p_{j(\ell)}c_{j(\ell)+1}^2}{
        c_{j(\ell)}^2+c_{j(\ell)+1}^2}\,\varepsilon
     +O(\varepsilon^2).
\end{equation}
Since 
$r_{j(\ell)}(\varepsilon)=\eta_{j(\ell)}(\varepsilon)^2$,
by these estimations and \eqref{l.reflpot.H_m.23}, we arrive at
the convergence 
\begin{align}
 \label{l.reflpot.H_m.25}
   \frac{p_{j(\ell)+1}(\varepsilon)
       +\eta_{j(\ell)}(\varepsilon)}{
       p_{j(\ell)}(\varepsilon)-\eta_{j(\ell)}(\varepsilon)}
   & =\frac{p_{j(\ell)}(\varepsilon)
       +\eta_{j(\ell)}(\varepsilon)}{
       p_{j(\ell)+1}(\varepsilon)-\eta_{j(\ell)}(\varepsilon)}
      \frac{p_{j(\ell)+1}(\varepsilon)^2
       -\eta_{j(\ell)}(\varepsilon)^2}{
       p_{j(\ell)}(\varepsilon)^2-\eta_{j(\ell)}(\varepsilon)^2}
   \\
   & \to (-1)\times
         \frac{2p_{j(\ell)}c_{j(\ell)+1}^2}{
                  -2p_{j(\ell)}c_{j(\ell)}^2}
         =\frac{c_{j(\ell)+1}^2}{c_{j(\ell)}^2}
     \quad\text{as }\varepsilon\to0.
 \nonumber
\end{align}
Plugging \eqref{l.reflpot.H_m.210},
\eqref{l.reflpot.H_m.22},  \eqref{l.reflpot.H_m.23}, and
this into \eqref{l.reflpot.H_m.211}, we see that 
$m_{j(\ell)}(\varepsilon)$ converges to $m_j$. 
\end{proof}

We next extend the probabilistic representation to soliton
solutions to the KdV equation. 
Fix $\ba=\sum\limits_{j=1}^n c_j^2\delta_{p_j}
        \in \boldsymbol{\Sigma}_n$ 
and $(\eta_j,m_j)_{1\le j\le d}=\boldsymbol{s}(\ba)$ as above. 
For $(x,t)\in \mathbb{R}^2$, set 
\[
    \zeta_j(x,t)=x\eta_j+t\eta_j^3
    \quad\text{for } 1\le j\le n
    \quad\text{and}\quad
    \zeta(x,t)=\text{\rm diag}[\zeta_1(x,t),
        \dots,\zeta_n(x,t)].
\]
Put
\[
    \widehat{G}(x,t)
    =\biggl(\frac{\sqrt{m_im_j}
          e^{-(\zeta_i(x,t)+\zeta_j(x,t))}}{
          \eta_i+\eta_j}\biggr)_{1\le i,j\le n}.
\]
As was seen in the beginning of this section, the function
\[
    2\Bigl(\frac{\partial}{\partial x}\Bigr)^2
      \log\det(I_n+\widehat{G})
\]
is the soliton solution to the KdV equation \eqref{eq.kdv}. 
Define 
$\widehat{\phi}_\ba\in
 C^\infty(\mathbb{R}^2;\mathbb{R}^{n\times n})$ by
\[
    \widehat{\phi}_\ba(x,t)
    =U\bigl\{\ch[\zeta(x,t)]
         -\snh[\zeta(x,t)]R^{-1}U^{-1}D_\ba U\bigr\}
     U^{-1},
\]
where we have continued to use the symbols $U$ and $R$
employed in the proof of Theorem~\ref{t.reflpot}.
As will be seen in Lemma~\ref{l.sol.nondeg} below,
$\det\widehat{\phi}_\ba(y,t)\ne0$ for every 
$(y,t)\in \mathbb{R}^2$.
Thus we can define 
$\widehat{\psi}_{\ba,x,t}:[0,x]\to \mathbb{R}^{n\times n}$ 
as
\[
    \widehat{\psi}_{\ba,x,t}(y)
    =-\biggl(\Bigl(
      \frac{\partial\widehat{\phi}_\ba}{\partial x}
      \Bigr)(\cdot,t) 
      \widehat{\phi}_\ba(\cdot,t)^{-1}\biggr)(x-y)
    \quad\text{for }y\in[0,x].
\]
Set
\begin{align*}
    \widehat{\Psi}_\ba(x,t)
    & =\mathbb{E}\biggl[\exp\biggl(
       -\frac12 \int_0^x 
         \langle \boldsymbol{c}_\ba,\xi_\ba(y)
         \rangle^2 dy
      +\frac12 \bigl\langle
         (\widehat{\psi}_{\ba,x,t})_S(x)-D_\ba)
          \xi_\ba(x),\xi_\ba(x) \bigr\rangle
   \\
   &  \hphantom{=\mathbb{E}\biggl[\exp\biggl(}
      -\frac12\int_0^x \bigl|
          (\widehat{\psi}_{\ba,x,t})_A(y)
             \xi_\ba(y)\bigr|^2 dy\biggr)\biggr]
      \quad\text{for }(x,t)\in[0,\infty)\times \mathbb{R},
\end{align*}
where $(\widehat{\psi}_{\ba,x,t})_S$ 
(resp. $(\widetilde{\psi}_{\ba,x,t})_A$) is 
the symmetric (resp. anti-symmetric) part of 
$\widehat{\psi}_{\ba,x,t}$.

\begin{theorem}\label{t.kdv}
It holds that
\begin{align*}
    \log\widehat{\Psi}_\ba(x,t)
    =& -\frac12\log\det(I_n+\widehat{G}(x,t))
     \\
     & +\frac12\log\det(I_n+\widehat{G}(0,t))
     -\frac{x}2\sum_{j=}^n (p_j+\eta_j)
   \quad\text{for }(x,t)\in[0,\infty)\times \mathbb{R}.
\end{align*}
In particular, 
\[
    v=-4\Bigl(\frac{\partial}{\partial x}\Bigr)^2
      \log \widehat{\Psi}_\ba
\]
is a soliton solution to the KdV equation \eqref{eq.kdv}.
\end{theorem}

For the proof, we prepare a lemma.
Under $\text{\bf (H)}_m$, set
\[
    J=\{(j(\ell)+1,j(\ell))\mid 1\le \ell\le m\}
\]
and define
\begin{align*}
    Z_\ba
    =& (-1)^m \prod_{i=1}^n \frac{c_i}{2\eta_i}
     \prod_{1\le i<j\le n}(p_i-p_j)(\eta_i-\eta_j)
    \\
    & \times\biggl\{
      \prod_{(i,j)\notin J} (p_i+\eta_j)
      \prod_{k=1}^{n-m} |(D_\ba^2-r_kI_n)^{-1}
                \boldsymbol{c}_\ba|
      \prod_{\ell=1}^m \frac{c_{j(\ell)+1}(c_{j(\ell)}^2
          +c_{j(\ell)+1}^2)^{1/2}}{2p_{j(\ell)}c_{j(\ell)}}     
      \biggr\}^{-1},
\end{align*}
where, when $m=0$, $J=\emptyset$, 
``$(i,j)\notin J$'' means ``$1\le i,j\le n$'', and 
$\prod\limits_{\ell=1}^m\cdots=1$.

\begin{lemma}\label{l.sol.nondeg}
It holds that 
\begin{equation}\label{l.sol.nondeg.1}
    \det\widehat{\phi}_\ba(y,t)
    =\det(I_n+\widehat{G}(y,t))e^{\tr\zeta(y,t)}Z_\ba
    \quad\text{for every } (y,t)\in \mathbb{R}^2.
\end{equation}
In particular, 
$\det\widehat{\phi}_\ba(y,t)\ne0$ for every 
$(y,t)\in \mathbb{R}^2$.
\end{lemma}

\begin{proof}
We first assume that $\text{\bf(H)}_0$ holds.
In repetition of the argument in the proof of
Lemma~\ref{l.reflpot.H_0} with $e^{\zeta(y,t)}$ for
$e^{yR}$, we have that
\[
    U^{-1}\widehat{\phi}_\ba(y,t)U
    =-\frac12 V(\boldsymbol{c}_\ba) R^{-1}B
      (I_n+\widehat{G}(y,t))e^{\zeta(y,t)}B^{-1}X
      C(\boldsymbol{c}_\ba)U.
\]
Due to this and \eqref{l.reflpot.H_0.211}, we see that
\begin{align*}
    \det\widehat{\phi}_\ba(y,t)
    & =\det(I_n+\widehat{G}(y,t))e^{\tr\zeta(y,t)}
    \\ 
    & \hphantom{=}
      \times(-1)^n \prod_{i=1}^n 
       \frac{c_i}{2\eta_i|(D_\ba^2-\eta_i^2I_n)^{-1}
          \boldsymbol{c}_\ba|}
     \frac{\prod_{1\le i<j\le n}(p_i-p_j)(\eta_i-\eta_j)}{
           \prod_{1\le i,j\le n} (p_i+\eta_j)}
     \det U
    \\
    & =(-1)^n\det U
       \det(I_n+\widehat{G}(y,t))e^{\tr\zeta(y,t)}Z_\ba.
\end{align*}
By Cauchy' identity \eqref{eq.cauchy} and \eqref{eq.p264},
it holds that 
\[
    \det U=\det V(\boldsymbol{c}_\ba)
    \det C(\boldsymbol{c}_\ba)
    \frac{\prod_{1\le i<j\le n}
        (p_i^2-p_j^2)(\eta_j^2-\eta_i^2)}{
        \prod_{1\le i,j\le n}(p_i^2-\eta_j^2)}.
\]
It follows from \eqref{r.m_j>0.1} that
\[
    \text{\rm sgn}\biggl[
     \frac{\prod_{1\le i<j\le n}
        (p_i^2-p_j^2)(\eta_j^2-\eta_i^2)}{
        \prod_{1\le i,j\le n}(p_i^2-\eta_j^2)}\biggr]
    =(-1)^n.
\]
Since $|\det U|=1$ and 
$\det V(\boldsymbol{c}_\ba)
    \det C(\boldsymbol{c}_\ba)>0$,
we obtain that $\det U=(-1)^n$.
Thus \eqref{l.sol.nondeg.1} follows.

We next assume that $\text{\bf(H)}_m$ with $m\ge1$ holds.
We continue to use the notation introduced in the proof of
Lemma~\ref{l.reflpot.H_m} like 
$\ba(\varepsilon),\eta_j(\varepsilon),m_j(\varepsilon)$
and so on.
Put
\[
    R(\varepsilon)
    =\text{\rm diag}[\eta_1(\varepsilon),
     \dots,\eta_n(\varepsilon)],\quad
    \zeta_j(y,t;\varepsilon)
    =y\eta_j(\varepsilon)+t\eta_j(\varepsilon)^3,
\]
\[
    \zeta(y,t;\varepsilon)
    =\text{\rm diag}[\zeta_1(y,t;\varepsilon),
     \dots,\zeta_n(y,t;\varepsilon)],
\]
and 
\[
    \widehat{G}(y,t;\varepsilon)
    =\biggl(\frac{\sqrt{m_i(\varepsilon)m_j(\varepsilon)}}{
          \eta_i(\varepsilon)+\eta_j(\varepsilon)}
     e^{-(\zeta_i(y,t;\varepsilon)+\zeta_j(y,t;\varepsilon))}
     \biggr)_{1\le i,j\le n}.
\]
Due to the above observation, we know that
\[
    \det\widehat{\phi}_{\ba(\varepsilon)}(y,t)
    =\det(I_n+\widehat{G}(y,t;\varepsilon))
     e^{\tr\zeta(y,t;\varepsilon)}
     Z_{\ba(\varepsilon)}.
\]
By definition, 
$\lim\limits_{\varepsilon\to0}p_j(\varepsilon)=p_j$
for $1\le j\le n$.
Moreover, as was seen in the proof of
Lemma~\ref{l.reflpot.H_m},
we have that
\begin{equation}\label{l.sol.nondeg.211}
    \lim_{\varepsilon\to0}\eta_j(\varepsilon)=\eta_j
    \quad\text{and}\quad
    \lim_{\varepsilon\to0}m_j(\varepsilon)=m_j
    \quad\text{for }1\le j\le n.
\end{equation}
Hence it holds that
\[
    \lim_{\varepsilon\to0} \zeta(y,t;\varepsilon)
    =\zeta(y,t)
    \quad\text{and}\quad
    \lim_{\varepsilon\to0} \widehat{G}(y,t;\varepsilon)
    =\widehat{G}(y,t).
\]
Thus, to show \eqref{l.sol.nondeg.1},  it suffices to prove
that
\[
    \lim_{\varepsilon\to0}
       Z_{\ba(\varepsilon)}=Z_\ba \quad\text{and}\quad
    \lim_{\varepsilon\to0}
       U(\varepsilon)=U,
\]
where 
\[
    U(\varepsilon)
    =\biggl(\frac{c_i}{
     |(D_{\ba(\varepsilon)}^2-\eta_j(\varepsilon)^2I_n)^{-1}c_\ba|
      (p_i(\varepsilon)^2-\eta_j(\varepsilon)^2)}
     \biggr)_{1\le i,j\le n}.
\]

We first show that $Z_{\ba(\varepsilon)}$ converges to $Z_\ba$.
It is easy to see the following convergence of the factors
in $Z_{\ba(\varepsilon)}$.
\begin{align*}
    & \lim_{\varepsilon\to0} \prod_{i=1}^n 
          \frac{c_i}{2\eta_i(\varepsilon)}
      =\prod_{i=1}^n \frac{c_i}{2\eta_i},
    \\
    & 
    \lim_{\varepsilon\to0} \prod_{1\le i<j\le n}
          (p_i(\varepsilon)-p_j(\varepsilon))
       (\eta_i(\varepsilon)-\eta_j(\varepsilon))
      =\prod_{1\le i<j\le n}
          (p_i-p_j)(\eta_i-\eta_j),
    \\
    & \lim_{\varepsilon\to0}\prod_{(i,j)\notin J} 
      (p_i(\varepsilon)+\eta_j(\varepsilon))
      =\prod_{(i,j)\notin J}  (p_i+\eta_j)\ne0,
\end{align*}
and
\[
    \lim_{\varepsilon\to0}
         \prod_{j\notin \{j(1),\dots,j(m)\}}
         |(D_{\ba(\varepsilon)}^2-\eta_j(\varepsilon)^2 I_n)^{-1}
           \boldsymbol{c}_\ba|
      =\prod_{k=1}^{n-m}
       |(D_\ba^2-r_k I_n)^{-1}\boldsymbol{c}_\ba|.
\]
Thus, to see the convergence of $Z_{\ba(\varepsilon)}$, it remains 
to compute the limit of the factor
\[
    \prod_{\ell=1}^m (p_{j(\ell)+1}(\varepsilon)
         +\eta_{j(\ell)}(\varepsilon))
    \prod_{\ell=1}^m |(D_{\ba(\varepsilon)}^2
       -\eta_{j(\ell)}(\varepsilon)^2 I_n)^{-1}
           \boldsymbol{c}_\ba|.
\]

By \eqref{l.reflpot.H_m.21-1}, we have that
\begin{equation}\label{l.sol.nondeg.24}
    0>p_{j(\ell)+1}(\varepsilon)
      +\eta_{j(\ell)}(\varepsilon)=O(\varepsilon).
\end{equation}
Since 
\[
    \liminf_{\varepsilon\to0} \min\{
     |p_k(\varepsilon)^2-\eta_{j(\ell)}(\varepsilon)^2|
     \mid k\ne j(\ell),j(\ell)+1\}>0,
\]
this implies that
\begin{align*}
    & (p_{j(\ell)+1}(\varepsilon)
         +\eta_{j(\ell)}(\varepsilon))
      |(D_{\ba(\varepsilon)}^2
       -\eta_{j(\ell)}(\varepsilon)^2 I_n)^{-1}
           \boldsymbol{c}_\ba|
    \\
    & 
    =-\biggl\{\sum_{k=1}^n 
       \frac{c_k^2(p_{j(\ell)+1}(\varepsilon)
         +\eta_{j(\ell)}(\varepsilon))^2}{
       (p_k(\varepsilon)^2
         -\eta_{j(\ell)}(\varepsilon)^2)^2}
     \biggr\}^{1/2}
    \\
    & 
    =-\biggl\{
       \frac{c_{j(\ell)}^2(p_{j(\ell)+1}(\varepsilon)
         +\eta_{j(\ell)}(\varepsilon))^2}{
       (p_{j(\ell)}(\varepsilon)^2
         -\eta_{j(\ell)}(\varepsilon)^2)^2}
      +\frac{c_{j(\ell)+1}^2(p_{j(\ell)+1}(\varepsilon)
         +\eta_{j(\ell)}(\varepsilon))^2}{
       (p_{j(\ell)+1}(\varepsilon)^2
         -\eta_{j(\ell)}(\varepsilon)^2)^2}
       +O(\varepsilon^2)
     \biggr\}^{1/2}.
\end{align*}
It follows from \eqref{l.reflpot.H_m.25} and
\eqref{l.sol.nondeg.24} that 
\[
    \lim_{\varepsilon\to0}
        \frac{c_{j(\ell)}^2(p_{j(\ell)+1}(\varepsilon)
         +\eta_{j(\ell)}(\varepsilon))^2}{
       (p_{j(\ell)}(\varepsilon)^2
         -\eta_{j(\ell)}(\varepsilon)^2)^2}
      =\frac{c_{j(\ell)+1}^4}{4p_{j(\ell)}^2c_{j(\ell)}^2}
\]
and
\[
    \lim_{\varepsilon\to0}
        \frac{c_{j(\ell)+1}^2(p_{j(\ell)+1}(\varepsilon)
         +\eta_{j(\ell)}(\varepsilon))^2}{
           (p_{j(\ell)+1}(\varepsilon)^2
             -\eta_{j(\ell)}(\varepsilon)^2)^2}
      =\frac{c_{j(\ell)+1}^2}{4p_{j(\ell)}^2}.
\]
Hence we see that
\begin{equation}\label{l.sol.nondeg.25}
    \lim_{\varepsilon\to0}
      (p_{j(\ell)+1}(\varepsilon)
         +\eta_{j(\ell)}(\varepsilon))
      |(D_{\ba(\varepsilon)}^2
       -\eta_{j(\ell)}(\varepsilon)^2 I_n)^{-1}
           \boldsymbol{c}_\ba|
    =-\frac{c_{j(\ell)+1}}{2p_{j(\ell)}c_{j(\ell)}}
       (c_{j(\ell)}^2+c_{j(\ell)+1}^2)^{\frac12}.
\end{equation}
Thus $Z_{\ba(\varepsilon)}$ converges to $Z_\ba$.

We next give the proof of the convergence of $U(\varepsilon)$ to $U$.
If $j\notin\{j(\ell)\mid 1\le \ell\le m\}$, then, by 
\eqref{l.sol.nondeg.211}, we see that 
\[
    \lim_{\varepsilon\to0}U(\varepsilon)_j^i
    =\lim_{\varepsilon\to0}
      \frac1{|(D_{\ba(\varepsilon)}^2
       -\eta_j(\varepsilon)^2 I_n)^{-1}
           \boldsymbol{c}_\ba|}
      \frac{c_i}{p_i(\varepsilon)^2
              -\eta_j(\varepsilon)^2}
    =U_j^i
    \quad\text{for }1\le i\le n.
\]
If $j=j(\ell)$, then, due to \eqref{l.sol.nondeg.211} and
the existence of the limit described in
\eqref{l.sol.nondeg.25}, we see that 
\[
    \lim_{\varepsilon\to0}
      |(D_{\ba(\varepsilon)}^2
       -\eta_{j(\ell)}(\varepsilon)^2 I_n)^{-1}
           \boldsymbol{c}_\ba|=\infty.
\]
This implies that
\begin{align*}
    \lim_{\varepsilon\to0} U(\varepsilon)_{j(\ell)}^i
    & =\lim_{\varepsilon\to0}
      \frac1{|(D_{\ba(\varepsilon)}^2
       -\eta_{j(\ell)}(\varepsilon)^2 I_n)^{-1}
           \boldsymbol{c}_\ba|}
      \frac{c_i}{p_i(\varepsilon)^2
              -\eta_{j(\ell)}(\varepsilon)^2}
    \\
    & =0\times\frac{c_i}{p_i^2-p_{j(\ell)}^2}
    =0
    \quad\text{for }i\ne j(\ell),j(\ell)+1.
\end{align*}
By \eqref{l.reflpot.H_m.25} and \eqref{l.sol.nondeg.25}, 
we obtain that
\begin{align*}
    & \lim_{\varepsilon\to0} U(\varepsilon)_{j(\ell)}^{j(\ell)}
    \\
    & =\lim_{\varepsilon\to0} 
      \frac{c_{j(\ell)}}{p_{j(\ell)}(\varepsilon)
                       +\eta_{j(\ell)}(\varepsilon)}
      \bigl\{(p_{j(\ell)}(\varepsilon)
                       -\eta_{j(\ell)}(\varepsilon))
       |(D_{\ba(\varepsilon)}^2-\eta_{j(\ell)}(\varepsilon)^2 I_n)^{-1}
           \boldsymbol{c}_\ba|\bigr\}^{-1}
    \\
   &  =\frac{c_{j(\ell)}}{2p_{j(\ell)}}
       \lim_{\varepsilon\to0} \biggl\{-\biggl(
            \frac{c_{j(\ell)}^2}{(p_{j(\ell)}(\varepsilon)
               +\eta_{j(\ell)}(\varepsilon))^2}
           +\frac{c_{j(\ell)+1}^2(p_{j(\ell)}(\varepsilon)
               -\eta_{j(\ell)}(\varepsilon))^2}{
            (p_{j(\ell)+1}(\varepsilon)^2
               -\eta_{j(\ell)}(\varepsilon)^2)^2}
           \biggr)^{\frac12}\biggr\}^{-1}
    \\
    & =-c_{j(\ell)+1}
       (c_{j(\ell)}^2+c_{j(\ell)+1}^2)^{-\frac12}
      =U_{j(\ell)}^{j(\ell)}
\end{align*}
and
\begin{align*}
    & \lim_{\varepsilon\to0} U(\varepsilon)_{j(\ell)}^{j(\ell)+1}
    \\
    & =\lim_{\varepsilon\to0} 
       \frac{c_{j(\ell)+1}}{
          p_{j(\ell)+1}(\varepsilon)-\eta_{j(\ell)}(\varepsilon)}
       \bigl\{
        (p_{j(\ell)+1}(\varepsilon)+\eta_{j(\ell)}(\varepsilon))
        |(D_{\ba(\varepsilon)}^2-\eta_{j(\ell)}(\varepsilon)^2 I_n)^{-1}
           \boldsymbol{c}_\ba|\bigr\}^{-1}
    \\
    & =c_{j(\ell)}(c_{j(\ell)}^2+c_{j(\ell)+1}^2)^{-\frac12}
      =U_{j(\ell)}^{j(\ell)+1}.
\end{align*}
Thus $U(\varepsilon)$ converges to $U$.
\end{proof}

\begin{proof}[Proof of Theorem~\ref{t.kdv}]
Due the last assertion of Lemma~\ref{l.sol.nondeg} and  
Lemma~\ref{l.reflpot.cv0} with $\phi=\widehat{\phi}_\ba$, we
obtain that
\[
    \widehat{\Psi}_\ba(x,t)
    =\biggl(
       \frac{\det\widehat{\phi}_\ba(0,t) e^{-x\tr D_\ba}}{
       \det\widehat{\phi}_\ba(x,t)} 
      \biggr)^{\frac12}.
\]
By \eqref{l.sol.nondeg.1}, we see that
\[
    \frac{\det\widehat{\phi}_\ba(0,t)}{
       \det\widehat{\phi}_\ba(x,t)}
    =\frac{\det(I_n+\widehat{G}(0,t))}{              
           \det(I_n+\widehat{G}(x,t))
           e^{x\sum\limits_{j=1}^n \eta_j}}.
\]
Combining these two identities, we obtain the
desired expression of $\log\widehat{\Psi}_\ba(x,t)$.
\end{proof}

\chapter{Applications of square root transformations}
\label{chap.sqrt}

\begin{quote}{\small
Applications of square root transformations of order one 
are discussed.
Especially, the evaluation of Laplace transformations will
be extended to pinned measures, 
and the variation of evaluation accordingly 
as the pinning changes will be represented in terms of 
Pl\"ucker coordinate.
Several examples on the variation are also given.
}
\end{quote}

\section{Selfdecomposability}
\label{sec.selfd}
In this subsection, we apply Theorem~\ref{t.tways} to 
characteristic functions of quadra\-tic forms.
For this purpose, let $\eta\in\stwo$ 
and take an ONB $\{h_n\}_{n=1}^\infty$ of $\mathcal{H}$ 
such that $B_\eta$ is developed as
\[
    B_\eta=\sum_{n=1}^\infty a_n h_n\otimes h_n,
\]
where $a_n\in \mathbb{R}$ for $n\in \mathbb{N}$ and 
$\sum\limits_{n=1}^\infty a_n^2<\infty$.
Define $\f_\eta:\mathbb{R}\to[0,\infty)$ by
\[
    \f_\eta(x)
    =\begin{cases}
       \displaystyle
       \frac12 \frac1{|x|} \sum_{n;xa_n>0} 
             e^{-|\frac{x}{a_n}|} 
       & \text{if }x\in \mathbb{R}\setminus\{0\},
       \\
       0 & \text{if }x=0,
     \end{cases}
\]
where $\sum\limits_{n;xa_n>0}$ means that the sum is taken
over $n$s with the property that $xa_n>0$.
Since
\begin{equation}\label{eq.selfd.1}
    e^{-|\frac{x}{a}|} 
    \le k!\Bigl|\frac{a}{x}\Bigr|^k
    \quad\text{for $x,a\in \mathbb{R}\setminus\{0\}$ and 
     $k\in \mathbb{N}\cup\{0\}$,}
\end{equation}
the series for $\f_\eta$ converges uniformly on compacts in  
$\mathbb{R}\setminus\{0\}$. 
Thus $\f_\eta$ is continuous on
$\mathbb{R}\setminus\{0\}$, and hence measurable on
$\mathbb{R}$.
Observe that
\begin{align*}
    \int_{-1}^1 |x|^2 \f_\eta(x)dx
    & =\frac12\sum_{n;a_n<0} \int_{-1}^0 |x| 
        e^{-|\frac{x}{a_n}|} dx
     +\frac12\sum_{n;a_n>0} \int_0^1 |x| 
        e^{-|\frac{x}{a_n}|} dx
    \\
    & \le \frac12 \biggl(\sum_{n=1}^\infty a_n^2\biggr)
        \int_0^\infty y e^{-y} dy
      =\frac12\|B_\eta\|_{\htwo}^2.
\end{align*}
Combining this with \eqref{eq.selfd.1} for $k=3$ and
$|x|>1$, we obtain that
\begin{equation}\label{eq.selfd.2}
    \int_{-\infty}^\infty |x|^2 \f_\eta(x)dx
    <\infty.
\end{equation}
We shall show the following expression of the characteristic
function of $\q_\eta$.

\begin{theorem}\label{t.selfd}
It holds that
\begin{equation}\label{t.selfd.1}
    \int_{\mathcal{W}} e^{\kyosu\lambda\q_\eta} d\mu
    =\exp\biggl(\int_{-\infty}^\infty 
     \{e^{\kyosu\lambda x}-1-\kyosu\lambda x\}
     \f_\eta(x)dx \biggr)
    \quad\text{for every }\lambda\in \mathbb{R}.
\end{equation}
\end{theorem}

\begin{remark}\label{r.selfd}
(i)
Since 
$|e^{\kyosu\lambda x}-1-\kyosu\lambda x|
 \le \frac12\lambda^2|x|^2$ 
for $x\in \mathbb{R}$, by virtue of \eqref{eq.selfd.2},
$\{e^{\kyosu\lambda x}-1-\kyosu\lambda x\}
     \f_\eta(x)$
is integrable with respect to the Lebesgue measure.
Hence the integral in the exponent of the RHS of the identity
\eqref{t.selfd.1} is finitely definite.
\\
(ii) 
The expression \eqref{t.selfd.1} says that the distribution
of $\q_\eta$ is selfdecomposable and the corresponding
L\'evy measure is $\f_\eta(x)dx$ 
(cf.\cite[Corollary~15.11]{sato}). 
\end{remark}

\begin{proof}
Define $\eta^{(+)},\eta^{(-)}\in\stwo$ by
\[
    \eta^{(\pm)}(t,s)
    =\sum_{n;\pm a_n>0} a_n
     [h_n^\prime(s)\otimes h_n^\prime(t)]
    \quad\text{for }(t,s)\in[0,T]^2 
    \quad\text{(double-sign corresponds)}.
\]
By Theorem~\ref{t.q.eta}, 
$\q_\eta=\q_{\eta^{(+)}}+\q_{\eta^{(-)}}$, and
$\q_{\eta^{(+)}}$ and
$\q_{\eta^{(-)}}$ are independent.
Hence it holds that
\begin{equation}\label{t.selfd.21}
    \int_{\mathcal{W}} e^{\kyosu\lambda\q_\eta} d\mu
    =\biggl(\int_{\mathcal{W}} 
            e^{\kyosu\lambda\q_{\eta^{(+)}}} d\mu\biggr)
     \biggl(\int_{\mathcal{W}} 
            e^{\kyosu\lambda\q_{\eta^{(-)}}} d\mu\biggr)
    \quad\text{for }\lambda\in \mathbb{R}.
\end{equation}

We first consider $\q_{\eta^{(+)}}$.
Let $\lambda_0>0$.
Since 
$B_{\eta^{(+)}}=\sum\limits_{n;a_n>0} a_n h_n\otimes h_n$, we know that
\[
    \Lambda(B_{-\lambda_0\eta^{(+)}})
    =\sup_{n;a_n>0}(-\lambda_0 a_n)\le 0<1.
\]
By Theorem~\ref{t.tways}, we have that
\[
    \int_{\mathcal{W}} e^{-\lambda_0 q_{\eta^{(+)}}} d\mu
    =\{\dettwo(I+\lambda_0 B_{\eta^{(+)}})\}^{-\frac12}
     =\exp\biggl(\frac12\sum_{n;a_n>0}
       \{-\log(1+a_n\lambda_0)+a_n\lambda_0\}\biggr).
\]

Let $a>0$.
Applying the integration by parts formula successively, we
have that 
\[
    \int_0^\infty x^n e^{-\frac{x}{a}} dx=n! a^{n+1}
    \quad\text{for }n\in \mathbb{N}\cup\{0\}.
\]
Hence it holds that
\[
    \int_0^\infty \{e^{-\lambda_0 x}-1+\lambda_0 x\}
     \frac1{x} e^{-\frac{x}{a}} dx
    =\sum_{n=1}^\infty \frac{(-1)^n\lambda_0^n a^n}{n}
     + a\lambda_0
    =-\log(1+a\lambda_0)+a\lambda_0.
\]

Since 
\[
    |e^{-\lambda_0 x}-1+\lambda_0 x|
     \le \frac{\lambda_0^2}2|x|^2
    \quad\text{and}\quad
    0\le \frac12\sum_{n\le N;a_n>0}
      \frac1{x}e^{-\frac{x}{a_n}} \le \f_{\eta^{(+)}}(x)
\]
for $x\ge0$ and $N\in \mathbb{N}$,
by \eqref{eq.selfd.2} and the dominated convergence
theorem, we obtain that
\begin{align*}
    \frac12 \sum_{n;a_n>0} 
     \{-\log(1+a_n\lambda_0)+a_n\lambda_0\}
    & =\lim_{N\to\infty} \int_0^\infty 
         \{e^{-\lambda_0 x}-1+\lambda_0 x\}
      \biggl(\frac1{2x} \sum_{n\le N;a_n>0} 
      e^{-\frac{x}{a_n}}\biggr) dx
    \\
    & =\int_0^\infty \{e^{-\lambda_0 x}-1+\lambda_0 x\} 
        \f_{\eta^{(+)}}(x) dx.
\end{align*}
Thus we arrive at the identity that
\begin{equation}\label{t.selfd.22}
    \int_{\mathcal{W}} e^{-\lambda_0\q_{\eta^{(+)}}} d\mu
    =\exp\biggl(\int_0^\infty 
      \{e^{-\lambda_0 x}-1+\lambda_0 x\}
      \f_{\eta^{(+)}}(x) dx\biggr)
    \quad\text{for any $\lambda_0>0$.}
\end{equation}

Let 
$\Omega=\{\zeta\in \mathbb{C}\mid \text{\rm Re}\zeta<0\}$.
Observe that 
\begin{equation}\label{t.selfd.23}
     |e^{\zeta x}-1-\zeta x|
      =\biggl|\int_0^x\int_0^y 
          \zeta^2 e^{\zeta z}dz \biggr|
      \le \frac{|\zeta|^2}2 |x|^2
\end{equation}
and 
\[
    \biggl|\frac{d}{d\zeta}\biggl(
        e^{\zeta x}-1-\zeta x\biggr)\biggr|
    =|x(e^{\zeta x}-1)|\le |\zeta||x|^2
\]
for any $\zeta\in\Omega$ and $x\ge0$.
By \eqref{eq.selfd.2} for $\eta^{(+)}$ and the
dominated convergence theorem of differentiation type, the
mapping 
\[
    \Omega\ni\zeta\mapsto 
    \int_0^\infty\{e^{\zeta x}-1-\zeta x\} 
       \f_{\eta^{(+)}}(x)dx\in \mathbb{C}
\]
is holomorphic.
Combined with Lemma~\ref{l.hol} for $\eta^{(+)}$ and
\eqref{t.selfd.22}, this holomorphy implies that
\[
    \int_{\mathcal{W}} e^{\zeta\q_{\eta^{(+)}}} d\mu
    =\exp\biggl(\int_0^\infty 
      \{e^{\zeta x}-1-\zeta x\}
      \f_{\eta^{(+)}}(x) dx\biggr)    
    \quad\text{for any }
    \zeta\in \mathcal{D}_{\eta^{(+)}}\cap\Omega.
\]
Hence for $n\in \mathbb{N}$ with
$n>\|B_{\eta^{(+)}}\|_\op$, it holds that
\[
    \int_{\mathcal{W}} 
        e^{(-\frac1{n}+\kyosu\lambda)\q_{\eta^{(+)}}} d\mu
    =\exp\biggl(\int_0^\infty 
           \biggl\{e^{(-\frac1{n}+\kyosu\lambda) x}-1
             -\Bigl(-\frac1{n}+\kyosu\lambda\Bigr)x
           \biggr\}
           \f_{\eta^{(+)}}(x) dx\biggr)
    \quad\text{for $\lambda\in \mathbb{R}$.}
\]
By Lemma~\ref{l.hol}, the LHS of the equation converges to 
\[
    \int_{\mathcal{W}} e^{\kyosu\lambda\q_{\eta^{(+)}}}d\mu.
\]
By \eqref{t.selfd.23} and \eqref{eq.selfd.2},
the exponent in the RHS of the equation converges to 
\[
    \int_0^\infty 
    \{e^{\kyosu\lambda x}-1-\kyosu\lambda x\}
    \f_{\eta^{(+)}}(x) dx.
\]
Thus we obtain that 
\begin{equation}\label{t.selfd.24}
    \int_{\mathcal{W}} e^{\kyosu\lambda\q_{\eta^{(+)}}}d\mu
    =\exp\biggl(\int_0^\infty 
      \{e^{\kyosu\lambda x}-1-\kyosu\lambda x\}
      \f_{\eta^{(+)}}(x) dx\biggr)
    \quad\text{for }\lambda\in \mathbb{R}.
\end{equation}

We next consider $\eta^{(-)}$.
Observe that
\[
    \f_{\eta^{(-)}}(y)
    =\frac1{2|y|}\sum_{n;a_n<0}
      e^{-|\frac{y}{a_n}|}
    =\frac1{2|-y|}\sum_{n;(-a_n)>0}
      e^{-|\frac{-y}{-a_n}|}
    =\f_{(-\eta)^{(+)}}(-y)
    \quad\text{for } y<0.
\]
Since $-\eta^{(-)}=(-\eta)^{(+)}$, we obtain from 
\eqref{t.selfd.24} that
\begin{align}
 \label{t.selfd.25}
    \int_{\mathcal{W}} e^{\kyosu\lambda\q_{\eta^{(-)}}} d\mu
    & =\int_{\mathcal{W}} 
        e^{\kyosu(-\lambda)\q_{(-\eta)^{(+)}}} d\mu
    =\exp\biggl(\int_0^\infty 
         \{e^{-\kyosu\lambda x}-1+\kyosu\lambda x\}
           \f_{(-\eta)^{(+)}}(x) dx\biggr)
   \\
    & =\exp\biggl(\int_{-\infty}^0
         \{e^{\kyosu\lambda y}-1-\kyosu\lambda y\}
           \f_{\eta^{(-)}}(y) dy\biggr)
    \quad\text{for }\lambda\in \mathbb{R}.
 \nonumber
\end{align}

Noticing that
$\f_\eta=\one_{(-\infty,0)}\f_{\eta^{(-)}}
  +\one_{(0,\infty)}\f_{\eta^{(+)}}$, 
and plugging \eqref{t.selfd.24} and \eqref{t.selfd.25} into  
\eqref{t.selfd.21}, we obtain \eqref{t.selfd.1}.
\end{proof}

\section{Pinned measure}
\label{sec.cond.expect}
In this section, we show that the evaluation
\eqref{t.tways.3} of Laplace transformations is extended
to pinned measures.

Let $k_1,\dots,k_N\in \mathcal{H}$ be linearly independent. 
Put $\boldsymbol{k}=(k_1,\dots,k_N)$, and denote by
$\pi_{\boldsymbol{k}}$ the orthogonal projection of
$\mathcal{H}$ to the subspace spanned by $k_1,\dots,k_N$.
By Lemma~\ref{l.b.kappa}, there is a
$\Pi_{\boldsymbol{k}}\in\stwo$ with the property that
$\pi_{\boldsymbol{k}}=B_{\Pi_{\boldsymbol{k}}}$.
For $\eta\in\stwo$, put
\[
    B_{\eta;\boldsymbol{k}}
    =\pi_{\boldsymbol{k}}^\perp
       B_\eta \pi_{\boldsymbol{k}}^\perp,
    \quad\text{where }
    \pi_{\boldsymbol{k}}^\perp=I-\pi_{\boldsymbol{k}}.
\]
Define $\eta_{\boldsymbol{k}}\in\stwo$ by
\begin{align*}
    \eta_{\boldsymbol{k}}(t,s)
    = & \eta(t,s)
      -\int_0^T\Bigl\{\Pi_{\boldsymbol{k}}(t,u)\eta(u,s)
      +\eta(t,u)\Pi_{\boldsymbol{k}}(u,s)\Bigr\} du
    \\
    & +\int_0^T\int_0^T \Pi_{\boldsymbol{k}}(t,u)
       \eta(u,v)\Pi_{\boldsymbol{k}}(v,s) dudv
    \quad\text{for }(t,s)\in[0,T]^2.
\end{align*}
We then have that     
$B_{\eta;\boldsymbol{k}}=B_{\eta_{\boldsymbol{k}}}$:
\[
    (B_{\eta;\boldsymbol{k}}h)^\prime(t)
    =\int_0^T \eta_{\boldsymbol{k}}(t,s)h^\prime(s)ds
    \quad\text{for }h\in \mathcal{H}
    \text{ and }t\in[0,T].
\]
The Weiner functional 
$\D^*\boldsymbol{k}=(\D^*k_1,\dots,\D^*k_N)$
is in $\mathbb{D}^\infty(\mathbb{R}^N)$.
Let 
\[
  C(\boldsymbol{k})
  =\bigl( \langle k_i,k_j\rangle_{\mathcal{H}}
    \bigr)_{1\le i,j\le N}.
\]
Due to the linear independence of $k_i$s, 
$\det C(\boldsymbol{k})\ne0$.
Since $\D(\D^*k_i)=k_i$ for $1\le i\le N$
(\cite[(5.1.9)]{mt-cambridge}, 
$\D^* \boldsymbol{k}$ is non-degenerate.
Let $\mathcal{P}$ be the space of Wiener functionals of the
form $p(\ell_1,\dots,\ell_n)$ with a polynomial
$p:\mathbb{R}^n\to \mathbb{R}$ and 
$\ell_1,\dots,\ell_n\in \mathcal{W}^*$.
For $\eta\in\stwo$ with $\Lambda(B_\eta)<1$, by
Lemma~\ref{l.q.eta.int}, 
$e^{\q_\eta}\in L^{1+}(\mu)$.
Due to \eqref{eq.e^a}, we can consider
$\int_{\mathcal{W}} f e^{\q_\eta+\D^*h}
 \delta_0(\D^*\boldsymbol{k}) d\mu$ for any
$f\in \mathcal{P}$ and $h\in \mathcal{H}$.

\begin{theorem}\label{t.lap.cond}
Let $\eta\in\stwo$ and assume that $\Lambda(B_\eta)<1$.
Then it holds that
\begin{align}\label{t.lap.cond.1}
    \int_{\mathcal{W}} fe^{\q_\eta+\D^*h} 
        \delta_0(\D^*\boldsymbol{k}) d\mu
    = & \{(2\pi)^N
          \dettwo(I-B_{\eta;\boldsymbol{k}})
          \det C(\boldsymbol{k}) \}^{-\frac12}
        e^{-\frac12 \tr(\pi_{\boldsymbol{k}}B_\eta \pi_{\boldsymbol{k}})}
    \\
    & \times
      \int_{\mathcal{W}}
       ((\pi_{\boldsymbol{k}}^\perp)_*f)
        (\iota+F_{\widehat{\kappa_S(\eta_{\boldsymbol{k}})}})
        e^{\D^*[C_{\eta_{\boldsymbol{k}}}^{-1}
         (\pi_{\boldsymbol{k}}^\perp h)]} d\mu
 \nonumber
\end{align}
for every $f\in \mathcal{P}$ and $h\in \mathcal{H}$,
where
\[
    (\pi_{\boldsymbol{k}}^\perp)_*f
    =p\bigl(\D^*(\pi_{\boldsymbol{k}}^\perp\ell_1),
      \dots,\D^*(\pi_{\boldsymbol{k}}^\perp\ell_n)
     \bigr)
    \quad\text{for }f=p(\ell_1,\dots,\ell_n).
\]
\end{theorem}

Since 
$\Lambda(B_{\eta_{\boldsymbol{k}}})\le \Lambda(B_\eta)<1$, 
$\widehat{\kappa_S(\eta_{\boldsymbol{k}})}\in\stwo$ is
defined well. 

\begin{proof}
Notice that $C(\boldsymbol{k})$ is symmetric and positive definite. 
Take a 
$V=(V_{ij})_{1\le i,j\le N}\in \mathbb{R}^{N\times N}$ 
with
$VC(\boldsymbol{k})V^\dagger=I_N$.
For 
$\varphi,\widetilde{\varphi}\in\mathscr{S}(\mathbb{R}^N)$,
it holds that
\[
    \int_{\mathbb{R}^N}  \varphi(V^{-1}x)
      \widetilde{\varphi}(x) dx
    =\int_{\mathbb{R}^N}  \varphi(y)
      \widetilde{\varphi}(Vy) |\det V| dy.
\]
Letting $\varphi\to\delta_0$, we obtain that
$\delta_0(V^{-1} dx)=|\det V|\delta_0(dx)$.

Define $\widehat{k}_1,\dots,\widehat{k}_N\in \mathcal{H}$ by 
$\widehat{k}_i=\sum\limits_{j=1}^N V_{ij}k_j$
for $1\le i\le N$.
Then $\widehat{k}_1,\dots,\widehat{k}_N$ are orthonormal.
Put 
$\D^*\widehat{\boldsymbol{k}}
 =(\D^*\widehat{k}_1,\dots,\D^*\widehat{k}_N)$.
Then 
$\D^*\boldsymbol{k}=V^{-1}\D^*\widehat{\boldsymbol{k}}$.
Since $(\det V)^2\det C(\boldsymbol{k})=1$, 
we have that
\[
    \int_{\mathcal{W}} f e^{\q_\eta+\D^*h}
    \delta_0(\D^*\boldsymbol{k}) d\mu
    =\{\det C(\boldsymbol{k})\}^{-\frac12}
      \int_{\mathcal{W}} f e^{\q_\eta+\D^*h}
        \delta_0(\D^*\widehat{\boldsymbol{k}}) d\mu.
\]
Furthermore, note that
$B_{\eta;\boldsymbol{k}}=B_{\eta;\widehat{\boldsymbol{k}}}$,
$\pi_{\boldsymbol{k}}=\pi_{\widehat{\boldsymbol{k}}}$, and
$\eta_{\boldsymbol{k}}=\eta_{\widehat{\boldsymbol{k}}}$.
Thus, to prove \eqref{t.lap.cond.1}, we may and will assume
that $k_1,\dots,k_N$ are orthonormal.

Take an ONB $\{h_n\}_{n=1}^\infty$ of the subspace
$\pi_{\boldsymbol{k}}^\perp(\mathcal{H})$ such that 
\[
    B_{\eta;\boldsymbol{k}}=\sum_{n=1}^\infty 
     \langle B_\eta h_n,h_n\rangle_{\mathcal{H}}
     h_n\otimes h_n.
\]
Then 
$\{k_1,\dots,k_N,h_n\mid n\in \mathbb{N}\}$ is an ONB of
$\mathcal{H}$.
By Theorem~\ref{t.q.eta}, it holds that
\begin{align*}
    \q_\eta= & \frac12\sum_{n=1}^\infty
      \langle B_\eta h_n,h_n\rangle_{\mathcal{H}}
      \{(\D^*h_n)^2-1\}
     +\sum_{n=1}^\infty \sum_{j=1}^N
      \langle B_\eta h_n,k_j\rangle_{\mathcal{H}}
      (\D^*h_n)(\D^*k_j)
    \\
    & +\frac12 \sum_{i,j=1}^N
      \langle B_\eta k_i,k_j\rangle_{\mathcal{H}}
      \{(\D^*k_i)(\D^*k_j)-\delta_{ij}\}.
\end{align*}
Note that $B_{\eta;\boldsymbol{k}}=B_{\eta_{\boldsymbol{k}}}$
and $\D^*\boldsymbol{k}=0$
$\delta_0(\D^*\boldsymbol{k})d\mu$-a.e.
By Theorem~\ref{t.q.eta} for $\q_{\eta_{\boldsymbol{k}}}$,
this expression of $\q_\eta$, and the definition of
$(\pi_{\boldsymbol{k}}^\perp)_*f$, we have that
\[
    \q_\eta =\q_{\eta_{\boldsymbol{k}}}
       -\frac12 
        \tr(\pi_{\boldsymbol{k}}B_\eta \pi_{\boldsymbol{k}})
    \quad\text{and}\quad
    f=(\pi_{\boldsymbol{k}}^\perp)_*f
    \quad\text{$\delta_0(\D^*\boldsymbol{k})d\mu$-a.e.}
\]
Furthermore, developing $h\in \mathcal{H}$ as
$h=\sum\limits_{j=1}^N \langle h,k_j\rangle_{\mathcal{H}}
 k_j+\pi_{\boldsymbol{k}}^\perp h$, we see that
\[
    \D^*h
    =\D^*(\pi_{\boldsymbol{k}}^\perp h)
    \quad\text{$\delta_0(\D^*\boldsymbol{k})d\mu$-a.e.}
\]
Thus we obtain that
\begin{equation}\label{t.lap.cond.21}
    \int_{\mathcal{W}} f e^{\q_\eta+\D^*h} 
      \delta_0(\D^*\boldsymbol{k})d\mu
    =e^{-\frac12\tr(\pi_{\boldsymbol{k}}
            B_\eta\pi_{\boldsymbol{k}})}
     \int_{\mathcal{W}} 
       ((\pi_{\boldsymbol{k}}^\perp)_*f)
       e^{\q_{\eta_{\boldsymbol{k}}}
          +\D^*(\pi_{\boldsymbol{k}}^\perp h)} 
      \delta_0(\D^*\boldsymbol{k})d\mu.    
\end{equation}

Let $p:\mathbb{R}^n\to \mathbb{R}$ be a polynomial.
Take a sequence 
$\{\varphi_m\}_{m=1}^\infty\subset\mathscr{S}(\mathbb{R}^N)$ 
converging to $\delta_0$ in
$\mathscr{S}^\prime(\mathbb{R}^N)$.
We then have that
\begin{align*}
    & \int_{\mathcal{W}} p(\D^*h_1,\dots,\D^*h_n)
     \delta_0(\D^*\boldsymbol{k}) d\mu
    =\lim_{m\to\infty}
      \int_{\mathcal{W}} p(\D^*h_1,\dots,\D^*h_n)
      \varphi_m(\D^*\boldsymbol{k}) d\mu
    \\
    & =\biggl(\int_{\mathcal{W}} p(\D^*h_1,\dots,\D^*h_n) 
       d\mu\biggr)
     \lim_{m\to\infty}
      \int_{\mathcal{W}} \varphi_m(\D^*\boldsymbol{k}) d\mu
    \\
    & 
    =\biggl(\int_{\mathcal{W}} p(\D^*h_1,\dots,\D^*h_n) 
       d\mu\biggr)
     \lim_{m\to\infty}
      \int_{\mathbb{R}^N} \varphi_m(x) 
       \frac1{\sqrt{(2\pi)^N}} e^{-\frac12|x|^2} dx
    \\
    & 
    =(2\pi)^{-\frac{N}2}
     \int_{\mathcal{W}} p(\D^*h_1,\dots,\D^*h_n) d\mu.
\end{align*}
Hence, if  
$\Psi\in \mathbb{D}^{\infty,1+}$ is 
$\sigma[\{\D^*h_n\mid n\in \mathbb{N}\}]$-measurable, then  
\[
    \int_{\mathcal{W}}\Psi \delta_0(\D^*\boldsymbol{k}) d\mu 
    =(2\pi)^{-\frac{N}2}
    \int_{\mathcal{W}}\Psi d\mu.
\]
Combining this with \eqref{t.lap.cond.21}, we obtain that
\[
    \int_{\mathcal{W}} f e^{\q_\eta+\D^*h} 
      \delta_0(\D^*\boldsymbol{k})d\mu
    =(2\pi)^{-\frac{N}2}
     e^{-\frac12\tr[\pi_{\boldsymbol{k}}
            B_\eta\pi_{\boldsymbol{k}}]}
     \int_{\mathcal{W}} 
       ((\pi_{\boldsymbol{k}}^\perp)_*f)
       e^{\q_{\eta_{\boldsymbol{k}}}
          +\D^*(\pi_{\boldsymbol{k}}^\perp h)} d\mu.
\]
Applying Theorem~\ref{t.tways}(ii), we obtain 
\eqref{t.lap.cond.1}.
\end{proof}

\section{Pl\"ucker coordinates}
\label{sec.plucker}
In this section, we see how 
$\dettwo(I-B_{\eta;\boldsymbol{k}})$ appearing in
Theorem~\ref{t.lap.cond} varies as $N$ changes. 
For this purpose, Assume that $\eta\in\stwo$ satisfy the
condition that 
\\[5pt]
\indent
{\bf (A)}
\begin{minipage}[t]{380pt}
there exist $A_I,A_F\in\htwo$ such that
$\dettwo(I-A_I)\ne0$, $A_F$ is of finite rank, i.e., the
range $\mathcal{R}(A_F)$ of $A_F$ is of finite dimension,
and $B_\eta=A_I+A_F$.
\end{minipage}
\\[5pt]
Take linearly independent $k_1,\dots,k_M\in \mathcal{H}$
with $M\in \mathbb{N}$ such that $\mathcal{R}(A_F)$
is included in the subspace spanned by them.
Noting that $I-A_I$ has a continuous inverse
(\cite[Theorem\,XII.1.2]{GGK}),
for each $p=(p_1,\dots,p_M)\in \mathbb{R}^M$, we define
$J_p\in \mathcal{H}$ by
\[
    J_p=(I-A_I)^{-1}\biggl(\sum_{j=1}^M p_jk_j\biggr).
\]
Define the projections 
$\pi^{(N)}:\mathcal{H}\to \mathcal{H}$ for $0\le N\le M$ as
$\pi^{(0)}=0$ and 
$\pi^{(N)}=\pi_{\boldsymbol{k}^{(N)}}$ if $0<N\le M$, 
where $\boldsymbol{k}^{(N)}=(k_1,\dots,k_N)$ and 
$\pi_{\boldsymbol{k}^{(N)}}$ is the projection defined 
in the paragraph before Theorem~\ref{t.lap.cond}.
Let 
\[
    B_\eta^{(N)}=(I-\pi^{(N)})B_\eta.
\]
Define 
$\J_{\!N}\in \mathbb{R}^{M\times M}$ by
\[
    \J_{\!N}\,p
    =\Bigl(\bigl\langle \{I-(B_\eta^{(N)}
      -B_\eta^{(M)})\}J_p,k_j\bigr\rangle_{\mathcal{H}}
     \Bigr)_{1\le j\le M}
     \quad\text{for } p\in \mathbb{R}^M.
\]
We show the following expression of 
$\int_{\mathcal{W}} e^{\q_\eta}
 \delta_0(\D^*\boldsymbol{k}^{(N)}) d\mu$.

\begin{theorem}\label{t.plucker}
Let $\eta\in\stwo$.
Assume that $\Lambda(B_\eta)<1$ and the condition~{\bf(A)}
is fulfilled. 
Then $\det \J_{\!N}\ne0$ for every
$0\le N\le M$, and it holds that
\begin{equation}\label{t.plucker.1}
    \int_{\mathcal{W}} e^{\q_\eta}
      \delta_0(\D^*\boldsymbol{k}^{(N)})d\mu
    =\biggl(\frac{\det C(\boldsymbol{k}^{(M)})}{
         (2\pi)^N \det C(\boldsymbol{k}^{(N)})
         \det \J_{\!N}} \biggr)^{\frac12}
     \{\dettwo(I-A_I)\}^{-\frac12}
     e^{-\frac12\tr A_F}
\end{equation}
for $0\le N\le M$,
where 
$\delta_0(\D^*\boldsymbol{k}^{(0)})d\mu=d\mu$.
\end{theorem}

\begin{remark}\label{r.plucker}
(i) 
As will be seen in the examples in the next section,
\eqref{t.plucker.1} gives us another explicit expression of 
Laplace transformations with respect to pinned measures.
\\
(ii)
Assume that $k_1,\dots,k_M$ are orthogonal.
Let 
$\Phi
  =\begin{pmatrix} \J_{\!M} \\ \J_0 \end{pmatrix}
  \in \mathbb{R}^{2M\times M}$.
It holds that
\begin{align*}
    \langle(I-(B_\eta^{(N)}-B_\eta^{(M)}))J_p,
      k_j\rangle_{\mathcal{H}}
    & =\langle J_p,k_j\rangle_{\mathcal{H}}
       -\langle B_\eta J_p,(\pi^{(M)}-\pi^{(N)})k_j
         \rangle_{\mathcal{H}}
    \\
    & =\begin{cases}
        \langle J_p,k_j\rangle_{\mathcal{H}}
        & \text{if }j\le N,
        \\
        \langle J_p,k_j\rangle_{\mathcal{H}}
         -\langle B_\eta J_p,k_j\rangle_{\mathcal{H}}
        & \text{if }j>N.
      \end{cases}
\end{align*}
Hence denoting by $\Phi^{(i)}$ the $i$th row of $\Phi$ for
$1\le i\le 2M$, we have that
\begin{equation}\label{r.plucker.1}
    \J_{\!N}=\begin{pmatrix} \Phi^{(1)} \\
    \vdots \\ \Phi^{(N)} \\ \Phi^{(M+N+1)} \\
    \vdots \\ \Phi^{(2M)} \end{pmatrix}.
\end{equation}
Thus $\det \J_{\!N}$ is the $(1,\dots,N,M+N+1,\dots,2M)$th
Pl\"ucker coordinate of $\Phi$.
\end{remark}

For the proof of Theorem~\ref{t.plucker}, we prepare a
lemma.  

\begin{lemma}\label{l.cond.det}
Let $\eta\in\stwo$.
Assume that $\Lambda(B_\eta)<1$ and the condition~{\bf(A)}
is fulfilled. 
Then it holds that
\begin{align}
 \label{l.cond.det.1}
    \int_{\mathcal{W}} e^{\q_\eta}
      \delta_0(\D^*\boldsymbol{k}^{(N)})d\mu
    = & \bigl\{(2\pi)^N \det C(\boldsymbol{k}^{(N)})
       \det(I-A_F^{(N)}(I-A_I)^{-1})\bigr\}^{-\frac12}
    \\
    & \times
     \{\dettwo(I-A_I)\}^{-\frac12}
     e^{-\frac12 \tr A_F}
    \quad\text{for }0\le N\le M,
 \nonumber
\end{align}
where 
\[
    A_F^{(N)}=-\pi^{(N)}A_I+(I-\pi^{(N)})A_F.
\]
\end{lemma}

\begin{proof}
Recall the identity 
$\dettwo(I+A)=\det(I+A)e^{-\tr A}$ for $A\in\htwo$ of trace 
class (\cite[Theorem~IX.2.1]{GGK}).
This implies that 
\begin{equation}\label{l.cond.det.21}
    \dettwo\bigl((I+A_1)(I+A_2)\bigr)
    =\det(I+A_1)\dettwo(I+A_2) e^{-\tr(A_1(I+A_2))}
\end{equation}
for $A_1\in\htwo$ of trace class and $A_2\in\htwo$.

We first consider the case when $N=0$.
By Theorem~\ref{t.tways}, we have that 
\[
    \int_{\mathcal{W}} e^{\q_\eta} d\mu
    =\{\dettwo(I-B_\eta)\}^{-\frac12}.
\]
Recalling that $B\in\htwo$ of finite rank is of trace class, 
and substituting $A_1=-A_F(I-A_I)^{-1}$ and $A_2=-A_I$ into
\eqref{l.cond.det.21}, we obtain that 
\[
    \dettwo(I-B_\eta)
     =\det\bigl(I-A_F(I-A_I)^{-1}\bigr)\dettwo(I-A_I)
       e^{\tr A_F}.
\]
Thus \eqref{l.cond.det.1} holds when $N=0$.

We next consider the case when $1\le N\le M$.
By Theorem~\ref{t.lap.cond}, we have that
\[
    \int_{\mathcal{W}} e^{\q_\eta}
       \delta_0(\D^*\boldsymbol{k}^{(N)})d\mu
    =\bigl\{ (2\pi)^N 
       \det C(\boldsymbol{k}^{(N)})
       \dettwo(I-B_{\eta;\boldsymbol{k}^{(N)}})
     \bigr\}^{-\frac12}
     e^{-\frac12\tr(\pi^{(N)}B_\eta \pi^{(N)})}.
\]
It holds that
\begin{equation}\label{l.cond.det.22}
    \dettwo(I-B_{\eta;\boldsymbol{k}^{(N)}})
    =\dettwo(I-B_\eta^{(N)}).
\end{equation}
Since
\[
    \bigl\{I-A_F^{(N)}(I-A_I)^{-1}\bigr\}(I-A_I)
    =I-A_I-A_F^{(N)}=I-B_\eta^{(N)},
\]
substituting $A_1=-A_F^{(N)}(I-A_I)^{-1}$ and $A_2=-A_I$
into \eqref{l.cond.det.21}, we obtain that
\[
    \dettwo(I-B_\eta^{(N)})
    =\det\bigl(I-A_F^{(N)}(I-A_I)^{-1}\bigr)
     \dettwo(I-A_I) e^{\tr A_F^{(N)}}.
\]
Furthermore, observe that 
\begin{align*}
   \tr\bigl(A_F^{(N)}+\pi^{(N)} B_\eta \pi^{(N)}\bigr)
   & =\tr\bigl(-\pi^{(N)}A_I+(I-\pi^{(N)})A_F
      +\pi^{(N)} B_\eta \pi^{(N)}\bigr)
   \\
   & =\tr\bigl(-\pi^{(N)} B_\eta (I-\pi^{(N)})+A_F\bigr)
     =\tr A_F.
\end{align*}
Thus \eqref{l.cond.det.1} holds for $1\le N\le M$.
\end{proof}

\begin{proof}[Proof of Theorem~\ref{t.plucker}]
Let $0\le N\le M$.
Set 
\[
    K_N=\pi^{(M)}(I-A_F^{(N)}(I-A_I)^{-1})\pi^{(M)},
\]
and define 
$q_N=(q_{N;ij})_{1\le i,j\le M}
 \in \mathbb{R}^{M\times M}$ by
\[
    K_Nk_i=\sum_{j=1}^M q_{N;ij} k_j
    \quad\text{for }1\le i\le M.
\]
It then holds that 
\[
    \det(I-A_F^{(N)}(I-A_I)^{-1})=\det q_N.
\]
By Lemma~\ref{l.cond.det}, we have that 
\begin{equation}\label{t.plucker.21}
    \int_{\mathcal{W}} e^{\q_\eta}
      \delta_0(\D^*\boldsymbol{k}^{(N)})d\mu
    =\bigl\{(2\pi)^N \det C(\boldsymbol{k}^{(N)})
      \det q_N \bigr\}^{-\frac12} 
     \{\dettwo(I-A_I)\}^{-\frac12}
     e^{-\frac12 \tr A_F}.
\end{equation}

Since 
$\Lambda(B_{\eta;\boldsymbol{k}^{(N)}})
 \le \Lambda(B_\eta)<1$, 
as was seen in the proof of Lemma~\ref{l.tw.(ii)to(iii)},
$\dettwo(I-B_{\eta;\boldsymbol{k}^{(N)}})\ne0$.
By \eqref{l.cond.det.22}, 
$\dettwo(I-B_\eta^{(N)})\ne0$, and hence 
$I-B_\eta^{(N)}$ has a continuous inverse.
Define $J_{N;p}\in \mathcal{H}$ for 
$p=(p_1,\dots,p_M)\in \mathbb{R}^M$ by
\[
    J_{N;p}=(I-B_\eta^{(N)})^{-1}\biggl(\sum_{j=1}^M
      p_jk_j\biggr).
\]
Notice that
\[ 
    \bigl(I-A_F^{(N)}(I-A_I)^{-1}\bigr)k_i=K_Nk_i
    \quad\text{for }1\le i\le M,
\]
and
\[
    I-A_F^{(N)}(I-A_I)^{-1}
      =I-(B_\eta^{(N)}-A_I)(I-A_I)^{-1}
      =(I-B_\eta^{(N)})(I-A_I)^{-1}.
\]
Then it holds that
\[
    (I-A_I)^{-1}k_i
    =(I-B_\eta^{(N)})^{-1}K_Nk_i
    =\sum_{j=1}^M q_{N;ij} (I-B_\eta^{(N)})^{-1}k_j.
\]
Hence we have that
\begin{equation}\label{t.plucker.22}
    J_p=J_{N;q_N^\dagger p}
    \quad\quad\text{for } p\in \mathbb{R}^M.
\end{equation}

Since $\langle B_\eta^{(M)}h,k_i\rangle_{\mathcal{H}}=0$ for 
any $h\in \mathcal{H}$ and $1\le i\le M$, we see that
\begin{align*}
    \langle J_{N;p},k_i\rangle_{\mathcal{H}}
    & =\langle (I-B_\eta^{(M)})J_{N;p},
           k_i\rangle_{\mathcal{H}}
      =\langle (I-B_\eta^{(N)})J_{N;p},
           k_i\rangle_{\mathcal{H}}
       +\langle (B_\eta^{(N)}-B_\eta^{(M)})J_{N;p},
           k_i\rangle_{\mathcal{H}}
    \\
    & =\biggl\langle \sum_{j=1}^M p_jk_j,
           k_i\biggr\rangle_{\mathcal{H}}
     +\langle (B_\eta^{(N)}-B_\eta^{(M)})J_{N;p},
           k_i\rangle_{\mathcal{H}}.
\end{align*}
In conjunction with \eqref{t.plucker.22}, this yields that
\[
    \J_{\!N}\,p=C(\boldsymbol{k}^{(M)})q_N^\dagger p
    \quad\text{for } p\in \mathbb{R}^M.
\]
Thus 
\[
    \det q_N=\frac{\det \J_{\!N}}{
         \det C(\boldsymbol{k}^{(M)})}.
\]
Plugging this into \eqref{t.plucker.21}, we arrive at
\eqref{t.plucker.1}.
\end{proof}

Given linearly independent $k_1,\dots,k_M\in \mathcal{H}$, 
$B_\eta$ with $\Lambda(B_\eta)<1$ has a trivial
decomposition satisfying the condition~{\bf (A)}.
In fact, denoting by $\pi^{(M)}$ the orthogonal projection
of $\mathcal{H}$ to the subspace spanned by 
$k_1,\dots,k_M$ as above, we set $A_I=B_\eta-\pi^{(M)}$ and
$A_F=\pi^{(M)}$.
Notice that $A_I\in\stwo$ and 
$\Lambda(A_I)\le \Lambda(B_\eta)<1$.
As was seen in the proof of Lemma~\ref{l.tw.(ii)to(iii)}, we
have that $\dettwo(I-A_I)\ne0$.
Obviously the range of $A_F$ is the subspace spanned by
$k_1,\dots,k_M$.
Thus the assumption {\bf(A)} is fulfilled with these $A_I$
and $A_F$, and Theorem~\ref{t.plucker} is applicable with
$k_1,\dots,k_M$.
Since $\tr A_F=M$, the identity \eqref{t.plucker.1} reads as 
\[
    \int_{\mathcal{W}} e^{\q_\eta}
      \delta_0(\D^*\boldsymbol{k}^{(N)})d\mu
    =\biggl(\frac{\det C(\boldsymbol{k}^{(M)})}{
         (2\pi)^N \det C(\boldsymbol{k}^{(N)})
         \det \J_{\!N}} \biggr)^{\frac12}
     \{\dettwo(I-B_\eta+\pi^{(M)})\}^{-\frac12}
     e^{-\frac{M}2}    
\]
for $0\le N\le M$.

\section{Examples}
\label{sec.example}
In this section, we give several examples for
Theorem~\ref{t.plucker}.
In the examples, the operator $A_I$ is a Volterra operator,
i.e., a compact operator with one-point spectrum $\{0\}$,
and we will encounter higher order linear ODEs.
In the following, we set 
\[
    \theta^{(N)}(t)=(\theta^1(t),\dots,\theta^N(t))
    \quad\text{for }0\le N\le d.
\]
Before getting into the examples, we prepare a lemma
corresponding to the condition~{\bf(A)}.
To state it, let 
\[
    \Delta_1=\{(t,s)\in[0,T]^2\mid t>s\}
    \quad\text{and}\quad
    \Delta_2=\{(t,s)\in[0,T]^2\mid t<s\}.
\]
For $\phi\in\ltwo$, define 
$\one_{\Delta_1}\phi,\one_{\Delta_2}\phi\in\ltwo$ as
\[
    (\one_{\Delta_i}\phi)(t,s)
    =\one_{\Delta_i}(t,s)\phi(t,s)
    \quad\text{for }(t,s)\in[0,T]^2
    \text{ and }i=1,2.
\]
As before, we set $\phi^*\in\ltwo$ so that
$\phi^*(t,s)=\phi(s,t)^\dagger$ for $(t,s)\in[0,T]^2$. 

\begin{lemma}\label{l.condA}
Take $\eta\in\stwo$ and $\phi\in\ltwo$ with 
$\one_{\Delta_2}\eta\eqltwo\one_{\Delta_2}\phi$.
Put $A_I=B_{\one_{\Delta_1}(\phi^*-\phi)}$ and $A_F=B_\phi$.
Then it holds that
\[
    \dettwo(I-A_I)=1
    \quad\text{and}\quad 
    B_\eta=A_I+A_F.
\]
In particular, if, in addition, $B_\phi$ is of finite rank,
then $B_\eta$ satisfies the condition~{\bf(A)}.
Furthermore, each $\kappa\in\ltwo$ with
$B_\eta=B_{\one_{\Delta_1}\kappa}+B_\phi$ satisfies that
$\one_{\Delta_1}\kappa\eqltwo\one_{\Delta_1}(\phi^*-\phi)$.
\end{lemma}

\begin{proof}
By Lemma~\ref{l.dettwo}, $\dettwo(I-A_I)=1$.
Since 
$\bigl(\one_{\Delta_1}\kappa\bigr)^*
 =\one_{\Delta_2}\kappa^*$
for any $\kappa\in\ltwo$, we have that
\[
    \one_{\Delta_1}\eta
    =\bigl(\one_{\Delta_2}\eta\bigr)^*
    \eqltwo \bigl(\one_{\Delta_2}\phi\bigr)^*
    =\one_{\Delta_1}\phi^*.
\]
Hence we see that
\[
    \eta\eqltwo \one_{\Delta_1}\eta+\one_{\Delta_2}\eta
    \eqltwo \one_{\Delta_1}\phi^*+\one_{\Delta_2}\phi
    \eqltwo \one_{\Delta_1}(\phi^*-\phi)+\phi.
\]
Thus the decomposition $B_\eta=A_I+A_F$ holds.

The last assertion is an immediate consequence of 
Lemma~\ref{l.b.kappa}, since 
$B_{\one_{\Delta_1}\kappa}=B_{\one_{\Delta_1}(\phi^*-\phi)}$ in that case.
\end{proof}

\begin{example}\label{e.vol.ode}
Let $\sigma\in C^1([0,T];\mathbb{R}^{d\times d})$ and define
$\eta\in\stwo$ as $\eta=\rho_\sigma$, that is,
\begin{equation}\label{e.vol.ode.0}
    \eta(t,s)=\one_{[0,t)}(s)\sigma(t)
       +\one_{(t,T]}(s) \sigma(s)^\dagger
    \quad\text{for }(t,s)\in[0,T]^2.
\end{equation}
Then we have that 
\[
    \q_\eta=\p_\sigma=\int_0^T \langle \sigma(t)\theta(t),
        d\theta(t)\rangle.
\]

Define $\kappa_F\in\ltwo$ as
$\kappa_F(t,s)=\sigma(s)^\dagger$
for $(t,s)\in[0,T]^2$.
By the definition \eqref{e.vol.ode.0} of $\eta$, we see that 
$\one_{\Delta_2}\eta=\one_{\Delta_2}\kappa_F$.
Let $e_1,\dots,e_d$ be an ONB of $\mathbb{R}^d$ and define
$k_j\in \mathcal{H}$ as $k_j^\prime=e_j$ for $1\le j\le d$.
Since 
\[
    (B_{\kappa_F}h)^\prime(t)
    =\int_0^T \sigma(s)^\dagger h^\prime(s)ds
    \quad\text{for }h\in \mathcal{H} 
    \text{ and }t\in[0,T],
\]
we see that
\[
    \mathcal{R}(B_{\kappa_F})
    \subset\Biggl\{\sum_{i=1}^d a_i k_i \,\Bigg|\,
       a_1,\dots,a_d\in \mathbb{R} \Biggr\}.
\]
Thus $B_{\kappa_F}$ is of finite rank.

Set $\kappa_I=\one_{\Delta_1}(\kappa_F^*-\kappa_F)$.
We then have that
\[
    \kappa_I(t,s)
      =\one_{[0,t)}(s)\{\sigma(t)-\sigma(s)^\dagger\}
    \quad\text{for }(t,s)\in[0,T]^2.
\]
By Lemma~\ref{l.condA}, $B_\eta$ satisfies the
condition~{\bf(A)} with $A_I=B_{\kappa_I}$ and
$A_F=B_{\kappa_F}$.

Assume that $\Lambda(B_\eta)<1$.
By Corollary~\ref{c.2nd.ode}, this condition is equivalent to
the non-singularity of the solution $\bS$ to the ODE 
\[
    \bS^{\prime\prime}-2\sigma_A\bS^\prime-\sigma^\prime\bS=0
    \quad \bS(T)=I_d,~\bS^\prime(T)=\sigma(T).
\]
Another equivalent condition in terms of $\sigma$ can be found in
Remark~\ref{r.riccati}.

To apply Theorem~\ref{t.plucker} to $\q_\eta$ with the above
$k_1,\dots,k_d$, we compute $J_p$ for 
$p=(p_1,\dots,p_d)\in \mathbb{R}^d$. 
Observe that
\[
    (A_Ih)^\prime(t)
    =\sigma(t)h(t)
     -\int_0^t \sigma(s)^\dagger h^\prime(s)ds
    \quad\text{for }h\in \mathcal{H}
    \text{ and }t\in[0,T].
\]
Since $(I-A_I)J_p=\sum\limits_{j=1}^d p_jk_j$, 
we have that
\[
    J_p^\prime(t)-\sigma(t)J_p(t)
    +\int_0^t \sigma(s)^\dagger J_p^\prime(s)ds
    =p
    \quad\text{for } t\in[0,T].
\]
This equation together with the initial condition $J_p(0)=0$ is
equivalent to the second order ODE on $\mathbb{R}^d$
\[
    J_p^{\prime\prime}-2\sigma_A J_p^\prime
    -\sigma^\prime J_p=0,
    \quad  J_p(0)=0,~~ J_p^\prime(0)=p,
    \quad\text{where }
    \sigma_A=\frac12(\sigma-\sigma^\dagger).
\]

Let $U\in C^2([0,T];\mathbb{R}^{d\times d})$ be the solution
to the second order ODE on $\mathbb{R}^{d\times d}$
\begin{equation}\label{e.vol.ode.1}
    U^{\prime\prime}-2\sigma_A U^\prime
    -\sigma^\prime U=0,
    \quad U(0)=0,~~U^\prime(0)=I_d.
\end{equation}
Then we have that
\[
    J_p=U(\cdot)p.
\]
By definition, it holds that
\[
    \J_d\, p
    =\bigl(\langle J_p,k_j\rangle_{\mathcal{H}}
        \bigr)_{1\le j\le d}
    =\bigl(\langle J_p(T),e_j\rangle
        \bigr)_{1\le j\le d}
    =U(T)p.
\]
Thus we obtain that
\[
    \J_d=U(T).
\]

Since $(I-\pi^{(d)})k_j=0$ for $1\le j\le d$, by definition,
\begin{align*}
    \J_0\, p
    & =\Bigl(\langle (I-\pi^{(d)}B_\eta)J_p,
           k_j\rangle_{\mathcal{H}}
        \Bigr)_{1\le j\le d}
      =\Bigl(\langle (I-B_\eta)J_p,k_j\rangle_{\mathcal{H}}
        \Bigr)_{1\le j\le d}
    \\
    & =\biggl(\biggl\langle \sum_{i=1}^d p_ik_i-A_FJ_p,
        k_j\biggr\rangle_{\mathcal{H}}
        \biggr)_{1\le j\le d}
      =Tp
       -T\biggl(\int_0^T \sigma(s)^\dagger 
        U^\prime(s) ds \biggr)p.
\end{align*}
Hence we see that
\[
     \J_0 = T(I_d-\Sigma),
     \quad\text{where }
     \Sigma=\int_0^T \sigma(s)^\dagger 
             U^\prime(s) ds.
\]

Define $\Phi\in \mathbb{R}^{2d\times d}$ by
\[
    \Phi=\begin{pmatrix} 
           U(T)
         \\
         T(I_d-\Sigma)
        \end{pmatrix}.
\]
Then each $\J_{\!N}$ is given by \eqref{r.plucker.1} with
this $\Phi$.

Observe that $C(\boldsymbol{k}^{(N)})=T^NI_N$ for 
$0\le N\le d$.
Furthermore, it holds that
\[
    \tr A_F
    =\frac1{T}
     \sum_{j=1}^d \langle A_F k_j,k_j\rangle_{\mathcal{H}}
    =\frac1{T}
     \sum_{j=1}^d \int_0^T \biggl\langle
       \int_0^T \sigma(s)^\dagger e_j ds,e_j\biggr\rangle dt
    =\int_0^T \tr\sigma(s)ds.
\]
Since $\D^*k_j=\theta^j(T)$, by Theorem~\ref{t.plucker}, we
obtain that 
\begin{equation}\label{e.vol.ode.2}
    \int_{\mathcal{W}} 
       e^{\int_0^T \langle\sigma(t)\theta(t),d\theta(t)\rangle}
       \delta_0(\theta^{(N)}(T)) d\mu
     =\biggl(\frac{T^{d-N}}{
        (2\pi)^N \det \J_{\!N}}\biggr)^{\frac12}
     e^{-\frac12 \int_0^T \tr\sigma(s)ds}
\end{equation}
for $0\le N\le d$.
\end{example}

\begin{example}\label{e.vol.ode.AD}
In this example, we consider a special case of
Example~\ref{e.vol.ode}:
the case when the ODE \eqref{e.vol.ode.1} has constant coefficients. 
Then the solution $U$ has an explicit expression. 

Let $C,D\in \mathbb{R}^{d\times d}$ and assume that
$D^\dagger=D$.
Define the quadratic Wiener functional
$\mathfrak{a}_{C,D}$ by 
\[
    \mathfrak{a}_{C,D}
    =\frac12 \int_0^T 
      \langle C \theta(t),d\theta(t)\rangle
     +\frac12\int_0^T
      \langle D\theta(t),\theta(t)\rangle dt.
\]
This Wiener functional is equal to
$\mathfrak{a}_{\phi,\psi}^x$ in Theorem~\ref{t.FK} with 
$x=0$, $\phi\equiv\frac12 C$, and $\psi\equiv D$.
Put 
\[
    \sigma(t)=\frac12 C+(T-t)D
    \quad\text{for }t\in[0,T].
\]
By \eqref{eq.a.p}, we know that
\begin{equation}\label{e.vol.ode.AD.1}
    \mathfrak{a}_{C,D}
    =\int_0^T \langle \sigma(t)\theta(t),
                   d\theta(t)\rangle
     +\frac{T^2}4 \tr D.
\end{equation}
The ODE \eqref{e.vol.ode.1} in Example~\ref{e.vol.ode} is 
rewritten as
\[
    U^{\prime\prime}-AU^\prime+DU=0,
    ~~U(0)=0,~~U^\prime(0)=I_d,
    \quad\text{where } A=\tfrac12(C-C^\dagger).
\]

Put
\[
    M_{A,D}
    =\begin{pmatrix} 0 & I_d \\ -D & A \end{pmatrix}
    \in \mathbb{R}^{2d\times 2d},
\]
and let $\lambda_1,\dots,\lambda_{2d}\in \mathbb{C}$ be its
eigenvalues counted with multiplicity.
Define $r_1,\dots,r_{2d}\in C^\infty(\mathbb{R};\mathbb{C})$ 
by the system of ODEs 
\[
    \left\{
      \begin{aligned}     
         & r_1^\prime=\lambda_1 r_1,
         \\
         & r_j^\prime=r_{j-1}+\lambda_j r_j
           \quad\text{for }2\le j\le 2d,
         \\
         & r_1(0)=1,~~
           r_2(0)=\dots=r_{2d}(0)=0,
      \end{aligned}
    \right.
\]
and complex matrices $Q_1,\dots,Q_{2d}\in \mathbb{C}^{d\times d}$
successively as
\[
    Q_1=0,\quad Q_2=I_d,
\]
and
\[
    Q_{3+n}
     =-D \biggl(\sum_{j=1}^{n+1} r_j^{(n)}(0) Q_j\biggr)
      +A \biggl(\sum_{j=1}^{n+2} r_j^{(n+1)}(0) Q_j\biggr)
      -\sum_{j=1}^{n+2} r_j^{(n+2)}(0) Q_j
\]
for $0\le n\le 2d-3$.
Due to Proposition~\ref{p.ode.U} for $k=2$ with $C_0=-D$,
$C_1=A$, $U_0=0$, and $U_1=I_d$, we know that $U$ is
expressed as  
\begin{equation}\label{e.vol.ode.AD.1+}
    U=\sum_{j=1}^{2d} r_j Q_j.
\end{equation}

Observe that $\Sigma$ in Example~\ref{e.vol.ode} is given as
\[
    \Sigma
    = \frac12 C^\dagger U(T)+D\int_0^T U(s)ds.
\]
The matrix $\J_{\!N}$ is given by \eqref{r.plucker.1} with
\[
    \Phi=\begin{pmatrix}
          U(T) \\ T(I_d-\Sigma)
        \end{pmatrix}.
\]

Since 
\[
    \int_0^T \tr\sigma(s) ds
    =\frac{T}2\tr C+\frac{T^2}2 \tr D,
\]
by \eqref{e.vol.ode.2} and \eqref{e.vol.ode.AD.1}, we obtain
that
\[
    \int_{\mathcal{W}} 
     e^{\frac12\int_0^T \langle C\theta(t),
            d\theta(t)\rangle
        +\frac12\int_0^T \langle D\theta(t),
             \theta(t)\rangle dt}
     \delta_0(\theta^{(N)}(T)) d\mu
    =\biggl(\frac{T^{d-N}}{
        (2\pi)^N \det \J_{\!N}}\biggr)^{\frac12}
         e^{-\frac{T}4\tr C}
\]
for $0\le N\le d$.

Suppose that $M_{A,D}$ has distinct $2d$ eigenvalues
$\nu_1,\dots,\nu_{2d}$.
Define 
$\widehat{Q}_1,\dots \widehat{Q}_{2d}
 \in \mathbb{C}^{d\times d}$
in the following two steps.
(i)~Define $\widehat{U}_n\in \mathbb{R}^{d\times d}$ for 
$n\in \mathbb{N}\cup\{0\}$ successively as
\[
    \widehat{U}_0=0,\quad
    \widehat{U}_1=I_d,\quad
    \widehat{U}_{2+n}=-D \widehat{U}_n
     +A \widehat{U}_{n+1}
    \quad\text{for }0\le n\le 2d-3.
\]
(ii)~Define each $(p,q)$-component $\widehat{Q}_j^{pq}$ of
$\widehat{Q}_j$ as 
\[
    \begin{pmatrix}
     \widehat{Q}_1^{pq} \\ \vdots \\ \vdots \\ 
     \widehat{Q}_{2d}^{pq}
    \end{pmatrix}
    ={\begin{pmatrix}
      1 & \dots & 1
      \\
      \nu_1 & \dots & \nu_{2d}
      \\
      \nu_1^2 & \dots & \nu_{2d}^2
      \\
      \vdots & & \vdots
      \\
      \nu_1^{2d-1} & \dots & \nu_{2d}^{2d-1}
     \end{pmatrix}}^{-1}
     \begin{pmatrix}
     \widehat{U}_0^{pq} \\ \vdots \\ \vdots \\ 
     \widehat{U}_{2d-1}^{pq}
     \end{pmatrix}.
\]
Due to Proposition~\ref{p.ode.U.nu} for $k=2$ with $C_0=-D$,
$C_1=A$, $U_0=0$, $U_1=I_d$, we have another expression of
$U$ as 
\[
    U(t)=\sum_{j=1}^{2d} e^{\nu_j t} \widehat{Q}_j
    \quad\text{for }t\in[0,T].
\]
Furthermore, in this case, we have that
\[
    \Sigma=\sum_{j=1}^d \biggl\{
     \frac12 e^{\nu_j T} C^\dagger
     +\frac{e^{\nu_j T}-1}{\nu_j} D\biggr\} \widehat{Q}_j,
   \quad\text{where }\frac{e^{0T}-1}{0}=1.
\]
\end{example}

The ODEs in the above two examples are of second order.
Such ODEs has already appeared in the previous chapter.
In the following examples, we deal with examples where ODEs
of order more than or equal to three appear.

\begin{example}\label{e.vol.ode.AD.D}
We continue to use the notation in Examples~\ref{e.vol.ode}
and \ref{e.vol.ode.AD}. 
By \eqref{e.vol.ode.0} and \eqref{e.vol.ode.AD.1},
$\D \mathfrak{a}_{C,D}=\D\q_\eta$.
Due to Proposition~\ref{p.h.eta}, we have that
\begin{equation}\label{e.vol.ode.AD.D.01}
    \frac12\|\D \mathfrak{a}_{C,D}\|_{\mathcal{H}}^2
    =\q_{c(\eta)}+\frac12\|\eta\|_2^2.
\end{equation}
We apply Theorem~\ref{t.plucker} to 
$-\frac12\|D \mathfrak{a}_{C,D}\|_{\mathcal{H}}^2
 =\q_{-c(\eta)}-\frac12\|\eta\|_2^2$.

We first see that Theorem~\ref{t.plucker} is applicable to  
$B_{-c(\eta)}$.
To do so, develop $B_{-c(\eta)}$ as
\[
    B_{-c(\eta)}=-B_{c(\eta)}=-B_\eta^2
    =-A_I^2+\Bigl((-A_I)A_F+(-A_F)(A_I+A_F)\Bigr).
\]
Recall that $A_I=B_{\kappa_I}$, and hence
\[
    A_I^2=B_{\kappa_I*\kappa_I},
\]
where
\[
    (\kappa_I*\kappa_I)(t,s)
    =\int_0^T \kappa_I(t,u)\kappa_I(u,s) du
    \quad\text{for }(t,s)\in[0,T]^2.
\]
Since 
\[
    \kappa_I(t,s)=\one_{[0,t)}(s)\{A-(t-s)D\}
    \quad\text{for }(t,s)\in[0,T]^2,
\]
we know that
\[
    (\kappa_I*\kappa_I)(t,s)
    = \one_{[0,t)}(s)\biggl\{
       (t-s)A^2-\frac{(t-s)^2}2(DA+AD)
       +\frac{(t-s)^3}6 D^2 \biggr\}.
\]
Due to Lemma~\ref{l.dettwo}, we obtain that 
\begin{equation}\label{e.vol.ode.AD.D.02}
    \tr[A_I^2]=0
    \quad\text{and}\quad
    \dettwo(I+A_I^2)=1.
\end{equation}
Let $e_1,\dots,e_d$ be an ONB of $\mathbb{R}^d$ and define 
$k_1,\dots,k_{3d}\in \mathcal{H}$ by 
\[
    k_i^\prime(t)=e_i,\quad
    k_{d+i}^\prime(t)=\biggl(t-\frac{T}2\biggr)e_i,\quad
    k_{2d+i}^\prime(t)
    =\biggl(t^2-Tt+\frac{T^2}6\biggr)e_i
\]
for $t\in[0,T]$ and $1\le i\le d$.
Then $k_1,\dots,k_{3d}$ are orthogonal.
As was seen in the previous example, we have that
\[
    \mathcal{R}(A_F)
    \subset
    \biggl\{ \sum_{j=1}^d a_jk_j \,\bigg|\, 
      a_1,\dots,a_d\in \mathbb{R} \biggr\}.
\]
Since
\[
    (A_Ik_i)^\prime(t)
    =\biggl(tA-\frac{t^2}2D\biggr)e_i
    \quad\text{for }t\in[0,T]\text{ and }1\le i\le d,
\]
we see that
\[
    \mathcal{R}\Bigl((-A_I)A_F+(-A_F)(A_I+A_F)\Bigr)
    \subset
    \biggl\{\sum_{i=1}^{3d} a_i k_i \,\bigg|\,
         a_1,\dots,a_{3d}\in \mathbb{R} \biggr\}.
\]
Thus $B_{-c(\eta)}$ satisfies the condition {\bf(A)} with
$-A_I^2$ for $A_I$ and $(-A_I)A_F+(-A_F)(A_I+A_F)$ for
$A_F$.
Since $B_{c(\eta)}\in \mathcal{S}_+(\mathcal{H})$, 
$\Lambda(B_{-c(\eta)})\le0<1$.
Therefore Theorem~\ref{t.plucker} is applicable to
$\q_{-c(\eta)}$ with $k_1,\dots,k_{3d}$.

We compute $J_p$ associated with $\q_{-c(\eta)}$ for 
$p=(p^{(1)},p^{(2)},p^{(3)})\in 
 (\mathbb{R}^d)^3=\mathbb{R}^{3d}$.
Observe that
\[
    (A_Ih)^\prime(t)=\int_0^t [Ah^\prime-Dh](s)ds
\]
and 
\[
    (A_I^2 h)^\prime(t)
      =\int_0^t\biggl\{ A\biggl(
        \int_0^s[Ah^\prime-Dh](u)du\biggr)
       -D\biggl(\int_0^s\biggl(\int_0^u 
           [Ah^\prime-Dh](v)dv\biggr)du \biggr)
        \biggr\} ds
\]
for $h\in \mathcal{H}$ and $t\in[0,T]$.
Since $(I+A_I^2)J_p=\sum\limits_{i=1}^{3d} p_ik_i$, where
$p_1,\dots,p_{3d}$ is determined as
$p^{(i)}=(p_{(i-1)d+1},\dots,p_{id})$ for $i=1,2,3$,
$J_p$ is characterized by the equation 
\begin{align*}
    & J_p^{(1)}(t)+\int_0^t \biggl\{A\biggl(
        \int_0^s[AJ_p^{(1)}-DJ_p](u)du\biggr)
    \\
    & \hphantom{J_p^{(1)}(t)-\int_0^t \biggl\{}
     -D\biggl(\int_0^s\biggl(\int_0^u 
          [AJ_p^{(1)}-DJ_p](v)dv\biggr) du \biggr)
      \biggr\} ds
    \\
    & =p^{(1)}+\Bigl(t-\frac{T}2\Bigr)p^{(2)}
       +\Bigl(t^2-Tt+\frac{T^2}6\Bigr)p^{(3)}
    \quad\text{for }t\in[0,T].
\end{align*}
If we set 
\[
    \widehat{J}_0(p)=0,
      \quad
      \widehat{J}_1(p)
      =p^{(1)}-\frac{T}2 p^{(2)}+\frac{T^2}6 p^{(3)},
      \quad
      \widehat{J}_2(p)=p^{(2)}-T p^{(3)},
\]
and
\[
    \widehat{J}_3(p)=-A^2\Bigl(p^{(1)}-\frac{T}2 p^{(2)}
               +\frac{T^2}6 p^{(3)}\Bigr)
            +2p^{(3)},
\]
then this equation with the initial condition $J_p(0)=0$ is 
equivalent to the fourth order ODE on $\mathbb{R}^d$  
\[
    J_p^{(4)}+A^2J_p^{(2)}-(AD+DA)J_p^{(1)}+D^2J_p=0,
    \quad
    J_p^{(i)}(0)=\widehat{J}_i(p)
    \quad\text{for }0\le i\le 3.
\]

Let $\lambda_1,\dots,\lambda_{4d}\in \mathbb{C}$ be the
eigenvalues of the matrix
\[
    \hat{M}_{A,D}
    \equiv 
     \begin{pmatrix}
       0 & I_d & 0 & 0
       \\
       0 & 0 & I_d & 0 
       \\
       0 & 0 & 0& I_d 
       \\
       -D^2 & AD+DA & -A^2 & 0
     \end{pmatrix} \in \mathbb{R}^{4d\times 4d}
\]
counted with multiplicity.
Define 
$r_1,\dots,r_{4d}\in C^\infty(\mathbb{R};\mathbb{C})$ as 
\[
    \left\{
      \begin{aligned}     
         & r_1^\prime=\lambda_1 r_1,
         \\
         & r_j^\prime=r_{j-1}+\lambda_j r_j
           \quad\text{for }2\le j\le 4d,
         \\
         & r_1(0)=1,~~
           r_2(0)=\dots=r_{4d}(0)=0,
      \end{aligned}
    \right.
\]
and the multi-linear mappings 
$q_1,\dots q_{4d}:(\mathbb{C}^d)^3\to \mathbb{C}^d$ 
as
\[
    q_1(p)=0,\quad
      q_2(p)=\widehat{J}_1(p), \quad
      q_{m+1}(p)
      =\widehat{J}_m(p)-\sum_{j=1}^m r_j^{(m)}(0)q_j(p)
       \quad\text{for }2\le m\le 3,
\]
and
\begin{align*}
    q_{5+n}(p)
    = & -D^2 \biggl(\sum_{j=1}^{n+1} r_j^{(n)}(0)q_j(p)\biggr)
      +(AD+DA) \bigg(\sum_{j=1}^{n+2} r_j^{(n+1)}(0)q_j(p)\biggr)
    \\
    & -A^2 \biggl(\sum_{j=1}^{n+3} r_j^{(n+2)}(0)q_j(p)\biggr)
      -\sum_{j=1}^{n+4} r_j^{(n+4)}(0)q_j(p)
      \quad\text{for }0\le n\le 4d-5.
\end{align*}
By Corollay~\ref{c.ode.U} for $k=4$ with $C_0=-D^2$,
$C_1=AD+DA$, $C_2=-A^2$, and $C_3=0$, $J_p$ is represented as
\[
    J_p=\sum_{j=1}^{4d} r_j q_j(p).
\]

Using this $J_p$, we define the matrices
$\J_0,\J_{3d}\in \mathbb{R}^{3d\times 3d}$ and 
$\Phi\in \mathbb{R}^{6d\times 3d}$ by 
\[
    \J_0\, p
    =\begin{pmatrix}
      \langle J_p,(I+B_{c(\eta)})k_1\rangle_{\mathcal{H}}
      \\
      \vdots
      \\
      \langle J_p,(I+B_{c(\eta)})k_{3d}\rangle_{\mathcal{H}}
     \end{pmatrix},
    \quad 
    \J_{3d}\,p
    =\begin{pmatrix}
      \langle J_p,k_1\rangle_{\mathcal{H}}
      \\
      \vdots
      \\
      \langle J_p,k_{3d}\rangle_{\mathcal{H}}
     \end{pmatrix}
    \quad\text{for }p\in \mathbb{R}^{3d}
\]
and 
\[
    \Phi=\begin{pmatrix} \J_{3d} \\ \J_0\end{pmatrix}.
\]
The matrix $\J_{\!N}$ is given by \eqref{r.plucker.1} with
this $\Phi$.

The above 
$\langle J_p,k_i\rangle_{\mathcal{H}}$s are computed as
\[
    \langle J_p,k_i\rangle_{\mathcal{H}}
     = \langle J_p(T),e_i\rangle,
\]
\[
    \langle J_p,k_{d+i}\rangle_{\mathcal{H}}
      =\biggl\langle -\frac{T}2 J_p(T)
       +\int_0^T t J_p^{(1)}(t) dt,e_i\biggr\rangle,
\]
and
\[
    \langle J_p,k_{2d+i}\rangle_{\mathcal{H}}
      = \biggl\langle -\frac{T^2}6 J_p(T)
        +\int_0^T (t^2-Tt)J_p^{(1)}(t)dt,
       e_i\biggr\rangle.
\]
Furthermore, 
$\langle J_p,(I+B_{c(\eta)})k_i\rangle_{\mathcal{H}}$s 
are also computed as follows.
For $e\in \mathbb{R}^d$ and $n\in \{0,1,2\}$, 
define $h_{e;n}\in \mathcal{H}$ by 
$h_{e;n}^\prime(t)=t^n e$ for $t\in[0,T]$. 
Determine $K_n,P_n(t)\in \mathbb{R}^{d\times d}$ as
\[
    K_n= \int_0^T s^n\sigma(s)^\dagger ds
       =\frac{T^{n+1}}{2(n+1)}C^\dagger
        +\frac{T^{n+2} n!}{(n+2)!}D
\]
and
\begin{align*}
    P_n(t) = 
       & \frac{n!t^{n+2}}{(n+2)!} A^2
        -\frac{n!t^{n+3}}{(n+3)!} (AD+DA)
        +\frac{n!t^{n+4}}{(n+4)!} D^2
    \\
    &  +\frac1{n+1} K_{n+1}A
        -\frac{n!}{(n+2)!} K_{n+2}D
        +\biggl(tA-\frac{t^2}2 D+K_0\biggr) K_n.
\end{align*}
Since 
\[
    (B_\eta h)^\prime(t)
    =\int_0^t [Ah^\prime-Dh](s)ds
     +\int_0^T \sigma(s)^\dagger h^\prime(s) ds
    \quad\text{for $h\in \mathcal{H}$ and $t\in[0,T]$,}
\]
by a straightforward
computation, we obtain that
\[
    (B_{c(\eta)}h_{e;n})^\prime(t) 
    =P_n(t)e \quad\text{for }t\in[0,T].
\]
Rewriting as
\[
    k_i=h_{e_i;0},\quad
    k_{d+i}=h_{e_i;1}-\frac{T}2h_{e_i;0},\quad
    k_{2d+i}=h_{e_i;2}-T h_{e_i;1}+\frac{T^2}6 h_{e_i;0}
    \quad\text{for }1\le i\le d, 
\]
in conjunction with the above evaluation
of $\langle J_p,k_i\rangle_{\mathcal{H}}$s, we obtain that
\[
    \langle J_p,(I+B_{c(\eta)})k_i\rangle_{\mathcal{H}}
      =\biggl\langle J_p(T)
          +\int_0^T P_0(t)^\dagger J_p^{(1)}(t) dt,
        e_i\biggr\rangle,
\]
\[
    \langle J_p,(I+B_{c(\eta)})k_{d+i}\rangle_{\mathcal{H}}
      =\biggl\langle -\frac{T}2 J_p(T)
        +\int_0^T \Bigl(t I_d +P_1(t)^\dagger
            -\frac{T}2 P_0(t)^\dagger \Bigr)J_p^{(1)}(t)
         dt, e_i\biggr\rangle,
\]
and
\begin{align*}
    & \langle
      J_p,(I+B_{c(\eta)})k_{2d+i}\rangle_{\mathcal{H}}
    \\
    &  =\biggl\langle -\frac{T^2}6 J_p(T)
       +\int_0^T \Bigl((t^2-tT)I_d
            +P_2(t)^\dagger-T P_1(t)^\dagger
            +\frac{T^2}6 P_0(t)^\dagger\Bigr)J_p^{(1)}(t)
        dt, e_i \biggr\rangle
\end{align*}
for $1\le i\le d$.

We compute the remaining quantities in
Theorem~\ref{t.plucker}.
By a direct computation, we know that
\[
    \|k_i\|_{\mathcal{H}}^2=T, \quad
    \|k_{d+i}\|_{\mathcal{H}}^2=\frac{T^3}{12}, \quad
    \|k_{2d+i}\|_{\mathcal{H}}^2=\frac{T^5}{180}
    \quad\text{for }1\le i\le d.
\]
Hence we have that
\begin{equation}\label{e.vol.ode.AD.D.03}
    \det C(\boldsymbol{k}^{(N)})=T^N
    \quad\text{for }0\le N\le d
    \quad\text{and}\quad
    \det\bigl(C(\boldsymbol{k}^{(3d)})\bigr)
    =\frac{T^{9d}}{12^{d}180^{d}}.
\end{equation}
By \eqref{e.vol.ode.AD.D.02}, we know that
\begin{align*}
    & -\|\eta\|_2^2
      -\tr\Bigl[(-A_I)A_F+(-A_F)(A_I+A_F)\Bigr]
    \\
    & =-\|B_\eta\|_{\htwo}^2
       -\tr\Bigl[(-A_I)A_F+(-A_F)(A_I+A_F)\Bigr]
    \\
    & =\tr\Bigl[-B_\eta^2+A_IA_F+A_F(A_I+A_F)\Bigr]
      =\tr[-A_I^2]=0.
\end{align*}
Applying Theorem~\ref{t.plucker} to $\q_{-c(\eta)}$ 
together with this, \eqref{e.vol.ode.AD.D.01},
\eqref{e.vol.ode.AD.D.02}, and \eqref{e.vol.ode.AD.D.03}, 
we obtain that 
\[
    \int_{\mathcal{W}} 
       e^{-\frac12\|\D \mathfrak{a}_{C,D}\|_{\mathcal{H}}^2}
     \delta_0(\theta^{(N)}(T))d\mu
    =\biggl( \frac{T^{9d-N}}{
        (2\pi)^N 12^d 180^d \det \J_{\!N}} 
     \biggr)^{\frac12}
    \quad\text{for }0\le N\le d.
\]

Define $\hat{k}_{d+1},\dots,\hat{k}_{2d}\in \mathcal{H}$ by
$\hat{k}_{d+i}^\prime(t)=(t^2-\frac{T^2}3)e_i$ for
$1\le i\le d$ and $t\in[0,T]$.
If $A=0$, which always occurs when $d=1$, we can carry out
the above program with 
$k_1,\dots,k_d,\hat{k}_{d+1},\dots,\hat{k}_{2d}$
instead of $k_1,\dots,k_{3d}$.
We omit the details.
\end{example}

\begin{example}\label{e.samp.var}
In the previous examples, we encountered second and fourth
order ODEs.
In this example, we deal with the case when a third order 
ODE appears.

Let $D\in \mathbb{R}^{d\times d}$ and assume that
$D^\dagger=D$ and  
\[
    \Lambda(-D)
    \equiv\sup_{x\in \mathbb{R}^d,|x|=1}
        \langle (-D)x,x\rangle 
    <1.
\]
Define the quadratic Wiener functional $\mathfrak{v}_D$ by 
\[
    \mathfrak{v}_D=\frac12\int_0^T 
     \bigl\langle D(\theta(t)-\overline{\theta}),
          \theta(t)-\overline{\theta}\bigr\rangle dt,
    \quad\text{where }
    \overline{\theta}=\frac1{T} \int_0^T \theta(t)dt.
\]
When $d=1$ and $D=1$, then $\frac{2}{T}\mathfrak{v}_D$ is 
the sample variance of sample path, which played a key role
to dominate the inverse of Malliavin covariance in the
early days of Malliavin calculus. 
See, for example, the first edition of Ikeda and
Watanabe's book (\cite{IW}). 

As an easy exercise of Malliavin calculus, we know that
\[
    \langle \D \mathfrak{v}_D,g\rangle_{\mathcal{H}}
    =\int_0^T \langle D(\theta(t)-\overline{\theta}),
          g(t)-\overline{g}\rangle dt
    \quad\text{for }g\in \mathcal{H}.
\]
Hence 
\[
    \int_{\mathcal{W}} \D \mathfrak{v}_D d\mu=0.
\]
Taking the $\mathcal{H}$-derivative of the above, and
substituting the identity 
$\int_0^T (h(t)-\overline{h})dt=0$ into the resulting
identity, we see that
\begin{align}
 \label{e.samp.var.1}
    \langle \D^2\mathfrak{v}_D,
        h\otimes g\rangle_{\htwo}
      & =\int_0^T \langle D(h(t)-\overline{h}),
             g(t)-\overline{g}\rangle dt
    \\
    & =\int_0^T \langle D(h(t)-\overline{h}),
             g(t)\rangle dt
 \nonumber
    \\
    & =\int_0^T \biggl\langle 
           D\biggl(\int_t^T (h(s)-\overline{h}) ds
             \biggr), g^\prime(t)\biggr\rangle dt
 \nonumber
    \\
    & =\int_0^T \biggl\langle 
           -D\biggl(\int_0^t (h(s)-\overline{h}) ds
             \biggr), g^\prime(t)\biggr\rangle dt
    \quad\text{for }h,g\in \mathcal{H}.
 \nonumber
\end{align}
Setting $\eta\in\stwo$ as
\begin{equation}\label{e.samp.var.eta}
    \eta(t,s)=\biggl(-\one_{[0,t)}(s)(t-s)
      +\frac{t}{T}(T-s)\biggr)D
    =\biggl(\min\{t,s\}-\frac{ts}{T}\biggr)D
\end{equation}
for $(t,s)\in[0,T]^2$, we have that
\begin{equation}\label{e.samp.var.1+}
    \int_0^T \eta(t,s)h^\prime(s)ds
    =-D\biggl(\int_0^t (h(s)-\overline{h}) ds\biggr)
    \quad\text{for }h\in \mathcal{H}.
\end{equation}
Thus it holds that 
\[
    \D^2 \mathfrak{v}_D=B_\eta.
\]
Furthermore, it is easily seen that
\[
    \int_{\mathcal{W}} \mathfrak{v}_D d\mu
    =\frac{T^2}{12} \tr D.
\]
Due to Lemma~\ref{l.D^3=0} and Theorem~\ref{t.q.eta}, we
obtain that  
\begin{equation}\label{e.samp.var.2}
    -\mathfrak{v}_D=\q_{-\eta}-\frac{T^2}{12} \tr D.
\end{equation}

Define $\kappa_F\in\ltwo$ as 
\[
    \kappa_F(t,s)=-\frac{t}{T}(T-s)D
    \quad\text{for }(t,s)\in[0,T]^2.
\]
By the definition \eqref{e.samp.var.eta} of $\eta$, we have
that $\one_{\Delta_2}(-\eta)=\one_{\Delta_2}\kappa_F$.
Let $e_1,\dots,e_d$ be an ONB of $\mathbb{R}^d$ and define 
orthogonal $k_1,\dots,k_{2d}\in \mathcal{H}$ by 
\[
    k_i^\prime(t)=e_i 
    \quad\text{and}\quad
    k_{d+i}^\prime(t)=\biggl(t-\frac{T}2\biggr)e_i
    \quad\text{for }t\in[0,T]
    \text{ and }1\le i\le d.
\]
Since
\[
    (B_{\kappa_F}h)^\prime(t)
      =\int_0^T\biggl(-\frac{t}{T}(T-s)D\biggr)h^\prime(s)ds
      =-tD\overline{h}
    \quad\text{for }h\in \mathcal{H}
    \text{ and }t\in[0,T],    
\]
we have that
\[
    \mathcal{R}(B_{\kappa_F})
    \subset \biggl\{ \sum_{i=1}^{2d} a_ik_i \,\bigg|\,
       a_1,\dots,a_{2d}\in \mathbb{R}\biggr\}.
\]
Thus $B_{\kappa_F}$ is of finite rank.
Set $\kappa_I\in\ltwo$ so that
\[
    \kappa_I(t,s)=\one_{[0,t)}(s)(t-s)D
    \quad\text{for }(t,s)\in[0,T]^2.
\]
By \eqref{e.samp.var.eta} and Lemma~\ref{l.condA}, we see
that $B_{-\eta}$ satisfies the condition~{\bf(A)} with
$A_I=B_{\kappa_I}$ and $A_F=B_{\kappa_F}$.
Moreover, by the first identity in \eqref{e.samp.var.1}, we
see that 
\begin{align*}
    \Lambda(B_{-\eta})
    \le \Lambda(-D)\sup_{\|h\|_{\mathcal{H}}=1}
        \int_0^T |h(t)-\overline{h}|^2 dt
    \le \Lambda(-D)<1.
\end{align*}
Therefore Theorem~\ref{t.plucker} is applicable to
$\q_{-\eta}$ with $k_1,\dots,k_{2d}$.

We compute $J_p$ associated with $\q_{-\eta}$ for 
$p=(p^{(1)},p^{(2)})\in (\mathbb{R}^d)^2=\mathbb{R}^{2d}$.
Note that 
\[
    (A_Ih)^\prime(t)
    =\int_0^T \one_{[0,t)}(s)(t-s)Dh^\prime(s)ds
    =\int_0^t Dh(s)ds
    \quad\text{for }h\in \mathcal{H} 
    \text{ and }t\in[0,T].   
\]
Since $(I-A_I)J_p=\sum\limits_{i=1}^{2d} p_ik_i$,
where $p_1,\dots,p_{2d}$ is determined as
$p^{(i)}=(p_{(i-1)d+1},\dots,p_{id})$ for $i=1,2$, we have that
\[
    J_p^{(1)}(t)-\int_0^t DJ_p(s)ds
    =p^{(1)}+\Bigl(t-\frac{T}2\Bigr)p^{(2)}
    \quad\text{for }t\in[0,T].
\]
If we set
\[
    \widehat{J}_0(p)=0,\quad
    \widehat{J}_1(p)=p^{(1)}-\frac{T}2 p^{(2)},\quad
    \widehat{J}_2(p)=p^{(2)},
\]
then this identity with the initial condition $J_p(0)=0$ is
equivalent to the third order ODE on $\mathbb{R}^d$ 
\[
    J_p^{(3)}-DJ_p^{(1)}=0,
    \quad
    J_p^{(i)}(0)=\widehat{J}_i(p)
    \quad\text{for }0\le i\le 2.
\]

We solve this ODE as a second order ODE for
$J_p^{(1)}$ and integrate it to have $J_p$.
To do so, let $\lambda_1,\dots,\lambda_{2d}\in \mathbb{C}$
be the eigenvalues of the matrix
\[
    M_D=\begin{pmatrix}
          0 & I_d \\ D & 0 
        \end{pmatrix}
    \in \mathbb{R}^{2d\times 2d}
\]
counted with multiplicity.
Define 
$r_1,\dots,r_{2d}\in C^\infty(\mathbb{R};\mathbb{C})$ as
\[
    \left\{
      \begin{aligned}     
         & r_1^\prime=\lambda_1 r_1,
         \\
         & r_j^\prime=r_{j-1}+\lambda_j r_j
           \quad\text{for }2\le j\le 2d,
         \\
         & r_1(0)=1,~~
           r_2(0)=\dots=r_{2d}(0)=0,
      \end{aligned}
    \right.
\]
and the multi-linear mappings 
$q_1,\dots,q_{2d}:(\mathbb{C}^d)^2\to \mathbb{C}^d$ as
\[
    q_1(p)=\widehat{J}_1(p), \quad
      q_2(p)=\widehat{J}_2(p)-\lambda_1 \widehat{J}_1(p),
\]
and
\[
    q_{3+n}(p)
      =D\biggl(\sum_{j=1}^{n+1} r_j^{(n)}(0)q_j(p)\biggr)
       -\sum_{j=1}^{n+2} r_j^{(n+2)}(0) q_j(p)
      \quad\text{for }0\le n\le 2d-3.
\]
By Corollary~\ref{c.ode.U} for  $k=2$ with $C_0=D$ and $C_1=0$,
we have that 
\[
    J_p(t)=\sum_{j=1}^{2d} 
       \biggl(\int_0^t r_j(s)ds\biggr) q_j(p)
    \quad\text{for }t\in[0,T].
\]

Using this $J_p$, we define the matrices
$\J_{0},\J_{2d}\in \mathbb{R}^{2d\times 2d}$ by
\[
    \J_{0}\, p
    =\begin{pmatrix}
      \langle J_p,(I+B_\eta)k_1\rangle_{\mathcal{H}}
      \\
      \vdots
      \\
      \langle J_p,(I+B_\eta)k_{2d}\rangle_{\mathcal{H}}
     \end{pmatrix}
    \quad\text{and}\quad
    \J_{2d}\,p
    =\begin{pmatrix}
      \langle J_p,k_1\rangle_{\mathcal{H}}
      \\
      \vdots
      \\
      \langle J_p,k_{2d}\rangle_{\mathcal{H}}
     \end{pmatrix}
    \quad\text{for }p\in \mathbb{R}^{2d}.
\]
The matrix $\J_{\!N}$ are determined by \eqref{r.plucker.1}
with  
\[
    \Phi=\begin{pmatrix} \J_{2d} \\ \J_{0}\end{pmatrix}
    \in \mathbb{R}^{4d\times 2d}.
\]

The above $\langle J_p,k_i\rangle_{\mathcal{H}}$s and
$\langle J_p,(I+B_\eta)k_i\rangle_{\mathcal{H}}$s are
computed as follows.
It follows from \eqref{e.samp.var.1+} that
\[
    (B_\eta k_i)^\prime(t)=\frac{t(T-t)}2 De_i
    \quad\text{and}\quad
    (B_\eta k_{d+i})^\prime(t)
    =-\biggl(\frac{t^3}6-\frac{T}4 t^2+\frac{T^2}{12} t
      \biggr) De_i.
\]
Hence we have that
\[
    \langle J_p,k_i\rangle_{\mathcal{H}}
      =\langle J_p(T),e_i\rangle,
\]
\[
   \langle J_p,k_{d+i}\rangle_{\mathcal{H}}
      =\biggl\langle -\frac{T}2 J_p(T)
       +\int_0^T t J_p^{(1)}(t) dt, e_i\biggr\rangle,
\]
\[
    \langle J_p,(I+B_\eta)k_i\rangle_{\mathcal{H}}
      =\biggl\langle J_p(T)
       +\int_0^T \frac{t(T-t)}2 D J_p^{(1)}(t) dt,
        e_i\biggr\rangle,
\]
and
\[
    \langle J_p,(I+B_\eta)k_{d+i}\rangle_{\mathcal{H}}
      =\biggl\langle -\frac{T}2 J_p(T)
       +\int_0^T \Bigl\{ tI_d-\Bigl(\frac{t^3}6
           -\frac{T}4 t^2+\frac{T^2}{12} t \Bigr)D
           \Bigr\} J_p^{(1)}(t) dt,e_i\biggr\rangle
\]
for $1\le i\le d$.

We compute the remaining quantities in Theorem~\ref{t.plucker}.
As was seen in the previous example, we have that
\[
    \|k_i\|_{\mathcal{H}}^2=T
    \quad\text{and}\quad
    \|k_{d+i}\|_{\mathcal{H}}^2=\frac{T^3}{12}
    \quad\text{for }1\le i\le d.
\]
Hence we obtain that
\begin{equation}\label{e.samp.var.6}
    \det\bigl(C(\boldsymbol{k}^{(N)})\bigr)
    =T^N
    \quad\text{for }0\le N\le d
    \quad\text{and}\quad
    \det C(\boldsymbol{k}^{(2d)})
    =\frac{T^{4d}}{12^d}.
\end{equation}
By Lemma~\ref{l.trace.detail}, we have that 
\[
    -\frac12 \tr A_F 
    =-\frac12\int_0^T \tr[\kappa_F(t,t)] dt
    =\frac12 \tr\biggl[\int_0^T 
        \biggl(\frac{t}{T}(T-t)D\biggr) dt\biggr]
    =\frac{T^2}{12} \tr D.
\]

Due to Theorem~\ref{t.plucker}, Lemma~\ref{l.condA}, and the
identities \eqref{e.samp.var.2} and \eqref{e.samp.var.6}, we
obtain that 
\[
    \int_{\mathcal{W}} e^{-\mathfrak{v}_D}
     \delta_0(\theta^{(N)}(T))d\mu
    =\biggl(\frac{T^{4d-N}}{
       (2\pi)^N 12^d\det \J_{\!N}} \biggr)^{\frac12}
    \quad\text{for }0\le N\le d.
\]
\end{example}

\begin{example}\label{e.ikm}
In this example, we consider the case where ODEs of 
order higher than $4$ appear.

Let $\mathcal{I}:\mathcal{W}\to \mathcal{W}$ be the integral
operator such that
\[
    \mathcal{I}[w](t)=\int_0^t w(s)ds
    \quad\text{for }t\in[0,T]
    \text{ and }w\in \mathcal{W}.
\]
Fix an $N\in \mathbb{N}$ and define the stochastic process 
$\{X_N(t)\}_{t\in[0,T]}$ by
\[
    X_N(t)=\mathcal{I}^N[\theta](t)
    \quad\text{for }t\in[0,T].
\]

Define $\gamma_N:[0,T]^2\to \mathbb{R}$ as 
\[
    \gamma_N(t,s)
    =\one_{[0,t)}(s)\frac{(t-s)^N}{N!}
    \quad\text{for }(t,s)\in[0,T]^2.
\]
Exchanging the order of integration, we easily see that
\begin{equation}\label{e.ikm.1}
    \mathcal{I}^N[h](t)
    =\int_0^T \gamma_N(t,s)h^\prime(s)ds
    \quad\text{for }h\in \mathcal{H} 
    \text{ and }t\in[0,T].
\end{equation}
Approximating $\{\theta(t)\}_{t\in[0,T]}$ 
piecewise-linearly, we obtain from this that 
\begin{equation}\label{e.ikm.2}
    X_N(t)=\int_0^T \gamma_N(t,s)d\theta(s).
\end{equation}
Hence $\{X_N(t)\}_{t\in[0,T]}$ is a Gaussian process.
When $d=2$, the stochastic area surrounded by such a
Gaussian process was studied by Ikeda, Kusuoka, and
Manabe (\cite{IKM1}). 

For $\sigma\in C([0,T];\mathbb{R}^{d\times d})$, we consider 
the Wiener functional $\q:\mathcal{W}\to \mathbb{R}$ by
\[
    \q=\int_0^T 
        \bigl\langle \sigma(t)X_N(t),dX_N(t)\bigr\rangle
    =\int_0^T 
        \bigl\langle \sigma(t)X_N(t),X_{N-1}(t)\bigr\rangle dt.
\]
Define $\eta\in\stwo$ by
\[
    \eta(t,s)=\int_0^T \Bigl\{
      \gamma_{N-1}(u,t)\gamma_N(u,s)\sigma(u)
      +\gamma_N(u,t)\gamma_{N-1}(u,s)\sigma(u)^\dagger
      \Bigr\}du
    \quad\text{for }(t,s)\in[0,T]^2, 
\]
we shall show that
\begin{equation}\label{e.ikm.3}
    \q=\q_\eta
     +\frac1{2(N!)^2} \int_0^T t^{2N} \tr[\sigma(t)]dt.
\end{equation}
To see this, we first compute $\D\q$.
By \eqref{e.ikm.2}, we know that
\begin{align*}
    \langle \D\q, g\rangle_{\mathcal{H}}
    = & \int_0^T \biggl\langle \sigma(t)X_N(t),
          \int_0^T \gamma_{N-1}(t,s)g^\prime(s)ds
        \biggr\rangle dt
    \\
    & +\int_0^T \biggl\langle \sigma(t)
         \int_0^T \gamma_N(t,s)g^\prime(s)ds,
       X_{N-1}(t) \biggr\rangle dt
    \quad\text{for }g\in \mathcal{H}.
\end{align*}
By \eqref{e.ikm.2} again, this implies that
\[
    \int_{\mathcal{W}} \D\q d\mu=0.
\]
Next, taking the $\mathcal{H}$-derivative of $\D\q$, we have that
\begin{align}
 \label{e.ikm.4}
    \langle \D^2\q, h\otimes g\rangle_{\htwo}
    & = \int_0^T \biggl\langle \sigma(t)
         \int_0^T \gamma_N(t,u)h^\prime(u)du,
          \int_0^T \gamma_{N-1}(t,s)g^\prime(s)ds
        \biggr\rangle dt
    \\
    & \hphantom{=}
      +\int_0^T \biggl\langle \sigma(t)
         \int_0^T \gamma_N(t,s)g^\prime(s)ds,
          \int_0^T \gamma_{N-1}(t,u)h^\prime(u)du
        \biggr\rangle dt
 \nonumber
    \\
    & =\int_0^T \biggl\langle
        \int_0^T \eta(s,u)h^\prime(u)du,g^\prime(s)
        \biggr\rangle ds
    \quad\text{for }h,g\in \mathcal{H}.
 \nonumber
\end{align}
Hence it holds that
\[
    \D^2\q=B_\eta.
\]
Finally, by \eqref{e.ikm.2} and the isometry for It\^o
integral, we see that 
\[
    \int_{\mathcal{W}} \q d\mu
    =\int_0^T \biggl\{\biggl(
        \int_0^T \gamma_N(t,s)\gamma_{N-1}(t,s)ds
        \biggr)\tr[\sigma(t)] \biggr\}dt
    =\int_0^T \frac{t^{2N}}{2(N!)^2} 
         \bigl(\tr\sigma(t)\bigr)dt.
\]
By Lemma~\ref{l.D^3=0}, we obtain the identity \eqref{e.ikm.3}.

Let $N_1,N_2\in \mathbb{N}$.
By definition, we know that 
\[
    \gamma_{N_2}(u,s)
    =\one_{[0,u)}(s)\gamma_{N_2}(u,s)
    \quad\text{for }(u,s)\in[0,T]^2.
\]
Since
\[
    \one_{[0,t]}(u)\one_{[0,u)}(s)
    =\one_{[0,t)}(s)\one_{[0,t]}(u)\one_{[0,u)}(s),
\]
we have that
\begin{align}
    \gamma_{N_1}(u,t)\gamma_{N_2}(u,s)
    & =\frac{(u-t)^{N_1}}{N_1!} \gamma_{N_2}(u,s)
      -\one_{[0,t]}(u)\frac{(u-t)^{N_1}}{N_1!}
       \gamma_{N_2}(u,s)
 \label{e.ikm.101}
    \\
    & =\frac{(u-t)^{N_1}}{N_1!} \gamma_{N_2}(u,s)
      -\one_{[0,t)}(s) \one_{[0,t]}(u)
       \frac{(u-t)^{N_1}}{N_1!}\gamma_{N_2}(u,s)
 \nonumber
\end{align}
for $(u,t),(u,s)\in[0,T]^2$.
Define $\kappa_F\in\ltwo$ as
\begin{align*}
    \kappa_F(t,s)
    & =\int_0^T\biggl\{
       \frac{(u-t)^{N-1}}{(N-1)!} \gamma_N(u,s)\sigma(u)
       +\frac{(u-t)^N}{N!} \gamma_{N-1}(u,s)\sigma(u)^\dagger
       \biggr\} du
    \\
    & =\sum_{j=0}^{N-1} \frac{(-1)^j}{j!(N-1-j)!} 
        \biggl(
        \int_0^T u^{N-j-1}\gamma_N(u,s)\sigma(u)du
        \biggr)t^j
    \\
    & \hphantom{=}
      +\sum_{j=0}^N \frac{(-1)^j}{j!(N-j)!} \biggl(
        \int_0^T u^{N-j}\gamma_{N-1}(u,s)\sigma(u)^\dagger 
         du \biggr)t^j
    \quad\text{for }(t,s)\in[0,T]^2.
\end{align*}
By \eqref{e.ikm.101}, we know that
$\one_{\Delta_2}\eta=\one_{\Delta_2}\kappa_F$.
Let $e_1,\dots,e_d$ be an ONB of $\mathbb{R}^d$ and take
monic polynomials $f_0,\dots,f_N:\mathbb{R}\to \mathbb{R}$
such that each $f_i$ is of degree $i$ and 
$\int_0^T f_i(t)f_j(t)dt=0$ if $i\ne j$.
Define $k_1,\dots,k_{(N+1)d}\in \mathcal{H}$ by
\[
    k_{nd+i}^\prime=f_n e_i
    \quad\text{for }0\le n\le N
    \text{ and }1\le i\le d.
\] 
Then it holds that
\[
    \mathcal{R}(B_{\kappa_F})
    \subset\Biggl\{ \sum_{i=1}^{(N+1)d} a_i k_i \,\Bigg|\,
     a_1,\dots,a_{(N+1)d}\in \mathbb{R}\Biggr\}.
\]
Thus $B_{\kappa_F}$ is of finite rank.
Define $\kappa_I\in\ltwo$ by 
\[
    \kappa_I(t,s)
    =-\one_{[0,t)}(s) \int_0^t \biggl\{
       \frac{(u-t)^{N-1}}{(N-1)!}\gamma_N(u,s)\sigma(u)
      +\frac{(u-t)^N}{N!}\gamma_{N-1}(u,s)\sigma(u)^\dagger
      \biggr\}du
\]
for $(t,s)\in[0,T]^2$.
By \eqref{e.ikm.101} and Lemma~\ref{l.condA}, $B_\eta$
satisfies the condition~{\bf(A)} with $A_I=B_{\kappa_I}$ and
$A_F=B_{\kappa_F}$.

We assume that $\Lambda(B_\eta)<1$ and apply
Theorem~\ref{t.plucker} to $\q_\eta$.
This assumption is fulfilled, for example, if 
\begin{equation}\label{e.ikm.4-}
    \frac{2}{N!(N-1)!} \int_0^T t^{2N}
            |\sigma(t)| dt <1.
\end{equation}
In fact, by \eqref{e.ikm.4}, we have that
\begin{align*}
    \langle B_\eta h,h\rangle_{\mathcal{H}}
    & \le 2\int_0^T |\sigma(t)|
        \frac{t^N}{N!}\frac{t^{N-1}}{(N-1)!}
        \biggl(\int_0^t |h^\prime(s)|ds\biggr)^2 dt
    \\
    & \le \frac{2}{N!(N-1)!} \biggl(\int_0^T t^{2N}
            |\sigma(t)| dt\biggr)\|h\|_{\mathcal{H}}^2
    \quad\text{for }h\in \mathcal{H}.
\end{align*}
In the remaining of the example, we only discuss how to
characterize $J_p$ via a $(2N+2)$th order ODE. 

Let 
$p=(p^{(1)},\dots,p^{(N+1)})\in (\mathbb{R}^d)^{N+1}
  =\mathbb{R}^{(N+1)d}$.
By the very definition of $\kappa_I$ and \eqref{e.ikm.1}, we
see that
\begin{align*}
    (A_Ih)^\prime(t)
    & =-\int_0^t \biggl\{
        \frac{(u-t)^{N-1}}{(N-1)!} 
          \Bigl(\mathcal{I}\bigl[\sigma 
            \mathcal{I}^N[h]\bigr]\Bigr)^\prime(u)
        +\frac{(u-t)^N}{N!} 
          \Bigl(\mathcal{I}\bigl[\sigma^\dagger
            \mathcal{I}^{N-1}[h]\bigr]\Bigr)^\prime(u)
         \biggr\} du
    \\
    & =(-1)^N\Bigl\{
           \mathcal{I}^N\bigl[
            \sigma \mathcal{I}^N[h]\bigr]
            -\mathcal{I}^{N+1}\bigl[
            \sigma^\dagger \mathcal{I}^{N-1}[h]\bigr]
            \Bigr\}(t)
     \quad\text{for $h\in \mathcal{H}$ and $t\in[0,T]$.}
\end{align*}
Since 
\[
    (I-A_I)J_p=\sum_{n=0}^N 
     \biggl(\int_0^\bullet f_n(s)ds\biggr) p^{(n)},
\]
$J_p$ is specified by the equation
\begin{equation}\label{e.ikm.5}
    J_p^{(1)}
    -(-1)^N \mathcal{I}^N\bigl[
            \sigma \mathcal{I}^N[J_p]\bigr]
    +(-1)^N \mathcal{I}^{N+1}\bigl[
            \sigma^\dagger \mathcal{I}^{N-1}[J_p]\bigr]
    =\sum_{n=0}^N f_n p^{(n)}.
\end{equation}
Setting 
\[
    K_p=\mathcal{I}^N[J_p],
\]
and differentiating \eqref{e.ikm.5} $(N+1)$-times, we have
that 
\begin{equation}\label{e.ikm.6}
    K_p^{(2N+2)}-2(-1)^N\sigma_A K_p^{(1)}
    -(-1)^N\sigma^{(1)}K_p=0.
\end{equation}
The initial condition for $K_p$ follow from its definition
and \eqref{e.ikm.5} as follows.
\[
     K_p^{(i)}(0)=0, \quad
     K_p^{(N+1+i)}(0)=J_p^{(i+1)}(0)
       =\sum_{n=0}^N f_n^{(i)}(0)p^{(n)}
     \quad\text{for }0\le i\le N.
\]
$K_p$ is obtained by solving \eqref{e.ikm.6} with this
initial condition, and then $J_p$ is done by differentiating
$K_p$ $N$-times.

In a special case, which is a modification of
Example~\ref{e.vol.ode.AD} with $X_N$, we have a more
explicit expression of $J_p$.
To see this, let $C,D\in \mathbb{R}^{d\times d}$ and assume
that $D^\dagger=D$.
Define the Wiener functional
$\mathfrak{a}_{C,D}^{(N)}:\mathcal{W}\to \mathbb{R}$ by
\[
    \mathfrak{a}_{C,D}^{(N)}
    =\frac12\int_0^T \langle CX_N(t),X_{N-1}(t)\rangle dt
     +\frac12\int_0^T \langle DX_N(t),X_N(t)\rangle dt.
\]
Since 
\[
    \frac12\int_0^T \langle DX_N(t),X_N(t)\rangle dt
    =\int_0^T \langle (T-t)DX_N(t),X_{N-1}(t)\rangle dt,
\]
using the same $\sigma\in C([0,T];\mathbb{R}^{d\times d})$
given by
\[
    \sigma(t)=\frac12 C+(T-t)D
    \quad\text{for }t\in[0,T],    
\]
as in Example~\ref{e.vol.ode.AD}, we have that
\[
    \mathfrak{a}_{C,D}^{(N)}
    =\int_0^T \langle \sigma(t)X_N(t),X_{N-1}(t)\rangle dt.
\]
We assume that
\[
    \frac{T^{2N+1}}{N!(N-1)!(2N+1)}
    \biggl\{2|C|+\frac{T}{N+1}|D|\biggr\}<1,
\]
by which the condition \eqref{e.ikm.4-} is fulfilled.
We apply the above observation to $\mathfrak{a}_{C,D}^{(N)}$
so to compute $J_p$ associated with it. 

Since $C,C^\dagger,D$ and $\mathcal{I}$ commute, the
ODE \eqref{e.ikm.6} reads as
\[
    K_p^{(2N+2)}-(-1)^N AK_p^{(1)}+(-1)^N DK_p=0
    \quad\text{with }A=\tfrac12\{C-C^\dagger\},
\]
that is,
\begin{equation}\label{e.ikm.7-}
    J_p^{(N+2)}-(-1)^N A \mathcal{I}^{N-1}[J_p]
    +(-1)^N D \mathcal{I}^N[J_p]=0.
\end{equation}
Differentiate this equation $N$-times to have the ODE for
$J_p$ such that
\begin{equation}\label{e.ikm.7}
    J_p^{(2N+2)}-(-1)^N AJ_p^{(1)}+(-1)^N DJ_p=0.    
\end{equation}
Since $f_n^{(i)}=0$ if $i\ge N+1$, the initial conditions
for $J_p$, which are obtained from \eqref{e.ikm.5}
and \eqref{e.ikm.7-}, are 
\[
    J_p(0)=0,
    \quad
    J_p^{(i)}(0)=\sum_{n=0}^N f_n^{(i-1)}(0) p^{(n+1)}
      \quad\text{for }1\le i\le 2N+1.    
\]

Let $\lambda_1,\dots,\lambda_{(2N+2)d}$ be the eigenvalues
of the matrix
\[
    \begin{pmatrix}
       0 & \hspace{7pt}I_d \hspace{7pt}
       & 0  & \cdots & 0 \\
       0 &\ddots  & \ddots & \ddots& \vdots   \\
       \vdots & \ddots & \ddots & \ddots & 0   \\
       0 & \cdots & \hspace{10pt}0\hspace{10pt}
       & 0 & I_d \\
       \hspace{10pt}0 \hspace{10pt}
       & \cdots & 0  & (-1)^N A & -(-1)^N D
    \end{pmatrix}
    \in \mathbb{R}^{(2N+2)d\times (2N+2)d}
\]
counted with multiplicity.
Define 
$r_j\in C^\infty(\mathbb{R};\mathbb{C})$ as
\[
    \left\{
      \begin{aligned}     
         & r_1^\prime=\lambda_1 r_1,
         \\
         & r_j^\prime=r_{j-1}+\lambda_j r_j
           \quad\text{for }2\le j\le (2N+2)d,
         \\
         & r_1(0)=1,~~
           r_2(0)=\dots=r_{(2N+2)d}(0)=0,
      \end{aligned}
    \right.
\]
and the multi-linear mappings 
$q_j:(\mathbb{C}^d)^{(N+1)d}\to \mathbb{C}^d$ for 
$1\le j\le (2N+2)d$ as
\[
    q_1(p)=0,\quad 
      q_{m+1}(p)
      =\sum_{n=0}^N f_n^{(m-1)}(0)p^{(n+1)}
       -\sum_{j=1}^m r_j^{(m)}(0)q_j(p)
      \quad\text{for }1\le m\le 2N+1,
\]
and
\begin{align*}
    q_{2N+3+n}(p)
    = & (-1)^{N+1} D \biggl(\sum_{j=1}^{n+1}
          r_j^{(n)}(0)q_j(p)\biggr)
       +(-1)^N A \biggl(\sum_{j=1}^{n+2}
          r_j^{(n+1)}(0)q_j(p)\biggr)
    \\
    & -\sum_{j=1}^{2N+2+n} 
          r_j^{(2N+2+n)}(0)q_j(p)
      \quad\text{for }0\le n\le (2N+2)(d-1)-1.
\end{align*}
By Corollary~\ref{c.ode.U} with 
$k=2N+2$, $C_0=(-1)^{N+1}D$, $C_1=(-1)^NA$, and 
$C_2=\cdots=C_{2N+1}=0$,  we have that
\[
    J_p=\sum_{j=1}^{(2N+2)d} r_j q_j(p).
\]

Ikeda, Kusuoka, and Manabe's stochastic area surrounded by
$X_N$ (\cite{IKM1}) is our 
$\mathfrak{a}_{C,D}^{(N)}$ with $d=2$, $C$ being skew
symmetric, and $D=0$. 
\end{example}

\begin{remark}\label{r.volterra}
(i)~
All $A_I$s in the above examples are Volterra operators.
Taking advantage of the decomposition of $B_\eta$ into a sum
of a Volterra operator and a finite rank operator was found
out by Ikeda, Kusuoka, and Manabe (\cite{IKM1,IKM2,IM2}).   
Their result was used by Hara-Ikeda (\cite{hara-ikeda}) to
see the underlying Grassmannian structure (Pl\"ucker
coordinates).
In such a case, the author also showed in 
\cite{st-kjm-jacobi.f} the similar assertions to
Theorem~\ref{t.plucker} and Remark~\ref{r.plucker}. 
The general condition~{\bf(A)} is introduced anew in this
monograph.
\end{remark}

\appendix

\chapter{Analytic or algebraic assertions}
\label{chap.anal.alg}

\section{Functional analytic lemmas}
\label{chap.funct.anal}

\begin{lemma}\label{l.cont.inv}
Every self-adjoint continuous linear operator 
$A:\mathcal{H}\to \mathcal{H}$ with the property that
\[
    \inf_{\|h\|_{\mathcal{H}}=1} \|Ah\|_{\mathcal{H}}
    \ge \varepsilon
    \quad\text{for some $\varepsilon>0$}
\]
has a continuous inverse.
Furthermore, the inverse operator $A^{-1}$ satisfies that
\[
    \|A^{-1}\|_\op\le \varepsilon^{-1}.    
\]
\end{lemma}

\begin{proof}
We first show that $A$ has a continuous inverse.
The assumed lower estimation implies the inequality that 
\begin{equation}\label{l.cont.inv.21}
    \|Ah\|_{\mathcal{H}} \ge \varepsilon\|h\|_{\mathcal{H}}
    \quad\text{for every }h\in \mathcal{H}.
\end{equation}
Hence $A$ is injective.
Due to the bounded inverse theorem, it suffices to show that 
$A$ is surjective. 

Let $\mathcal{R}(A)$ be the range of $A$.
If $h_n\in \mathcal{R}(A)$ for $n\in \mathbb{N}$ and $Ah_n$ 
converges to $g$ in $\mathcal{H}$, then the inequality
\eqref{l.cont.inv.21} yields that
\[
    \|h_n-h_m\|_{\mathcal{H}}
    \le \varepsilon^{-1} \|Ah_n-Ah_m\|_{\mathcal{H}}
    \to 0 \quad\text{as }n,m\to\infty,
\]
that is, $\{h_n\}_{n=1}^\infty$ is a Cauchy sequence 
in $\mathcal{H}$.
Let $h$ be its limit in $\mathcal{H}$.
Then we see that
\[
    \|Ah-g\|_{\mathcal{H}}
    =\lim_{n\to\infty} \|Ah_n-g\|_{\mathcal{H}}
    =0.
\]
Thus $\mathcal{R}(A)$ is closed.

If $h\in \mathcal{H}$ is perpendicular to
$\mathcal{R}(A)$, then, by the self-adjointness of $A$, we
have that
\[
    \langle Ah,g\rangle_{\mathcal{H}}
    =\langle h,Ag\rangle_{\mathcal{H}}=0
    \quad\text{for any }g\in \mathcal{H}.
\]
In conjunction with the injectivity of $A$, this yields that 
$h=0$.
Thus $\mathcal{R}(A)=\mathcal{H}$, that is, $A$ is
surjective.

Substituting $A^{-1}h$ for $h$ in \eqref{l.cont.inv.21}, 
we obtain the desired domination of $\|A^{-1}\|_\op$.
\end{proof}

\begin{lemma}\label{l.trace.detail}
If $\kappa\in\ltwo$ is continuous on $[0,T]^2$ and 
$B_\kappa$ is of trace class, then 
\[
    \tr B_\kappa=\int_0^T \tr[\kappa(s,s)]ds.
\]
\end{lemma}

\begin{proof}
Let $\{\varphi_n\}_{n=1}^\infty$ be an orthonormal
trigonometric basis of $L^2([0,T];\mathbb{R})$,
that is, an ONB obtained by orthonormalizing
$\bigl\{\cos(\frac{2m\pi t}{T}),
 \sin(\frac{2m\pi t}{T});m\in \mathbb{N}\cup\{0\}
 \bigr\}$.
Take an ONB $e_1,\dots,e_d$ of $\mathbb{R}^d$.
For $n\in \mathbb{N}$ and $1\le i\le d$, define
$h_{n;i}\in \mathcal{H}$ by 
$h_{n;i}^\prime=\varphi_ne_i$.
Let $N\in \mathbb{N}$ and denote by $\pi_N$ the orthogonal 
projection of $\mathcal{H}$ onto the subspace spanned by 
$\{h_{n;i}\mid 1\le n\le N, 1\le i\le d\}$.
Define $\kappa_N\in\ltwo$ so that
\[
    B_{\kappa_N}=\sum_{i,j=1}^d \sum_{n,m=1}^N
      \langle B_\kappa h_{n;i},h_{m;j}
          \rangle_{\mathcal{H}}
      h_{n;i}\otimes h_{m;j},
\]
that is,
\[
    \kappa_N(t,s)
    =\sum_{i,j=1}^d \sum_{n,m=1}^N
      \langle B_\kappa h_{n;i}, h_{m;j}
        \rangle_{\mathcal{H}} 
      [(\varphi_n(s)e_i)\otimes
           (\varphi_m(t)e_j)]
     \quad\text{for } (t,s)\in[0,T]^2.
\]
It holds that
\[
    \int_0^T \tr[\kappa_N(s,s)]ds
    =\sum_{i=1}^d \sum_{n=1}^N
      \langle B_\kappa h_{n;i},h_{n;i}
          \rangle_{\mathcal{H}}.
\]
Since $B_{\kappa_N}=\pi_N B_\kappa \pi_N$,
by \cite[Theorem\,I.3.1]{GGK}, we obtain that
\begin{equation}\label{l.trace.21}
    \tr(\pi_N B_\kappa \pi_N)
    =\tr B_{\kappa_N}
    =\sum_{i=1}^d \sum_{n=1}^N
      \langle B_\kappa h_{n;i},h_{n;i}
          \rangle_{\mathcal{H}}
    =\int_0^T [\tr\kappa_N(s,s)]ds.
\end{equation}

Put 
\[
    s_M=\frac1{M}\sum_{N=1}^M \kappa_N
    \quad\text{for }M\in \mathbb{N}.
\]
By the Fej\'{e}r theorem, the sequence
$\{s_M\}_{M=1}^\infty$ converges uniformly to $\kappa$ on
$[0,T]^2$.
Since $B_\kappa$ is of trace class, due to \eqref{l.trace.21}, 
we have that
\begin{align*}
    \tr B_\kappa
    & =\lim_{M\to\infty} \tr\bigl(\pi_M B_\kappa\pi_M\bigr)
      =\lim_{M\to\infty} \frac1{M}\sum_{N=1}^M
         \tr\bigl(\pi_N B_\kappa\pi_N\bigr)
    \\
    & =\lim_{M\to\infty} \int_0^T \tr[s_M(s,s)]ds
      =\int_0^T \tr[\kappa(s,s)]ds.
\end{align*}
Thus the desired identity holds.
\end{proof}

Recall that $\mathcal{S}_+(\mathcal{H})$ is used to denote 
the totality of self-adjoint, continuous, and non-negative definite
linear operators of $\mathcal{H}$ to itself.

\begin{lemma}\label{l.sq.root}
Every $A\in \mathcal{S}_+(\mathcal{H})$ has a unique $B\in
\mathcal{S}_+(\mathcal{H})$ with $B^2=A$. 
If a continuous linear $C:\mathcal{H}\to \mathcal{H}$
commutes with $A$, then so with $B$.
\end{lemma}

\begin{proof}
This is widely known as the square root lemma.
We give a proof by modifying the proof
of \cite[Theorem\,VI.9]{RS}, where complex Hilbert spaces are
dealt with. 

It is sufficient to consider the case when 
$\|A\|_\op\le 1$.
Since $I-A\in \mathcal{S}_+(\mathcal{H})$, 
\[
    \|I-A\|_\op=\sup_{\|h\|_{\mathcal{H}}=1}
     \langle (I-A)h,h\rangle_{\mathcal{H}}\le 1.
\]

Let $(1-a)^{\frac12}=1+\sum\limits_{n=1}^\infty c_n a^n$ be
the power series of $(1-a)^{\frac12}$ about $a=0$.
As was seen in the lemma for \cite[Theorem\,VI.9]{RS},
$c_n<0$ for any $n\in \mathbb{N}$ and the series converges
absolutely for $a\in[-1,1]$.
Hence the series 
$I+\sum\limits_{n=1}^\infty c_n (I-A)^n$ converges in the
operator norm to a continuous linear operator
$B:\mathcal{H}\to \mathcal{H}$.
Since the convergence is absolute, we can square the series
and rearrange terms, which proves that $B^2=A$.
Since 
\[
    \langle (I-A)^nh,h\rangle_{\mathcal{H}} 
     \le \|I-A\|_\op^n\le 1
    \quad\text{for any $n\in \mathbb{N}$ and 
    $h\in \mathcal{H}$ with $\|h\|_{\mathcal{H}}=1$}, 
\]
we have that 
\[
    \langle Bh,h\rangle_{\mathcal{H}}
    \ge 1+\sum_{n=1}^\infty c_n=0
    \quad\text{for any 
    $h\in \mathcal{H}$ with $\|h\|_{\mathcal{H}}=1$}.
\]
Thus $B\in \mathcal{S}_+(\mathcal{H})$.
The absolute convergence of the series also implies that
$CB=BC$ for any continuous linear $C$ with $CA=AC$. 

Suppose that $B^\prime\in \mathcal{S}_+(\mathcal{H})$
satisfies that $(B^\prime)^2=A$.
Since $B^\prime A=(B^\prime)^3=AB^\prime$, $B^\prime$
commutes with $B$. 
Therefore we have that
\[
    (B-B^\prime)B(B-B^\prime)
    +(B-B^\prime)B^\prime(B-B^\prime)
    =(B^2-(B^\prime)^2)(B-B^\prime)
    =0.
\]
Since $(B-B^\prime)B(B-B^\prime)\ge0$ and 
$(B-B^\prime)B^\prime(B-B^\prime)\ge0$, they must vanish,
so their difference $(B-B^\prime)^3=0$.
Since $B-B^\prime$ is self-adjoint, we have that
\begin{align*}
    \|B-B^\prime\|_\op^4
    & =\sup_{\|h\|_{\mathcal{H}}=1}
        \|(B-B^\prime)h\|_{\mathcal{H}}^4
      \le \sup_{\|h\|_{\mathcal{H}}=1}
         \|(B-B^\prime)^2h\|_{\mathcal{H}}^2
    \\
    & \le \sup_{\|h\|_{\mathcal{H}}=1}
         \|(B-B^\prime)^4h\|_{\mathcal{H}}
      =0.
\end{align*}
Hence $B-B^\prime=0$, which implies the uniqueness.
\end{proof}

\section{Special factorization}
\label{chap.spec.factrization}

\begin{proposition}\label{p.sigma.gamma}
Let $\sigma\in\stwo$.
Assume that 
\begin{equation}\label{p.sigma.gamma.1}
    \inf_{\xi\in[0,T]} \Lambda(\pi_\xi B_\sigma\pi_\xi)<1,
\end{equation}
where $\pi_\xi:\mathcal{H}\to \mathcal{H}$ is the projection
used in \eqref{l.tw.(ii)to(iv).22}.
Define $\gamma\in\stwo$ by $B_\gamma=(I-B_\sigma)^{-1}-I$.
Then there exists a $\nu\in\stwo$ with
\begin{equation}\label{p.sigma.gamma.2}
    \gamma(t,s)=\nu(t,s)+\int_t^T \nu(t,u)\nu(u,s)du
    \quad\text{for }0\le s<t\le T.
\end{equation}
\end{proposition}

To show the proposition, we recall the special factorization
due to Gohberg-Krein (\cite{GK}).  
Let 
\[
    \ltwo^{\mathbb{C}}
    =\{\tau_1+\kyosu \tau_2\mid \tau_1,\tau_2\in\ltwo\}.
\]
For $\tau\in \ltwo^{\mathbb{C}}$, define the continuous
linear operator $\mathcal{L}_\tau$ of
$L^2([0,T];\mathbb{C}^d)$ to itself by 
\[
    (\mathcal{L}_\tau f)(t)
    =\int_0^T \tau(t,s)f(s)ds
    \quad\text{for $t\in[0,T]$ and 
    $f\in L^2([0,T];\mathbb{C}^d)$}.
\]
The resolvent kernel of $\mathcal{L}_\tau$ is the function 
$\gamma\in\ltwo^{\mathbb{C}}$ with 
$(I-\mathcal{L}_\tau)^{-1}=I+\mathcal{L}_\gamma$.

\begin{lemma}\label{l.gk}
Let $\tau\in\ltwo^{\mathbb{C}}$.
\\
{\rm(i)}
In order that there exists a resolvent kernel
$\gamma\in\ltwo^{\mathbb{C}}$ for $\mathcal{L}_\tau$ in the
form 
\begin{equation}\label{l.gk.1}
    \mathcal{L}_\gamma=\mathcal{L}_{\nu_+}
     +\mathcal{L}_{\nu_-}+\mathcal{L}_{\nu_+}\mathcal{L}_{\nu_-},
\end{equation}
where $\nu_\pm\in\ltwo^{\mathbb{C}}$ satisfies that
\begin{equation}\label{l.gk.2}
    \nu_+(t,s)=\one_{[t,T]}(s)\nu_+(t,s)
    \quad\text{and}\quad
    \nu_-(t,s)=\one_{[0,t]}(s)\nu_-(t,s)
    \quad\text{for $(t,s)\in[0,T]^2$, }
\end{equation}
it is necessary and sufficient that the equation 
\[
    g(t)-\int_0^\xi \tau(t,s)g(s)ds=f(t)
    \quad\text{for }t\in[0,\xi]
\]
is solvable in $L^2([0,\xi];\mathbb{C}^d)$ for any
$\xi\in(0,T]$. 
\\
{\rm(ii)}
If the condition in (i) holds and $\mathcal{L}_\tau$ is
self-adjoint, then  
$\mathcal{L}_{\nu_+}^*=\mathcal{L}_{\nu_-}$, that is,
\[
    \overline{\nu_+(s,t)^\dagger}=\nu_-(t,s)
    \quad\text{for }(t,s)\in[0,T]^2.
\]
{\rm(iii)}
If the condition in (i) holds, then it holds that
\[
    \mathcal{L}_{\nu_-}
    =\int_0^T \mathcal{L}_\tau P_\xi
         (I-P_\xi \mathcal{L}_\tau P_\xi)^{-1} d\xi,
\]
where $(P_\xi f)(t)=\one_{[0,\xi]}(t)f(t)$ for $t\in[0,T]$.
\end{lemma}

\begin{proof}
(i) See the proposition~$1^\circ$ after (8.5) 
in \cite[p.180]{GK}.
\\
(ii)
See the remark after (2.5) in \cite[p.161]{GK}.
\\
(iii)
See Theorem~3.1 in \cite[p.162]{GK}.
\end{proof}

\begin{proof}[Proof of Proposition~\ref{p.sigma.gamma}]
Let $\sigma\in\ltwo$ and assume that \eqref{p.sigma.gamma.1}
holds.

Due to the inclusion
$\mathbb{R}^{d\times d}\subset\mathbb{C}^{d\times d}$,
$\sigma\in\ltwo^{\mathbb{C}}$.
For $f=f_1+\kyosu f_2\in L^2([0,T];\mathbb{C}^d)$ with
$f_1,f_2\in L^2([0,T];\mathbb{R}^d)$, since
$\sigma\in\stwo$, we have that
\begin{align*}
    & \langle P_\xi \mathcal{L}_\sigma P_\xi f,
        f\rangle_{L^2([0,T];\mathbb{C}^d)}
    \\
    & =\int_0^T \biggl\langle \int_0^T 
       \one_{[0,\xi]}(t) \sigma(t,s)\one_{[0,\xi]}(s)
        \{f_1(s)+\kyosu f_2(s)\} ds,
       \{f_1(t)-\kyosu f_2(t)\} 
       \biggr\rangle dt
    \\
    & =\sum_{i=1}^2 \int_0^T \biggl\langle \int_0^T
       \sigma(t,s)\one_{[0,\xi]}(s)f_i(s),
       \one_{[0,\xi]}(t) f_i(t)\biggr\rangle dt
    \\
    & \le \Lambda(\pi_\xi B_\sigma\pi_\xi) 
        \|f\|_{L^2([0,T];\mathbb{C}^d)}.
\end{align*}
Hence we obtain that
\[
    \inf_{\xi\in(0,T]} 
    \sup_{\|f\|_{L^2([0,T];\mathbb{C}^d)}=1}
       \langle P_\xi \mathcal{L}_\sigma P_\xi f,
        f\rangle_{L^2([0,T];\mathbb{C}^d)}
    <1.
\]
By Lemma~\ref{l.gk}(i), there exists a resolvent kernel
$\gamma\in\ltwo^{\mathbb{C}}$ for $\mathcal{L}_\tau$ 
and $\nu_\pm\in\ltwo^{\mathbb{C}}$ satisfying \eqref{l.gk.1}
and \eqref{l.gk.2}.

The identity \eqref{l.gk.1} is rewritten as
\[
    \gamma(t,s)=\nu_+(t,s)+\nu_-(t,s)
     +\int_0^T \nu_+(t,u)\nu_-(u,s)du
    \quad\text{for }(t,s)\in[0,T]^2.
\]
In particular, by \eqref{l.gk.2}, we have that
\begin{equation}\label{p.sigma.gamma.21}
    \gamma(t,s)=\nu_-(t,s)+\int_t^T \nu_+(t,u)\nu_-(u,s)du
    \quad\text{for }0\le s<t\le T.
\end{equation}

Since $\sigma\in\ltwo$, we have that
\[
    \overline{\sigma(s,t)^\dagger}=\sigma(t,s)
    \quad\text{for }(t,s)\in[0,T]^2.
\]
Hence $\mathcal{L}_\sigma^*=\mathcal{L}_\sigma$.
By Lemma~\ref{l.gk}(ii), we know that
\[
    \overline{\nu_+(s,t)^\dagger}=\nu_-(t,s)
    \quad\text{for }(t,s)\in[0,T]^2.
\]

It follows from Lemma~\ref{l.gk}(iii) that 
$\nu_-\in\ltwo$.
Due to the above identity, $\nu_\pm\in\ltwo$ and 
\begin{equation}\label{p.sigma.gamma.22}
    \nu_+(s,t)^\dagger=\nu_-(t,s)
    \quad\text{for }(t,s)\in[0,T]^2.
\end{equation}
Define $\nu\in\ltwo$ by
\[
    \nu(t,s)=\one_{[0,t)}(s)\nu_-(t,s)
      +\one_{(t,T]}(s)\nu_+(t,s)
    \quad\text{for }(t,s)\in[0,T]^2.
\]
By \eqref{p.sigma.gamma.22}, $\nu\in\stwo$.
Furthermore, due to \eqref{p.sigma.gamma.21}, we see
that \eqref{p.sigma.gamma.2} holds.
\end{proof}

\section{Matrix-valued linear ODE with constant coefficients}
\label{chap.e(t)}
In this section, we give an explicit expression of solutions
to matrix-valued linear ODEs with constant coefficients.

Let $k\in \mathbb{N}$ and
$C_0,\dots,C_{k-1},U_0,\dots,U_{k-1}
 \in \mathbb{R}^{d\times d}$.
Denote by $\bU\in C^k(\mathbb{R};\mathbb{R}^{d\times d})$
the solution to the ODE on $\mathbb{R}^{d\times d}$
\begin{equation}\label{eq.ode.U}
    \bU^{(k)}-C_{k-1}\bU^{(k-1)}-\cdots
         -C_1 \bU^{(1)}-C_0\bU=0,
    \quad
    \bU^{(i)}(0)=U_i
    \quad\text{for }0\le i\le k-1,
\end{equation}
where $\bU^{(i)}=\bigl(\frac{d}{dt}\bigr)^i \bU$.
Put
\[
    M=\begin{pmatrix}
       0 & I_d & 0 & \cdots & 0 
       \\
       0 & 0 & I_d & \ddots & \vdots
       \\
       \vdots & \vdots & \ddots & \ddots & 0
       \\
       0 & 0 & \cdots & 0 & I_d
       \\
       C_0 & C_1 & \cdots & \cdots & C_{k-1}
      \end{pmatrix}
   \in \mathbb{R}^{kd\times kd}.
\]
Let $\lambda_1,\dots,\lambda_{kd}\in \mathbb{C}$ be the
eigenvalues of $M$ counted with multiplicity.
Define $r_1,\dots,r_{kd}\in C^\infty(\mathbb{R};\mathbb{C})$
by the system of ODEs
\begin{equation}\label{eq.ode.r_j}
    \left\{
      \begin{aligned}     
         & r_1^\prime=\lambda_1 r_1,
         \\
         & r_j^\prime=r_{j-1}+\lambda_j r_j
           \quad\text{for }2\le j\le kd,
         \\
         & r_1(0)=1,~~
           r_2(0)=\dots=r_{kd}(0)=0,
      \end{aligned}
    \right.
\end{equation}
and $Q_1,\dots,Q_{kd}\in \mathbb{C}^{d\times d}$
successively by 
\begin{equation}\label{eq.Q}
    \left\{
      \begin{aligned}     
         & Q_1=U_0,\quad
           Q_{m+1}=U_m-\sum_{j=1}^m r_j^{(m)}(0) Q_j
           \quad\text{for }1\le m\le k-1,
         \\
         & Q_{k+n+1}=\sum_{i=0}^{k-1} C_i \biggl(
            \sum_{j=1}^{i+n+1} r_j^{(i+n)}(0) Q_j \biggr)
            -\sum_{j=1}^{k+n} r_j^{(k+n)}(0) Q_j
         \\
         & \hspace{200pt}
           \text{for }0\le n\le k(d-1)-1.
      \end{aligned}
    \right.
\end{equation}

With these $r_j$s and $Q_j$s, $\bU$ is represented as
follows. 

\begin{proposition}\label{p.ode.U}
Under the above notation, it holds that
\[
    \bU=\sum_{j=1}^{kd} r_j Q_j.
\]
\end{proposition}

Before proceeding to the proof of this proposition, we see that
the above expression is extended to ODEs on $\mathbb{R}^d$.

\begin{corollary}\label{c.ode.U}
Let $C_0,\dots,C_{k-1}\in \mathbb{R}^{d\times d}$ and
$J_0,\dots,J_{k-1}\in \mathbb{R}^d$.
The solution $J\in C^k(\mathbb{R};\mathbb{R}^d)$ to the ODE
on $\mathbb{R}^d$ 
\[
    J^{(k)}-C_{k-1} J^{(k-1)}-\cdots
         -C_1 J^{(1)}-C_0 J=0,
    \quad
    J^{(i)}(0)=J_i
    \quad\text{for }0\le i\le k-1
\]
is expressed as 
\[
    J=\sum_{j=1}^{kd} r_j q_j,
\]
where 
$r_1,\dots,r_{kd}\in C^\infty(\mathbb{R};\mathbb{C})$ 
are defined by \eqref{eq.ode.r_j} and 
$q_1,\dots,q_{kd}\in \mathbb{C}^d$ are determined
successively as
\[
    q_1=J_0,\quad
    q_{m+1}=J_m-\sum_{j=1}^m r_j^{(m)}(0) q_j
           \quad\text{for }1\le m\le k-1,
\]
and
\[
    q_{k+n+1}=\sum_{i=0}^{k-1} C_i \biggl(
            \sum_{j=1}^{i+n+1} r_j^{(i+n)}(0) q_j \biggr)
            -\sum_{j=1}^{k+n} r_j^{(k+n)}(0) q_j
           \quad\text{for }0\le n\le k(d-1)-1.
\]
\end{corollary}

\begin{proof}
For each $0\le \nu\le k-1$, let 
$\bU_\nu
\in C^k(\mathbb{R};\mathbb{R}^{d\times d})$ be the
solution to the ODE \eqref{eq.ode.U} with
$\bU^{(i)}(0)=\delta_{\nu i} I_d$ for $0\le i\le k-1$.
Then we have that
\[
    J=\sum_{\nu=0}^{k-1} \bU_\nu J_\nu.
\]
Define 
$Q_{\nu;1},\dots,Q_{\nu;kd}\in \mathbb{C}^{d\times d}$ 
by \eqref{eq.Q} with $U_m=\delta_{m\nu} I_d$ for 
$0\le m\le k-1$ and set
\[
    q_j=\sum_{\nu=0}^{k-1} Q_{\nu;j} J_\nu
    \quad\text{for }1\le j\le kd.
\]
Due to Proposition~\ref{p.ode.U} and the above expression of
$J$ with $\bU_\nu$s, we know that
\[
    J=\sum_{\nu=0}^{k-1} \biggl(\sum_{j=1}^{kd}
       r_j Q_{\nu;j}\biggr)J_\nu
     =\sum_{j=1}^{kd} r_j q_j.
\]

The identities in the first line of \eqref{eq.Q} for $\bU_\nu$ are 
written as  
\[
    Q_{\nu;m+1}=\delta_{\nu m} I_d
      -\sum_{j=1}^m r_j^{(m)}(0) Q_{\nu;j}
    \quad\text{for }0\le m\le k-1,
\]
where $\sum\limits_{j=1}^0\cdots=0$.
This implies that
\[
    q_{m+1}=J_m-\sum_{j=1}^m r_j^{(m)}(0) q_j
    \quad\text{for }0\le m\le k-1.
\]
The identities in the second line of \eqref{eq.Q} for $\bU_\nu$
imply that  
\[
    q_{k+n+1}=\sum_{i=0}^{k-1} C_i \biggl(
            \sum_{j=1}^{i+n+1} r_j^{(i+n)}(0) q_j\biggr)
            -\sum_{j=1}^{k+n} r_j^{(k+n)}(0) q_j
           \quad\text{for }0\le n\le k(d-1)-1.
\]
Thus $q_j$s satisfy the desired successive relationship.
\end{proof}

We proceed to the proof of Proposition~\ref{p.ode.U}.
To do so, we recall an expression, due to E.J. Putzer 
(\cite{putzer}), of the exponential function 
\[
    \e[tM]=\sum_{n=0}^\infty \frac{t^n}{n!}M^n
    \quad\text{for }t\in \mathbb{R}.
\]

\begin{proposition}\label{p.e[tM].r}
Let $r_1,\dots,r_{kd}$ be as in \eqref{eq.ode.r_j}, and put
\[
    P_1=I_{kd}
    \quad\text{and}\quad
    P_j=(M-\lambda_{j-1}I_{kd})\cdots
        (M-\lambda_1I_{kd}) 
    \quad\text{for } 2\le j\le kd.
\]
Then it holds that
\[
    \e[tM]=\sum_{j=1}^{kd} r_j(t) P_j
    \quad\text{for }t\in \mathbb{R}.
\]
\end{proposition}

\begin{proof}
Define 
$\Phi\in C^\infty(\mathbb{R};\mathbb{R}^{kd\times kd})$
as
\[
    \Phi(t)=\sum_{j=1}^{kd} r_j(t) P_j
    \quad\text{for }t\in \mathbb{R}.
\]
We have that
\[
    \Phi^\prime
    =\sum_{j=1}^{kd} r_j^\prime P_j
    =\lambda_1 r_1 P_1 
       +\sum_{j=2}^{kd} (r_{j-1}+\lambda_j r_j) P_j
    =\sum_{j=1}^{kd} \lambda_j r_j P_j
       +\sum_{j=1}^{kd-1} r_j P_{j+1}.
\]
This implies that
\[
    \Phi^\prime-\lambda_{kd}\Phi
    =\sum_{j=1}^{kd-1} 
       r_j\{P_{j+1}+(\lambda_j-\lambda_{kd})P_j\}.
\]

By definition and the Cayley-Hamilton theorem, we know that
\[
    P_{j+1}=(M-\lambda_jI_{kd})P_j
    \quad\text{for }j\ge1
    \quad\text{and}\quad P_{kd+1}=0.
\]
Hence we see that
\begin{align*}
    \Phi^\prime-\lambda_{kd}\Phi
    & =\sum_{j=1}^{kd-1} 
       r_j\{(M-\lambda_jI_{kd})
            +(\lambda_j-\lambda_{kd})I_{kd}\}P_j
      =\sum_{j=1}^{kd-1} r_j(M-\lambda_{kd}I_{kd})P_j
    \\
    & =(M-\lambda_{kd}I_{kd})(\Phi-r_{kd}P_{kd})
      =(M-\lambda_{kd}I_{kd}) \Phi.
\end{align*}
This implies that
\[
    \Phi^\prime=M\Phi.
\]
By definition, $\Phi(0)=I_{kd}$.
Thus $\e[tM]=\Phi(t)$ for $t\in \mathbb{R}$.
\end{proof}

\begin{proof}[Proof of Proposition~\ref{p.ode.U}]
Rewriting \eqref{eq.ode.U} as
\[
    \begin{pmatrix}
      \bU \\ \bU^{(1)} \\ \vdots \\ \bU^{(k-1)}
    \end{pmatrix}^\prime
    =M 
    \begin{pmatrix}
      \bU \\ \bU^{(1)} \\ \vdots \\ \bU^{(k-1)}
    \end{pmatrix},
    \quad
    \begin{pmatrix}
      \bU(0) \\ \bU^{(1)}(0) \\ \vdots \\ \bU^{(k-1)}(0)
    \end{pmatrix}
    =\begin{pmatrix}
      U_0 \\ U_1 \\ \vdots \\ U_{k-1}
     \end{pmatrix},
\]
by Proposition~\ref{p.e[tM].r}, we know that
\[
    \begin{pmatrix}
      \bU(t) \\ \bU^{(1)}(t) \\ \vdots \\ \bU^{(k-1)}(t)
    \end{pmatrix}
      =\biggl(\sum_{j=1}^{kd} r_j(t)P_j\biggr)
       \begin{pmatrix}
         U_0 \\ U_1 \\ \vdots \\ U_{k-1}
       \end{pmatrix}
    \quad\text{for }t\in \mathbb{R}.
\]
Hence we have that
\begin{equation}\label{p.ode.U.20}
    \bU=\sum_{j=1}^{kd} r_j Q_j
    \quad\text{for some }Q_j\in \mathbb{C}^{d\times d}.
\end{equation}
Thus the proof is completed, once we have shown that these
$Q_j$s satisfy the successive relationship \eqref{eq.Q}.

To show \eqref{eq.Q}, plug \eqref{p.ode.U.20} into
\eqref{eq.ode.U} to see that
\[
    \sum_{j=1}^{kd} r_j^{(k+n)} Q_j
    =\sum_{i=0}^{k-1} C_i \biggl(
        \sum_{j=1}^{kd} r_j^{(i+n)} Q_j \biggr)
    \quad\text{for }n\ge0.
\]
By \eqref{p.ode.U.20} and this, we obtain that
\begin{equation}\label{p.ode.U.21}
    \left\{
      \begin{aligned}
       & \sum_{j=1}^{kd} r_j^{(m)}(0) Q_j=U_m
         \quad\text{for }0\le m\le k-1,
       \\
       & \sum_{j=1}^{kd} r_j^{(k+n)}(0) Q_j
         =\sum_{i=0}^{k-1} C_i \biggl(\sum_{j=1}^{kd} 
            r_j^{(i+n)}(0) Q_j \biggr)
         \quad\text{for }n\ge0.
      \end{aligned}
    \right.
\end{equation}

Define $\Lambda\in \mathbb{R}^{kd\times kd}$ by
\[
    \Lambda
    =\begin{pmatrix}
       \lambda_1 & 0 & \cdots & \cdots & 0
       \\
       1 & \ddots & \ddots& & \vdots
       \\
       0 &\ddots & \ddots & \ddots & \vdots 
       \\
       \vdots & \ddots& \ddots & \ddots & 0
       \\
       0 & \cdots & 0 & 1 & \lambda_{kd}
     \end{pmatrix}.
\]
By \eqref{eq.ode.r_j}, we have that
\[
    \begin{pmatrix}    
      r_1(0) \\ r_2(0) \\ \vdots \\ r_{kd}(0) 
    \end{pmatrix}
    =
    \begin{pmatrix}    
      1 \\ 0 \\ \vdots \\ 0
    \end{pmatrix}
    \text{ and }
    \begin{pmatrix}    
      r_1^{(n)}(0) \\ r_2^{(n)}(0) \\ \vdots 
      \\ r_{kd}^{(n)}(0) 
    \end{pmatrix}
    =\Lambda
    \begin{pmatrix}    
      r_1^{(n-1)}(0) \\ r_2^{(n-1)}(0) \\ \vdots 
      \\ r_{kd}^{(n-1)}(0) 
    \end{pmatrix}
    \quad\text{for }n\ge 1.
\]
We obtain from this by induction that
\[
    r_{n+1}^{(n)}(0)=1 \text{ and }
    r_{n+j}^{(n)}(0)=0
    \quad\text{for }n\in \mathbb{N}\cup\{0\}
    \text{ and }j\ge 2.
\]
Substituting this into \eqref{p.ode.U.21}, we arrive at the 
successive relationship \eqref{eq.Q} for $Q_j$s.
\end{proof}

Represent the characteristic polynomial of $M$ as the
product function
\[
    \det(\lambda I_{kd}-M)
    =\prod_{j=1}^n (\lambda-\nu_j)^{m_j}
    \quad\text{for }\lambda\in \mathbb{C},
\]
where $n\in \mathbb{N}$, $\nu_j\in \mathbb{C}$,
$\nu_i\ne\nu_j$ if $i\ne j$, and $m_j\in \mathbb{N}$.
We have another expression of $\e[tM]$ as follows.

\begin{proposition}\label{p.e[tM].nu}
There are $C_{ji}\in \mathbb{C}^{kd\times kd}$ for 
$1\le j\le n$ and $1\le i\le m_j$ such that
\[
    \e[tM]=\sum_{j=1}^n \sum_{i=1}^{m_j}
       t^{i-1} e^{\nu_j t} C_{ji}
    \quad\text{for }t\in \mathbb{R}.
\]
\end{proposition}

\begin{proof}
Set $\hat{m}_0=0$ and 
$\hat{m}_r=m_1+\dots+m_r$ for $1\le r\le n$.
Arrange the eigenvalues $\lambda_1,\dots,\lambda_{kd}$ of $M$
counted with multiplicity as
\[
    \lambda_j=\nu_r
    \quad\text{if }
    1\le r\le n\text{ and }
    \hat{m}_{r-1}<j\le \hat{m}_r.
\]
On account of Proposition~\ref{p.e[tM].r}, it suffices to
show the existence of complex numbers $c_{j;p,q}$s such
that, if $\hat{m}_{r-1}<j\le \hat{m}_r$, then
\begin{equation}\label{c.e[tM].r.21}
    r_j(t)
    =\sum_{p=1}^{r-1}\sum_{q=1}^{m_p} c_{j;p,q}
       t^{q-1} e^{\nu_p t}
     +\sum_{q=1}^{j-\hat{m}_{r-1}} c_{j;r,q} 
          t^{q-1} e^{\nu_r t}
    \quad\text{for }t\in \mathbb{R},
\end{equation}
where
$\sum\limits_{p=1}^{r-1}\sum\limits_{q=1}^{m_p} \dots=0$ if
$r=1$. 
We show by induction that \eqref{c.e[tM].r.21} holds.

Before carrying out the induction, we remark two identities.
To state the first one, for $m\in \mathbb{N}$ and $\beta\ne0$, 
define $b_0(m;\beta),\dots,b_m(m;\beta)\in \mathbb{R}$ as
\[
    b_0(1;\beta)=-\frac1{\beta},\quad
    b_1(1;\beta)=\frac1{\beta},\quad
    b_m(m;\beta)=\frac1{\beta},
\]
and
\[
    b_k(m;\beta)=-\frac{m-1}{\beta} b_k(m-1;\beta)
      \quad\text{for }m\ge2,~0\le k\le m-1.
\]
By virtue of the integration by parts on $[0,t]$, we have the first
identity that 
\[
    \int_0^t s^{m-1} e^{\beta s} ds
    =b_0(m;\beta)
     +\sum_{k=1}^m b_k(m;\beta) t^{k-1} e^{\beta t}
    \quad\text{for }t\in \mathbb{R}.
\]
The very definition of $r_j$ yields the second identity that
\[
    \lambda_{j+1}(t)
    =e^{\lambda_{j+1} t}\int_0^t r_j(s)e^{-\lambda_{j+1}s}
    ds
    \quad\text{for }t\in \mathbb{R}.
\]

We now proceed to the induction.
By the definition, we have that
\[
    r_1(t)=e^{\nu_1 t}
    \quad\text{for }t\in \mathbb{R}.
\]
Thus \eqref{c.e[tM].r.21} holds for $j=1$ with $c_{1;1,1}=1$.
Next let $\hat{m}_{r-1}<j\le \hat{m}_r$ and suppose
that \eqref{c.e[tM].r.21} holds for this $j$.
The rest of the induction is divided into two cases:
when $j=\hat{m}_r$ and when $\hat{m}_{r-1}<j<j+1\le\hat{m}_r$.
If $j=\hat{m}_r$, then $\lambda_{j+1}=\nu_{r+1}$.
Due to the identity \eqref{c.e[tM].r.21} for $j=\hat{m}_r$, we see
that 
\begin{align*}
    r_{j+1}(t)
    & = e^{\lambda_{j+1}t} \int_0^t e^{-\lambda_{j+1}s} r_j(s)ds
      = \sum_{p=1}^r \sum_{q=1}^{m_p} c_{j;p,q}
        e^{\nu_{r+1}t} 
        \int_0^t s^{q-1} e^{(\nu_p-\nu_{r+1})s} ds
    \\
    & = \sum_{p=1}^r \sum_{q=1}^{m_p} c_{j;p,q}
         b_0(q;\nu_p-\nu_{r+1})e^{\nu_{r+1}t}
       +\sum_{p=1}^r \sum_{q=1}^{m_p} c_{j;p,q}
        \biggl(\sum_{k=1}^q b_k(q;\nu_p-\nu_{r+1})t^{k-1}
        \biggr) e^{\nu_p t}.
\end{align*}
Thus \eqref{c.e[tM].r.21} holds for $j+1$ when $j=\hat{m}_r$
with 
\[
    c_{j+1;p,q}=\sum_{k=q}^{m_p} c_{j;p,k}
       b_q(k;\nu_p-\nu_{r+1})
      \quad\text{for }p\le r
      \text{ and }1\le q\le m_p
\]
and 
\[
    c_{j+1;r+1,1} 
      =\sum_{p=1}^r \sum_{q=1}^{m_p} c_{j;p,q}
        b_0(q;\nu_p-\nu_{r+1}).
\]
If $\hat{m}_{r-1}<j<j+1\le \hat{m}_r$, then we have that
\begin{align*}
    & r_{j+1}(t)
      = \sum_{p=1}^{r-1} \sum_{q=1}^{m_p} c_{j;p,q}
        e^{\nu_r} 
        \int_0^t s^{q-1} e^{(\nu_p-\nu_r)s} ds
       +\sum_{q=1}^{j-\hat{m}_{r-1}} c_{j;r,q} e^{\nu_r t}
         \int_0^t s^{q-1} ds
    \\
    & = \sum_{p=1}^{r-1} \sum_{q=1}^{m_p} c_{j;p,q}
        \biggl\{
          b_0(q;\nu_p-\nu_r) e^{\nu_r t}
          +\sum_{k=1}^q b_k(q;\nu_p-\nu_r) t^{k-1} e^{\nu_p t}
        \biggr\}
        +\sum_{q=1}^{j-\hat{m}_{r-1}} \frac{c_{j;r,q}}{q!} 
           t^q e^{\nu_r t}.
\end{align*}
Thus \eqref{c.e[tM].r.21} holds for $j+1$ when 
$\hat{m}_{r-1}<j<j+1\le \hat{m}_r$ with 
\[
    c_{j+1;p,q} =\sum_{k=q}^{m_p} c_{j;p,k}
        b_q(k;\nu_p-\nu_r)
      \quad\text{for } p\le r-1
      \text{ and }1\le q\le m_p,
\]
\[
    c_{j+1;r,1} = \sum_{p=1}^{r-1} \sum_{q=1}^{m_p}
        c_{j;p,q} b_0(q;\nu_p-\nu_r),
\]
and
\[
    c_{j+1;r,q} = \frac{c_{j;r,q-1}}{(q-1)!}
      \quad\text{for } 2\le q\le j+1-\hat{m}_{r-1}.
\]
The proof is completed.
\end{proof}

With the help of the above expression, we have another
expression of $\bU$.

\begin{proposition}\label{p.ode.U.nu}
Assume that $M$ has $kd$ eigenvalues $\nu_1,\dots,\nu_{kd}$
different from each other, that is, $n=kd$ and $m_j=1$ for
$1\le j\le n$ in Proposition~\ref{p.e[tM].nu}.
Define $\widehat{U}_n\in \mathbb{R}^{d\times d}$ for 
$0\le n\le kd-1$ successively by
\[
    \widehat{U}_n
    =\begin{cases}
      U_n & \text{for }0\le n\le k-1
      \\
      {\displaystyle
      \sum_{i=0}^{k-1} C_i \widehat{U}_{n-k+i}}
      & \text{for }k\le n\le kd-1,
     \end{cases}
\]
and 
$Q_j=\bigl(Q_j^{pq}\bigr)_{1\le p,q\le d}
  \in \mathbb{C}^{d\times d}$ 
for $1\le j\le kd$ by 
\begin{equation}\label{p.ode.U.nu.1}
    \begin{pmatrix}
      Q_1^{pq} \\ \vdots \\ \vdots \\ Q_{kd}^{pq}
    \end{pmatrix}
    =
    \begin{pmatrix}
      1 & \cdots & \cdots & 1 
      \\
      \nu_1 & \cdots & \cdots & \nu_{kd}
      \\
      \vdots  &&& \vdots
      \\
      \nu_1^{kd-1} & \cdots & \cdots & \nu_{kd}^{kd-1}
    \end{pmatrix}^{-1}
    \begin{pmatrix}
      \widehat{U}_0^{pq} \\ \vdots \\ \vdots \\ 
      \widehat{U}_{kd-1}^{pq}
    \end{pmatrix}.
\end{equation}
Then 
\[
    \bU(t)=\sum_{j=1}^{kd} e^{\nu_j t} Q_j
    \quad\text{for }t\in \mathbb{R}.
\]
\end{proposition}

\begin{proof}
By Proposition~\ref{p.e[tM].nu}, there are 
$Q_1,\dots,Q_{kd}\in \mathbb{C}^{d\times d}$ such that
\begin{equation}\label{p.ode.U.nu.21}
    \bU(t)=\sum_{j=1}^{kd} e^{\nu_j t} Q_j
    \quad\text{for }t\in \mathbb{R}.
\end{equation}
Set 
\[
    K_n=\sum_{j=1}^{kd} \nu_j^n Q_j
    \quad\text{for }n\in \mathbb{N}\cup\{0\}.
\]

It follows from \eqref{p.ode.U.nu.21} that
\[
    \bU^{(n)}(t)=\sum_{j=1}^{kd} \nu_j^n e^{\nu_j t} Q_j
    \quad\text{for }n\in \mathbb{N}\cup\{0\}.
\]
Plugging this into the ODE \eqref{eq.ode.U}, we have that
\[
    \bU^{(k+n)}(t)=\sum_{i=0}^{k-1} C_i \biggl(
     \sum_{j=1}^{kd} \nu_j^{i+n} e^{\nu_j t} Q_j \biggr)
    \quad\text{for }n\in \mathbb{N}\cup\{0\}.
\]
Evaluating $\bU^{(n)}(t)$ at $t=0$, by using these two expressions,
we obtain that
\[
    K_n=U_n 
    \quad\text{for }0\le n\le k-1
    \quad\text{and}\quad
    K_n=\sum_{i=0}^{k-1} C_i K_{n-k+i}
    \quad\text{for }k\le n\le kd-1.
\]
Hence $K_n=\widehat{U}_n$ for $0\le n\le kd-1$.
This implies that 
\[
    \sum_{j=1}^{kd} \nu_j^n Q_j^{pq}
    =K_n^{pq}
    =\widehat{U}_n^{pq}
    \quad\text{for $0\le n\le kd-1$ and
    $1\le p,q\le d$}.
\]
Thus $Q_j^{pq}$s are determined by \eqref{p.ode.U.nu.1}.
\end{proof}

\section{Cauchy's and Lagrange's identities}
\label{chap.cauchy.lagrange}

We show two identities named after Cauchy and Lagrange.

\begin{lemma}\label{l.kdv.Cauchy_Lagrange}
Let 
$\alpha_i,\beta_i,a_i,b_j,z\in \mathbb{C}$
for $1\le i\le n,1\le j\le n-1$, and assume that
$a_i\ne a_j$ for $i\ne j$.
Then the following two identity hold.
\begin{align}
    & \det\biggl(\Bigl(\frac1{\alpha_i+\beta_j}
       \Bigr)_{1\le i,j\le n}\biggr)
    =\frac{\prod_{1\le i<j\le n}(\alpha_i-\alpha_j)
            (\beta_i-\beta_j)}{
               \prod_{1\le i,j\le n}(\alpha_i+\beta_j)},
 \label{eq.cauchy}
    \\[5pt]
    & \sum_{k=1}^n \frac{
           \prod_{j=1}^{n-1}(a_k+b_j)
           \prod_{i\ne k}(a_i-z)}{
              \prod_{i\ne k}(a_i-a_k)}
    =\prod_{k=1}^{n-1} (z+b_k).
 \label{eq.lagrange}
\end{align}
\end{lemma}

\eqref{eq.cauchy} is called Cauchy's identity and 
\eqref{eq.lagrange} is done Lagrange's.

\begin{proof}
We first show \eqref{eq.cauchy} by induction.

When $n=2$, \eqref{eq.cauchy} can be shown by an elementary
computation.
Assume that \eqref{eq.cauchy} holds up to $n-1$.
Repeating the following manipulation
\begin{enumerate}
\item
subtracting the first line from the other lines, and
\item
taking the factor out of the determinant,
\end{enumerate}
we obtain that 
\begin{align*}
    & \det\biggl(\frac1{\alpha_i+\beta_j}
                   \biggr)_{1\le i,j\le n}
    =\det\begin{pmatrix}
            \dfrac1{\alpha_1+\beta_1} & 
            \\
            \vdots
            & \Biggl(\dfrac{\beta_1-\beta_j}{\alpha_i+\beta_1}
                  \,\dfrac1{\alpha_i+\beta_j}
              \Biggr)_{\substack{1\le i\le n\\ 2\le j\le n}}  
            \\
            \dfrac1{\alpha_n+\beta_1} & 
         \end{pmatrix}
    \\[5pt]
    & 
    =\prod_{i=1}^n \frac1{\alpha_i+\beta_1}
     \prod_{j=2}^n (\beta_1-\beta_j)
     \times\det\begin{pmatrix}
            1 & 
            \\
            \vdots
            & \Biggl(\dfrac1{\alpha_i+\beta_j}
              \Biggr)_{\substack{1\le i\le n\\ 2\le j\le n}}  
            \\
            1 &
         \end{pmatrix}
    \\
    & 
    =\prod_{i=1}^n \frac1{\alpha_i+\beta_1}
     \prod_{j=2}^n (\beta_1-\beta_j)
     \times\det\Biggl(
             \dfrac{\alpha_1-\alpha_i}{\alpha_1+\beta_j}\,
             \dfrac1{\alpha_i+\beta_j}
              \Biggr)_{2\le i,j\le n}
    \\
    & 
    =\prod_{i=1}^n \frac1{\alpha_i+\beta_1}
     \prod_{j=2}^n \frac1{\alpha_1+\beta_j}
     \prod_{i=2}^n (\alpha_1-\alpha_i)
                  (\beta_1-\beta_i)
     \times\det\Biggl( \dfrac1{\alpha_i+\beta_j}
              \Biggr)_{2\le i,j\le n}.
\end{align*}
In conjunction with the induction assumption, we obtain 
\eqref{eq.cauchy} for $n$.

We next prove \eqref{eq.lagrange}.
We write the left hand side of \eqref{eq.lagrange} as
$\sum\limits_{i=0}^{n-1} \gamma_i z^i$, and the right as 
$\sum\limits_{i=0}^{n-1} \delta_i z^i$.
If we substitute $z=a_j$, $1\le j\le n$,
they coincide.
This means that
\[
    (\gamma_0,\dots,\gamma_{n-1})
    \begin{pmatrix}
      1 & 1 & \dots & 1 \\
      a_1 & a_2 & \dots & a_n \\
      a_1^2 & a_2^2 & \dots & a_n^2 \\
      \vdots & \vdots & \ddots & \vdots\\
      a_1^{n-1} & a_2^{n-1} & \dots & a_n^{n-1} \\
    \end{pmatrix}
    =(\delta_0,\dots,\delta_{n-1})
    \begin{pmatrix}
      1 & 1 & \dots & 1 \\
      a_1 & a_2 & \dots & a_n \\
      a_1^2 & a_2^2 & \dots & a_n^2 \\
      \vdots & \vdots & \ddots & \vdots\\
      a_1^{n-1} & a_2^{n-1} & \dots & a_n^{n-1} \\
    \end{pmatrix}.
\]
Since $a_i\ne a_j$ for $i\ne j$,
the Vandermonde matrix for $a_1,\dots,a_n$ is not singular.
Hence $\gamma_i=\delta_i$ for $0\le i\le n-1$.
Thus \eqref{eq.lagrange} holds.
\end{proof}

\chapter{Without Malliavin calculus}
\label{chap.without.MC}
\setcounter{section}{1}

In this chapter, we prove the assertions in
Chapters~\ref{chap.qwf} and \ref{chap.transf} without using
Malliavin calculus. 

Let $\I:\mathcal{H}\to L^2(\mu)$ be the mapping of Wiener integral; 
we define $\I(\ell)=\ell$ for $\ell\in \mathcal{W}^*$ and 
extend it to $\mathcal{H}$ continuously by using the identity 
$\int_{\mathcal{W}}|\I(\ell)|^2
  d\mu=\|\ell\|_{\mathcal{H}}^2$ 
and the denseness of $\mathcal{W}^*$ in $\mathcal{H}$.
It is also expressed with It\^o integral as follows.
\[
    \I(h)=\int_0^T \langle h^\prime(t),d\theta(t)\rangle
    \quad\text{for }h\in \mathcal{H}.
\]
Since $\D^*h=\I(h)$ for $h\in \mathcal{H}$
(\cite[(5.1.8)]{mt-cambridge}), the assertions
in Section~\ref{sec.development} continue to hold with
$\I(h)$ for $\D^*h$. 
We give alternative proofs for the assertions which are
proved by using Malliavin calculus.
To make the correspondence clear, we add the prime
${}^\prime$ to the number of the assertion.

\noindent
{\bf Theorem~\ref{t.q.eta}$\boldsymbol{{}^\prime}$.}
{\it
Let $\eta\in\stwo$.
For any ONB $\{h_n\}_{n=1}^\infty$ of $\mathcal{H}$, 
$\q_\eta$ is expanded in $L^2(\mu)$ as
\[
    \q_\eta=\frac12\sum_{n,m=1}^\infty
     \langle B_\eta h_n,h_m\rangle_{\mathcal{H}}
     \{\I(h_n)\I(h_m)-\delta_{nm}\}.
\]
}

\begin{proof}
As was seen in the proof of Lemma~\ref{l.b.kappa}, we have
that 
\[
    \eta(t,s)
    =\sum_{n,m=1}^\infty
     \langle B_\eta h_n,h_m\rangle_{\mathcal{H}}
     [h_n^\prime(s)\otimes h_m^\prime(t)]
    \quad\text{for }(t,s)\in[0,T]^2.
\]
Define $\eta^{(N)}\in\stwo$ for $N\in \mathbb{N}$ by
\[
    \eta^{(N)}(t,s)
    =\sum_{n,m=1}^N
     \langle B_\eta h_n,h_m\rangle_{\mathcal{H}}
     [h_n^\prime(s)\otimes h_m^\prime(t)]
    \quad\text{for }(t,s)\in[0,T]^2.
\]
Then it holds that
\[
    \int_{\mathcal{W}}|\q_\eta-\q_{\eta^{(N)}}|^2 d\mu
    =\frac12 \|\eta-\eta^{(N)}\|_2^2
   \to 0
   \quad\text{as }N\to\infty.
\]

By definition, we know that
\[
    \q_{\eta^{(N)}}
    =\sum_{n,m=1}^N
     \langle B_\eta h_n,h_m\rangle_{\mathcal{H}}
     \int_0^T\biggl\langle\biggl(\int_0^t 
          \langle h_n^\prime(s),d\theta(s)\rangle
         \biggr) h_m^\prime(t),d\theta(t)
         \biggr\rangle.
\]
Since $B_\eta\in \mathcal{S}(\htwo)$, switching $n$ and $m$,
we have that
\[
   \q_{\eta^{(N)}}=\sum_{n,m=1}^N
     \langle B_\eta h_n,h_m\rangle_{\mathcal{H}}
     \int_0^T\biggl\langle\biggl(\int_0^t 
          \langle h_m^\prime(s),d\theta(s)\rangle
         \biggr) h_n^\prime(t),d\theta(t)
         \biggr\rangle.    
\]

Apply It\^o's formula to see that
\begin{align*}
    \I(h_n)\I(h_m)-\delta_{nm}
    = & \int_0^T\biggl\langle\biggl(\int_0^t 
       \langle h_n^\prime(s),d\theta(s)\rangle
         \biggr) h_m^\prime(t),d\theta(t)
         \biggr\rangle
    \\
    & +\int_0^T\biggl\langle\biggl(\int_0^t 
          \langle h_m^\prime(s),d\theta(s)\rangle
         \biggr) h_n^\prime(t),d\theta(t)
         \biggr\rangle. 
\end{align*}
Thus we obtain that
\[
    \q_{\eta^{(N)}}
    =\frac12 \sum_{n,m=1}^N
     \langle B_\eta h_n,h_m\rangle_{\mathcal{H}}
     \{\I(h_n)\I(h_m)-\delta_{nm}\}.
\]
Letting $N\to\infty$, we obtain the desired expansion in $L^2(\mu)$.
\end{proof}

\noindent
{\bf Theorem~\ref{t.w.chaos}$\boldsymbol{{}^\prime}$.}
{\it $\mathcal{C}_2=\{\q_\eta\mid \eta\in\stwo\}$.}

\begin{proof}
By Theorem~\ref{t.q.eta}${}^\prime$, 
$\q_\eta\in \mathcal{C}_2$ for any $\eta\in\stwo$.

To see the converse inclusion, let $\q\in \mathcal{C}_2$.
In repetition of the argument in the proof of
Theorem~\ref{t.w.chaos} with $\I(h_n)$s for $\D^*h_n$s, we
find a double sequence $\{b_{nm}\}_{n,m=1}^\infty$ of real
numbers such that 
$\sum\limits_{n,m=1}^\infty b_{nm}^2<\infty$,
$b_{nm}=b_{mn}$ for $n,m\in \mathbb{N}$, and 
\begin{equation}\label{t.w.chaos.prime.21}
    \q=\sum_{n,m=1}^\infty b_{nm}
       \{\I(h_n)\I(h_m)-\delta_{nm}\}.
\end{equation}
Define $B\in \mathcal{S}(\htwo)$ by 
\[
    B=\sum_{n,m=1}^\infty 2b_{nm}
         h_n\otimes h_m.
\]
By Lemma~\ref{l.b.kappa}, there exists an $\eta\in\stwo$
with $B_\eta=B$.
Then $\langle B_\eta h_n,h_m\rangle_{\mathcal{H}}=2b_{nm}$
for $n,m\in \mathbb{N}$.
By Theorem~\ref{t.q.eta}${}^\prime$
and \eqref{t.w.chaos.prime.21}, we obtain that
\[
    \q_\eta
    =\frac12\sum_{n,m=1}^\infty 
     \langle B_\eta h_n,h_m\rangle_{\mathcal{H}}
       \{\I(h_n)\I(h_m)-\delta_{nm}\}
    =\sum_{n,m=1}^\infty b_{nm}
       \{\I(h_n)\I(h_m)-\delta_{nm}\}
    =\q.
\]
The proof is completed.
\end{proof}

\noindent
{\bf Proposition~\ref{p.h.osc}$\boldsymbol{{}^\prime}$.}
{\it
Let $\kappa\in\ltwo$ and $x\in \mathbb{R}^d$.
Then it holds that
\[
    \mathfrak{h}(\kappa;x)
    =\q_{c(\kappa;x)}
       +\frac12\int_0^T\int_0^T 
           |\kappa(t,s)^\dagger x|^2 ds dt.
\]
}

\begin{proof}
Let $\kappa_1,\kappa_2\in\ltwo$.
By a straightforward computation, we see that 
\[
    \int_{\mathcal{W}}|\mathfrak{h}(\kappa_1;x)
      -\mathfrak{h}(\kappa_2;x)|d\mu,
    \quad
    \|c(\kappa_1;x)-c(\kappa_2;x)\|_2
\]
and 
\[
 \biggl|\int_0^T \biggl(\int_0^T 
        |\kappa_1(t,s)^\dagger x|^2 ds\biggr)dt
       -\int_0^T \biggl(\int_0^T 
        |\kappa_2(t,s)^\dagger x|^2 ds\biggr)dt
      \biggr|
\]
are all dominated by 
\[
        |x|^2 \bigl\{\|\kappa_1\|_2+\|\kappa_2\|_2\bigr\}
        \|\kappa_1-\kappa_2\|_2.
\]
Hence the identity for $\kappa$ holds if $\kappa_N\to\kappa$ in
$\ltwo$ as $N\to\infty$ and the identity with $\kappa=\kappa_N$ holds
for every $N$.
Thus it suffices to show the identity for $\kappa$ of the
form 
\begin{equation}\label{p.h.osc.prime.21}
    \kappa(t,s)=\kappa(T_{N;n},s)
    \quad\text{for }(t,s)\in [T_{N;n},T_{N;n+1})\times[0,T]
    \text{ and }0\le n\le N-1,
\end{equation}
where $N\in \mathbb{N}$ and $T_{N;n}=\frac{nT}{N}$.
For such a $\kappa$, we have that
\[
    \mathfrak{h}(\kappa;x)
    =\frac12\sum_{n=0}^{N-1} \frac{T}{N}
     \biggl\langle x,\int_0^T \kappa(T_{N;n},s)d\theta(s)
        \biggr\rangle^2.
\]

For $k\in \mathcal{H}\setminus\{0\}$, define
$\eta^\prime(k)\in\stwo$ by  
\[
    \eta^\prime(k)(t,s)
    =k^\prime(s)\otimes k^\prime(t)
    \quad\text{for }(t,s)\in[0,T]^2.
\]
Noting that $B_{\eta^\prime(k)}=k\otimes k$, and 
applying Theorem~\ref{t.q.eta}${}^\prime$ with an
ONB $\{h_n\}_{n=1}^\infty$ such that
$h_1=\frac{1}{\|k\|_{\mathcal{H}}}k$, we obtain that
\begin{equation}\label{p.h.osc.prime.22}
    \I(k)^2=2\q_{\eta^\prime(k)}+\|k\|_{\mathcal{H}}^2.
\end{equation}

Let $\kappa\in\ltwo$ be of the form described
in \eqref{p.h.osc.prime.21}. 
For $0\le n\le N-1$, define $k_n\in \mathcal{H}$ by 
\[
    k_n^\prime(s)=\kappa\bigl(T_{N;n},s)^\dagger x
    \quad\text{for }s\in[0,T].
\]
Since 
\[
    \biggl\langle x,\int_0^T 
       \kappa(T_{N;n},s)d\theta(s) \biggr\rangle
    =\int_0^T \bigl\langle x,
       \kappa(T_{N;n},s)d\theta(s) \bigr\rangle
    =\I(k_n),
\]
by \eqref{p.h.osc.prime.22}, it holds that
\[
    \mathfrak{h}(\kappa;x)
    =\frac12\sum_{n=0}^{N-1} \frac{T}{N} 
        \I(k_n)^2
    =\sum_{n=0}^{N-1} \frac{T}{N} 
        \q_{\eta^\prime(k_n)}
       +\frac12 \sum_{n=0}^{N-1} \frac{T}{N} 
        \|k_n\|_{\mathcal{H}}^2.
\]
Observe that
\begin{align*}
    \biggl(\sum_{n=0}^{N-1} \frac{T}{N} 
        \eta^\prime(k_n)\biggr)(t,s)
    &  =\sum_{n=0}^{N-1} \frac{T}{N} 
       \bigl[\bigl(\kappa(T_{N;n},s)^\dagger x\bigr)\otimes
         \bigl(\kappa(T_{N;n},t)^\dagger x\bigr)\bigr]
    \\
    & =\int_0^T \bigl[\bigl(\kappa(u,s)^\dagger x\bigr)
         \otimes 
         \bigl(\kappa(u,t)^\dagger x\bigr)\bigr]du
    \\
    &  =c(\kappa;x)(t,s)
    \quad\text{for }(t,s)\in[0,T]^2
\end{align*}
and 
\begin{align*}
    \sum_{n=0}^{N-1} \frac{T}{N} 
         \|k_n\|_{\mathcal{H}}^2
    & =\sum_{n=0}^{N-1} \frac{T}{N} \biggl(
        \int_0^T \bigl|
        \kappa(T_{N;n},s)^\dagger x\bigr|^2 
        ds \biggr)
      =\int_0^T \biggl(\int_0^T |\kappa(t,s)^\dagger x|^2 
         ds\biggr) dt.
\end{align*}
Thus we arrive at the desired identity.
\end{proof}

We now consider the the achievements in Chapter~\ref{chap.transf}. 

\noindent
{\bf Theorem~\ref{t.transf}$\boldsymbol{{}^\prime}$.}
{\it
Let $\kappa\in\ltwo$ and assume that
$\Lambda(B_{\eta(\kappa)})<1$.
Then it holds that
\[
    |\dettwo(I+B_\kappa)| 
    \int_{\mathcal{W}} f(\iota+F_\kappa)
       e^{\q_{\eta(\kappa)}} d\mu
    =e^{\frac12 \|\kappa\|_2^2} \int_{\mathcal{W}} f d\mu
    \quad\text{for every }f\in C_b(\mathcal{W}).
\]
}

For the proof, we modify Lemma~\ref{l.f.kappa}.

\noindent
{\bf Lemma~\ref{l.f.kappa}$\boldsymbol{{}^\prime}$.}
{\it
Let $\kappa\in\ltwo$.
For any ONB $\{h_n\}_{n=1}^\infty$ of $\mathcal{H}$, it
holds that
\[
    F_\kappa=\sum_{n,m=1}^\infty 
     \langle B_\kappa h_n,h_m\rangle_{\mathcal{H}}
     \I(h_n)h_m,
\]
where the series converges in $L^2(\mu;\mathcal{H})$.
}

\begin{proof}
By Lemma~\ref{l.b.kappa}, we have that
\[
    \kappa(t,s)=\sum_{n,m=1}^\infty
     \langle B_\kappa h_n,h_m\rangle_{\mathcal{H}}
     [h_n^\prime(s)\otimes h_m^\prime(t)]
    \quad\text{for }(t,s)\in[0,T]^2.
\]
Define $\kappa^{(N)}\in\ltwo$ for $N\in \mathbb{N}$ by
\[
    \kappa^{(N)}(t,s)=\sum_{n,m=1}^N
     \langle B_\kappa h_n,h_m\rangle_{\mathcal{H}}
     [h_n^\prime(s)\otimes h_m^\prime(t)]
    \quad\text{for }(t,s)\in[0,T]^2.
\]
Then, by the isometry for It\^o integral, we have that
\[
    \int_{\mathcal{W}}
     \|F_{\kappa^{(N)}}-F_\kappa\|_{\mathcal{H}}^2 
     d\mu
    =\|\kappa^{(N)}-\kappa\|_2^2\to0
    \quad\text{as }N\to\infty.
\]
By the very definition of $F_{\kappa^{(N)}}$, we see that 
\[
    F_{\kappa^{(N)}}=\sum_{n,m=1}^N
     \langle B_\kappa h_n,h_m\rangle_{\mathcal{H}}
     \I(h_n)h_m.
\]
Thus the proof is completed.
\end{proof}

\begin{proof}[Proof of Theorem~\ref{t.transf}\,$\boldsymbol{{}^\prime}$]
In the first paragraph of the proof of
Theorem~\ref{t.transf}, the continuity of $\D^*$ is used to
see that 
\[
    \q_{\eta(\kappa^{(N)})}\to\q_{\eta(\kappa)}
    ~~\text{in }L^p(\mu)
    \quad\text{and}\quad
    F_{\kappa^{(N)}}\to F_\kappa
    ~~\text{in }L^p(\mu;\mathcal{H})
    \quad\text{for any }p\in(1,\infty).
\]
To carry out the argument after the paragraph, 
it is enough to show these convergences in the $L^2$-sense. 
Such $L^2$-convergences can be seen by observing that
\[
    \int_{\mathcal{W}}
       |\q_{\eta(\kappa^{(N)})}-\q_{\eta(\kappa)}|^2 d\mu
      =\frac12\|\eta(\kappa^{(N)})-\eta(\kappa)\|_2^2
      =\frac12\|B_{\eta(\kappa^{(N)})}-B_{\eta(\kappa)}\|_{\htwo}^2
\]
and
\[
    \int_{\mathcal{W}}    
        \|F_{\kappa^{(N)}}-F_\kappa\|_{\mathcal{H}}^2 
       d\mu
      = \|\kappa^{(N)}-\kappa\|_2^2
      = \|B_{\kappa^{(N)}}-B_\kappa\|_{\htwo}^2.
\qedhere
\]
\end{proof}

\noindent
{\bf Theorem~\ref{t.inv.transf}$\boldsymbol{{}^\prime}$.}
{\it
The identities \eqref{t.inv.transf.1}
and \eqref{eq:inv.transf} and the last identity hold
with $\I(\cdot)$ for $\D^*$.
}

For the proof we need to prove Lemma~\ref{l.d*h.f.hat.kappa}
without using Malliavin calculus.

\medskip\noindent
{\bf Lemma~\ref{l.d*h.f.hat.kappa}$\boldsymbol{{}^\prime}$.}
{\it
Let $\kappa\in\ltwo$ and assume that
$\Lambda(B_{\eta(\kappa)})<1$.
It then holds that
\[
    \I(h)(\iota+F_{\widehat{\kappa}})
    =\I((I+B_\kappa^*)^{-1}h)
    \quad\text{for any }h\in \mathcal{H}.
\]
}

\begin{proof}
Let $h\in \mathcal{H}$.
Since $\I(\ell)=\ell$ for $\ell\in \mathcal{W}^*$, in
repetition of the argument in the proof of
\cite[Lemma\,5.7.7]{mt-cambridge}, we find an
$\mathcal{H}$-invariant set $X$ such that $\mu(X)$ and 
\begin{equation}\label{t.inv.transf.prime.20}
    \I(h)(w+g)=\I(h)(w)
     +\langle h,g\rangle_{\mathcal{H}}
    \quad\text{for every }w\in X
    \text{ and }g\in \mathcal{H}.
\end{equation}

By Lemma~\ref{l.f.kappa}${}^\prime$, we have the series
expansion in $L^2(\mu;\mathcal{H})$ as
\[
    F_{\widehat{\kappa}}
    =\sum_{n,m=1}^\infty \langle B_{\widehat{\kappa}} h_n,
       h_m\rangle_{\mathcal{H}} \I(h_n)h_m
\]
with an ONB $\{h_n\}_{n=1}^\infty$ of $\mathcal{H}$. 
In conjunction with \eqref{t.inv.transf.prime.20}, this
implies that
\begin{align*}
    \I(h)(\iota+F_{\widehat{\kappa}})
    & =\I(h)+\langle h,
        F_{\widehat{\kappa}}\rangle_{\mathcal{H}}
      =\I(h)+\sum_{n,m=1}^\infty 
        \langle B_{\widehat{\kappa}} h_n,
         h_m\rangle_{\mathcal{H}} \I(h_n)
        \langle h,h_m\rangle_{\mathcal{H}}
    \\
    & =\I(h)+\sum_{n=1}^\infty 
        \langle B_{\widehat{\kappa}} h_n,
         h\rangle_{\mathcal{H}} \I(h_n)
      =\I(h)+\sum_{n=1}^\infty 
        \langle B_{\widehat{\kappa}}^*h,
         h_n\rangle_{\mathcal{H}} \I(h_n)
    \\
    & =\I((I+B_{\widehat{\kappa}}^*)h),
\end{align*}
where to see the last equality, we have used the continuity of 
$\I:\mathcal{H}\to L^2(\mu)$ and the expansion that
$B_{\widehat{\kappa}}^*h
 =\sum\limits_{n=1}^\infty \langle B_{\widehat{\kappa}}^*h,
    h_n\rangle_{\mathcal{H}} h_n$.
Since $I+B_{\widehat{\kappa}}^*=(I+B_\kappa^*)^{-1}$, we
obtain the desired identity.
\end{proof}

\begin{proof}[Proof of Theorem~\ref{t.inv.transf}\,${}^\prime$]
By Lemma~\ref{l.f.kappa}${}^\prime$ and 
\eqref{t.inv.transf.prime.20}, we can show the assertion in
Lemma~\ref{l.h.inv} without using Malliavin calculus.
Hence we have that
\begin{equation}\label{t.inv.transf.prime.21}
    (\iota+F_\kappa)\circ(\iota+F_{\widehat{\kappa}})
    =\iota+F_\kappa+F_{\widehat{\kappa}}
     +B_\kappa F_{\widehat{\kappa}}.
\end{equation}

Let $\{h_n\}_{n=1}^\infty$ be an ONB of $\mathcal{H}$.
By Lemma~\ref{l.f.kappa}${}^\prime$, we have that
\[
    F_\kappa  =\sum_{n,m=1}^\infty 
      \langle B_\kappa h_n,h_m\rangle_{\mathcal{H}}
      \I(h_n)h_m
    \quad\text{and}\quad
    F_{\widehat{\kappa}} =\sum_{n,p=1}^\infty 
       \langle B_{\widehat{\kappa}} h_n,
            h_p \rangle_{\mathcal{H}}  
       \I(h_n)h_p.
\]
It follows from the second identity that
\begin{align*}
    B_\kappa F_{\widehat{\kappa}}
    & =\sum_{n,p=1}^\infty 
       \langle B_{\widehat{\kappa}} h_n,
            h_p \rangle_{\mathcal{H}}  
       \I(h_n)B_\kappa h_p
      =\sum_{n,m,p=1}^\infty 
       \langle B_{\widehat{\kappa}} h_n,
            h_p \rangle_{\mathcal{H}}  
       \langle B_\kappa h_p,h_m \rangle_{\mathcal{H}}  
       \I(h_n) h_m
    \\
    & =\sum_{n,m=1}^\infty 
       \langle B_{\widehat{\kappa}} h_n,
            B_\kappa^*h_m\rangle_{\mathcal{H}}  
       \I(h_n)h_m
      =\sum_{n,m=1}^\infty 
       \langle B_\kappa B_{\widehat{\kappa}} h_n,
            h_m\rangle_{\mathcal{H}}  
       \I(h_n)h_m.
\end{align*}
Since 
\[
    B_\kappa+B_{\widehat{\kappa}}
    +B_\kappa B_{\widehat{\kappa}}
    =(I+B_\kappa)(I+B_{\widehat{\kappa}})-I
    =0,
\]
plugging the series expansions of $F_\kappa$,
$F_{\widehat{\kappa}}$, and 
$B_\kappa F_{\widehat{\kappa}}$
into \eqref{t.inv.transf.prime.21}, we see that
\[
    (\iota+F_\kappa)\circ(\iota+F_{\widehat{\kappa}})
    =\iota.
\]
The other equality in \eqref{t.inv.transf.1} can be seen
similarly. 

The identity \eqref{eq:inv.transf} and the last identity
with $\I(\cdot)$ for $\D^*$ follow in repetition of the
argument in the proof of Theorem~\ref{t.inv.transf} with
Theorem~\ref{t.transf}${}^\prime$ and
Lemma~\ref{l.d*h.f.hat.kappa}${}^\prime$ for 
Theorem~\ref{t.transf} and Lemma~\ref{l.d*h.f.hat.kappa}.
\end{proof}

\noindent
{\bf Lemma~\ref{l.transf.cm}$\boldsymbol{{}^\prime}$.}
{\it
For $\phi\in\ltwo$ with 
$\Lambda(B_{\eta(\kappa_\phi)})<1$, it holds that
\[
    |\dettwo(I+B_{\kappa_\phi})|\int_{\mathcal{W}} 
      f(\iota+\mathbb{F}_\phi) 
      e^{\widetilde{\Psi}_\phi} d\mu
    =\int_{\mathcal{W}} f d\mu
    \quad\text{for every }f\in C_b(\mathcal{W}).
\]
}

\begin{proof}
For $\kappa\in\ltwo$, define $s(\kappa)\in\stwo$ by
\[
    s(\kappa)(t,s)=\kappa(t,s)+\kappa(s,t)^\dagger
    \quad\text{for }(t,s)\in[0,T]^2.
\]
Let $\{h_n\}_{n=1}^\infty$ be an ONB of $\mathcal{H}$.
By Lemma~\ref{l.f.kappa}${}^\prime$, it holds that
\[
    F_\kappa=\sum_{n,m=1}^\infty 
     \langle B_\kappa h_n,h_m\rangle_{\mathcal{H}}
     \I(h_n)h_m.
\]
Due to Theorem~\ref{t.q.eta}${}^\prime$
and \eqref{r.transf.1}, we have that
\begin{align*}
    \q_{\eta(\kappa)}
    & = -\q_{s(\kappa)}
        -\frac12 \sum_{n,m=1}^\infty 
         \langle B_\kappa^*B_\kappa h_n,h_m
            \rangle_{\mathcal{H}}
         \{\I(h_n)\I(h_m)-\delta_{nm}\}
    \\
    & =-\q_{s(\kappa)}-\frac12\|F_\kappa\|_{\mathcal{H}}^2
       +\frac12\|B_\kappa\|_{\htwo}^2.
\end{align*}
Thus, by Theorem~\ref{t.transf}${}^\prime$, it suffices to
show that
\begin{align}
 \label{t.transf.cm.prime.21}
    & \int_0^T \biggl\langle \int_0^t 
      s(\kappa_\phi)(t,s) d\theta(s), d\theta(t)
      \biggr\rangle
    \\
    & 
    =\int_0^T\biggl\langle \int_0^T \phi(t,s)^\dagger
      d\theta(t),\theta(s)\biggr\rangle ds
     -\int_0^T\biggl(\int_0^s \tr\phi(t,s) dt\biggr)ds.
 \nonumber
\end{align}

Recalling the definition of $s(\kappa)$, we first compute 
\[
    \int_0^T \biggl\langle \int_0^t 
      \kappa_\phi(t,s) d\theta(s), d\theta(t)
      \biggr\rangle.
\]
By It\^o's formula, it holds that
\begin{align*}
    \int_0^t \kappa_\phi(t,s) d\theta(s)
    & =\int_0^t\biggl(\int_s^T \phi(t,u)du\biggr)
       d\theta(s)
      =\int_0^T \phi(t,u)\biggl(\int_0^u 
        \one_{[0,t)}(s)d\theta(s)\biggr) du
    \\
    & =\int_0^T \phi(t,u)\theta(t\wedge u) du.
\end{align*}
This implies that
\begin{align}
 \label{t.transf.cm.prime.22}
    & \int_0^T \biggl\langle \int_0^t 
        \kappa_\phi(t,s) d\theta(s), d\theta(t)
        \biggr\rangle
       =\int_0^T \biggr\langle \int_0^T 
        \phi(t,u)\theta(t\wedge u) du,d\theta(t)
        \biggr\rangle
    \\
    &
      =\int_0^T \biggl(\int_0^T \langle 
        \phi(t,u)\theta(t\wedge u),d\theta(t)\rangle
        \biggr)du
 \nonumber
    \\
    & =\int_0^T \biggl(\int_0^u \langle \theta(t),
        \phi(t,u)^\dagger d\theta(t)\rangle
        \biggr)du
       +\int_0^T \biggl\langle \theta(u),
         \int_u^T \phi(t,u)^\dagger d\theta(t)
         \biggr\rangle du,
 \nonumber
\end{align}
where, to see the second equality, we have applied the
Fubini theorem to the double integral with respect to $du$
and $d\theta(t)$.

We next compute 
\[
    \int_0^T \biggl\langle \int_0^t 
      \kappa_\phi(s,t)^\dagger d\theta(s), d\theta(t)
      \biggr\rangle.
\]
Put 
\[
    M(t)=\int_0^t \kappa_\phi(s,t)^\dagger d\theta(s)
        =\int_0^t \biggl(\int_t^T \phi(s,u)^\dagger du
         \biggr)d\theta(s)
    \quad\text{for }t\in[0,T].
\]
$\{M(t)\}_{t\in[0,T]}$ is an It\^o process with stochastic
differential
\[
    dM(t)=\biggl(\int_t^T \phi(t,u)^\dagger du\biggr)
          d\theta(t)
          -\biggl(\int_0^t \phi(s,t)^\dagger d\theta(s)
           \biggr)dt.
\]
By It\^o's formula, we see that
\begin{align*}
    \langle M(r),\theta(r)\rangle
    =& \int_0^r \langle M(t),d\theta(t)\rangle
     +\int_0^r \biggl\langle \theta(t),
        \biggl(\int_t^T\phi(t,u)^\dagger du\biggr)
        d\theta(t)\biggr\rangle
    \\
     & -\int_0^r \biggl\langle \theta(t),
        \int_0^t \phi(s,t)^\dagger d\theta(s)
        \biggr\rangle dt
       +\int_0^r\biggl(\int_t^T \tr\phi(t,u) du\biggr)dt
\end{align*}
for $r\in[0,T]$.
Substituting $r=T$ and noting that $M(T)=0$, we obtain that 
\begin{align*}
    & \int_0^T \biggl\langle \int_0^t 
        \kappa_\phi(s,t)^\dagger d\theta(s), d\theta(t)
        \biggr\rangle
    \\
    & =-\int_0^T \biggl\langle \theta(t),
           \biggl(\int_t^T \phi(t,u)^\dagger du\biggr)
           d\theta(t)\biggr\rangle
       +\int_0^T \biggl\langle \theta(t),
           \int_0^t \phi(s,t)^\dagger d\theta(s)
           \biggr\rangle dt
    \\
    &\quad
       -\int_0^T\biggl(\int_t^T \tr\phi(t,u) du
         \biggr)dt
    \\
    & =-\int_0^T \biggl(\int_0^u \langle\theta(t),
         \phi(t,u)^\dagger d\theta(t)\rangle\biggr)du
       +\int_0^T \biggl\langle \theta(t),
           \int_0^t \phi(s,t)^\dagger d\theta(s)
           \biggr\rangle dt
    \\
    &\quad
       -\int_0^T\biggl(\int_0^u \tr\phi(t,u) dt
         \biggr)du,
\end{align*}
where to see the last equality, we have applied the Fubini
theorem to the double integral with respect to $du$ and
$d\theta(t)$ in the first term and switched the order of
integration in the third term.
Combining this with \eqref{t.transf.cm.prime.22}, we
obtain \eqref{t.transf.cm.prime.21}.  
\end{proof}

\chapter{Eigenvalue expansions}
\label{chap.lap.another}
\setcounter{section}{1}

Let $\eta\in\stwo$.
Since $B_\eta\in \mathcal{S}(\htwo)$
(Lemma~\ref{l.b.kappa}), there is an ONB 
$\{h_n\}_{n=1}^\infty$ such that 
\[
    \langle B_\eta h_n,h_m\rangle_{\mathcal{H}}
    =\langle B_\eta h_n,h_n\rangle_{\mathcal{H}}\delta_{nm}
    \quad\text{for }n,m\in \mathbb{N}.
\]
By Theorem~\ref{t.q.eta}, $\q_\eta$ has the series expansion
in $L^p(\mu)$ for any $p\in(1,\infty)$ as
\[
    \q_\eta=\frac12\sum_{n=1}^\infty
     \langle B_\eta h_n,h_n\rangle_{\mathcal{H}}
     \{(\D^*h_n)^2-1\}.
\]
With the help of this expansion, we can evaluate the Laplace
transformation of $\q_\eta$ as given in \eqref{t.tways.3} without
using the square root  transformation $\iota+F_{\kappa_S(\eta)}$,
where $\kappa_S(\eta)\in\stwo$ is defined as in
Lemma~\ref{l.tw.(ii)to(iii)}.
In this chapter, we shall present this different proof. 

\begin{lemma}\label{l.eta.hat}
Let $N\in \mathbb{N}$.
For orthonormal $h_1,\dots,h_N\in \mathcal{H}$ and
$a_1,\dots,a_N\in \mathbb{R}$ with 
$\max\limits_{1\le n\le N} a_n<1$, define $\eta\in\stwo$ by 
\begin{equation}\label{l.eta.hat.1}
    \eta(t,s)=\sum_{n=1}^N a_n 
    [h_n^\prime(s)\otimes h_n^\prime(t)]
    \quad\text{for }(t,s)\in[0,T]^2.
\end{equation}
Then $\dettwo(I+B_{\kappa_S(\eta)})\ne0$ and 
\[
    B_{\widehat{\kappa_S(\eta)}}=\sum_{n=1}^N
    \Bigl(\frac1{\sqrt{1-a_n}}-1\Bigr) h_n\otimes h_n.
\]
\end{lemma}

\begin{proof}
Note that
\[
    B_\eta=\sum_{n=1}^N a_n h_n\otimes h_n.
\]
Hence 
$\Lambda(B_\eta)=\max\limits_{1\le n\le N} a_n<1$.
Since $\eta(\kappa_S(\eta))\eqltwo\eta$
(Theorem~\ref{t.tways}), due to
\eqref{r.transf.2}, $\dettwo(I+B_{\kappa_S(\eta)})\ne0$.

Put
\[
    B=\sum_{n=1}^N \Bigl(\frac1{\sqrt{1-a_n}}-1\Bigr)
    h_n\otimes h_n. 
\]
Then 
\[
    \langle(I+B)h,g\rangle_{\mathcal{H}}
    =\sum_{n=1}^N \frac1{\sqrt{1-a_n}}
     \langle h_n,h\rangle_{\mathcal{H}}
     \langle h_n,g\rangle_{\mathcal{H}}
    \quad\text{for }h,g\in \mathcal{H}.
\]
Hence $I+B\in \mathcal{S}_+(\mathcal{H})$ and 
$(I+B)^2(I-B_\eta)=I$.
By the uniqueness of square root (Lemma~\ref{l.sq.root}), we
see that $I+B=C_\eta^{-1}$.
Recalling the definition of $\kappa_S(\eta)$ that
$C_\eta-I=B_{\kappa_S(\eta)}$, we obtain that 
\[
    B_{\widehat{\kappa_S(\eta)}}
    =(I+B_{\kappa_S(\eta)})^{-1}-I
    =B.
\qedhere
\]
\end{proof}

\begin{lemma}\label{l.w*.finite}
Let $N\in \mathbb{N}$ and 
$\ell_1,\dots,\ell_N\in \mathcal{W}^*$ be orthonormal in
$\mathcal{H}$. 
For $a_1,\dots,a_N\in \mathbb{R}$ with 
$\max\limits_{1\le n\le N} a_n<1$, define $\eta\in\stwo$ by
\eqref{l.eta.hat.1} with $h_n=\ell_n$ for $1\le n\le N$.
Then \eqref{t.tways.3} with $h=0$ holds for every
$f\in C_b(\mathcal{W})$.
\end{lemma}

\begin{proof}
For $a\in \mathbb{R}$ with $a<1$, it holds that
\[
    \int_{\mathbb{R}} \varphi(x) e^{\frac12 a(x^2-1)}
      \frac1{\sqrt{2\pi}} e^{-\frac12 x^2} dx
    =\frac{e^{-\frac12 a}}{\sqrt{1-a}} \int_{\mathbb{R}}
      \varphi\biggl(y+
        \Bigl(\frac1{\sqrt{1-a}}-1\Bigr)y\biggr)
      \frac1{\sqrt{2\pi}} e^{-\frac12 y^2} dy
\]
for any $\varphi\in C_b(\mathbb{R})$.
Since $B_\eta=\sum\limits_{n=1}^N a_n \ell_n\otimes \ell_n$
and $\{\D^*\ell_n\mid 1\le n\le N\}$ is a family of
IID random variables obeying the normal distribution
$N(0,1)$, this yields that 
\begin{align*}
    & \int_{\mathcal{W}} \prod_{n=1}^N \Bigl\{
      \varphi_n(\D^*\ell_n)
      e^{\frac12 a_n\{(\D^*\ell_n)^2-1\}}\Bigr\} d\mu
    \\
    & =\{\dettwo(I-B_\eta)\}^{-\frac12} \int_{\mathcal{W}}
      \prod_{n=1}^N \biggl\{\varphi_n\Bigl(\D^*\ell_n
      +\Bigl(\frac1{\sqrt{1-a_n}-1}\Bigr)\D^*\ell_n
      \Bigr)\biggr\} d\mu
\end{align*}
for any $\varphi_1,\dots,\varphi_N\in C_b(\mathbb{R})$.
By Lemmas~\ref{l.eta.hat} and \ref{l.f.kappa}, it holds that 
\[
    \ell_n(F_{\widehat{\kappa_S(\eta)}})
     =\Bigl(\frac1{\sqrt{1-a_n}}-1\Bigr)\D^*\ell_n
    \quad\text{for }1\le n\le N.
\]
Recalling that $\D^*\ell_n=\ell_n$ for $1\le n\le N$
(see \cite[(5.1.5)]{mt-cambridge}), by
Theorem~\ref{t.q.eta}, we see that   
\[
    \int_{\mathcal{W}}\biggl[\prod_{n=1}^N 
        (\varphi_n\circ\ell_n)\biggr] e^{\q_\eta}d\mu
    =\{\dettwo(I-B_\eta)\}^{-\frac12}
      \int_{\mathcal{W}}\biggl[\prod_{n=1}^N 
        (\varphi_n\circ\ell_n)\biggr] 
      \bigl(\iota+F_{\widehat{\kappa_S(\eta)}}\bigr) d\mu.
\]
On account of the splitting property of the Wiener measure, 
from this it is easy to conclude that \eqref{t.tways.3}
with $h=0$ holds for every $f\in C_b(\mathcal{W})$.
\end{proof}

\begin{lemma}\label{l.w*.approx}
Let $N\in \mathbb{N}$ and 
$h_1,\dots,h_N\in \mathcal{H}$ be orthonormal.
For every $\varepsilon>0$, there exist
$\ell_1,\dots,\ell_N\in \mathcal{W}^*$, which are
orthonormal in $\mathcal{H}$ and satisfy that
$\sum\limits_{n=0}^N \|h_n-\ell_n\|_{\mathcal{H}}
 <\varepsilon$.
\end{lemma}

\begin{proof}
We show the assertion by induction for $N\in \mathbb{N}$. 
When $N=1$, the assertion follows from the denseness of
$\mathcal{W}^*$ in $\mathcal{H}$. 

Suppose that the assertion holds for $N$, and let
$h_1,\dots,h_{N+1}\in \mathcal{H}$ be orthonormal.
Fix an $\varepsilon>0$ and take a $\delta>0$ such that
$\delta<\frac13\wedge\frac{\varepsilon}7$.
By the hypothesis of induction, there exist 
$\ell_1,\dots,\ell_N\in \mathcal{W}^*$, which are
orthonormal in $\mathcal{H}$ and satisfy that
$\sum\limits_{n=1}^N \|h_n-\ell_n\|_{\mathcal{H}}<\delta$.
Denote by $\pi_N$ the orthogonal projection of $\mathcal{H}$
onto the subspace spanned by $\ell_1,\dots,\ell_N$.
It holds that
\[
    \|\pi_Nh_{N+1}\|_{\mathcal{H}}
    =\biggl\|\sum_{n=1}^N \langle h_{N+1},\ell_n
       \rangle_{\mathcal{H}}\ell_n\biggr\|_{\mathcal{H}}
    =\biggl\|\sum_{n=1}^N \langle h_{N+1},\ell_n-h_n
       \rangle_{\mathcal{H}}\ell_n\biggr\|_{\mathcal{H}}
    \le \sum_{n=1}^N \|h_n-\ell_n\|_{\mathcal{H}}<\delta.
\]

Choose an $\ell\in \mathcal{W}^*$ such that
$\|h_{N+1}-\ell\|_{\mathcal{H}}<\delta$.
We then have that
\[
    \|\pi_N \ell\|_{\mathcal{H}}
    \le \|\pi_N (\ell-h_{N+1})\|_{\mathcal{H}}
        +\|\pi_N h_{N+1}\|_{\mathcal{H}}<2\delta.
\]
This implies that
\begin{equation}\label{l.w*.approx.21}    
    \|(I-\pi_N)\ell\|_{\mathcal{H}}
    \ge \|\ell\|_{\mathcal{H}}-2\delta
    >\|h_{N+1}\|_{\mathcal{H}}-3\delta
    =1-3\delta>0.
\end{equation}

Define $\ell_{N+1}\in \mathcal{W}^*$ by
\[
    \ell_{N+1}=\frac1{\|(I-\pi_N)\ell\|_{\mathcal{H}}}
          (I-\pi_N)\ell.
\]
Then $\ell_1,\dots,\ell_{N+1}$ are orthonormal in
$\mathcal{H}$.
Note that
\[
    \|(I-\pi_N)\ell\|_{\mathcal{H}}
    \le \|\ell\|_{\mathcal{H}}
    \le \|h_{N+1}\|_{\mathcal{H}}+\delta=1+\delta.
\]
In conjunction with \eqref{l.w*.approx.21}, this yields that
\[
    \|(I-\pi_N)\ell-\ell_{N+1}\|_{\mathcal{H}}
     =\bigl|\|(I-\pi_N)\ell\|_{\mathcal{H}}-1\bigr|<3\delta.
\]

It holds that
\[
    \|h_{N+1}-\ell_{N+1}\|_{\mathcal{H}}
    \le \|h_{N+1}-\ell\|_{\mathcal{H}}
        +\|\pi_N \ell\|_{\mathcal{H}}
        +\|(I-\pi_N)\ell-\ell_{N+1}\|_{\mathcal{H}}
    <6\delta.
\]
Combined with the hypothesis of induction, this implies
that 
\[
    \sum_{n=1}^{N+1}\|h_n-\ell_n\|_{\mathcal{H}}
    \le \delta+6\delta=7\delta<\varepsilon.
\]
Thus the assertion holds for $N+1$.
\end{proof}

\begin{lemma}\label{l.h.finite}
Let $N\in \mathbb{N}$.
For orthonormal 
$h_1,\dots,h_N\in \mathcal{H}$ and 
$a_1,\dots,a_N\in\mathbb{R}$ with 
$\max\limits_{1\le n\le N} a_n<1$, define $\eta\in\stwo$ by
\eqref{l.eta.hat.1}.
Then \eqref{t.tways.3} with $h=0$ holds for any 
$f\in C_b(\mathcal{W})$.
\end{lemma}

\begin{proof}
By Lemma~\ref{l.w*.approx}, there are
$\ell_1^{(m)},\dots,\ell_N^{(m)}$, $m\in \mathbb{N}$, such
that $\ell_1^{(m)},\dots,\ell_N^{(m)}$ are
orthonormal in $\mathcal{H}$ for each $m\in \mathbb{N}$ and  
$\sum\limits_{n=1}^N \|h_n-\ell_n^{(m)}\|_{\mathcal{H}}\to0$ as
$m\to\infty$.
For $m\in \mathbb{N}$, define $\eta^{(m)}\in\stwo$ by
\[
    \eta^{(m)}(t,s)=\sum_{n=1}^N a_n
    [(\ell_n^{(m)})^\prime(s)\otimes 
     (\ell_n^{(m)})^\prime(t)]
    \quad\text{for }(t,s)\in[0,T]^2.
\]
Obviously $\|B_{\eta^{(m)}}-B_\eta\|_{\htwo}\to0$ as
$m\to\infty$.
Furthermore, by Lemma~\ref{l.eta.hat}, it holds that
\[
    \|B_{\widehat{\kappa_S(\eta^{(m)})}}
     -B_{\widehat{\kappa_S(\eta)}}\|_{\htwo}\to0.
\]
Due to Theorem~\ref{t.q.eta} and Lemma~\ref{l.f.kappa}, we
have that
\[
    \q_{\eta^{(m)}}\to\q_\eta\text{ in }L^p(\mu)
    \quad\text{and}\quad
    F_{\widehat{\kappa_S(\eta^{(m)})}}\to
    F_{\widehat{\kappa_S(\eta)}}
   \text{ in }L^p(\mu;\mathcal{H})
   \quad\text{as }m\to\infty
\]
for any $p\in(1,\infty)$.

Since $\max\limits_{1\le n\le N} a_n<1$ and 
\[
    \dettwo(I-B_{\eta^{(m)}})
    =\prod_{n=1}^N (1-a_n)e^{a_n}
    =\dettwo(I-B_\eta)>0
    \quad\text{for }m\in \mathbb{N},
\]
applying Lemma~\ref{l.w*.finite}, we have that
\begin{equation}\label{l.h.finite.21}
    \int_{\mathcal{W}} f e^{\q_{\eta^{(m)}}} d\mu
    =\{\dettwo(I-B_\eta)\}^{-\frac12}
     \int_{\mathcal{W}} 
      f(\iota+F_{\widehat{\kappa_S(\eta^{(m)})}})
     d\mu
\end{equation}
for every $f\in C_b(\mathcal{W})$ and $m\in \mathbb{N}$.

Let $M=\max\limits_{1\le n\le N}a_n$ and take a $p\in(1,\infty)$ such
that $pM<1$. 
Since 
$\Lambda(B_{p\eta^{(m)}})=p\Lambda(B_{\eta^{(m)}})
 =p(0\vee M)$
and $\|\eta^{(m)}\|_2=\|\eta\|_2$ for every $m\in \mathbb{N}$,
by Lemma~\ref{l.q.eta.int}, we see that 
\[
    \sup_{m\in \mathbb{N}}
    \int_{\mathcal{W}} e^{p\q_{\eta^{(m)}}} d\mu
    \le \exp\biggl(\frac12\biggl\{\frac12
       +\frac{p(0\vee M)}{
               3\{1-p(0\vee M)\}^3}\biggr\}
       p^2\|\eta\|_2^2\biggr).
\]
Thus the family $\{\exp(\q_{\eta^{(m)}})\mid m\in \mathbb{N}\}$
is uniformly integrable.
Letting $m\to\infty$ in \eqref{l.h.finite.21}, we arrive 
at \eqref{t.tways.3} with $h=0$. 
\end{proof}

\begin{lemma}\label{l.eta.hat.norm}
Let $\rho,\eta\in\stwo$.
Assume that $\Lambda(B_\rho)\vee \Lambda(B_\eta)<1$ and 
$B_\rho B_\eta=B_\eta B_\rho$.
Set $\Lambda^\prime(B)=1-\Lambda(B)$ for 
$B\in \mathcal{S}(\htwo)$.
Then it holds that
\[
    \|B_{\widehat{\kappa_S(\rho)}}
       -B_{\widehat{\kappa_S(\eta)}}\|_{\htwo}
    \le \frac{\|\rho-\eta\|_2}{
         \sqrt{\Lambda^\prime(B_\rho)}
         \sqrt{\Lambda^\prime(B_\eta)}
          (\sqrt{\Lambda^\prime(B_\rho)}
           +\sqrt{\Lambda^\prime(B_\eta)})}.
\]
\end{lemma}

\begin{proof}
We first show that 
\begin{equation}\label{l.eta.hat.norm.21}
    \inf_{\|h\|_{\mathcal{H}}=1}
    \langle C_\rho h,h\rangle_{\mathcal{H}}
    \ge\sqrt{\Lambda^\prime(B_\rho)}.
\end{equation}
Since $B_\rho\in \mathcal{S}(\htwo)$
(Lemma~\ref{l.b.kappa}), there is  an ONB
$\{h_n\}_{n=1}^\infty$ of $\mathcal{H}$ such that  
\[
    B_\rho=\sum_{n=1}^\infty a_n h_n\otimes h_n,
\]
where $a_n\in \mathbb{R}$ for $n\in \mathbb{N}$ and 
$\sum\limits_{n=1}^\infty a_n^2<\infty$.
Recalling that 
$\sup\limits_{n\in \mathbb{N}} a_n=\Lambda(B_\rho)<1$, 
we define the linear operator $C:\mathcal{H}\to \mathcal{H}$
by 
\[
    Ch=\sum_{n=1}^\infty \sqrt{1-a_n} 
       \langle h,h_n\rangle_{\mathcal{H}} h_n
    \quad\text{for }h\in \mathcal{H}.
\]
Then it holds that
\[
    \|Ch\|_{\mathcal{H}}^2
    =\sum_{n=1}^\infty (1-a_n)
       \langle h,h_n\rangle_{\mathcal{H}}^2
    \le \|h\|_{\mathcal{H}}^2
\]
and
\[
    \langle Ch,g\rangle_{\mathcal{H}}
    =\sum_{n=1}^\infty \sqrt{1-a_n} 
       \langle h,h_n\rangle_{\mathcal{H}} 
       \langle g,h_n\rangle_{\mathcal{H}} 
    =\langle h,Cg\rangle_{\mathcal{H}}
    \quad\text{for }h,g\in \mathcal{H}.
\]
This also implies that
\begin{equation}\label{l.eta.hat.norm.22}
    \langle Ch,h\rangle_{\mathcal{H}}
    \ge \sqrt{\inf_{n\in \mathbb{N}}(1-a_n)} 
         \|h\|_{\mathcal{H}}^2
    =\sqrt{\Lambda^\prime(B_\rho)} \|h\|_{\mathcal{H}}^2
    \quad\text{for }h\in \mathcal{H}.
\end{equation}
Thus $C\in \mathcal{S}_+(\mathcal{H})$.
Moreover, we see that 
\[
    C^2h=\sum_{n=1}^\infty (1-a_n)
       \langle h,h_n\rangle_{\mathcal{H}} h_n
    =(I-B_\eta)h
    \quad\text{for }h\in \mathcal{H}.
\]
Due to the uniqueness of square root (Lemma~\ref{l.sq.root}),
we see that $C=C_\eta$.
The inequality \eqref{l.eta.hat.norm.22} is nothing but the
desired inequality \eqref{l.eta.hat.norm.21}.

It follows from \eqref{l.eta.hat.norm.21} that
\[
    \inf_{\|h\|_{\mathcal{H}}=1} \|C_\rho h\|
    \ge \inf_{\|h\|_{\mathcal{H}}=1} 
      \langle C_\rho h,h\rangle_{\mathcal{H}}
    \ge \sqrt{\Lambda^\prime(B_\rho)}.
\]
Due to Lemma~\ref{l.cont.inv}, we see that $C_\rho$ has a 
continuous inverse and 
\begin{equation}\label{l.eta.hat.norm.23}
    \|C_\rho^{-1}\|_\op
    \le \frac1{\sqrt{\Lambda^\prime(B_\rho)}}.
\end{equation}

By \eqref{l.eta.hat.norm.21}, we obtain that
\[
    \|(C_\rho+C_\eta)h\|_{\mathcal{H}}
    \ge \langle (C_\rho+C_\eta)h,h\rangle_{\mathcal{H}}
    \ge \sqrt{\Lambda^\prime(B_\rho)}
        +\sqrt{\Lambda^\prime(B_\eta)}
\]
for $h\in \mathcal{H}$  with $\|h\|_{\mathcal{H}}=1$.
By Lemma~\ref{l.cont.inv}, this implies that
\begin{equation}\label{l.eta.hat.norm.24}
    \|(C_\rho+C_\eta)^{-1}\|_\op
    \le \frac1{\sqrt{\Lambda^\prime(B_\rho)}
           +\sqrt{\Lambda^\prime(B_\eta)}}.
\end{equation}

Using the commutativity of $B_\rho$, $B_\eta$, $C_\rho$,
$C_\eta$, $C_\rho^{-1}$, and $C_\eta^{-1}$, we rewrite as
\[
    B_{\widehat{\kappa_S(\rho)}}
       -B_{\widehat{\kappa_S(\eta)}}
    =C_\rho^{-1}-C_\eta^{-1}
    =C_\rho^{-1} C_\eta^{-1}
    (C_\rho+C_\eta)^{-1}(B_\rho-B_\eta).
\]
Plugging the estimations \eqref{l.eta.hat.norm.23} for
$\rho$ and $\eta$ and \eqref{l.eta.hat.norm.24} into this,
we arrive at the desired inequality.
\end{proof}

\begin{proof}[Another proof of \eqref{t.tways.3}]
It suffices to show \eqref{t.tways.3} with $h=0$.
In fact, if $\ell\in \mathcal{W}^*$, then $\D^*\ell=\ell$. 
Substituting $f_N=(N\wedge e^{\ell})f$ into
\eqref{t.tways.3} and letting $N\to\infty$, we obtain
\eqref{t.tways.3} with $\ell\in \mathcal{W}^*$ for $h$. 
Due to the denseness of $\mathcal{W}^*$ in $\mathcal{H}$,
the identity extend to $\mathcal{H}$.

Let $\eta\in\stwo$ and assume that $\Lambda(B_\eta)<1$.
Take an ONB $\{h_n\}_{n=1}^\infty$ of $\mathcal{H}$ such
that
$B_\eta=\sum\limits_{n=1}^\infty a_n h_n\otimes h_n$, where 
$a_n=\langle B_\eta h_n,h_n\rangle_{\mathcal{H}}$ for 
$n\in \mathbb{N}$. 
For $N\in \mathbb{N}$, define $\eta^{(N)}\in\stwo$ by 
\[
    \eta^{(N)}(t,s)=\sum_{n=1}^N a_n
     [h_n^\prime(s)\otimes h_n^\prime(t)]
    \quad\text{for }(t,s)\in[0,T]^2.
\]
Then $\|B_{\eta^{(N)}}-B_\eta\|_{\htwo}\to0$ as $N\to\infty$
and 
$B_{\eta^{(N)}} B_\eta=B_\eta B_{\eta^{(N)}}$ for
$N\in \mathbb{N}$. 
By Theorem~\ref{t.q.eta} and Lemmas~\ref{l.eta.hat.norm} and
\ref{l.f.kappa}, 
$\q_{\eta^{(N)}}\to\q_\eta$ in $L^p(\mu)$
and 
$F_{\widehat{\kappa_S(\eta^{(N)})}}\to
 F_{\widehat{\kappa_S(\eta)}}$ in $L^p(\mu;\mathcal{H})$ as
$N\to\infty$ for any $p\in(1,\infty)$.

Since 
$\max\limits_{1\le n\le N} a_n \le \Lambda(B_\eta)<1$, 
by Lemma~\ref{l.h.finite},  we have that 
\begin{equation}\label{t.surj.1st.101}
    \int_{\mathcal{W}} f e^{\q_{\eta^{(N)}}} d\mu
    =\{\dettwo(I-B_{\eta^{(N)}})\}^{-\frac12}
     \int_{\mathcal{W}} f(\iota+F_{\widehat{\kappa_S(\eta^{(N)})}})
     d\mu
\end{equation}
for every $f\in C_b(\mathcal{W})$ and $N\in \mathbb{N}$. 

Take a $p\in(1,\infty)$ such that $p\Lambda(B_\eta)<1$.
By Lemma~\ref{l.q.eta.int} and the dominations that
$\Lambda(B_\eta^{(N)})\le \Lambda(B_\eta)$ and
$\|\eta^{(N)}\|_2\le \|\eta\|_2$, we have that
\[
    \sup_{N\in \mathbb{N}}
       \int_{\mathcal{W}} e^{p\q_{\eta^{(N)}}} d\mu
    \le \exp\biggl(\frac12\biggl\{\frac12
        +\frac{p\Lambda(B_\eta)}{
           3(1-p\Lambda(B_\eta))^3}
     \biggr\}p^2\|\eta\|_2^2\biggr).
\]
Hence the family $\{\exp(\q_{\eta^{(N)}})\mid N\in \mathbb{N}\}$ 
is uniformly integrable.
Letting $N\to\infty$ in \eqref{t.surj.1st.101}, we arrive at  
\eqref{t.tways.3} with $h=0$.
\end{proof}

\chapter*{Notations}
\label{chap.notations}
\begingroup
\renewcommand{\arraystretch}{2}
\begin{center}
Fundamentals
\\[5pt]
\begin{tabular}{|l|l|}
 \hline 
 $\mathbb{N}$ & natural numbers 
 \\ \hline
 $\mathbb{R}$ & real numbers
 \\ \hline
 $\mathbb{C}$ & complex numbers
 \\ \hline
 $\mathbb{R}^n$ & $n$-dimensional real Euclidean space
 \\ \hline
 $\langle \cdot,\cdot\rangle$ & 
 the inner product in $\mathbb{R}^n$
 \\ \hline
 $|x|$ &  the Euclidean norm of $x\in\mathbb{R}^n$
 \\ \hline 
 $\mathbb{R}^{d\times d}$ & $d\times d$ real matrices
 \\ \hline 
 $\mathbb{C}^{d\times d}$ &  $d\times d$ complex matrices
 \\ \hline 
 $|M|$ & the Euclidean norm of $M\in \mathbb{R}^{d\times d}$
 \\ \hline 
 $M^\dagger$ & the transpose of 
              $M\in \mathbb{C}^{d\times d}$
 \\ \hline 
 $\one_S$ & the indicator function of $S$
 \\ \hline 
\end{tabular}
\end{center}

\begin{center}
Wiener and Cameron-Martin spaces
\\[5pt]
\begin{tabular}{|l|l|}
\hline
$\mathcal{W}$ & 
$\mathbb{R}^d$-valued continuous functions on $[0,T]$
vanishing at $0$
\\
\hline
$\mu$ & the Wiener measure on $\mathcal{W}$
\\
\hline
$\mathcal{H}$ & 
\begin{minipage}[t]{300pt}
the space of absolutely continuous
$h\in \mathcal{W}$ possessing square integrable derivative
$h^\prime$ on $[0,T]$.
\end{minipage}
\\ \hline
$\langle h,g\rangle_{\mathcal{H}}$ &
$\displaystyle
  \int_0^T \langle h^\prime(t),g^\prime(t)
           \rangle dt$
\\ \hline
$\|h\|_{\mathcal{H}}$ & 
$\langle h,h\rangle_{\mathcal{H}}^{\frac12}$
\\
\hline
\end{tabular}
\end{center}

\begin{center}
Function spaces 
\\[5pt]
\begin{tabular}{|l|l|}
\hline
$L^p(\mu;E)$ & $E$-valued $p$th integrable functions with
respect to $\mu$
\\ \hline
$L^p(\mu)$ & the abbreviation for $L^p(\mu;\mathbb{R})$
\\ \hline
$L^{1+}(\mu)$ & $\bigcup\limits_{p\in(1,\infty)} L^p(\mu)$
\\ \hline
$L^p([a,b];E)$ & $E$-valued $p$th integrable functions on
$[a,b]$ 
\\ \hline
$C_b(X)$ &  continuous and bounded real functions on $X$
\\ \hline
$C([0,T];\mathbb{R}^{d\times d})$ & 
$\mathbb{R}^{d\times d}$-valued continuous functions on 
$[0,T]$
\\ \hline
$C^k([0,T];\mathbb{R}^{d\times d})$ &
$\mathbb{R}^{d\times d}$-valued $k$-times continuously
differentiable functions on $[0,T]$
\\ \hline
$\mathscr{S}(\mathbb{R}^N)$ & rapidly decreasing functions
on $\mathbb{R}^N$ 
\\ \hline
$\mathscr{S}^\prime(\mathbb{R}^N)$ & 
tempered distributions on $\mathbb{R}^N$
\\ \hline
\end{tabular}
\end{center}

\begin{center}
Malliavin calculus
\\[5pt]
\begin{tabular}{|l|l|}
\hline
 $\D$ & the $\mathcal{H}$-derivative 
 \\ \hline
 $\D^*$ &  the adjoint of $\D$
 \\ \hline
 $\mathbb{D}^{k,p}(E)$ &
 \begin{minipage}[t]{300pt}
 $k$-times $\mathcal{H}$-differentiable $E$-valued Wiener
 functionals with $p$th integrable $\mathcal{H}$-derivatives
 of all orders 
 \end{minipage}
 \\ \hline
 $\mathbb{D}^\infty(E)$ &
 $\bigcap\limits_{k\in \mathbb{N}\cup\{0\}}
   \bigcap\limits_{p\in(1,\infty)} \mathbb{D}^{k,p}(E)$
 \\ \hline
 $\mathbb{D}^{k,p}$ and $\mathbb{D}^\infty$ & 
 the abbreviations for $\mathbb{D}^{k,p}(\mathbb{R})$ and
 $\mathbb{D}^\infty(\mathbb{R})$, respectively
 \\ \hline
 $\delta_z(\Phi)d\mu$ & 
 the pinned measure at $z$ by $\Phi$
 \\ \hline
\end{tabular}
\end{center}

\newpage
\begin{center}
Hilbert-Schmidt Operators
\\[5pt]
\begin{tabular}{|l|l|}
\hline
$\htwo$ &
Hilbert-Schmidt operator of 
$\mathcal{H}$ to $\mathcal{H}$
\\
\hline
$\mathcal{S}(\htwo)$ &
self-adjoint Hilbert-Schmidt operator of 
$\mathcal{H}$ to $\mathcal{H}$
\\
\hline
$\mathcal{S}_+(\mathcal{H})$ &
self-adjoint, continuous, and non-negative definite 
operators of $\mathcal{H}$ to $\mathcal{H}$
\\
\hline
$\Lambda(B)$ &
$\sup\limits_{\|h\|_{\mathcal{H}}=1}
  \langle Bh,h\rangle_{\mathcal{H}}$
~($B\in \mathcal{S}(\htwo)$)
\\
\hline
$h\otimes g\in \htwo$ & 
$(h\otimes g)k=\langle h,k\rangle g$
~($k\in \mathcal{H}$)
\\ \hline
$B_\kappa\in \htwo$ &
$\displaystyle
   \langle B_\kappa h,g\rangle_{\mathcal{H}}
    =\int_0^T \biggl\langle \int_0^T 
       \kappa(t,s)h^\prime(s)ds,g^\prime(t)\biggr\rangle
       dt$
~($h,g\in \mathcal{H}$)
\\ \hline
\end{tabular}
\end{center}

\begin{center}
Kernels
\\[5pt]
\begin{tabular}{|l|l|}
\hline
$\ltwo$ & 
$\mathbb{R}^{d\times d}$-valued square integrable functions
on $[0,T]^2$
\\
\hline
$\stwo$ & the totality of 
$\eta\in\ltwo$ with 
$\eta(t,s)^\dagger=\eta(s,t)$
for $(t,s)\in[0,T]^2$
\\
\hline
$\|\kappa\|_2$ &
$\displaystyle
  \biggl(\int_0^T \int_0^T |\kappa(t,s)|^2 ds dt
     \biggr)^{\frac12}$
\\
\hline
$\kappa\eqltwo\kappa^\prime$ & 
if
$\|\kappa-\kappa^\prime\|_2=0$
\\ \hline
$\eta(\kappa)\in\stwo$ & 
$\displaystyle
 \eta(\kappa)(t,s)=-\biggl\{\kappa(t,s)
  +\kappa(s,t)^\dagger
  +\int_0^T \kappa(u,t)^\dagger \kappa(u,s)du\biggr\}$
\\ \hline
$\widehat{\kappa}\in\ltwo$ &
$B_{\widehat{\kappa}}=(I+B_\kappa)^{-1}-I$
\\ \hline
$\kappa_S(\eta)\in\stwo$ & 
$B_{\kappa_S(\eta)}=C_\eta-I$ with
$C_\eta\in \mathcal{S}_+(\mathcal{H})$ such that
$C_\eta^2=I-B_\eta$
\\ \hline
$\kappa_A(\rho)\in\ltwo$ & 
$\kappa_A(\rho)(t,s)=-\one_{[0,t)}(s) \rho(t,s)$
\\ \hline
$\rho_\phi\in\stwo$ &
$\rho_\phi(t,s)=\one_{[0,t)}(s)\phi(t)
  +\one_{(t,T]}(s)\phi(s)^\dagger$
\\ \hline
$\sigma(\phi)\in C([0,T];\mathbb{R}^{d\times d})$ &
$\displaystyle
 \sigma(\phi)(t)=\phi(t)
   -\int_t^T\phi(u)^\dagger\phi(u)du$
\\ \hline
\end{tabular}
\end{center}

\newpage
\begin{center}
Wiener functionals
\\[5pt]
\begin{tabular}{|l|l|}
\hline
$F_\kappa:\mathcal{W}\to \mathcal{H}$ &
$\displaystyle
 \langle F_\kappa,h\rangle_{\mathcal{H}}
    =\int_0^T \biggl\langle 
         \int_0^T \kappa(t,s)d\theta(s),h^\prime(t)
         \biggr\rangle dt$
~($h\in \mathcal{H}$)
\\ \hline
$\q_\eta:\mathcal{W}\to \mathbb{R}$ & 
$\displaystyle
 \int_0^T \biggl\langle \int_0^t
      \eta(t,s)d\theta(s),d\theta(t)
      \biggr\rangle$
\\ \hline
$\mathfrak{h}(\kappa;x)$ &
$\displaystyle
 \frac12\int_0^T \biggl\langle x,
   \int_0^T \kappa(t,s)d\theta(s) \biggr\rangle^2 dt$
\\ \hline
$\mathfrak{h}(\kappa)$ & 
$\displaystyle
 \frac12\int_0^T \biggl|\int_0^T \kappa(t,s)d\theta(s)
  \biggr|^2 dt$
\\ \hline
$\p_\sigma:\mathcal{W}\to \mathbb{R}$ & 
$\displaystyle
 \int_0^T\langle \sigma(t)\theta(t)d\theta(t)\rangle$
\\ \hline
\end{tabular}
\end{center}

\begin{center}
Matrix functions
\\[5pt]
\begin{tabular}{|l|l|}
 \hline 
 $\e[M]$ & 
 $\displaystyle
   \sum_{n=0}^\infty \frac1{n!}M^n$ 
 \\ \hline
 $\ch[M]$ & 
 $\displaystyle
  \sum_{n=0}^\infty \frac1{(2n)!}M^{2n}$
 \\ \hline
 $\sh[M]$ & 
 $\displaystyle
  \sum_{n=0}^\infty \frac1{(2n+1)!}M^{2n}$
 \\ \hline
 $\mathfrak{sh}[M]$ &
 $\displaystyle
  \sum_{n=0}^\infty \frac1{(2n+1)!}M^{2n+1}$
 \\ \hline
 $\tnh[M]$ & $\sh[M](\ch[M])^{-1}$ (if $\det\ch[M]\ne0$)
 \\ \hline
\end{tabular}
\end{center}

\endgroup


\end{document}